%% file: motiviccompletion.tex
\documentclass{article}
\input{preamble.tex}

\title{Unstable $p$-completion in motivic homotopy theory}

\begin{document}

\maketitle
\begin{abstract}
    We define unstable $p$-completion in general $\infty$-topoi and the unstable 
    motivic homotopy category, and prove that the $p$-completion 
    of a nilpotent sheaf or motivic space can be computed on its Postnikov tower.
    We then show that the ($p$-completed) homotopy groups of 
    the $p$-completion of a nilpotent motivic space $X$ fit into short exact sequences 
    $0 \to \mathbb L_0 \pi_n(X) \to \pi_n^p(\complete{X}) \to \mathbb L_1 \pi_{n-1}(X) \to 0$,
    where the $\mathbb L_i$ are (versions of) the derived $p$-completion functors,
    analogous to the classical situation.
\end{abstract}
\tableofcontents
\newpage

\section{Introduction}
In their seminal paper \cite{bousfield2009homotopy}, Bousfield and Kan defined the $p$-completion 
functor on (nilpotent) spaces/anima. This process associates to every nilpotent anima $X$ 
another anima $\complete{X}$, together with a map $X \to \complete{X}$,
which is universal among $\finfld{p}$-equivalences, i.e.\ maps $f \colon X \to Y$ which induce 
isomorphisms on $\finfld{p}$-homology.
Roughly, the $p$-completion functor ``derived $p$-completes the homotopy groups of $X$'',
in the following sense:
Write $L_i \colon \Ab \to \Ab$ for the derived $p$-completion functors on abelian groups,
i.e.\ the composition 
\begin{equation*}
    \Ab \hookrightarrow \Cat D(\Ab) \xrightarrow{\limil{n} (-) \sslash p^n} \Cat D(\Ab) \xrightarrow{H_{i}} \Ab,
\end{equation*}
where the map in the middle is understood to be the derived limit of the cofibers (or cones) of the multiplication-by-$p^n$-maps.
Then, one has the following theorem:
\begin{thm} [Bousfield-Kan]
    Let $X$ be a nilpotent pointed anima (resp.\ a spectrum).
    Then for every $n \ge 1$ (resp.\ any $n \in \Z$) there is a short exact sequence 
    \begin{equation*}
        0 \to L_0 \pi_n(X) \to \pi_n(\complete{X}) \to L_1 \pi_{n-1}(X) \to 0.
    \end{equation*}
\end{thm}

In this paper, we want to show that there is an analogous functor in 
unstable motivic homotopy theory over a perfect field, which behaves similar to the classical situation.

Let $k$ be a perfect field.
Recall that $\MotSpc{k} \subset \PrShv{\smooth{k}}$ is the full subcategory of presheaves of anima on $\smooth{k}$ 
(the category of smooth quasi-compact $k$-schemes)
consisting of those presheaves which are $\AffSpc{1}$-invariant and satisfy Nisnevich descent,
called the category of motivic spaces.
Similarly, write $\SH{k} \subset \PrShvVal{\smooth{k}}{\Sp}$ for the full subcategory of presheaves 
of spectra, consisting of the $\AffSpc{1}$-invariant Nisnevich sheaves, called 
the category of $S^1$-spectra.
There is an adjunction $\pSus \colon \MotSpc{k} \rightleftarrows \SH{k} \colon \Loop$.
We regard $\SH{k}$ as equipped with the homotopy t-structure $\tstruct{\SH{k}}$, with heart $\heart{\SH{k}}$.
Write $\completebr{-} \colon \SH{k} \to \SH{k}$ for the $p$-adic completion functor,
i.e.\ the functor $E \mapsto \limil{k} E \sslash p^k$,
and $\completebr{\SH{k}}$ for its essential image.

Note that in this setting, classical theorems cannot be true on the nose:
For example, one cannot expect that for every $A \in \heart{\SH{k}}$ 
there are short exact sequences 
\begin{align*}
    0 \to L_0 \pi_n(A) \to \pi_n(\complete{A}) \to L_1 \pi_{n-1} (A) \to 0,
\end{align*}
(where the $L_i$ are defined analogously to the case of abelian groups),
since it is unreasonable to expect that $\complete{A}$
has no negative homotopy groups (although the negative homotopy groups 
are always uniquely $p$-divisible,
see \cref{lemma:upd-homotopy-of-completion-below-connectivity-of-base}).
These negative homotopy groups appear since infinite products 
are not t-exact in $\SH{k}$.

Luckily, there is a new t-structure (the $p$-adic t-structure)
one can associate to $\SH{k}$ which solves those problems.
Our main theorem can now be summarized as follows:
\begin{thm} \label{thm:intro:main}
    There is a localization functor $\completebr{-} \colon \MotSpc{k} \to \MotSpc{k}$
    which inverts $p$-equivalences, i.e. morphisms $f \colon X \to Y$ in $\MotSpc{k}$,
    such that $(\pSus f) \sslash p$ is an equivalence.

    One can define the $p$-adic t-structure $\tpstruct{\SH{k}}$ on $\SH{k}$ with heart $\pheart{\SH{k}}$,
    and derived $p$-completion functors
    \begin{align*}
        \mathbb L_i \colon \heart{\SH{k}} &\to \pheart{\SH{k}}, \\
        A &\mapsto \pi_i^p(A) \coloneqq \Omega^i \tau^p_{\le i} \tau^p_{\ge i} A. 
    \end{align*}

    For every $X \in \MotSpc{k}_*$, there is a functorial sequence of $p$-completed homotopy groups 
    $\pi_n^p(X) \in \pheart{\SH{k}}$ for $n \ge 2$.
    There is a simliar construction if $n = 1$.

    These constructions satisfy the following:
    \begin{enumerate} [label= (\arabic*),ref= (\arabic*),itemsep=0em] 
        \item 
            The $p$-adic t-structure is not left-separated. 
            Write 
            \begin{equation*}
                \tpcon[\infty]{\SH{k}} \coloneqq \bigcap_n \tpcon[n]{\SH{k}}.
            \end{equation*}
            Then there is a canonical equivalence 
            \begin{equation*}
                \SH{k} / \tpcon[\infty]{\SH{k}} \cong \completebr{\SH{k}}.
            \end{equation*}
        \item 
            A morphism $f \colon A \to B$ is a $p$-equivalence in $\heart{\SH{k}}$ (i.e.\ $f \sslash p$ is an equivalence)
            if and only if $\mathbb L_i(f)$ is an equivalence for all $i$.

            More generally, a morphism $f \colon E \to F$ in $\SH{k}$ is a $p$-equivalence 
            if and only if $\pi_n^p(f)$ is an equivalence for all $n \in \Z$.
        \item
            An object $E \in \SH{k}$ lives inside the $p$-adic heart $\pheart{\SH{k}}$
            if and only if $E \sslash p \in \tcon{\SH{k}}$, $E \in \tcocon{\SH{k}}$,
            $E \cong \complete{E}$ and $\pi_0(E)$ is of bounded $p$-divisibility (i.e.\ 
            has no map from a $p$-divisible object $A \in \heart{\SH{k}}$).
        \item \label{eq:intro:stable-ses}
            If $E \in \SH{k}$, then there are functorial short exact sequences 
            \begin{equation*} 
                0 \to \mathbb L_0 \pi_n(E) \to \pi_n^p(\complete{E}) \to \mathbb L_1 \pi_{n-1} (E) \to 0.
            \end{equation*}
        \item \label{eq:intro:peq-iso-on-p-complete-htpy}
            If $f \colon X \to Y$ is a $p$-equivalence of pointed nilpotent motivic spaces,
            then $\pi_n^p(f)$ is an isomorphism for all $n \ge 1$.

            The converse holds if moreover $\pi_1(X)$ and $\pi_1(Y)$ are abelian.
        \item \label{eq:intro:unstable-ses}
            Moreover, if $X \in \MotSpc{k}_*$ is a pointed nilpotent motivic space,
            then for every $n \ge 2$ there is a functorial short exact sequence in $\pheart{\SH{k}}$
            \begin{equation*}
                0 \to \mathbb L_0 \pi_n(X) \to \pi_n^p(X) \to \mathbb L_1 \pi_{n-1}(X) \to 0.
            \end{equation*}
            (and there is also a similar sequence for $n = 1$).
    \end{enumerate}
\end{thm}
\begin{proof}
    The $p$-completion functor is constructed in \cref{lemma:topos:completion-functor}.
    The $p$-adic t-structure is defined in \cref{def:t-struct:defn},
    and the derived $p$-completion functors are constructed in \cref{def:t-struct:Li}.
    The definition of the $p$-completed homotopy groups is \cref{def:motivic:completed-htpy}.
    For proofs of the other statements, see:
    \begin{enumerate} [label= (\arabic*),ref= (\arabic*),itemsep=0em]
        \item \cref{rmk:t-struct:choices},
        \item \cref{lemma:t-struct:p-eq-iso-on-htpy-in-t-structure},
        \item \cref{lemma:t-struct:mod-p-in-con-implies-in-tpcon,lemma:t-struct:char-of-cocon},
        \item \cref{lemma:short-exact-sequence-of-t-structure},
        \item \cref{lemma:motivic:p-complete-htpy-iso-if-peq,prop:motivic:peq-of-iso-in-phtpy}, and
        \item \cref{thm:motivic:ses}.
    \end{enumerate}
\end{proof}

\begin{rmk}
    The results about the $p$-adic t-structure are very general:
    One can associate a $p$-adic t-structure with the same properties to any presentable
    stable $\infty$-category which is equipped with a (right-separated) t-structure.
\end{rmk}

The situation is somewhat more complicated than the classical situation, for the following reasons:
First, as already remarked above, if we have an $S^1$-spectrum $A \in \heart{\SH{k}}$,
then (contrary to the classical situation) $\complete{A}$ is no longer concentrated in degrees $0$ and $1$,
since there are no connectivity bounds on sequential limits of connective $S^1$-spectra.
Nonetheless, we can fix this problem by introducing the $p$-adic t-structure and the derived $p$-completion functors $\mathbb L_i$.
In this t-structure, the $p$-completion $\complete{A}$ is concentrated in degrees $0$ and $1$.
It follows that the derived $p$-completion functors vanish for all $i \neq 0, 1$,
see \cref{lemma:t-struct:Li-zero-if-i-neq-zero-one}.

In particular, the $p$-adic heart $\pheart{\SH{k}}$ does not live inside the standard heart $\heart{\SH{k}}$.
Therefore, in order for the short exact sequence \ref{eq:intro:unstable-ses} to make sense, 
we cannot use the homotopy groups $\pi_n(\complete{X})$,
but need a more elaborate construction.

Note that in the classical situation, our constructions give the same results as before,
because here the heart of the $p$-adic t-structure on $\Sp$ (the $\infty$-category of spectra) actually lives inside the normal heart,
and the (new) derived $p$-completion functors $\mathbb L_i$ agree with the classical derived $p$-completion 
functors $L_i$. A proof of this fact can be found in \cref{lemma:anima:t-struct-description}.

In order to prove the above theorem, we introduce a notion of $p$-completion on a general $\infty$-topos $\topos X$,
and then use this in the special case of the $\infty$-topos of Nisnevich sheaves on smooth $k$-schemes.
In particular, we obtain the following:
\begin{lem}
    Let $\topos X$ be an $\infty$-topos (or more generally any presentable $\infty$-category).
    Then there is a localization functor $\completebr{-} \colon \topos X \to \topos X$,
    which inverts $p$-equivalences (i.e.\ maps $f$ such that $(\pSus f) \sslash p$ is an equivalence).
\end{lem}
\begin{proof}
    The construction can be found in \cref{lemma:topos:completion-functor}.
\end{proof}

Note that the short exact sequence \ref{eq:intro:stable-ses} in \cref{thm:intro:main} is unsatisfying:
It relates the $p$-completed homotopy groups of $X$ to the derived $p$-completions 
of the homotopy groups of $X$.
But this does (a priori) not say anything about the ($p$-completed) homotopy groups of $\complete{X}$!
In particular, note that we cannot use that the canonical $p$-equivalence $X \to \complete X$ induces 
an equivalence $\pi_n^p(X) \to \pi_n^p(\complete{X})$ via \ref{eq:intro:peq-iso-on-p-complete-htpy} of \cref{thm:intro:main},
since it is not clear (and probably wrong) that $\complete{X}$ is nilpotent even if $X$ is.
But we are nonetheless able to say more:
By the above lemma, we get $p$-completion functors 
in the categories of Zariski sheaves, Nisnevich sheaves, motivic spaces and 
connected motivic spaces, denote them by $L^p_{\zar}$, $L^p_{\nis}$, $L^p_{\AffSpc{1}}$ and $L^p_{\AffSpc{1},\ge 1}$, respectively.
We can relate the different functors:
\begin{prop} \label{prop:intro:different-completions}
    Let $X \in \MotSpc{k}_*$ be a nilpotent motivic space, it is in particular connected.
    Then there are equivalences 
    \begin{align*}
        L^p_{\zar}(X) \cong L^p_{\nis}(X) \cong L^p_{\AffSpc{1},\ge 1}(X).
    \end{align*}
    In particular, the $p$-completion of $X$ as a Nisnevich or Zariski sheaf 
    is again an $\AffSpc{1}$-invariant Nisnevich sheaf!

    If \cref{conj:motivic:conjecture-A} is true (i.e.\ if the $p$-completion $L^p_{\AffSpc{1}}(Y)$ 
    is connected for every nilpotent motivic space $Y$),
    then we also get an equivalence $L^p_{\nis}(X) \cong L^p_{\AffSpc{1}}(X)$.
\end{prop}
\begin{proof}
    The equivalences can be found in \cref{thm:motivic:iota-a1-completion,lemma:motivic:iota-completion}.
\end{proof}
Note that here a small problem arises: Currently we do not know whether 
the $p$-completion of a nilpotent motivic spaces is still connected.
The corresponding fact in an $\infty$-topos is true, see \cref{lemma:peq-respects-pi0}.
This introduces some complications, but at least for $L^p_{\zar}$, $L^p_{\nis}$ and $L^p_{\AffSpc{1},\ge 1}$
we have the following:
\begin{thm}
    Let $X \in \MotSpc{k}_*$ be a nilpotent motivic space.
    We have equivalences 
    \begin{equation*}
        \pi_n^p(X) \cong \pi_n^p(L^p_{\zar}(X)) \cong \pi_n^p(L^p_{\nis}(X)) \cong \pi_n^p(L^p_{\AffSpc{1},\ge 1}(X))
    \end{equation*}
    for all $n$.
    In particular, we get a short exact sequence
    \begin{equation*}
        0 \to \mathbb L_0 \pi_n(X) \to \pi_n^p(\complete{X}) \to \mathbb L_1 \pi_{n-1}(X) \to 0,
    \end{equation*}
    where $\complete{X}$ is any of $L^p_{\zar}(X) \cong L^p_{\nis}(X) \cong L^p_{\AffSpc{1},\ge 1}(X)$.
\end{thm}
\begin{proof}
    See \cref{lemma:motivic:p-complete-htpy-ios-to-completion} together with 
    the above \cref{prop:intro:different-completions} for the first claim.
    For the short exact sequence, see \cref{cor:motivic:ses-with-p-completion}.
\end{proof}

In order to be able to compute $p$-completions of nilpotent sheaves, 
we will be using the following theorem:
\begin{thm} \label{thm:intro:post-tower}
    If $\topos X$ is locally of finite uniform homotopy dimension (see \cref{def:topos:locally-finite-uniform-htpy-dim},
    this is a mild generalization of the notion of being of homotopy dimension $\le n$,
    which is in particular satisfied by the Nisnevich and Zariski topoi),
    then the $p$-completion of a nilpotent sheaf $X \in \topos X$ 
    (see \cref{def:nilpotent:defn} for the definition of nilpotence in an $\infty$-topos) can be computed on
    its Postnikov tower, i.e.\ there is an equivalence 
    \begin{equation*}
        \complete{X} \cong \limil{k} \completebr{\tau_{\le k} X}.
    \end{equation*}
\end{thm}
\begin{proof}
    This can be found in \cref{thm:topos:p-comp-commutes-with-post-tower}.     
\end{proof}
The above result about the Postnikov tower is extremely useful in computing the 
$p$-completions of nilpotent sheaves:
Let $X \in \topos X$ be a nilpotent sheaf, where $\topos X$ is an $\infty$-topos locally of finite uniform homotopy dimension.
Then the Postnikov tower has a principal refinement (see \cref{def:nilpotent:principal-refinement}),
i.e.\ there are positive integers $m_n$, $n$-truncated spaces $X_{n, k}$, 
abelian group objects $A_{n, k} \in \AbObj{\Disc{\topos X}}$ in the associated $1$-topos 
of discrete objects for all $n$ and all $0 \le k \le m_n$,
and fiber sequences $X_{n, k+1} \xrightarrow{p_{n, k}} X_{n, k} \to K(A_{n,k+1}, n+1)$
that refine the Postnikov tower (in the sense that $X_{n, 0} = \tau_{\le n} X$
and that the truncation map $\tau_{\le n} X \to \tau_{\le n-1} X$ 
can be factored as $p_{n, m_n - 1} \circ \dots \circ p_{n, 0}$).
Now we have the following proposition:
\begin{prop} \label{prop:intro:post-sections}
    For every $n$ and $k$ we have an equivalence 
    \begin{equation*}
        \completebr{X_{n, k+1}} \cong \tau_{\ge 1} \Fib{\completebr{X_{n, k}} \to \completebr{K(A_{n, k}, n+1)}}.
    \end{equation*}
    Moreover, there is an equivalence 
    \begin{equation*}
        \completebr{K(A, n)} \cong \tau_{\ge 1} \pLoop (\completebr{\Sigma^n HA})
    \end{equation*}
    for every abelian group object $A \in \AbObj{\Disc{\topos X}}$.
\end{prop}
\begin{proof}
    See \cref{cor:completion-of-EM-space,lemma:postnikov-fiber-sequence-completion}.
\end{proof}
The above proposition, together with \cref{thm:intro:post-tower} about the Postnikov tower,
allows us to compute the $p$-completion of nilpotent sheaves by reducing to the much easier case 
of the $p$-completion of sheaves of spectra,
which is just given by the $p$-adic limit $E \mapsto \complete{E} \cong \limil{k} E\sslash p^k$.
This computational tool will power almost all of our results.

\subsection*{Outline}
We will start with the construction of and some basic results about the stable $p$-completion functor
on a stable $\infty$-category in \cref{section:stable}.
We will then construct the $p$-adic t-structure on a stable $\infty$-category, which is a t-structure which behaves 
exceptionally well with respect to $p$-completion. In particular, we will show that this t-structure 
admits an analog of the fundamental short exact sequence for the 
(stable) $p$-completion of spectra in \cref{lemma:short-exact-sequence-of-t-structure}. 

In \cref{section:topos:main}, we will first construct the unstable $p$-completion functor on an 
arbitrary presentable $\infty$-category $\topos X$,
and then show that if $\topos X$ is moreover an $\infty$-topos, then this functor is very well-behaved.
In particular, we prove our fundamental computational result,
that we can calculate the $p$-completion of a nilpotent sheaf by 
reducing to its Postnikov tower; and then to the much easier case of Eilenberg MacLane spaces,
see \cref{thm:intro:post-tower,prop:intro:post-sections}.

In order to show that there is a short exact sequence as in \cref{thm:intro:main},
we will use the following diagram of right adjoints:
\begin{center}
    \begin{tikzcd}
        \PrShv{W}&\arrow[l, "\iota_{\Sigma}"] \PSig{W} \arrow[r, "\cong"] &\ShvTop{\prozartop}{\prozar{\smooth{k}}} \arrow[d, "\nu_*"] \\
        \MotSpc{k} \arrow[r, "\iota_{\AffSpc{1}}"] &\MotShv{k} \arrow[r, "\iota_{\nis}"] &\ZarShv{k}.
    \end{tikzcd}
\end{center}
First, we will show that there is a short exact sequence for objects in $\PrShv{W}$ in \cref{section:embedding:presheaf},
by using the classical short exact sequence on each level.
Then, we will show in \cref{section:embedding:psig} that this also gives short exact sequences on the nonabelian derived category $\PSig{W}$ for suitable $W$.
We will then show in \cref{section:embedding} that if we have an embedding of $\infty$-topoi $\nu^* \colon \topos X \rightleftarrows \PSig{W} \colon \nu_*$,
then (at least in good cases), we also get a short exact sequence for objects in $\topos X$.
An example of such an embedding of $\infty$-topoi is the embedding of the Zariski topos into the 
pro-Zariski topos, constructed in \cref{appendix:pro-zar}.
Thus, we get a short exact sequence for (certain) objects in $\ZarShv{k}$.
Then, in \cref{section:motivic}, we will show that the sequence on the Zariski topos 
actually induces a sequence for objects in the Nisnevich topos, and then, finally, for nilpotent motivic spaces.

Note that in $\ZarShv{k}$, the short exact sequence only exists for a nilpotent Zariski sheaf $X$
if the following technical condition is satisfied: $(\mathbb L_1 (\pi_n \nu^* X)) \sslash p$
must be classical (i.e.\ in the essential image of $\nu^*$).
Therefore, we will spend some time in \cref{section:zar} to find a geometric condition that 
will always imply this technical statement: Gersten injectivity of $\pi_n(X) / p^k$,
see \cref{def:pro-zar:gersten}.
If $X$ is a motivic space, then we will deduce Gersten injectivity of $\pi_n(X) / p^k$
from the Gabber presentation lemma in \cref{section:motivic}.

In the remainder of \cref{section:motivic} we will compare the various 
different notions of $p$-completion (we can $p$-complete as a (connected) motivic space, as a Nisnevich sheaf or as a Zariski sheaf),
see \cref{prop:intro:different-completions}.

\subsection*{Notation}
We will write $\An$ for the $\infty$-category of anima/homotopy types/spaces, and $\Sp$
for the stable $\infty$-category of spectra. More generally,
if $\Cat V$ is a presentable $\infty$-category, we write $\Stab{\Cat V}$ for the stabilization 
of $\Cat V$.

\subsection*{Conventions}
We will adhere to the following derived convention:

If $\Cat D$ and $\Cat E$ are stable $\infty$-categories equipped with $t$-structures
and $F \colon \Cat D \to \Cat E$ is an exact functor,
we will also write $F$ for the composition 
\begin{equation*}
    \tstructheart{\Cat D} \hookrightarrow \Cat D \xrightarrow{F} \Cat E.
\end{equation*}
In contrast, we write $\heart F$ for the functor 
\begin{equation*}
    \tstructheart{\Cat D} \hookrightarrow \Cat D \xrightarrow{F} \Cat E \xrightarrow{\pi_0} \tstructheart{\Cat E}.
\end{equation*}

Note that in particular limits in $\tstructheart{\Cat D}$ are calculated 
as $\heart{\limil{I}}(-) = \pi_0(\limil{I}(-))$, and similar for colimits.
To avoid awkward notation, if $f \colon X \to Y \in \tstructheart{\Cat D}$ is a morphism,
we will write $\ker(f)$ for the kernel of $f$ in the abelian category $\tstructheart{\Cat D}$ (instead of e.g.\ $\heart{\operatorname{fib}}(f)$),
whereas $\Fib{f}$ refers to the fiber of $f$ in the stable $\infty$-category $\Cat D$,
and similar for $\coker(f)$ and $\Cofib{f}$.
If $n \in \Z$ is an integer, then $n$ induces an endomorphism $n \colon X \to X$.
We will write $X / n \coloneqq \coker(X \xrightarrow{n} X) \in \tstructheart{\Cat D}$
and $X \sslash n \coloneqq \Cofib{X \xrightarrow{n} X} \in \Cat D$.

Moreover, suppose that $\topos X$ and $\topos Y$ are $\infty$-topoi,
and that $F \colon \topos X \to \topos  Y$ is a functor that respects $n$-truncated objects for every $n \ge 0$ and finite limits 
(e.g.\ the left adjoint or the right adjoint of a geometric morphism).
Then $F$ induces a functor on the stabilizations $\spectra{X} \to \spectra{Y}$, which we also denote by $F$.
Note that there is a standard t-structure on $\spectra{X}$, 
and an equivalence $\AbObj{\Disc{\topos X}} \cong \tstructheart{\spectra X}$,
where the left-hand side denotes the abelian group objects in the underlying 1-topos of discrete objects in $\topos X$.
Using this equivalence, we will identify the homotopy object functors $\pi_n$ 
with functors
\begin{equation*}
    \pi_n \colon \topos X \to \tstructheart{\spectra{X}}
\end{equation*}
for $n \ge 2$.
Since $F$ commutes with finite products, it also induces a functor 
\begin{equation*}
    \AbObj{\Disc{\topos X}} \to \AbObj{\Disc{\topos Y}}.
\end{equation*}
Under the above identifications, we will refer to this functor as 
\begin{equation*}
    \heart{F} \colon \tstructheart{\spectra{X}} \to \tstructheart{\spectra{Y}}.
\end{equation*}
Note that this coincides with the earlier use of the symbol $\heart{F}$ from above.
If $F$ is the left adjoint of a geometric morphism, it induces a t-exact functor on the stabilization.
Therefore, the functors $\heart{F}$ and $F$ (restricted to the heart) are equivalent,
and we will usually omit the heart.
However, if $F$ is the right adjoint of a geometric morphism,
this is usually not the case, and we will always write $\heart{F}$ if we refer to the 
functor on the hearts. (Although, in many of our cases, the right adjoint will actually be t-exact.)

Let $(\Cat C, \tau)$ is a site and $\topos X \coloneqq \ShvTop{\tau}{\Cat C}$
is the associated $\infty$-topos.
Suppose that $A \in \heart{\spectra{X}} \cong \heart{\ShvTop{\tau}{\Cat C, \Sp}}$.
For $U \in \Cat C$, we will write $\Gamma(U, A) \in \Sp$ for the value of $A$ at $U$
(note that this spectrum knows about the $\tau$-cohomology of $A$ at $U$!).
In contrast, $\heart{\Gamma}(U, A) = \pi_0(\Gamma(U, A))$ are the global sections 
of $A \in \AbObj{\Disc{\topos X}}$.
Note that in particular, the equivalence $\heart{\spectra{X}} \to \AbObj{\Disc{\topos X}}$
is realized by the functor $A \mapsto \heart{\Gamma}(-, A)$.

\subsection*{Acknowledgement}
I want to thank my advisor Tom Bachmann for suggesting the topic,
for answering an uncountable amount of questions, 
and for his excellent comments on drafts of this paper.
I also want to thank Lorenzo Mantovani, Luca Passolunghi and Timo Weiß 
for helpful discussions about parts of this paper.

\section{Stable \texorpdfstring{$p$}{p}-Completion}\label{section:stable}
Let $\Cat D$ be a presentable stable $\infty$-category \cite[Definition 1.1.1.9]{higheralgebra}.
Suppose that $\Cat D$ is equipped with an accessible t-structure $\tstruct{\Cat D}$ \cite[Definition 1.2.1.4 and Definition 1.4.4.12]{higheralgebra}.
Suppose moreover that this t-structure is right-separated (i.e.\ $\bigcap_{n} \tcocon[n]{\Cat D} = 0$).
We will call this t-structure the \emph{standard t-structure} (on $\Cat D$).
Let $\tstructheart{\Cat D} \coloneqq \tcon{\Cat D} \cap \tcocon{\Cat D}$
be the heart of the standard t-structure.
This is an abelian category, see \cite[Remark 1.2.1.12]{higheralgebra}.
We write $\tau_{\le n}$ and $\tau_{\ge n}$ for the truncation functors,
and $\pi_n \colon \Cat D \to \tstructheart{\Cat D}$ for the $n$-th homotopy object.
We say that an object $E \in \Cat D$ is \emph{$k$-connective} (resp.\ \emph{$k$-coconnective} or \emph{$k$-truncated})
for some $k \in \Z$
if $E \in \tcon[k]{\Cat D}$ (resp.\ $E \in \tcocon[k]{\Cat D}$).

\subsection{Properties of the Stable \texorpdfstring{$p$}{p}-Completion Functor}
In this section, we define the stable $p$-completion functor and prove some basic properties.
Most of the results are well-known, see for example \cite[Section 2.2]{mathew2017nilpotence} or \cite[Section 2.1]{bachmann2021rigidity}.

\begin{defn}
    Let $(-)\sslash p$ be the endofunctor on $\Cat D$ given on objects by
    $E \mapsto \Cofib{E \xrightarrow{p} E}$.

    We say that a morphism $f \colon E \to F$ in $\Cat D$ is
    a \emph{$p$-equivalence} if $f\sslash p$ is an equivalence.
    We say that an object $E \in \Cat D$ is \emph{$p$-complete}
    if for all $p$-equivalences $F \to F'$
    the induced map on mapping spaces $\Map{\Cat D}(F', E) \to \Map{\Cat D}(F, E)$
    is an equivalence.

    Write $\complete{\Cat D}$ for the subcategory of $p$-complete objects.
\end{defn}

\begin{rmk}
    If $\Cat D$ is equipped with a symmetric monoidal structure $\otimes$
    that is exact in each variable, then the endofunctor $(-)\sslash p$
    is equivalent to the functor $- \otimes (\mathbb S \sslash p)$,
    where $\mathbb S$ is the unit of the symmetric monoidal structure.
    This follows immediately from the assumption that the tensor product 
    is exact in each variable.
\end{rmk}

\begin{lem}\label{lemma:stable:small-generation}
    The class $S$ of $p$-equivalences in $\Cat D$ is strongly saturated and of small generation.
\end{lem}
\begin{proof}
    Using \cite[Proposition 5.5.4.16]{highertopoi}, it suffices to show that $S = f^{-1}(S')$
    for some colimit-preserving functor $f$ and a strongly saturated class $S'$ of small generation.
    This holds for $f = (- \sslash p)$ and $S'$ the collection of equivalences in $\Cat{D}$.
    $S'$ is of small generation because it is the smallest saturated class of morphisms in $\Cat D$,
    see \cite[Example 5.5.4.9]{highertopoi}, and therefore generated by the empty collection.
\end{proof}

\begin{lem}
    The category $\complete{\Cat D}$ is presentable,
    and the inclusion $\complete{\Cat D} \to \Cat D$
    has a left adjoint $\completebr{-} \colon \Cat D \to \complete{\Cat D}$.
    In other words, $\completebr{-}$ is a localization functor.
\end{lem}
\begin{proof}
    This is an application of \cite[Proposition 5.5.4.15]{highertopoi},
    using that the class $S$ of $p$-equivalences in $\Cat D$ is of small generation,
    see \cref{lemma:stable:small-generation}.
\end{proof}

This localization functor is called the \emph{(stable) $p$-completion functor}.
By abuse of notation, we will also write $\completebr{-} \colon \Cat D \to \Cat D$
for the composition of the localization functor with the inclusion.
The $p$-completion functor has an easy description:

\begin{lem}\label{lemma:stable:completion-functor-description}
    There is a natural isomorphism of functors
    $\completebr{-} \cong \limil{n} (- \sslash p^n)$.
\end{lem}
\begin{proof}
    Suppose that $\Cat D$ is equipped with a symmetric monoidal structure $\otimes$
    that is exact in each variable.
    Then this follows from the discussion before \cite[Proposition 2.23]{mathew2017nilpotence}.

    We also give a second proof, which does not require the existence of a stably symmetric monoidal structure:
    Let $L_p \colon \Cat D \to \Cat D$ be the functor
    given by $E \mapsto \limil{n} (E \sslash p^n)$.
    It suffices to show that $L_p(E)$ is $p$-complete for every $E$ and that the 
    canonical map $\alpha_E \colon E \to L_p(E)$ induced by the maps $E \to E\sslash p^n$ is a $p$-equivalence.

    Since the inclusion of $p$-complete objects is a right adjoint,
    it commutes with limits.
    In particular, in order to show that $\limil{n} E\sslash p^n$ is $p$-complete,
    it suffices to show that $E \sslash p^n$ is $p$-complete for all $n$.

    First, let $f \colon X \to Y$ be a $p$-equivalence, i.e. $f \sslash p$ is an equivalence.
    For every $n$, there is a fiber sequence (in the stable category $\Fun{}(\Delta^1, \Cat D)$)
    \begin{equation*}
        f\sslash p \to f\sslash p^n \to f\sslash p^{n-1}.
    \end{equation*}
    By induction, we deduce that if $f\sslash p$ is an equivalence,
    so is $f\sslash p^n$ for all $n$.

    We now show that $E \sslash p^n$ is $p$-complete for all $n$.
    For this, let $f \colon X \to Y$ be a $p$-equivalence.
    We have the following chain of natural equivalences:
    \begin{align*}
        \Map{\Cat D}(f, E\sslash p^n)
         & \cong \Map{\Cat D} \left(f, \Fib{\Sigma E \xrightarrow{p^n} \Sigma E}\right)             \\
         & \cong \Fib{\Map{\Cat D}(f, \Sigma E) \xrightarrow{p^n} \Map{\Cat D}(f, \Sigma E)} \\
         & \cong \Map{\Cat D}\left(f \sslash p^n, \Sigma E\right).
    \end{align*}
    Here, we use that the mapping space functor is left exact in both variables,
    and that $\Cofib{g} = \Fib{\Sigma g}$ for every morphism $g$
    in a stable category.
    Thus, since $f \sslash p^n$ is an equivalence by the above,
    we conclude that $\Map{\Cat D}(f, E \sslash p^n)$ is an equivalence.
    In other words, $E \sslash p^n$ is $p$-complete.

    Thus, we are left to show that for every $E$, $\alpha_E \sslash p \colon E \sslash p \to (\limil{n} E \sslash p^n) \sslash p$
    is an equivalence.
    Indeed, we can write
    \begin{align*}
        (\limil{n} E \sslash p^n) \sslash p
         & \cong \limil{n} ((E \sslash p^n) \sslash p)            \\
         & \cong \limil{n} ((E \sslash p) \sslash p^n)            \\
         & \cong \limil{n} (E\sslash p \oplus \Sigma E \sslash p) \\
         & \cong E \sslash p.
    \end{align*}
    The first equivalence holds because $\Cat D$ is stable, and thus the cofiber
    $(- \sslash p)$ is also a (suspension of) a limit, and limits commute with limits.
    The last equality holds, because in the limit, the transition maps
    on the left part are the identity, and are multiplication by $p$
    on the right part.
\end{proof}

From now on we will use the equivalence from \cref{lemma:stable:completion-functor-description}
without reference.

\begin{lem}\label{lemma:stable:peq-description}
    Let $f \colon E \to F$ be a morphism in $\Cat D$.
    The following are equivalent:
    \begin{enumerate} [label= (\arabic*),ref= (\arabic*),itemsep=0em]
        \item $f$ is a $p$-equivalence,
        \item $\completebr{f}$ is an equivalence,
        \item $\Map{\Cat D}(f, T)$ is an equivalence of anima for every $T \in \complete{\Cat D}$.
    \end{enumerate}

    In particular, for any object $E$
    the unit $E \to \complete{E}$ is a $p$-equivalence,
    and $E$ is $p$-complete if and only if $E \cong \complete{E}$.
\end{lem}
\begin{proof}
    This follows immediately from the fact that $\completebr{-}$ is a localization functor,
    and that the class of $p$-equivalences is strongly saturated by \cref{lemma:stable:small-generation}.
    See \cite[Proposition 5.5.4.2 and Proposition 5.5.4.15 (4)]{highertopoi}.
\end{proof}

\begin{lem}\label{lemma:stable:truncation-of-completion}
    Let $E \in \Cat D$ be $k$-truncated.
    Then $\complete{E}$ is $(k+1)$-truncated.
\end{lem}
\begin{proof}
    For each $n$, we see that
    $E \sslash p^n = \Cofib{E \xrightarrow{p^n} E} \cong \Sigma \Fib{E \xrightarrow{p^n} E}$.
    Since $\tcocon[k]{\Cat D}$ is stable under limits (see \cite[Corollary 1.2.1.6]{higheralgebra}),
    we conclude that $E \sslash p^n \in \tcocon[k+1]{\Cat D}$.
    By the same corollary we now get that $\complete{E} = \limil{n} (E \sslash p^n)$
    is $(k+1)$-truncated.
\end{proof}

\begin{defn}
    Let $\Cat A$ be an abelian category, and let $A \in \Cat A$.
    We say that $A$ is \emph{uniquely $p$-divisible},
    if $A \xrightarrow{p} A$ is an isomorphism.
    Similarly, we say that $A$ is \emph{$p$-divisible},
    if $\coker(A \xrightarrow{p} A) = 0$.
\end{defn}

\begin{lem} \label{lemma:upd-homotopy-of-completion-below-connectivity-of-base}
    Let $f \colon F \to E$ be a $p$-equivalence in $\Cat D$
    such that $F$ is $k$-connective for some $k$.
    Then $\pi_n{E}$ is uniquely $p$-divisible for all $n < k$.
\end{lem}
\begin{proof}
    Since $f$ is a $p$-equivalence, it induces an equivalence $F\sslash p \to E\sslash p$.
    We have a cofiber sequence $E \xrightarrow{p} E \to E\sslash p$,
    and thus a cofiber sequence $E \xrightarrow{p} E \to F\sslash p$.
    This induces a long exact sequence on homotopy objects, which gives us
    \begin{equation*}
        \pi_{i+1}(F\sslash p) \to \pi_{i}(E) \xrightarrow{p} \pi_i(E) \to \pi_i(F\sslash p)
    \end{equation*}
    for all $i$.

    If $i \le k-2$, then the outer terms vanish ($F\sslash p$ is $k$-connective).
    Thus, $\pi_i(E)$ is uniquely $p$-divisible.

    If $i = k-1$, we get
    \begin{center}
        \begin{tikzcd}
            \pi_k(E\sslash p) \arrow[r, "\alpha"] &\pi_{k-1}(E) \arrow[r, "p"] &\pi_{k-1}(E) \arrow[r] &\pi_{k-1}(F\sslash p) = 0 \\
            \pi_k(F\sslash p) \arrow[u, "\cong"] \arrow[r] &\pi_{k-1}(F) = 0 \arrow[u]
        \end{tikzcd}
    \end{center}
    Commutativity of the square implies that $\alpha = 0$.
    Thus, also $\pi_{k-1}(E)$ is uniquely $p$-divisible.
\end{proof}

\begin{lem}\label{lemma:stable-peq-to-zero-iff-upd}
    Let $E \in \Cat D$.
    Consider the following statements:
    \begin{enumerate}[label= (\arabic*),ref= (\arabic*),itemsep=0em]
        \item\label{lemma:stable:completion-zero-iff:completion-zero}
              $\complete E = 0$,
        \item\label{lemma:stable:completion-zero-iff:htpy-completion-zero}
              $\pi_n(\complete E) = 0$ for all $n$,
        \item\label{lemma:stable:completion-zero-iff:htpy-mod-p-zero}
              $\pi_n(E \sslash p) = 0$ for all $n$, and
        \item\label{lemma:stable:completion-zero-iff:htpy-pdiv}
              $\pi_n(E)$ is uniquely $p$-divisible for all $n$.
    \end{enumerate}
    Then \ref{lemma:stable:completion-zero-iff:completion-zero} $\implies$ \ref{lemma:stable:completion-zero-iff:htpy-completion-zero} $\implies$ \ref{lemma:stable:completion-zero-iff:htpy-mod-p-zero} $\iff$ \ref{lemma:stable:completion-zero-iff:htpy-pdiv}.
    If $\tcon[\infty]{\Cat D} \coloneqq \bigcap_n \tcon[n]{\Cat D}$ is stable under sequential limits, then also \ref{lemma:stable:completion-zero-iff:htpy-completion-zero} $\iff$ \ref{lemma:stable:completion-zero-iff:htpy-mod-p-zero}.
    If the t-structure is moreover left-separated, then also \ref{lemma:stable:completion-zero-iff:completion-zero} $\iff$ \ref{lemma:stable:completion-zero-iff:htpy-completion-zero}.

    Note that if the t-structure is left-separated, then $\tcon[\infty]{\Cat D} = 0$ is in particular 
    stable under sequential limits (i.e.\ in this case, all four statements are equivalent).
\end{lem}
\begin{proof}
    It is clear that (1) $\implies$ (2).
    Moreover, if the t-structure is left-separated,
    then it follows directly that (2) $\implies$ (1)
    (note that the t-structure is assumed to be right-separated).

    We now show (2) $\implies$ (3).
    Since $E \to \complete{E}$ is a $p$-equivalence,
    we have $E \sslash p \cong \complete{E} \sslash p$.
    In other words, there is a cofiber sequence
    \begin{equation*}
        \complete E \xrightarrow{p} \complete E \to E \sslash p.
    \end{equation*}
    This induces the following long exact sequence on homotopy:
    \begin{center}
        \begin{tikzcd}
            \cdots \arrow[r, "p"] &\pi_n(\complete{E}) \arrow[r] & \pi_n(E\sslash p) \arrow[r] &\pi_{n-1}(\complete{E}) \arrow[r, "p"] &\cdots.
        \end{tikzcd}
    \end{center}
    We conclude that $\pi_n(E \sslash p) = 0$ for all $n$.

    The equivalence (3) $\iff$ (4) 
    follows immediately from the long exact sequence associated to the fiber sequence $E \xrightarrow{p} E \to E\sslash p$,
    similar to the proof of \cref{lemma:upd-homotopy-of-completion-below-connectivity-of-base}.

    We are left to show that (4) $\implies$ (2) if we assume that $\tcon[\infty]{\Cat D}$ 
    is stable under sequential limits.
    Using the cofiber sequence $E \xrightarrow{p^k} E \to E \sslash p^k$
    we conclude as above that $\pi_n(E \sslash p^k) = 0$ for all $k \ge 1$ and all $n$.
    In particular, since the standard t-structure is right-separated, we see that
    $E \sslash p^k \in \tcon[\infty]{\Cat D}$.
    But now we conclude that $\complete{E} = \limil{k} E \sslash p^k \in \tcon[\infty]{\Cat D}$.
    This implies that $\pi_n(\complete E) \cong 0$ for all $n$.
\end{proof}

\begin{cor}\label{cor:stable-peq-iff-upd-fiber}
    Let $f \colon E \to F$ be a morphism in $\Cat D$.
    Consider the following statements:
    \begin{enumerate}[label=(\arabic*),ref=(\arabic*),itemsep=0em]
        \item $f$ is a $p$-equivalence,
        \item $\completebr{\Fib{f}} = 0$,
        \item $\completebr{\Cofib{f}} = 0$,
        \item $\Fib{f}$ has uniquely $p$-divisible homotopy objects,
        \item $\Cofib{f}$ has uniquely $p$-divisible homotopy objects.
    \end{enumerate}
    Then (1) $\iff$ (2) $\iff$ (3) $\implies$ (4) $\iff$ (5).
    If moreover the standard t-structure is left-separated,
    then also (4) $\implies$ (3).
\end{cor}
\begin{proof}
    The equivalence of $(1)$ and $(2)$ follows from the
    fiber sequence $\Fib{f} \to E \to F$.
    That $(2)$ implies $(4)$ was proven in \cref{lemma:stable-peq-to-zero-iff-upd}.
    If the t-structure is left-separated, then also $(4)$ implies $(2)$, again by
    \cref{lemma:stable-peq-to-zero-iff-upd}.
    The other equivalences follow because $\Cat D$ is stable
    and thus there is an equivalence $\Cofib{f} \cong \Sigma \Fib{f}$.
\end{proof}

\begin{lem} \label{lemma:stable:p-eq-of-spectra-smash-product}
    Suppose that $\Cat D$ is equipped with a symmetric monoidal structure $\otimes$
    that is exact in each variable.

    Let $f_i \colon E_i \to F_i$ be a $p$-equivalence in $\Cat D$ for $i = 1, \dots, n$.
    Then also $\bigotimes_i f_i \colon \bigotimes_i E_i \to \bigotimes_i F_i$
    is a $p$-equivalence, i.e.\ \emph{$p$-completion is compatible with the symmetric monoidal 
    structure} in the language of \cite[Definition 2.2.1.6]{higheralgebra}.
\end{lem}
\begin{proof}
    Note that by \cite[Example 2.2.1.7]{higheralgebra},
    it suffices to show that for every $p$-equivalence $f \colon E \to F$,
    and any object $Z \in \Cat D$,
    also $f \otimes Z \colon E \otimes Z \to F \otimes Z$ is a $p$-equivalence.
    We thus have to show that $(f \otimes Z) \sslash p$ is an equivalence.
    Since the symmetric monoidal structure is exact in each variable,
    we can write $(f \otimes Z) \sslash p \cong (f \sslash p) \otimes Z$, which is an equivalence by assumption.
\end{proof}

\subsection{The \texorpdfstring{$p$}{p}-adic t-structure}
The aim of this section is to define a t-structure on
$\Cat D$ which is suited for $p$-completions.

\begin{defn} \label{def:t-struct:defn}
    For $i \in \Z$, let $\tpcon[i]{\Cat D}$ be the full subcategory of $\Cat D$
    given by objects
    \begin{equation*}
        \set{E \in \Cat D}{\pi_j(E) \text{ uniquely p-divisible } \forall j<i-1, \pi_{i-1}(E) \text{ p-divisible}}.
    \end{equation*}
    Let $\tpcocon[i]{\Cat D}$ be the right orthogonal complement of $\tpcon[i+1]{\Cat D}$,
    i.e. $E \in \tpcocon[i]{\Cat D}$ if and only if for all $F \in \tpcon[i+1]{\Cat D}$
    the mapping space $\Map{}(F, E)$ is contractible.
    We will show below in \cref{lemma:t-struct:proof}
    that this defines a t-structure $\tpstruct{\Cat D}$ on $\Cat D$.
    We will call this t-structure the \emph{$p$-adic t-structure} on $\Cat D$.
    Denote by $\pi^p_n$ the $n$-th homotopy object and by $\tau^p_{\le n}$ and $\tau^p_{\ge n}$ the truncations
    of this t-structure.
    Moreover, denote by $\tpstructheart{\Cat D} \coloneqq \tpcon{\Cat D} \cap \tpcocon{\Cat D} \subset \Cat D$
    the heart.
\end{defn}

\begin{rmk}
    Note that the $p$-adic t-structure $\tpstruct{\Cat D}$ depends on the t-structure $\tstruct{\Cat D}$.
    This is suppressed in our notation. Later, $\Cat D$ will be the
    stabilization of a presentable $\infty$-category, which admits a canonical t-structure,
    so this slight abuse of notation will not be a problem.
\end{rmk}

In order to prove that the $p$-adic t-structure is a t-structure, we will need the following lemma:
\begin{lem} \label{lemma:t-struct:mod-p-in-con-implies-in-tpcon}
    Let $E \in \Cat D$.
    The following are equivalent:
    \begin{enumerate}[label=(\arabic*),ref=(\arabic*),itemsep=0em]
        \item $E \in \tpcon{\Cat D}$,
        \item $E \sslash p^n \in \tcon{\Cat D}$ for all $n$ and
        \item $E \sslash p \in \tcon{\Cat D}$.
    \end{enumerate}
\end{lem}
\begin{proof}
    The fiber sequence $E \xrightarrow{p^n} E \to E\sslash p^n$ yields the long exact sequence
    \begin{equation*}
        \cdots \to \pi_{k+1}(E\sslash p^n) \to \pi_k(E) \xrightarrow{p^n} \pi_k(E) \to \pi_k(E\sslash p^n) \to \cdots.
    \end{equation*}
    We conclude that $\pi_k(E\sslash p^n) = 0$ for all $k < 0$
    if and only if $\pi_k(E)$ is uniquely $p^n$-divisible for all $k < -1$
    and $\pi_{-1}(E)$ is $p^n$-divisible.
    But being (uniquely) $p^n$-divisible is the same as being (uniquely) $p$-divisible.
    Since the standard t-structure is right-separated by assumption,
    we see that $\pi_k(E \sslash p^n) = 0$ for all $k < 0$ is equivalent to
    $E \sslash p^n \in \tcon{\Cat D}$.
\end{proof}

\begin{lem} \label{lemma:t-struct:proof}
    The pair $\tpstruct{\Cat D}$ from \cref{def:t-struct:defn}
    defines an accessible t-structure on $\Cat D$.
\end{lem}
\begin{proof}
    Using \cite[Proposition 1.4.4.11]{higheralgebra},
    it suffices to show that $\tpcon{\Cat D}$ is presentable and closed under colimits and extensions.
    Note that by \cref{lemma:t-struct:mod-p-in-con-implies-in-tpcon},
    we see that $\tpcon{\Cat D} = \set{E \in \Cat D}{E \sslash p \in \tcon{\Cat D}}$.

    We first show that $\tpcon{\Cat D}$ is presentable.
    Note that by assumption, the standard t-structure is accessible, i.e.\ $\tcon{\Cat D}$ is presentable.
    Using \cref{lemma:t-struct:mod-p-in-con-implies-in-tpcon},
    we see that there is a cartesian diagram of $\infty$-categories
    \begin{center}
        \begin{tikzcd}
            \tpcon{\Cat D} \arrow[r] \arrow[d, hook] &\tcon{\Cat D} \arrow[d, hook] \\
            \Cat D \arrow[r, "(-)\sslash p"] &\Cat D.
        \end{tikzcd}
    \end{center}
    The inclusion $\tcon{\Cat D} \hookrightarrow \Cat D$ and the functor $(-) \sslash p$
    commute with colimits (by \cite[Corollary 1.2.1.6]{higheralgebra}, and since colimits commute with colimits).
    By assumption, $\Cat D$ and $\tcon{\Cat D}$ are presentable.
    Thus, the limit of the above diagram can be computed in the $\infty$-category
    of presentable $\infty$-catgories $\mathcal{P}r^L$
    (see \cite[Proposition 5.5.3.13]{highertopoi}),
    and we conclude that $\tpcon{\Cat D}$ is presentable.
    In particular, we see that the functor $\tpcon{\Cat D} \hookrightarrow \Cat D$ commutes with colimits,
    i.e.\ the subcategory $\tpcon{\Cat D}$ is closed under colimits.

    We are left to show that $\tpcon{\Cat D}$ is closed under extensions.
    This follows immediately from the fact that $\tcon{\Cat {D}}$
    is closed under extensions (this is true for the connective part of any t-structure),
    and that $(-) \sslash p \colon \Cat D \to \Cat D$
    commutes with extensions (because it is an exact functor).
\end{proof}

\begin{rmk} \label{rmk:t-struct:choices}
    We quickly explain why we made those choices.
    We will see in \cref{lemma:t-struct:left-complete}
    that the $p$-adic t-structure is not left-separated, with
    \begin{equation*}
        \bigcap_n \tpcon[n]{\Cat D} = \set{E \in \Cat D}{\pi_n(E) \text{ uniquely } p\text{-divisible for all } n}.
    \end{equation*}
    Note that this is exactly the kernel of the $p$-completion functor $\completebr{-}$,
    hence the left-seperation of this t-structure (i.e.\ the Verdier quotient
    $\Cat D / \bigcap_n \tpcon[n]{\Cat D}$) is given by $\complete{\Cat D}$.
    There is another t-structure with the same property:

    Let $\Cat C \coloneqq \set{E \in \Cat D}{\pi_j(E) \text{ uniquely p-divisible } \forall j < 0}$.
    Then we could define a t-structure by declaring the $-1$-coconnective objects to be the
    right orthogonal complement of $\Cat C$, and the connective objects to be the
    left orthogonal complement of the $-1$-coconnective objects. (Note that
    $\Cat C$ itself cannot be the subcategory of connective objects of a t-structure
    since it is not closed under extensions).
    If $\Cat D = \Sp$ is the category of spectra
    (or more generally the stabilization of an $\infty$-topos locally of homotopy dimension $0$),
    then these two t-structures agree.
    But in general, this is not true:
    Let $\topos X$ be the $\infty$-topos of étale sheaves (of anima)
    on the small étale site of $\Spec{\Q}$.
    Let $\mu_{p^\infty}$ be the sheaf of $p$-power roots of unity,
    i.e.\ $\mu_{p^\infty}(\Spec k) = \set{x \in k}{\exists n, x^{p^n} = 1}$.
    The associated Eilenberg-MacLane spectrum $H\mu_{p^\infty}$
    lies inside $\tpcon[1]{\spectra X}$ (since $\mu_{p^\infty}$ is $p$-divisible), but is only
    $0$-connective in this second t-structure.
    If one views $\mu_{p^\infty}$ as an étale version of the (ordinary) spectrum
    $H(\Z[p^{-1}]/\Z)$, then one would expect this shift.
\end{rmk}

\begin{defn}
    Let $\Cat A$ be an abelian category. Let $A \in \Cat A$.
    We say that $A$ has \emph{bounded $p$-divisibility}
    if for all $p$-divisible $B \in \Cat A$ we have $\Map{}(B, A) = 0$.
\end{defn}

\begin{lem} \label{lemma:t-struct:char-of-cocon}
    Let $E \in \Cat D$.
    Then $E \in \tpcocon{\Cat D}$ if and only if
    $E = \tau_{\le 0} E$, $E = \complete E$ and $\pi_0(E)$ has bounded $p$-divisibility.
\end{lem}
\begin{proof}
    Suppose first that $E \in \tpcocon{\Cat D}$.
    Note that $\tcon[1]{\Cat D} \subseteq \tpcon[1]{\Cat D}$
    since the zero object $0 \in \tstructheart{\Cat D}$ is (uniquely) $p$-divisible.
    Thus, $\tcocon{\Cat D} \supseteq \tpcocon{\Cat D}$.
    Hence, $E = \tau_{\le 0} E$.
    In order to show that $E$ is $p$-complete, it suffices to show that
    $\Map{}(A, E) = 0$ for all $A$ with $\complete A = 0$.
    So let $\complete A = 0$.
    Hence, by \cref{lemma:stable-peq-to-zero-iff-upd}, $\pi_n(A)$ is uniquely $p$-divisible for all $n$.
    Thus, $A \in \tpcon[1]{\Cat D}$.
    Thus, by definition of $\tpcocon{\Cat D}$ we know that $\Map{}(A, E) = 0$.
    For the bounded $p$-divisibility, suppose that $B$ is a $p$-divisible object of $\tstructheart{\Cat D}$.
    Then $B \in \tpcon[1]{\Cat D}$.
    Hence, $\Map{}(B, \pi_0(E)) \cong \Map{}(B, E) = 0$,
    where we used that $E \in \tcocon{\Cat D}$ and thus $\pi_0(E) \cong \tau_{\ge 0} E$ in the first equivalence,
    and that $E \in \tpcocon{\Cat D}$ in the second.

    For the other direction, assume that $E = \tau_{\le 0} E$, $E = \complete E$
    and that $\pi_0(E)$ has bounded $p$-divisiblity.
    Let $F \in \tpcon[1]{\Cat D}$.
    We need to show that $\Map{}(F, E) = 0$.
    But by assumption on $F$, $\tau_{\ge 0} F \to F$
    is a $p$-equivalence (see e.g. \cref{cor:stable-peq-iff-upd-fiber}) and $\pi_0(F)$ is $p$-divisible.
    Thus,
    \begin{align*}
        \Map{}(F, E) & \cong \Map{}(\tau_{\ge 0} F, E)       \\
                     & \cong \Map{}(\pi_0(F), E)            \\
                     & \cong \Map{}(\pi_0(F), \pi_0(E)) = 0.
    \end{align*}
    The first equivalence holds because $E$ is $p$-complete and the second exists because $E$ is coconnective.
    The third follows because $\pi_0(F)$ is connective.
    The last equality holds because $\pi_0(E)$ has bounded $p$-divisibility and $F \in \tpcon[1]{\Cat D}$.
\end{proof}

\begin{lem} \label{lemma:t-struct:left-complete}
    We have
    \begin{equation*}
        \bigcap_n \tpcocon[n]{\Cat D} = 0
    \end{equation*}
    and
    \begin{equation*}
        \bigcap_n \tpcon[n]{\Cat D} = \set{E \in \Cat D}{\pi_n(E) \text{ uniquely } p\text{-divisible for all } n}.
    \end{equation*}
    In particular, we have that $\pi_n^p(E) = 0$ for all $n$ if and only if
    $\pi_n(E)$ is uniquely $p$-divisible for all $n$.
\end{lem}
\begin{proof}
    Note that $\tpcocon[n]{\Cat D} \subset \tcocon[n]{\Cat D}$ by \cref{lemma:t-struct:char-of-cocon}.
    Hence, $\bigcap_n \tpcocon[n]{\Cat D} \subset \bigcap_n \tcocon[n]{\Cat D} = 0$ since
    $\tstruct{\Cat D}$ is right-separated.
    The second statement is clear since uniquely $p$-divisible abelian group objects
    are in particular $p$-divisible.
    For the last part, note that $\pi_n^p(E) = 0$ for all $n$
    if and only if $E$ lives in the stable subcategory of $\Cat{D}$ generated by
    $\bigcap_n \tpcon[n]{\Cat D}$ and $\bigcap_n \tpcocon[n]{\Cat D}$.
    But by the above, the latter is zero, and the former consists of exactly those spectra 
    which have uniquely $p$-divisible homotopy objects.
\end{proof}

\begin{cor} \label{lemma:t-struct:p-eq-iso-on-htpy-in-t-structure}
    Suppose that the standard t-structure is left-separated.
    Let $f \colon E \to F$ be a morphism in $\Cat D$.
    Then $f$ is a $p$-equivalence if and only if $\pi_n^p (E) \to \pi_n^p (F)$ is an isomorphism for all $n$.
\end{cor}
\begin{proof}
    From \cref{cor:stable-peq-iff-upd-fiber} we see that $f$ is a $p$-equivalence
    if and only if $\Fib{f}$ has uniquely $p$-divisible homotopy objects.
    Using \cref{lemma:t-struct:left-complete},
    this is equivalent to $\Fib{f} \in \bigcap_n \tpcon[n]{\Cat D}$.
    The long exact sequence now implies that this is the case if and only if
    $\pi_n^p(f)$ is an equivalence for all $n$.
\end{proof}

\begin{defn} \label{def:t-struct:Li}
    For every $n \in \Z$ define a functor $\mathbb{L}_n \colon \tstructheart{\Cat D} \to \tpstructheart{\Cat D}$
    via $A \mapsto \pi^p_n(A)$, i.e.\ the restriction of $\pi_n^p$ to the heart.
\end{defn}

\begin{defn}
    Let $\Cat A$ be an abelian category.
    Let $A \in \Cat A$, and $n \in \N$.
    Denote by $A[p^n] \coloneqq \ker(A \xrightarrow{p^n} A)$
    the \emph{$p^n$-torsion} of $A$.
\end{defn}

\begin{lem} \label{lemma:t-struct:pi1-bdd-divisible}
    Let $A \in \tcocon{\Cat D}$.
    Then $\pi_1(\complete{A}) \cong \limilheart{n} \pi_0(A)[p^n]$
    is of bounded $p$-divisibility.
    Here, the transition maps in the limit are multiplication by $p$.
\end{lem}
\begin{proof}
    Let $E \coloneqq \complete{A} \cong \limil{n} A \sslash p^n$.
    Note that $A \sslash p^n$ is $1$-truncated,
    with $\pi_1(A \sslash p^n) \cong \pi_0(A)[p^n]$.
    (This can be seen from the long exact sequence associated to the
    fiber sequence $A \xrightarrow{p^n} A \to A \sslash p^n$.)
    Since $\tau_{\ge 1}$ is a right adjoint, it commutes with limits.
    We now compute
    \begin{align*}
        \pi_1(E) & \cong \pi_1(\tau_{\ge 1} E)                         \\
                 & \cong \pi_1(\limil{n} \tau_{\ge 1} (A \sslash p^n)) \\
                 & \cong \pi_1(\limil{n} \Sigma (\pi_0(A)[p^n]))              \\
                 & \cong \pi_0(\limil{n} \pi_0(A)[p^n])                       \\
                 & = \limilheart{n} \pi_0(A)[p^n].
    \end{align*}

    In order to show that $\pi_1(E)$ has bounded $p$-divisibility,
    let $B \in \tstructheart{\Cat D}$ be $p$-divisible.
    We need to show that $\Map{}(B, \pi_1(E)) = 0$.
    By pulling out the limit (note that $\limilheart{n}$ is the categorical limit in $\tstructheart{\Cat D}$) we get
    \begin{equation*}
        \Map{\tstructheart{D}}(B, \pi_1(E)) \cong \limil{n} \Map{\tstructheart{D}}(B, \pi_0(A)[p^n]).
    \end{equation*}
    Thus, it suffices to show that for every $n$ we have $\Map{}(B, \pi_0(A)[p^n]) = 0$.
    So fix $n \ge 1$ and a map $\phi \colon B \to \pi_0(A)[p^n]$.
    Since $p^n \colon B \to B$ is an epimorphism ($B$ is $p$-divisible),
    in order to show that $\phi = 0$, it suffices show that $\phi \circ p^n = 0$.
    But $\phi \circ p^n = p^n \circ \phi$.
    Now we conclude by noting that the endomorphism $p^n \colon \pi_0(A)[p^n] \to \pi_0(A)[p^n]$ is zero.
\end{proof}

\begin{lem} \label{lemma:t-struct:Li-zero-if-upd}
    Let $A \in \tstructheart{\Cat D}$.
    If $A$ is uniquely $p$-divisible,
    then $\mathbb{L}_n A = 0$ for all $n$.
\end{lem}
\begin{proof}
    If $A$ is uniquely $p$-divisible, then $A \in \tpcon[k]{\Cat D}$ for all $k$.
    Hence, $\mathbb{L}_n A = \pi^p_n(A) = 0$ for all $n$.
\end{proof}

\begin{prop} \label{lemma:t-struct:Li-zero-if-i-neq-zero-one}
    Let $E \in \Cat D$.
    We have the following: 
    \begin{enumerate}[label=(\arabic*),ref=(\arabic*),itemsep=0em]
        \item If $E \in \tcocon{\Cat D}$, then $\complete{E} \in \tpcocon[1]{\Cat D}$, 
        \item if $E \sslash p \in \tcon{\Cat D}$, then $\complete{E} \in \tpcon{\Cat D}$,
        \item if $E \in \tcon{\Cat D}$, then $\complete{E} \in \tpcon{\Cat D}$, and
        \item if $E \in \tstructheart{\Cat D}$, then $\complete{E} \in \tpcon{\Cat D} \cap \tpcocon[1]{\Cat D}$.
    \end{enumerate}

    In particular, if $E \in \tstructheart{\Cat D}$,
    then $\mathbb L_n E = 0$ for all $n \neq 0, 1$.
\end{prop}
\begin{proof}
    We start with (1):
    We have seen in \cref{lemma:stable:truncation-of-completion} that $\pi_k(\complete{E}) = 0$ for all $k > 1$.
    $\complete{E}$ is $p$-complete by definition.
    By \cref{lemma:t-struct:pi1-bdd-divisible},
    we get that $\pi_1(\complete{E})$ is of bounded $p$-divisibility.
    Thus, $\complete{E} \in \tpcocon[1]{\Cat D}$ by \cref{lemma:t-struct:char-of-cocon}.

    We now prove (2):
    By \cref{lemma:t-struct:mod-p-in-con-implies-in-tpcon} we see that
    $\complete{E} \in \tpcon{\Cat D}$ if and only if $\complete{E} \sslash p \in \tcon{\Cat D}$.
    But $\complete{E} \sslash p \cong E \sslash p \in \tcon{\Cat D}$.

    Part (3) follows from (2), noting that $E \in \tcon{\Cat D}$ implies 
    that $E \sslash p \in \tcon{\Cat D}$, 
    since $\tcon{\Cat D}$ is stable under colimits, see \cite[Corollary 1.2.1.6]{higheralgebra}.

    Part (4) is an immediate consequence of (1) and (3).
    The last statement follows immediately from (4):
    \cref{lemma:t-struct:p-eq-iso-on-htpy-in-t-structure} implies that
    $\mathbb L_n E = \pi_n^p(E) \cong \pi_n^p(\complete{E})$,
    thus $\mathbb L_n(E) = 0$ for all $n \neq 0, 1$.
    This proves the lemma.
\end{proof}

\begin{lem} \label{lemma:t-struct:decomposition-via-Li}
    Let $A \in \tstructheart{\Cat D}$.
    Then there is a canonical fiber sequence
    \begin{equation*}
        \Sigma \mathbb L_1 A \to \complete{A} \to \mathbb L_0 A.
    \end{equation*}
\end{lem}
\begin{proof}
    \cref{lemma:t-struct:Li-zero-if-i-neq-zero-one}
    shows that $\complete{A} \in \tpcon{\Cat D} \cap \tpcocon[1]{\Cat D}$.
    Thus, using \cref{lemma:t-struct:p-eq-iso-on-htpy-in-t-structure},
    we conclude $\mathbb L_0 A \cong \pi_0^p(\complete{A}) \cong \tau_{\le 0}^p(\complete{A})$.
    Similar, we see $\Sigma \mathbb L_1(A) \cong \Sigma \pi_1^p(\complete{A}) \cong \tau_{\ge 1}^p(\complete{A})$.
    The lemma now immediately follows since we have a canonical fiber sequence
    \begin{equation*}
        \tau^p_{\ge 1} (\complete{A}) \to \complete{A} \to \tau^p_{\le 0} (\complete{A}).
    \end{equation*}
\end{proof}

\begin{lem} \label{lemma:t-struct:L1-mod-p}
    Let $A \in \tstructheart{\Cat D}$.
    Then $(\mathbb L_1 A) \sslash p \in \tstructheart{\Cat D}$
    and there is a short exact sequence in $\tstructheart{\Cat D}$
    \begin{equation*}
        0 \to (\mathbb L_1 A) \sslash p \to A[p] \to \pi_1((\mathbb L_0 A) \sslash p) \to 0,
    \end{equation*}
    coming from the fiber sequence of \cref{lemma:t-struct:decomposition-via-Li}.
\end{lem}
\begin{proof}
    Consider the fiber sequence
    \begin{equation*}
        \Sigma \mathbb L_1 A \to \complete{A} \to \mathbb L_0 A
    \end{equation*}
    from \cref{lemma:t-struct:decomposition-via-Li}.
    Applying $(-) \sslash p$ yields the fiber sequence
    \begin{equation*}
        \Sigma (\mathbb L_1 A) \sslash p \to \complete{A} \sslash p \to (\mathbb L_0 A) \sslash p.
    \end{equation*}
    Note that $\complete{A} \sslash p \cong A \sslash p$ is concentrated in degrees $0$ and $1$.
    Using \cref{lemma:t-struct:mod-p-in-con-implies-in-tpcon}
    we know that $\mathbb L_i A \sslash p \in \tcon{\Cat D}$ for $i = 0, 1$.
    Since $\mathbb L_0 A \in \tpcocon{\Cat D} \subset \tcocon{\Cat D}$,
    we conclude that $(\mathbb L_0 A) \sslash p \in \tcocon[1]{\Cat D}$.
    Now the long exact sequence in homotopy associated to the above fiber sequence yields
    that $\pi_i((\mathbb L_1 A) \sslash p) = 0$ for all $i \ge 1$.
    We therefore conclude that $(\mathbb L_1 A) \sslash p \in \tstructheart{\Cat D}$.
    The long exact sequence also gives us
    \begin{equation*}
        0 \to \pi_1 (\Sigma (\mathbb L_1 A) \sslash p) \to \pi_1 (A \sslash p) \to \pi_1 ((\mathbb L_0 A) \sslash p)\to 0.
    \end{equation*}
    We conclude by noting that $\pi_1 (\Sigma (\mathbb L_1 A) \sslash p) = \pi_0 ((\mathbb L_1 A) \sslash p) = (\mathbb L_1 A) \sslash p$
    and that $\pi_1 (A \sslash p) \cong A[p]$.
\end{proof}

\begin{lem} \label{lemma:short-exact-sequence-of-t-structure}
    Let $E \in \Cat D$ and $n \in \Z$.
    Then there is a short exact sequence
    \begin{equation*}
        0 \to \mathbb L_0 \pi_{n}(E) \to \mathbb \pi^p_n(E) \to \mathbb L_1 \pi_{n-1}(E) \to 0
    \end{equation*}
    natural in $E$.
\end{lem}
\begin{proof}
    Note that for any spectrum $F$ we have the following:
    If $F$ is $k$-connective, then $\pi^p_n(F) = 0$ for all $n < k$ ($\tcon[k]{\Cat D} \subseteq \tpcon[k]{\Cat D}$),
    and if $F$ is $k$-truncated, then $\pi^p_n(F) = 0$ for all $n > k+1$ (\cref{lemma:stable:truncation-of-completion}).

    Consider the fiber sequence
    \begin{equation*}
        \tau_{\ge n}E \to E \to \tau_{\le n-1}E.
    \end{equation*}
    This gives the following long exact sequence in $\tpstructheart{\Cat D}$:
    \begin{equation*}
        \pi_{n+1}^p(\tau_{\le n-1}E) \to \pi_n^p(\tau_{\ge n}E) \to \pi_n^p(E) \to \pi_n^p(\tau_{\le n-1}E) \to \pi_{n-1}^p(\tau_{\ge n}E).
    \end{equation*}
    Since $\tau_{\le n-1}E$ is $n-1$-truncated, we get that $\pi_{n+1}^p(\tau_{\le n-1}E) = 0$.
    Similarly, since $\tau_{\ge n}E$ is $n$-connective, we get that $\pi_{n-1}^p(\tau_{\ge n}E) = 0$.
    Thus, we arrive at a short exact sequence
    \begin{equation*}
        0 \to \pi_n^p(\tau_{\ge n}E) \to \pi_n^p(E) \to \pi_n^p(\tau_{\le n-1}E) \to 0.
    \end{equation*}
    Now consider the fiber sequence
    \begin{equation*}
        \Sigma^{n-1} \pi_{n-1} E \to \tau_{\le n-1}E \to \tau_{\le n-2}E,
    \end{equation*}
    which induces the following long exact sequence in $\tpstructheart{\Cat D}$:
    \begin{equation*}
        \pi^p_{n+1}(\tau_{\le n-2}E) \to \pi^p_{n}(\Sigma^{n-1} \pi_{n-1}E) \to \pi^p_n (\tau_{\le n-1}E) \to \pi^p_n(\tau_{\le n-2}E).
    \end{equation*}
    Again, since $\tau_{\le n-2}E$ is $n-2$-truncated, the outer terms vanish, and we
    are left with an isomorphism $\pi^p_n (\tau_{\le n-1}E) \cong \pi^p_{n}(\Sigma^{n-1} \pi_{n-1}E) \cong \pi^p_1 (\pi_{n-1}E) = \mathbb{L}_1(\pi_{n-1}E)$.

    Similarly, we can consider the fiber sequence
    \begin{equation*}
        \tau_{\ge n+1}E \to \tau_{\ge n}E \to \Sigma^n\pi_n(E),
    \end{equation*}
    which induces the following long exact sequence in $\tpstructheart{\Cat D}$:
    \begin{equation*}
        \pi^p_{n} (\tau_{\ge n+1}E) \to \pi^p_n(\tau_{\ge n}E) \to \pi^p_n (\Sigma^n\pi_n(E)) \to \pi^p_{n-1} (\tau_{\ge n+1}E).
    \end{equation*}
    Now $\tau_{\ge n+1}E$ is $n+1$-connective, so the outer terms vanish, and we are left with and isomorphism
    $\pi^p_n(\tau_{\ge n}E) \cong \pi^p_n (\Sigma^n\pi_n(E)) = \pi^p_0(\pi_n(E)) = \mathbb{L}_0(\pi_n(E))$.

    Plugging those isomorphisms into the short exact sequence from the beginning, we get a short exact sequence
    \begin{equation*}
        0 \to \mathbb L_0 \pi_{n}(E) \to \mathbb \pi^p_n(E) \to \mathbb L_1 \pi_{n-1}(E) \to 0.
    \end{equation*}
\end{proof}

\begin{cor} \label{cor:t-struct:dependence-on-truncation-or-cover}
    Let $E \in \Cat D$ and $n \in \Z$.
    We have equivalences
    $\pi_n^p(E) \cong \pi_n^p(\tau_{\ge k} E) \cong \pi_n^p(\tau_{\le l} E)$
    for all $k \le n-1$ and all $l \ge n$.
\end{cor}
\begin{proof}
    This follows immediately from \cref{lemma:short-exact-sequence-of-t-structure}.
\end{proof}

\begin{cor} \label{cor:recognizing-stable-p-equivalences}
    Suppose that the standard t-structure is left-separated.
    Let $f \colon E \to F$ be a map in $\Cat D$.
    If $f$ induces isomorphisms
    $\mathbb{L}_i\pi_n(E) \to \mathbb{L}_i\pi_n(F)$ for all $n$ and $i = 0, 1$,
    then $f$ is a $p$-equivalence.
\end{cor}
\begin{proof}
    Combine \cref{lemma:short-exact-sequence-of-t-structure,lemma:t-struct:p-eq-iso-on-htpy-in-t-structure}.
\end{proof}

\subsection{Comparison Results}
In this section, we will compare the $p$-adic t-structures on different stable categories.
For this suppose that $\Cat D$ and $\Cat E$ are two presentable stable categories, satisfying the assumptions from the
beginning of the section, i.e.\ they both come equipped with accessible right-separated t-structures $\tstruct{\Cat D}$
and $\tstruct{\Cat E}$. We again call those t-structures the standard t-structures,
in contrast to the $p$-adic t-structures.

\begin{lem} \label{lemma:t-struct:right-exact-commutes-with-completion}
    Let $F \colon \Cat D \to \Cat E$ be an exact functor.
    Then $F$ preserves $p$-equivalences.

    If moreover $F$ commutes with sequential limits (e.g.\ if $F$ is a right adjoint functor), then $F$ commutes with $p$-completion,
    and in particular preserves $p$-complete objects.
\end{lem}
\begin{proof}
    Since $F$ is exact, it commutes with $\Cofib{- \xrightarrow{p} -}$.
    Thus, since $F$ sends equivalences to equivalences, it follows that $F$
    preserves $p$-equivalences.

    Suppose now that $F$ commutes with sequential limits.
    Let $X \in \Cat D$.
    Then we compute $\completebr{FX} \cong \limil{n} (FX) \sslash p^n \cong \limil{n} F(X \sslash p^n) \cong F(\limil{n} X \sslash p^n) \cong F(\complete{X})$.
\end{proof}

\begin{lem} \label{lemma:t-struct:conservative-detects-peq}
    Let $F \colon \Cat D \to \Cat E$ be an exact conservative functor.
    Then $F$ detects $p$-equivalences, i.e.\ for every $f \colon E \to F$ in $\Cat D$
    the following holds:
    If $F(f)$ is a $p$-equivalence, then $f$ is a $p$-equivalence.
\end{lem}
\begin{proof}
    Let $f \colon E \to F$ in $\Cat D$ a morphism 
    such that $F(f)$ is a $p$-equivalence, i.e.\ $F(f) \sslash p$ is an equivalence.
    Note that since $F$ is exact, we have $F(f) \sslash p \cong F(f \sslash p)$.
    Now since $F$ is conservative, we conclude that $f \sslash p$ is an equivalence,
    i.e.\ $f$ is a $p$-equivalence.
\end{proof}

\begin{lem} \label{lemma:t-struct:right-t-exact}
    Let $L \colon \Cat D \to \Cat E$ be an exact functor which is right t-exact for the standard t-structures.
    Then $L$ is right t-exact for the $p$-adic t-structures.
    If $L$ has a right adjoint $R$, then $R$ is left t-exact for the $p$-adic t-structures.
\end{lem}
\begin{proof}
    Suppose that $X \in \tpcon{\Cat D}$.
    \cref{lemma:t-struct:mod-p-in-con-implies-in-tpcon} implies that $X \sslash p \in \tcon{\Cat D}$.
    Since $L$ is exact and right t-exact for the standard t-structures, we also have $LX \sslash p \cong L(X \sslash p) \in \tcon{\Cat E}$.
    But this now implies that $LX \in \tpcon{\Cat E}$, again by \cref{lemma:t-struct:mod-p-in-con-implies-in-tpcon}.

    The last statement is a general fact about t-structures, see e.g.\ \cite[Proposition 1.3.17 (iii)]{BeilinsonFaisceauxPerverse} (note that
    in the reference, cohomological indexing is used).
\end{proof}

\begin{lem} \label{lemma:t-struct:conservative}
    Let $L \colon \Cat D \to \Cat E$ be an exact conservative functor which is t-exact for the standard t-structures.
    Suppose that $X \in \Cat D$ such that $LX \in \tpcon[n]{\Cat E}$ for some $n$.
    Then $X \in \tpcon[n]{\Cat D}$.
\end{lem}
\begin{proof}
    Suppose $X \in \Cat D$ such that $LX \in \tpcon[n]{\Cat E}$ for some $n$.
    \cref{lemma:t-struct:mod-p-in-con-implies-in-tpcon} implies that $L(X \sslash p) \cong LX \sslash p \in \tcon[n]{\Cat E}$.
    Using the same lemma, it suffices to show that $X \sslash p \in \tcon[n]{\Cat D}$.
    Therefore, the lemma follows from the following more general statement,
    that any $Y \in \Cat{D}$ with $LY \in \tcon[n]{\Cat E}$ already lives in $\tcon[n]{\Cat D}$.
    So suppose that we have such a $Y \in \Cat D$.
    Then the map $\tau_{\ge n} LY \to LY$ is an equivalence.
    By t-exactness of $L$ for the standard t-structures,
    $L$ commutes with connective covers,
    i.e.\ $L\tau_{\ge n} Y \cong \tau_{\ge n} LY$.
    Conservativity of $L$ implies that $\tau_{\ge n} Y \to Y$ is an equivalence, 
    i.e.\ $Y \in \tcon[n]{\Cat D}$.
\end{proof}

\section{Unstable \texorpdfstring{$p$}{p}-Completion in \texorpdfstring{$\infty$}{infinity}-Topoi}
\label{section:topos:main}
Let $\topos X$ be a presentable $\infty$-category \cite[Definition 5.5.0.1]{highertopoi}.
We will have to deal with pointed and unpointed objects.
Write $\topos X_*$ for the category of pointed objects,
i.e.\ the category $\topos X_{*/}$ of objects under the terminal object $*$.
The forgetful functor $\topos X_* \to \topos X$
has a left adjoint $(-)_+ \colon \topos X \to \topos X_*$
given on objects by the formula $X \mapsto X \sqcup *$.

Let $\spectra X$ be the stabilization of $\topos X$.
See \cite[Section 1.4.2]{higheralgebra} for a discussion of the
stabilization of $\infty$-categories.
We have an adjoint pair of functors
\begin{equation*}
    \Sus \colon \topos X_* \rightleftarrows \spectra{X} \colon \pLoop.
\end{equation*}
Write $\pSus \colon \topos X \to \spectra{X}$
for the composition $\Sus \circ (-)_+$.
Hence, this is left adjoint to $\Loop \colon \spectra{X} \to \topos X$,
which forgets about the basepoint of the infinite loop space.

There is an accessible right-separated t-structure $\tstruct{\spectra{X}}$ on $\spectra{X}$,
given by $\tcocon[-1]{\spectra{X}} = \set{E \in \spectra{X}}{\pLoop E \cong *}$, see \cref{lemma:stabilization:t-structure}.
We will call this t-structure the \emph{standard t-structure} on $\spectra{X}$.
Therefore we can apply the results from \cref{section:stable}.

\begin{rmk}
    Later in this section, we will only work in the situation 
    where $\topos X$ is an $\infty$-topos.
    But since the category of motivic spaces is not an $\infty$-topos,
    we have to make some definitions in this more general setting.
    
    Later, we will reduce statements about the $p$-completion of motivic spaces 
    to the easier case of $p$-completion in suitable $\infty$-topoi. 
\end{rmk}

\subsection{Definition of the \texorpdfstring{$p$}{p}-Completion Functor}
\label{section:presentable-completion}
In this section, $\topos X$ will always be a presentable $\infty$-category.
We will define the unstable $p$-completion functor on the category $\topos X$.
As in the stable case, the $p$-completion functor is a localization
along a suitable class of $p$-equivalences:

\begin{defn}
    Let $g \colon X \to Y$ be a morphism in $\topos X_*$.
    We say that $g$ is a \emph{$p$-equivalence (of pointed objects)} if $\Sus g$ is a $p$-equivalence.

    Similarly, if $g \colon X \to Y$ is a morphism in $\topos X$,
    we say that $g$ is a \emph{$p$-equivalence (of unpointed objects)}
    if $g_+ \colon X_+ \to Y_+$ is a $p$-equivalence of pointed objects,
    i.e.\ if $\pSus g$ is a $p$-equivalence.
\end{defn}

As the next lemma shows, the distinction between pointed and unpointed $p$-equivalences
does not matter:

\begin{lem} \label{lemma:topos:unpointed-pointed:equivalence}
    Let $g \colon X \to Y$ be a morphism in $\topos X_*$.
    Then $g$ is a $p$-equivalence of pointed objects
    if and only if $g$ is a $p$-equivalence of the underlying unpointed objects.
\end{lem}
\begin{proof}
    We need to prove that $\Sus g$ is a $p$-equivalence if and only if
    $\pSus g$ is a $p$-equivalence.

    Note that we have natural cofiber sequences in $\topos X_*$ for every $X \in \topos X_*$:
    First we have the inclusion of the basepoint $\eta_X \colon * \to X$.
    This induces a morphism $\eta_{X,+} \colon *_+ \to X_+$.
    Second, we have the counit $c_X \colon X_+ \to X$.
    Both constructions are natural in $\topos X_*$.
    We claim that $*_+ \xrightarrow{\eta_{X,+}} X_+ \xrightarrow{c_X} X$ is a cofiber sequence.
    Consider the following diagram:
    \begin{center}
        \begin{tikzcd}
            * \arrow[r] \arrow[d, "\eta_X"] &*_+ \arrow[r, "c_*"] \arrow[d, "\eta_{X,+}"] & * \arrow[d, "\eta_X"] \\
            X \arrow[r] &X_+ \arrow[r, "c_X"] &X.
        \end{tikzcd}
    \end{center}
    The left horizontal arrows are the natural inclusions.
    The left square is clearly cocartesian.
    Since this is a retract diagram, the outer rectangle is also cocartesian.
    Thus, also the right square is cocartesian, see \cite[Lemma 4.4.2.1]{highertopoi}.
    In other words, the above sequence is a cofiber sequence.

    Since the sequence is natural in $\topos X_*$, we get a morphism of cofiber sequences in $\topos X_*$:
    \begin{center}
        \begin{tikzcd}
            *_+ \arrow[r, "\eta_{X,+}"] \arrow[d, equal] &X_+ \arrow[r, "c_X"] \arrow[d, "f_+"] & X \arrow[d, "f"] \\
            *_+ \arrow[r, "\eta_{Y,+}"] &Y_+ \arrow[r, "c_Y"] &Y.
        \end{tikzcd}
    \end{center}
    Since $\Sus$ and $(-) \sslash p$ commute with colimits (as $\Sus$ is left adjoint to $\pLoop$),
    and since $\pSus = \Sus \circ (-)_+$, we get
    a morphism of cofiber sequences
    \begin{center}
        \begin{tikzcd}
            \pSus * \sslash p \arrow[r, "\eta_{X,+}"] \arrow[d, equal] &\pSus X \sslash p\arrow[r, "c_X"] \arrow[d, "\pSus f \sslash p"] & \Sus X \sslash p \arrow[d, "\Sus f \sslash p"] \\
            \pSus * \sslash p \arrow[r, "\eta_{Y,+}"] &\pSus Y \sslash p \arrow[r, "c_Y"] &\Sus Y \sslash p.
        \end{tikzcd}
    \end{center}
    Taking cofibers of the vertical maps, we get a cofiber sequence
    \begin{equation*}
        0 \to \Cofib{\pSus f \sslash p} \to \Cofib{\Sus f \sslash p}.
    \end{equation*}
    Hence, $\Cofib{\pSus f \sslash p} \cong \Cofib{\Sus f \sslash p}$.
    Thus, $\Cofib{\pSus f \sslash p} = 0$ if and only if $\Cofib{\Sus f \sslash p} = 0$.
    This proves that $\pSus f$ is a $p$-equivalence if and only if $\Sus f$ is a $p$-equivalence.
\end{proof}

\begin{defn}
    We say that $X \in \topos X$ is \emph{(unpointed) $p$-complete} if every
    $p$-equivalence of unpointed objects $f \colon Y \to Y'$ induces on mapping spaces an equivalence
    $\Map{\topos X}(Y', X) \to \Map{\topos X}(Y, X)$.
    Denote by $\complete{\topos X}$ the full subcategory of $p$-complete objects.

    Similarly, we say that a pointed object $X \in \topos X_*$ is \emph{(pointed) $p$-complete}
    if every $p$-equivalence of pointed objects $f \colon Y \to Y'$ induces an equivalence
    $\Map{\topos X_*}(Y', X) \to \Map{\topos X_*}(Y, X)$.
    We write $\complete{\topos X_*}$ for the full subcategory of $p$-complete objects.
\end{defn}

Again, this distinction between pointed and unpointed objects does not matter:

\begin{lem}
    Let $X \in \topos X_*$.
    Then $X$ is pointed $p$-complete if and only if the underlying unpointed object is
    unpointed $p$-complete.
\end{lem}
\begin{proof}
    Suppose that the underlying unpointed object is unpointed $p$-complete.
    Let $f \colon Z \to Z'$ be a $p$-equivalence of pointed objects.
    Consider the following commutative cube:
    \begin{center}
        \begin{tikzcd}
            \Map{\topos X_*}(Z', X) \arrow[rd, "f^*"] \arrow[rr] \arrow[dd, hook] &&* \arrow[rd, equal] \arrow[dd] &\\
            &\Map{\topos X_*}(Z, X) \arrow[rr] \arrow[dd, hook] &&* \arrow[dd] \\
            \Map{\topos X}(Z', X) \arrow[rd, "f^*"] \arrow[rr]&&\Map{\topos X}(*, X) \arrow[rd, equal] &\\
            &\Map{\topos X}(Z, X) \arrow[rr]&&\Map{\topos X}(*, X).
        \end{tikzcd}
    \end{center}
    Here, the vertical maps $* \to \Map{\topos X}(*, X)$ select the map $* \to X$ given by the pointing of $X$.
    The horizontal map $\Map{\topos X}(Z, X) \to \Map{\topos X}(*, X)$ 
    is given by precomposition with the basepoint $* \to Z$, and similarly for $Z'$. 
    Note that the front and back squares are cartesian by definition of $\topos X_*$.
    Thus, since $f^* \colon \Map{\topos X}(Z', X) \to \Map{\topos X}(Z, X)$
    is an equivalence by assumption, also the map 
    $f^* \colon \Map{\topos X_*}(Z', X) \to \Map{\topos X_*}(Z, X)$
    is an equivalence.
    This proves that $X$ is pointed $p$-complete.

    For the other direction, we have to show that a $p$-equivalence of unpointed objects $g \colon Z \to Z'$
    induces an equivalence $\Map{\topos X}(Z', X) \to \Map{\topos X}(Z, X)$.
    By definition, $g_+$ is a $p$-equivalence of pointed objects.
    This implies that the induced map
    $\Map{\topos X_*}(Z'_+, X) \to \Map{\topos X_*}(Z_+, X)$
    is an equivalence, since $X$ was assumed to be pointed $p$-complete.
    But this gives
    \begin{equation*}
        \Map{\topos X}(Z', X) \cong \Map{\topos X_*}(Z'_+, X) \cong \Map{\topos X_*}(Z_+, X) \cong \Map{\topos X}(Z, X),
    \end{equation*}
    using that $(-)_+$ is left adjoint to the forgetful functor.
    In other words, $X$ is unpointed $p$-complete.
\end{proof}

In view of the last lemmas, being a $p$-equivalences or being $p$-complete
is independent of a choice of basepoint.
Below, we will use this without reference.

\begin{lem} \label{lemma:topos:p-equiv-strongly-saturated}
    The collection of $p$-equivalences in $\topos X$ (resp.\ in $\topos X_*$)
    is strongly saturated and of small generation.
\end{lem}
\begin{proof}
    Write $S$ for the class of $p$-equivalences in $\topos X$.
    Using \cite[Proposition 5.5.4.16]{highertopoi}, it suffices to show
    that $S = f^{-1}(S')$ for some colimit-preserving functor $f$
    and a strongly saturated class $S'$ of small generation.
    Then let $f = \pSus(-)$, and $S'$ be the collection of
    $p$-equivalences in $\spectra{X}$.
    $S'$ is strongly saturated and of small generation by \cref{lemma:stable:small-generation}.

    In the pointed case, on argues in the same way, using the functor $f = \pSus$.
\end{proof}

\begin{lem} \label{lemma:topos:completion-functor}
    The inclusion $\complete{\topos X} \to \topos X$ has a left adjoint
    $\completebr{-} \colon \topos X \to \complete{\topos X}$.
    We call this functor the $p$-completion functor.

    Similarly, the inclusion $\complete{\topos X_*} \to \topos X_*$
    has a left adjoint, which we also denote by $\completebr{-}$.
\end{lem}
\begin{proof}
    This is an application of \cite[Proposition 5.5.4.15]{highertopoi},
    using \cref{lemma:topos:p-equiv-strongly-saturated}.
\end{proof}

As in the stable case, the theory of Bousfield localizations gives us the following characterization of $p$-equivalences:
\begin{lem} \label{lemma:topos:peq-if-completion-eq}
    Let $f \colon X \to Y$ be a morphism in $\topos X$ (resp. $\topos X_*$).
    Then $f$ is a $p$-equivalence if and only if $\complete{f}$ is an equivalence.
\end{lem}
\begin{proof}
    This follows from \cite[Proposition 5.5.4.15 (4)]{highertopoi},
    where we use that the class of $p$-equivalences is strongly saturated, see \cref{lemma:topos:p-equiv-strongly-saturated}.
\end{proof}

\begin{lem} \label{lemma:topos:limit-of-p-complete}
    Let $I$ be a small $\infty$-category and
    $(X_i)_i$ an $I$-indexed diagram in $\topos X$.
    Suppose that $X_i$ is $p$-complete for each $i \in I$.
    Then $\limil{i \in I} X_i$ is $p$-complete.
    In particular, $* \in \topos X$ is $p$-complete.

    The same is true for limits in $\topos X_*$.
\end{lem}
\begin{proof}
    The inclusion $\complete{\topos X} \to \topos X$
    is a right adjoint by \cref{lemma:topos:completion-functor},
    hence it commutes with limits.
    The final object $*$ is the limit over the empty diagram, hence it is $p$-complete.

    For the pointed case, we can use the same proof, or note that
    $\topos X_*$ is presentable by \cite[Proposition 5.5.3.11]{highertopoi}.
    Thus, we can apply the above result to the presentable $\infty$-category $\topos X_*$.
\end{proof}

\begin{cor}\label{cor:loop-space-of-complete-is-complete}
    Let $X \in \topos X_*$ be $p$-complete.
    Then $\Omega X$ is $p$-complete.
\end{cor}
\begin{proof}
    $\Omega X$ is the limit of the diagram $* \rightarrow X \leftarrow *$.
    Since $X$ is $p$-complete by assumption, and $*$ is $p$-complete by \cref{lemma:topos:limit-of-p-complete},
    we conclude that $\Omega X$ is $p$-complete as a limit of $p$-complete objects
    (again by \cref{lemma:topos:limit-of-p-complete}).
\end{proof}

\begin{lem} \label{lemma:peq-via-points}
    Let $\topos{Y}_i$ be a collection of presentable $\infty$-categories.
    Suppose $s_i^* \colon \topos X \rightleftarrows \topos{Y}_i \colon s_{i,*}$ are adjunctions.
    Let $f \colon X \to X'$ be a morphism in $\topos X$.
    If $f$ is a $p$-equivalence, so is $s_i^* f$ for every $i$.
    The converse holds if the $s_i^*$ form a conservative family of functors, 
    and all of the $s_i^*$ are left-exact (i.e.\ commute with finite limits).

    In particular, if $\topos X$ is an $\infty$-topos with enough points,
    then $f$ is a $p$-equivalence if and only if it is a $p$-equivalence on stalks.
\end{lem}
\begin{proof}
    Using \cref{lemma:adjoints-on-stabilization},
    we see that the $s_i^* \dashv s_{i,*}$ induce exact functors on the stabilizations,
    such that the following diagram of functors commutes:
    \begin{center}
        \begin{tikzcd}
            \spectra{X} \arrow[r, "s_i^*"] &\Stab{\topos{Y}_i} \\
            \topos{X}_* \arrow[u, "\Sus"] \arrow[r, "s_i^*"] &\topos{Y}_{i,*} \arrow[u, "\Sus"]
        \end{tikzcd}
    \end{center}
    If the $s^*_i$ are left-exact, then 
    the functors on stabilizations are jointly conservative if the corresponding family of functors on $\topos X$ is,
    see \cref{lemma:stabilization:geometric-morphism-commutes-with-loopspace}.
    The lemma follows from \cref{lemma:t-struct:right-exact-commutes-with-completion,lemma:t-struct:conservative-detects-peq}.
\end{proof}

\subsection{Basic Properties of Unstable \texorpdfstring{$p$}{p}-Completion}
\label{section:topos}

From now on, we will assume that $\topos X$ is actually an $\infty$-topos \cite[Definition 6.1.0.4]{highertopoi},
it is in particular presentable \cite[Theorem 6.1.0.6]{highertopoi}.
If $\topos X$ is hypercomplete (see the discussion directly before \cite[Remark 6.5.2.11]{highertopoi}),
then the standard t-structure is left-separated:
If $E \in \spectra{X}$ is $\infty$-connective, then $\pLoop \Sigma^n E$ is $\infty$-connective for every $n$.
By hypercompleteness, we conclude $\pLoop \Sigma^n E \cong *$ for all $n$.
But this implies that $E \cong 0$, in other words, the t-structure is left-separated.

Write $\Disc{\topos X}$ for the category of discrete objects in $\topos X$,
i.e.\ the essential image of the truncation functor $\tau_{\le 0} \colon \topos X \to \topos X$.
This is an ordinary 1-topos. Write $\AbObj{\Disc{\topos X}}$ for the category of
abelian group objects in $\Disc{\topos X}$.
Note that there is an equivalence $\tstructheart{\spectra{X}} \cong \AbObj{\Disc{\topos X}}$
from the heart of the t-structure to the category of abelian group objects in $\topos X$, see \cite[Proposition 1.3.2.7 (4)]{sag}.
We will identify these two categories.
In particular, for $n \ge 2$ we will regard the homotopy object functors $\pi_n \colon \topos X \to \AbObj{\Disc{\topos X}}$ 
as functors $\pi_n \colon \topos X \to \tstructheart{\spectra{X}}$.

There is a symmetric monoidal structure $\otimes$ on $\spectra{X}$, see \cite[Proposition 1.3.4.6]{sag}.
Moreover, $\otimes$ is exact (and moreover cocontinuous) in each variable.
Note that $\pSus$ admits the structure of a symmetric monoidal
functor from $\topos{X}$ with the cartesian structure to $\spectra{X}$ with $\otimes$,
see again \cite[Proposition 1.3.4.6]{sag}.

\begin{lem}\label{lemma:peq-respects-pi0}
    Let $f \colon X \to Y$ be a $p$-equivalence in $\topos X$.
    Then $\pi_0(f) \colon \pi_0(X) \to \pi_0(Y)$ is an equivalence.
\end{lem}
\begin{proof}
    Consider the following diagram:
    \begin{center}
        \begin{tikzcd}
            \topos X \arrow[r, "\pSus"] \arrow[d, "\pi_0"] &\spectra{X} \arrow[r, "\tau_{\ge 0}"] &\tcon{\spectra{X}} \arrow[d, "\tau_{\le 0}"] \\
            \Disc{\topos X} \arrow[r, "{\Z[-]}"] & \AbObj{\Disc{\topos X}} \arrow[r, "\cong"] &\tstructheart{\spectra{X}},
        \end{tikzcd}
    \end{center}
    where $\Z[-]$ is the left adjoint to the forgetful functor $\AbObj{\Disc{\topos X}} \to \Disc{\topos X}$.
    This functor exists since all categories are presentable, and the forgetful functor commutes with limits and filtered colimits.
    The diagram commutes:
    We can see this by uniqueness of adjoints:
    Note that $\Z[\pi_0(-)]$ is left adjoint to the forgetful functor $\AbObj{\Disc{\topos X}} \to \topos X$,
    and $\tau_{\le 0} \tau_{\ge 0} \pSus$ is left adjoint to $\Loop \colon \tstructheart{\spectra{X}} \to \topos X$
    (note that $\pSus$ actually factors over $\tcon{\spectra{X}}$).
    But these two right adjoint functors agree under the identification $\AbObj{\Disc{\topos X}} \cong \heart{\spectra{X}}$.

    We can enlarge the diagram to the following:
    \begin{center}
        \begin{tikzcd}
            \topos X \arrow[r, "\pSus"] \arrow[d, "\pi_0"] &\tcon{\spectra{X}} \arrow[d, "\tau_{\le 0}"] \arrow[r, "{(-) \sslash p}"] &\tcon{\spectra X} \arrow[d, "\tau_{\le 0}"] \\
            \Disc{\topos X} \arrow[r, "{\Z[-]}"] &\tstructheart{\spectra{X}} \arrow[r, "{(-) / p}"] &\tstructheart{\spectra{X}},
        \end{tikzcd}
    \end{center}
    Here, $(-) / p$ is the functor given by $X \mapsto \coker(X \xrightarrow{p} X)$.
    We have seen above that the left square commutes. The commutativity of the right hand side can
    be easily seen from the long exact sequence.

    Since $f$ is a $p$-equivalence, $(\pSus f) \sslash p$ is an equivalence.
    This implies that $\Z[\pi_0(f)] / p$ is an isomorphism.
    Note that the functor $(\Z[-])/p$ can be identified with $\finfld{p}[-]$.
    Here, $\finfld{p}[-]$ is the left adjoint to the forgetful functor from 
    $p$-torsion abelian group objects (i.e. sheaves of $\finfld{p}$-vectorspaces) 
    in $\Disc{\topos X}$ to $\Disc{\topos X}$.
    Note that this functor is conservative, see \cref{prop:free:conservativity}.
    This implies that $\pi_0(f) \colon \pi_0(X) \to \pi_0(Y)$ is an isomorphism.
\end{proof}

\begin{lem} \label{lemma:discrete-objects-complete}
    Let $D \in \topos X$ a discrete space.
    Then $D$ is $p$-complete.
\end{lem}
\begin{proof}
    We need to show that $\Map{}(Y, D) \to \Map{}(X, D)$ is an equivalence
    for all $p$-equivalences $f \colon X \to Y$.
    But since $D$ is discrete, $\Map{}(Y, D) \cong \Map{}(\pi_0(Y), D)$
    and $\Map{}(X, D) \cong \Map{}(\pi_0(X), D)$.
    Thus, it suffices to show that $\pi_0(X) \to \pi_0(Y)$ is 
    an equivalence, which was proven in \cref{lemma:peq-respects-pi0}.
\end{proof}

\begin{cor} \label{cor:connected-cover-of-complete-is-complete}
    Let $X \in \topos X_*$ be $p$-complete.
    Then $\tau_{\ge 1}X$ is $p$-complete.
\end{cor}
\begin{proof}
    There is a fiber sequence $\tau_{\ge 1}X \to X \to \tau_{\le 0}X$.
    But $X$ is $p$-complete by assumption, and $\tau_{\le 0}X$ is $p$-complete because it is discrete,
    see \cref{lemma:discrete-objects-complete}.
    Thus, $\tau_{\ge 1}X$ is $p$-complete as a limit of $p$-complete objects, see \cref{lemma:topos:limit-of-p-complete}.
\end{proof}

\begin{lem} \label{lemma:topos:completion-products}
    Let $f_i \colon X_i \to Y_i$ be $p$-equivalences in $\topos X$ for $i = 1, \dots, n$.
    Then $\prod_i f_i \colon \prod_i X_i \to \prod_i Y_i$ is a $p$-equivalence,
    and hence $\completebr{\prod_i X_i} \cong \prod_i \complete{X_i}$.
\end{lem}
\begin{proof}
    We need to show that $\pSus (\prod_i f_i) \cong \bigotimes_i (\pSus f_i)$
    is a $p$-equivalence of spectra.
    This follows immediately from \cref{lemma:stable:p-eq-of-spectra-smash-product}.
    For the last point, it suffices to note that the canonical maps $X_i \to \complete{X_i}$
    are $p$-equivalences, and that $\prod_i \complete{X_i}$ is $p$-complete
    as a limit of $p$-complete objects, see \cref{lemma:topos:limit-of-p-complete}.
\end{proof}

\subsection{Completions via Postnikov-towers}
Suppose from now on that $\topos X$ has enough points, see \cite[Remark 6.5.4.7]{highertopoi}.
In particular, $\topos X$ is hypercomplete (again \cite[Remark 6.5.4.7]{highertopoi}).

\begin{lem} \label{lemma:loop-peq-iff-peq}
    Let $f \colon E \to F$ be a $p$-equivalence in $\spectra{X}$,
    with $E$ and $F$ $1$-connective.
    Then $\pLoop f \colon \pLoop E \to \pLoop F$ is a $p$-equivalence.
\end{lem}
\begin{proof}
    Since $\topos X$ has enough points and $\pLoop$ commutes with points
    (see \cref{lemma:stabilization:geometric-morphism-commutes-with-loopspace}),
    this statement can be checked on stalks, see \cref{lemma:peq-via-points}.
    Thus, the lemma follows from the corresponding statement about anima,
    see \cref{lemma:anima:infinite-loop-spaces-preserves-peq}.
\end{proof}

\begin{lem} \label{lemma:connective-inf-loop-peq-to-connective-cover}
    Let $E \in \spectra{X}$ such that $E$ is $k$-connective for some $k \ge 1$.
    Then $\pLoop E \to \pLoop \tau_{\ge k} (\complete E)$ is a $p$-equivalence.
    Moreover, $\completebr{\pLoop E} \cong \pLoop \tau_{\ge 1} (\complete E)$.
\end{lem}
\begin{proof}
    By the last \cref{lemma:loop-peq-iff-peq}, it is enough to show
    that $E \to \tau_{\ge k} \complete E$ is a $p$-equivalence.
    But $E \to \complete E$ is a $p$-quivalence, and since $E$ is $k$-connective,
    we conclude that $\pi_n(\complete E)$ is uniquely $p$-divisible for all $n < k$,
    see \cref{lemma:upd-homotopy-of-completion-below-connectivity-of-base}.
    Thus, $\tau_{<k} \complete E$ has uniquely $p$-divisible homotopy objects, and it follows that
    $\tau_{\ge k}\complete E \to \complete E$ is a $p$-equivalence,
    see \cref{cor:stable-peq-iff-upd-fiber}.
    Since $E \to \complete E$ is a $p$-equivalence, we conclude by 2-out-of-3
    (the class of $p$-equivalences is strongly saturated by \cref{lemma:topos:p-equiv-strongly-saturated}).

    For the last part, note that $\pLoop E \to \pLoop \tau_{\ge 1} (\complete E)$
    is a $p$-equivalence by the first part(since a $k$-connective spectrum
    is in particular $1$-connective).
    Thus, it suffices to show that $\pLoop \tau_{\ge 1} (\complete E)$ is $p$-complete.
    But we have an equivalence $\pLoop \tau_{\ge 1} (\complete E) \cong \tau_{\ge 1} \pLoop (\complete E)$.
    Since $\pLoop$ preserves $p$-complete objects
    (as a right adjoint to $\Sus$, which preserves $p$-equivalences),
    we conclude by \cref{cor:connected-cover-of-complete-is-complete}.
\end{proof}

\begin{cor} \label{cor:completion-of-EM-space}
    Let $K = K(A, n)$ be an Eilenberg-MacLane object in $\topos X_*$ with $n \ge 1$
    and $A \in \tstructheart{\spectra X}$.
    Then $\complete{K} \cong \pLoop\tau_{\ge 1}(\completebr{\Sigma^n A}) \cong \tau_{\ge 1}\pLoop(\completebr{\Sigma^n A})$.
    In particular, $\complete{K}$ is connected and $n+1$-truncated, and $\pi_i(\complete K)$ is abelian and uniquely $p$-divisible
    for all $1 \le i < n$.
\end{cor}
\begin{proof}
    Note that $K = \pLoop\Sigma^n A$.
    Thus, the result follows immediately from
    \cref{lemma:connective-inf-loop-peq-to-connective-cover,lemma:stable:truncation-of-completion,lemma:upd-homotopy-of-completion-below-connectivity-of-base}.
\end{proof}

In \cref{appendix:nilpotent} (in particular in \cref{def:nilpotent:defn}), we will define what a nilpotent object $X \in \topos X_*$ is.
Nilpotent objects have the property, that their Postnikov tower can be built by repeatedly
building in an Eilenberg-MacLane space $K(A, n)$,
see \cref{def:nilpotent:principal-refinement,lemma:nilpotent:principal-refinement}.
This allows one to prove statements about nilpotent objects by induction over the
(refined) Postnikov tower, and from the corresponding statement about Eilenberg-MacLane objects.

\begin{prop} \label{lemma:fiber-lemma}
    Let $f \colon X \to Y \in \topos X_*$ be a morphism of pointed nilpotent spaces,
    such that $\complete{X}$ and $\complete{Y}$ are also nilpotent.
    Then
    \begin{equation*}
        \completebr{\tau_{\ge 1} \Fib{X \xrightarrow{f} Y}} \cong \tau_{\ge 1} \Fib{\complete{X} \xrightarrow{\complete{f}} \complete{Y}}.
    \end{equation*}
\end{prop}
\begin{proof}
    The right-hand side is $p$-complete as the connected cover of a limit of $p$-complete spaces,
    see \cref{cor:connected-cover-of-complete-is-complete}.
    Thus, it suffices to show that the map $\tau_{\ge 1} \Fib{f} \to \tau_{\ge 1} \Fib{\complete{f}}$
    is a $p$-equivalence.
    This can be checked on stalks, see \cref{lemma:peq-via-points}. Since stalks preserve connected covers, nilpotent spaces, fibers and $p$-equivalences,
    this immediately follows from \cref{lemma:anima:refined-fiber-lemma},
    applied to the following diagram of fiber sequences of pointed anima (where $s$ is a point of $\topos X$)
    \begin{center}
        \begin{tikzcd}
            s^*\Fib{f} = \Fib{s^*f} \arrow[r] \arrow[d] & s^*X \arrow[r] \arrow[d] & s^*Y \arrow[d] \\
            s^*\Fib{\complete{f}} = \Fib{s^*(\complete{f})} \arrow[r] & s^*(\complete{X}) \arrow[r] & s^*(\complete{Y}),
        \end{tikzcd}
    \end{center}
    where the middle and right vertical maps are $p$-equivalences.
\end{proof}

\begin{prop} \label{lemma:postnikov-fiber-sequence-completion}
    Let $X \in \topos X_*$ be nilpotent and choose a principal refinement
    of the Postnikov tower as in \cref{lemma:nilpotent:principal-refinement}.
    Then for all $n \ge 1$ and all $1 \le k \le m_n$,
    $\completebr{X_{n, k}}$ is nilpotent and
    there is an equivalence
    \begin{equation*}
        \completebr{X_{n, k}} \cong \tau_{\ge 1} \Fib{\completebr{X_{n, k-1}} \to \complete{K(A_{n, k}, n + 1)}}.
    \end{equation*}
\end{prop}
\begin{proof}
    We prove the lemma by induction on $n$ and $k$, note that $X_{n, 0} \cong X_{n-1, m_n}$.
    Also note that $X_{1, 0} = * = \complete{*} = \completebr{X_{1, 0}}$ is nilpotent.

    $X_{n, k}$ is connected and fits into a fiber sequence of pointed spaces
    \begin{equation*}
        X_{n, k} \to X_{n, k-1} \to K(A_{n, k}, n+1).
    \end{equation*}
    $K(A_{n, k}, n+1)$ is nilpotent by \cref{lemma:nilpotent:loopspace} and
    $\completebr{X_{n, k-1}}$ is nilpotent by induction.
    Moreover, by \cref{cor:completion-of-EM-space} there is an equivalence
    $\completebr{K(A_{n, k}, n+1)} \cong \pLoop(\tau_{\ge 1} \completebr{\Sigma^{n+1} A_{n, k}})$,
    which is thus also nilpotent by \cref{lemma:nilpotent:loopspace}.
    We conclude by \cref{lemma:fiber-lemma}
    that $\completebr{X_{n, k}} \cong \tau_{\ge 1} \Fib{\completebr{X_{n, k-1}} \to \complete{K(A_{n, k}, n + 1)}}$.
    Note that $\completebr{X_{n, k}}$ is now nilpotent as the connected cover of a fiber of nilpotent spaces,
    see \cref{lemma:nilpotent:fiber}.
\end{proof}

\begin{prop} \label{cor:topos:truncation-of-completion}
    Let $X \in \topos X_*$ be nilpotent and $n$-truncated for some $n \in \Z$.
    Then $\complete{X}$ is $(n+1)$-truncated.
\end{prop}
\begin{proof}
    Choose a principal refinement $X_{m, k}$
    of the Postnikov tower, which is possible by \cref{lemma:nilpotent:principal-refinement}.
    Since $X$ is $n$-truncated, we see that $X = X_{n, 0}$.
    We proceed by induction on $m$ and $k$ as in the proof of \cref{lemma:postnikov-fiber-sequence-completion}.
    Note that $\completebr{X_{1, 0}} = \completebr{*} = *$ is clearly $(n+1)$-truncated.
    So suppose that $1 \le m < n$ and $1 \le k \le m_m$ and that $\completebr{X_{m, k-1}}$ is $(n+1)$-truncated.
    Now we have a fiber sequence
    \begin{equation*}
        \completebr{X_{m, k}} \cong \tau_{\ge 1} \Fib{\completebr{X_{m, k-1}} \to \complete{K(A_{m, k}, m + 1)}}
    \end{equation*}
    from \cref{lemma:postnikov-fiber-sequence-completion}.
    Since $(n+1)$-truncated objects are closed under limits (see \cite[Proposition 5.5.6.5]{highertopoi}),
    we conclude from the induction hypothesis and \cref{cor:completion-of-EM-space}
    that $\completebr{X_{m, k}}$ is $(n+1)$-truncated.
    If $m = n$, then the Postnikov tower stabilizes,
    and we conclude that $\complete X = \completebr{X_{n, 0}} = \completebr{X_{n-1, m_n}}$
    is $(n+1)$-truncated.
\end{proof}

Suppose now that $\topos X = \Shv{\Cat T}$ is the category of hypercomplete
sheaves on $\Cat T$ where $\Cat T$ is a Grothendieck site.

\begin{defn} \label{def:topos:locally-finite-uniform-htpy-dim}
    We say that $\topos X$ is \emph{locally of finite uniform homotopy dimension}
    if there is
    \begin{itemize}
        \item a conservative family of points $\Cat S$ of $\topos X$,
        \item for every $s \in \Cat S$ a pro-object $\mathcal I_s$ in $\Cat T$
              such that $s^*F \cong \colimil{U \in \mathcal I_s} F(U)$ for every $F \in \topos X$, and
        \item a function $\htpydim \colon \Cat S \to \N$,
    \end{itemize}
    such that for all $s \in S$ every object $U \in \mathcal I_s$ has homotopy dimension $\htpydim(s)$,
    i.e.\ if $F \in \topos{X}$ is $k$-connective, then $F(U)$ is ($k$-$\htpydim(s)$)-connective.
\end{defn}

Suppose from now on that $\topos X$ is locally of finite uniform homotopy dimension,
and choose $\Cat S$, $\Cat I_s$ and $\htpydim$ as in \cref{def:topos:locally-finite-uniform-htpy-dim}.
In the rest of this section we show that then $p$-completion of nilpotent spaces can be computed
on the Postnikov tower.

\begin{lem} \label{lemma:topos:connectivity-drop-bounded-on-stalks}
    Let $s \in \Cat S$, $U \in \mathcal I_s$ and $E \in \spectra{X}$.
    Suppose that $E$ is $m$-connective.
    Then $\complete{E}(U)$ is ($m$-$\htpydim(s)$-$1$)-connective.
\end{lem}
\begin{proof}
    We may assume $m = 0$.
    Since $E\sslash p^n = \Cofib{E \xrightarrow{p^n} E}$ is also connective,
    it suffices to prove the more general fact that a sequential limit $F = \limil{n} F_n$ of
    connective spectra $F_n$ has the property
    that $(\limil{n} F_n)(U)$ is (-$\htpydim(s)$-1)-connective for all $U \in \mathcal I_s$.
    By assumption, $F_n(U)$ is (-$\htpydim(s)$)-connective for all $n$.
    But then $(\limil{n} F_n)(U) = \limil{n} F_n(U)$ is (-$\htpydim(s)$-1)-connective
    as a sequential limit of (-$\htpydim(s)$)-connective spectra
    (see e.g.\ \cite[Proposition 2.2.9]{MCAT}
    for the corresponding fact about anima,
    then shift the $F_n$ such that they are ($\htpydim(s)$ + $l$)-connective for some $l \ge 1$,
    and use that $\pLoop$ commutes with limits, and with homotopy objects in non-negative degrees).
\end{proof}

\begin{lem} \label{lemma:topos:htpy-groups-of-sequential-limits-of-anima}
    Let $X_k$ be an $\N$-indexed inverse system of connected anima.
    Suppose that for all $n \ge 0$,
    there exists a $k_n > 0$ such that $\pi_n(X_k) = \pi_n(X_{k_n})$ for all $k \ge k_n$.
    Then $\pi_n(\limil{k} X_k) \cong \heart{\limil{k}} \pi_n(X_k) \cong \pi_n(X_{k_n})$ for all $n$.
\end{lem}
\begin{proof}
    See e.g. \cite[Proposition 2.2.9]{MCAT}.
    Note that the $\limone{}$-term vanishes because
    the homotopy groups get eventually constant, and hence satisfy the Mittag-Leffler property.
    The last equivalence holds because the limit is eventually constant.
\end{proof}

\begin{lem} \label{lemma:topos:limit-commutes-with-homotopy-objects-if-eventually-constant-on-sections}
    Let $X_k$ be an $\N$-indexed inverse system of connected objects in $\topos X_*$.
    Suppose that for all $n, d \ge 0$ there exists a $k_{d, n} > 0$
    such that $\pi_n(X_k(U)) \cong \pi_n(X_{k_{\htpydim(s), n}}(U))$ for
    all $s \in \Cat S$, $k \ge k_{\htpydim(s), n}$ and $U \in \mathcal I_s$.
    Then $s^*\limil{k}X_k \cong \limil{k}s^*X_k$ for all point $s \in \Cat S$.
\end{lem}
\begin{proof}
    Fix a point $s \in \Cat S$.
    Note that for $k \ge k_{\htpydim(s), n}$ and $n \ge 0$ we have
    \begin{equation*}
        \pi_n s^* X_k \cong \colim{U \in \Cat I_s} \pi_n(X_k(U)) \cong \colim{U \in \Cat I_s} \pi_n (X_{k_{\htpydim(s), n}}(U)) \cong \pi_n s^*X_{k_{\htpydim(s), n}}.
    \end{equation*}
    \cref{lemma:topos:htpy-groups-of-sequential-limits-of-anima}
    implies (use $k_n = k_{\htpydim(s), n}$) that for every $n$ and $U \in \mathcal I_s$ we have isomorphisms
    \begin{align*}
        \pi_n(\limil{k} X_k(U)) & \cong \pi_n(X_{k_{\htpydim(s), n}}(U))  \\
        \pi_n(\limil{k} s^*X_k) & \cong \pi_n(s^*X_{k_{\htpydim(s), n}}).
    \end{align*}
    We now compute
    \begin{align*}
        \pi_n(s^* \limil{k} X_k) & \cong s^* \pi_n(\limil{k} X_k)                                        \\
                                 & \cong \colimil{U \in \mathcal I_s} \pi_n (\limil{k} X_k(U))           \\
                                 & \cong \colimil{U \in \mathcal I_s} \pi_n (X_{k_{\htpydim(s), n}} (U)) \\
                                 & \cong \pi_n (s^* X_{k_{\htpydim(s), n}})                              \\
                                 & \cong \pi_n(\limil{k} s^* X_k).
    \end{align*}
    Since $n$ was arbitrary, we conclude that $s^* \limil{k} X_k \cong \limil{k} s^*X_k$,
    using Whitehead's theorem.
\end{proof}

\begin{lem} \label{lemma:topos:sections-of-completions-of-truncations-are-uniform}
    Let $X \in \topos{X}_*$ be nilpotent, $s \in \Cat S$ be a point
    and $n \in \mathbb{N}$.
    Define $k_{\htpydim(s), n} \coloneqq n + \htpydim(s) + 2$.
    Then for all $U \in \mathcal I_s$
    we have that $\pi_{n}(\completebr{\tau_{\le k}X}(U))$
    is independent of $k$ for $k \ge k_{\htpydim(s), n}$.
\end{lem}
\begin{proof}
    Fix $n \in \N$ and $U \in \mathcal I_s$.
    We proceed by induction on $k$, the case $k = k_{\htpydim(s), n}$ holds tautologically.
    Using \cref{lemma:nilpotent:principal-refinement}, we find a principal refinement
    of the Postnikov tower.
    For every $1 \le l \le m_k$, there is an equivalence
    \begin{equation*}
        \completebr{X_{k,l}} \cong \tau_{\ge 1} \Fib{\completebr{X_{k,l-1}} \to \completebr{K(A_{k,l}, k + 1)}},
    \end{equation*}
    see \cref{lemma:postnikov-fiber-sequence-completion}.
    Thus, it is enough to show that $\completebr{K(A_{k,l}, k+1)}(U)$ is $n+2$-connective.
    Using \cref{cor:completion-of-EM-space},
    it suffices to prove that $\completebr{\Sigma^{k+1} A_{k,l}}(U)$
    is $n+2$-connective.
    Note that the connectivity of $\Sigma^{k+1} A_{k,l}$ is at least $k_{\htpydim(s), n} + 1 = n + \htpydim(s) + 3$.
    Using \cref{lemma:topos:connectivity-drop-bounded-on-stalks},
    we conclude that the connectivity of $\completebr{\Sigma^{k+1} A_{k,l}}(U)$
    is at least $n + \htpydim(s) + 3 - \htpydim(s) - 1 = n + 2$.
\end{proof}

\begin{thm} \label{thm:topos:p-comp-commutes-with-post-tower}
    Let $X \in \topos{X}_*$ be nilpotent.
    Then $\complete{X} \cong \limil{k} \completebr{\tau_{\le k} X}$.
\end{thm}
\begin{proof}
    The right-hand side is $p$-complete because it is a limit of $p$-complete
    objects. Hence, it suffices to show that $X \to \limil{k} \completebr{\tau_{\le k}X}$
    is a $p$-equivalence. This can be checked on stalks.
    So let $s \in \Cat S$ be a point, we need to show that
    $s^* X \to s^* \limil{k} \completebr{\tau_{\le k}X}$ is a
    $p$-equivalence.
    Using \cref{lemma:topos:limit-commutes-with-homotopy-objects-if-eventually-constant-on-sections}
    and \cref{lemma:topos:sections-of-completions-of-truncations-are-uniform},
    we conclude that $s^* \limil{k} \completebr{\tau_{\le k}X} \cong \limil{k} s^* (\completebr{\tau_{\le k}X})$.
    The left-hand side is $s^* X = \limil{k} \tau_{\le k} s^*X \cong \limil{k} s^* \tau_{\le k}X$,
    using that $\An$ is Postnikov-complete and that $s^*$ commutes with truncations,
    see \cite[Proposition 5.5.6.28]{highertopoi}.
    Note that $s^* \tau_{\le k}X \to s^* (\completebr{\tau_{\le k}X})$
    is a $p$-equivalence for each $k$.
    Hence, the result follows from \cref{lemma:anima:seq-limit-of-peq-of-anima}.
\end{proof}

\section{Completions via Embeddings}
\subsection{Completions of Presheaves} 
\label{section:embedding:presheaf}
Let $\Cat C$ be a small $\infty$-category.
For every $\infty$-category $\Cat D$, 
denote by $\PrShvVal{\Cat C}{\Cat D} \coloneqq \operatorname{Fun}({\Cat C}^{\operatorname{op}}, \Cat D)$
the category of presheaves with values in $\Cat D$.
Denote by $\PrShv{\Cat C} \coloneqq \PrShv{\Cat C, \An}$ the category of presheaves
(of anima) on $\Cat C$.
Recall that there is a canonical equivalence of categories
$\Stab{\PrShv{\Cat C}} \cong \PrShvVal{\Cat C}{\Sp}$,
see \cite[Remark 1.4.2.9]{higheralgebra}.

\begin{lem}
    $\PrShv{\Cat C}$ is locally of homotopy dimension $0$,
    and thus in particular of cohomological dimension $0$.
    In particular, if $F \in \heart{\PrShvVal{\Cat C}{\Sp}}$ and $U \in \Cat C$,
    then $\heart{\Gamma}(U, F) \cong \Gamma(U, F)$ (i.e.\ there is no sheaf cohomology on presheaf topoi).
    Therefore, we will just write $F(U)$ for the abelian group $\heart{\Gamma}(U, F)$.

    Moreover, $\PrShv{\Cat C}$ is Postnikov-complete.
\end{lem}
\begin{proof}
    This follows from \cite[Example 7.2.1.9, Corollary 7.2.2.30 and Proposition 7.2.1.10]{highertopoi}.
\end{proof}

\begin{prop} \label{lemma:presheaf:completion-sectionwise}
    Let $f \colon X \to Y \in \PrShv{\Cat C}$ be a morphism of presheaves.
    Then $f$ is a $p$-equivalence
    if and only if $f(U) \colon X(U) \to Y(U)$ is a $p$-equivalence for all $U \in \Cat C$.
    Moreover, $X$ is $p$-complete if and only if $X(U)$ is $p$-complete for all $U \in \Cat C$.
    Thus, we have $\complete{X}(U) = \completebr{X(U)}$ for all $U \in \Cat C$.
\end{prop}
\begin{proof}
    By definition, $f$ is a $p$-equivalence if and only if $\pSus(f)\sslash p$ is an equivalence.
    Using the equivalence $\Stab{\PrShv{\Cat C}} \cong \PrShvVal{\Cat C}{\Sp}$,
    we see that this can be checked on sections.

    For the second point, suppose first that $X$ is $p$-complete.
    Let $U \in \Cat C$ an arbitrary object. Let $A \to A'$ be a $p$-equivalence of pointed anima.
    Denote by $c_A$ and $c_{A'}$ the presheaves on $\Cat C$ given by $j_U \otimes A$ and $j_U \otimes A'$, respectively
    (where $j_U$ denotes the Yoneda embedding of $U$),
    i.e.\ $c_A$ is the presheaf such that $c_A(V) = j_U(V) \times A = \sqcup_{\Hom{}(V, U)} A$ for all $V$,
    and similar for $c_{A'}$.
    By the above, $c_A \to c_{A'}$ is a $p$-equivalence.
    Thus, we get a chain of equivalences
    \begin{align*}
        \Map{}(A', X(U)) &\cong \Map{}(A', \Map{}(j_U, X)) \\
        &\cong \Map{}(c_{A'}, X) \\
        &\cong \Map{}(c_{A}, X) \\
        &\cong \Map{}(A, \Map{}(j_U, X)) \\
        &\cong \Map{}(A, X(U)),
    \end{align*}
    where the first and last equivalences follow from the Yoneda lemma,
    the second and fourth equivalences follow because $\otimes$ exhibits $\PrShv{\Cat C}$ as tensored over $\An$
    (note that $\An$ is the tensor unit of the Lurie tensor product of presentable $\infty$-categories,
    see \cite[Example 4.8.1.20]{higheralgebra}, and hence $\PrShv{\Cat C}$ is a module over $\An$),
    and the middle map is an equivalence because $X$ is $p$-complete.
    Thus, since $A \to A'$ was arbitrary, we conclude that $X(U)$ is $p$-complete.

    Suppose now that $X(U)$ is $p$-complete for all $U \in \Cat C$.
    We need to show that the $p$-equivalence $X \to \complete{X}$ is an equivalence.
    Note that for every $U$, $X(U) \to \complete{X}(U)$ is a $p$-equivalence.
    But since $\complete{X}$ is $p$-complete, we have already seen that $\complete{X}(U)$
    is $p$-complete.
    Since $X(U)$ is $p$-complete by assumption, we conclude that $X(U) \to \complete{X}(U)$
    is an equivalence.

    For the last point, let $F$ be the presheaf $\completebr{-} \circ X$.
    Then by the above, the canonical morphism $X \to F$ is a $p$-equivalence,
    and $F$ is $p$-complete. This shows that $F$ is the $p$-completion of $X$.
\end{proof}

\begin{lem} \label{lemma:completions-respects-connectiveness}
    Let $F \in \PrShv{\Cat C}$ be a presheaf.
    If $F$ is $n$-connective, then $\complete F$ is $n$-connective.
\end{lem}
\begin{proof}
    Since connectivity and $p$-completions can be computed on sections 
    (see \cref{lemma:presheaf:completion-sectionwise} for the statement about $p$-completions),
    the result follows from the analogous result
    in the category of anima, see \cref{lemma:anima:connectivenes-of-completion}.
\end{proof}

Recall the $p$-adic t-structure from \cref{def:t-struct:defn}.

\begin{lem} \label{lemma:presheaf:evaluation-t-exact}
    Let $U \in \Cat C$ be an object.
    Then the functor $ev_U \colon \PrShvVal{\Cat C}{\Sp} \to \Sp$
    (given by precomposition with the functor $\Delta^0 \to \Cat C, * \mapsto U$)
    is t-exact for the standard t-structures 
    and t-exact for the $p$-adic t-structures.

    Moreover, a presheaf of spectra $E \in \PrShvVal{\Cat C}{\Sp}$ 
    is connective or coconnective for the standard t-structure 
    (resp. the $p$-adic t-structure)
    if and only if $ev_U(E)$ is connective or coconnective for the standard t-structure on $\Sp$
    (resp. the $p$-adic t-structure on $\Sp$) for all $U \in \Cat C$.
\end{lem}
\begin{proof}
    The claim about the standard t-structures follows immediately from the fact that 
    $\pLoop$ is computed on section, and that the $ev_U$ are jointly conservative.

    Thus, $ev_U$ is also right t-exact for the $p$-adic t-structures by \cref{lemma:t-struct:right-t-exact} 
    (applied to $L = ev_U$).
    The last part about connective objects follows from \cref{lemma:t-struct:conservative}.
    
    So let $E \in \PrShvVal{\Cat C}{\Sp}$.
    We need to show that $E \in \tpcocon{\PrShvVal{\Cat C}{\Sp}}$ if and only if $E(U) \in \tpcocon{\Sp}$ for all $U$. 
    By \cref{lemma:t-struct:char-of-cocon},
    it thus suffices to show that
    \begin{enumerate}[label=(\arabic*),ref=(\arabic*),itemsep=0em]
        \item $E = \tau_{\le 0}E$ if and only if $E(U) = (\tau_{\le 0}E)(U)$ for all $U$,
        \item $\pi_0(E)$ has bounded $p$-divisibility if and only if $\pi_0(E)(U)$ has bounded $p$-divisiblity for all $U$ and
        \item $E$ is $p$-complete if and only if $E(U)$ is $p$-complete for all $U$.
    \end{enumerate}
    the first point follows because everything can be computed on sections.
    The third point is \cref{lemma:presheaf:completion-sectionwise},
    
    For the second point, assume first that $\pi_0(E)(U)$ has bounded $p$-divisibility for all $U$.
    Let $B \in \heart{\PrShvVal{\Cat C}{\Sp}}$ be $p$-divisible.
    Then $B(U)$ is $p$-divisible for all $U$.
    In particular, $\Map{}(B, \pi_0(E)) \subset \prod_U \Map{}(B(U), \pi_0(E)(U)) \cong 0$.
    On the other hand, suppose that $\pi_0(E)$ has bounded $p$-divisibility, and 
    suppose that $U \in \Cat C$. We have to show that $\pi_0(E)(U)$ has bounded $p$-divisibility.
    So let $B \in \heart{\Sp} \cong \Ab$ be $p$-divisible.
    As in the proof of \cref{lemma:presheaf:completion-sectionwise},
    let $c_B$ be the presheaf $j_U \otimes B$.
    Then we have $\Map{}(B, \pi_0(E)(U)) \cong \Map{}(c_B, \pi_0(E))$.
    Since $c_B$ is clearly $p$-divisible, the right mapping space is $0$.
    Thus, $\pi_0(E)(U)$ has bounded $p$-divisibility.
\end{proof}

\begin{lem}\label{lemma:presheaf:Li-sections}
    Let $E \in \PrShvVal{\Cat C}{\Sp}$ be a presheaf of spectra.
    Then there are natural equivalences $(\pi_n^p(E))(U) \cong \pi_n^p(E(U))$ for all $U \in \Cat C$.
    In particular, $\pi_n^p(E) \in \heart{\PrShvVal{\Cat C}{\Sp}}$.

    If $A \in \heart{\PrShvVal{\Cat C}{\Sp}}$ be a presheaf of abelian groups,
    then there are natural equivalences 
    $(\mathbb L_i A)(U) \cong \mathbb L_i (A(U))$ for all $U \in \Cat C$.
    In particular, $\mathbb L_i A \in \tstructheart{\PrShvVal{\Cat C}{\Sp}}$.
\end{lem}
\begin{proof}
    The second part is a special case of the first (note that $\mathbb L_i A = \pi_i^p A$).

    The lemma follows from t-exactness of the evaluation functors for the $p$-adic t-structures,
    see \cref{lemma:presheaf:evaluation-t-exact}.
    For the last statement, note that $\pi_n^p(E(U)) \in \heart{\Sp}$ by \cref{lemma:anima:t-struct-description}.
\end{proof}

In presheaf categories, the $p$-adic heart is particularly simple: it lives inside the normal heart,
and consists exactly of the $p$-complete objects therein:
\begin{lem} \label{lemma:presheaf:p-heart-description}
    We have $\pheart{\PrShvVal{\Cat C}{\Sp}} \subset \heart{\PrShvVal{\Cat C}{\Sp}}$,
    consisting exactly of the $p$-complete objects in the standard heart.

    In particular, for every $p$-complete $E \in \PrShvVal{\Cat C}{\Sp}$,
    we have $\pi_n(E) \cong \pi_n^p(E)$.
\end{lem}
\begin{proof}
    The inclusion is an immediate consequence of \cref{lemma:presheaf:Li-sections}.
    Suppose that $E \in \heart{\PrShvVal{\Cat C}{\Sp}}$ is $p$-complete.
    We now note that by \cref{lemma:anima:t-struct-description} 
    $E(U) \cong \pi_0(E(U)) \cong \pi_0^p(E(U))$ for all $U$ 
    (note that $E(U)$ is $p$-complete since evaluation commutes with limits),
    and thus $\pi_0^p(E) \cong E$, again by \cref{lemma:presheaf:Li-sections}.
\end{proof}

\begin{defn} \label{def:presheaf:L1G}
    Let $G \in \Grp{\Disc{\PrShv{\Cat C}}}$ be a nilpotent presheaf of groups
    (i.e.\  the conjugation action of $G$ on itself is nilpotent, see \cref{def:nilpotent:nilpotent-action}).
    We define
    \begin{equation*}
        \mathbb L_i G \coloneqq \pi_{i+1} (\completebr{BG})
    \end{equation*}
    for $i \ge 0$.
\end{defn}

\begin{rmk} \label{rmk:presheaf:LiG-zero}
    Since the $p$-completion of a 1-truncated nilpotent object is 2-truncated (see \cref{cor:topos:truncation-of-completion}),
    we see that $\mathbb L_i G = 0$ for all $i \ge 2$.
\end{rmk}

\begin{lem} \label{lemma:presheaf:L1G-eq-L1A}
    Let $A \in \heart{\PrShvVal{\Cat C}{\Sp}} \cong \AbObj{\Disc{\PrShv{\Cat C}}}$.
    Denote by $G$ the underlying nilpotent presheaf of groups (i.e.\ we forget that $A$ is abelian).
    Then $\mathbb L_i A \cong \mathbb L_i G$ for all $i \ge 0$.
\end{lem}
\begin{proof}
    Note first that $G$ is actually nilpotent, see \cref{lemma:nilpotent:abelian-implies-nilpotent}.
    Let $U \in \Cat C$.
    We have the following chain of natural equivalences
    \begin{align*}
        (\mathbb L_i A)(U) & \cong \mathbb L_i (A(U))  \\
                           & \cong \mathbb L_i (G(U))  \\
                           & \cong \pi_{i+1}(\completebr{B(G(U))})  \\
                           & \cong \pi_{i+1}(\completebr{BG})(U)    \\
                           & \cong (\mathbb L_i G)(U).
    \end{align*}
    Here, the first equivalence is \cref{lemma:presheaf:Li-sections},
    the second is \cref{lemma:anima:L1g-eq-L1A},
    the third and fifth equivalences hold by definition
    and the fourth equivalence exists because homotopy groups, Eilenberg-MacLane objects and $p$-completions can be
    computed on sections (see \cref{lemma:presheaf:completion-sectionwise} for the claim about $p$-completions).
\end{proof}

\begin{prop} \label{lemma:presheaf:short-exact-sequence}
    Let $F \in \PrShv{\Cat C}_*$ be a pointed nilpotent presheaf.
    Then for every $n \ge 2$ there exists a canonical short exact sequence in $\heart{\PrShv{\Cat C, \Sp}}$
    (or a short exact sequence in $\Grp{\Disc{\PrShv{\Cat C}}}$ if $n = 1$)
    \begin{equation*}
        0 \to \mathbb L_0 \pi_n(F) \to \pi_n(\complete{F}) \to \mathbb L_1 \pi_{n-1}(F) \to 0,
    \end{equation*}
    where we use \cref{def:presheaf:L1G} for $\mathbb L_i \pi_1(X)$.
    Note that this distinction does not matter if $\pi_1(X)$ is abelian,
    see \cref{lemma:presheaf:L1G-eq-L1A}.
    Here we define $\mathbb L_1 \pi_0(F) \coloneqq 0$, since $F$ is connected.
\end{prop}
\begin{proof}
    By \cref{lemma:anima:short-exact-sequence}, for every $U$ there are functorial short exact sequences
    \begin{equation*}
        0 \to \mathbb L_0 \pi_n(F(U)) \to \pi_n(\complete{F(U)}) \to \mathbb L_1 \pi_{n-1}(F(U)) \to 0.
    \end{equation*}
    But by \cref{lemma:presheaf:completion-sectionwise,lemma:presheaf:Li-sections},
    this is equivalently a short exact sequence
    \begin{equation*}
        0 \to (\mathbb L_0 \pi_n(F))(U) \to (\pi_n(\complete{F}))(U) \to (\mathbb L_1 \pi_{n-1}(F))(U) \to 0
    \end{equation*}
    for every $U \in \Cat C$.
    These sequences thus give
    \begin{equation*}
        0 \to \mathbb L_0 \pi_n(F) \to \pi_n(\complete{F}) \to \mathbb L_1 \pi_{n-1}(F) \to 0.
    \end{equation*}
\end{proof}

\subsection{Completions in the Nonabelian Derived Category}
\label{section:embedding:psig}

Let $\Cat C$ be an (essentially) small category with finite coproducts.
Recall that $\PSig{\Cat C} \subset \PrShv{\Cat C}$ is
the full subcategory of presheaves that transform finite coproducts into finite products.
It is the category freely generated by $\Cat C$ under sifted colimits.
Write $\iota \colon \PSig{\Cat C} \to \PrShv{\Cat C}$ for the inclusion,
and $L_{\Sigma} \colon \PrShv{\Cat C} \to \PSig{\Cat C}$ for the left adjoint.

\begin{defn} \label{def:psig:extensive}
    Recall from \cite[Definition 2.3]{bachmann2020norms} that a category is called
    \emph{extensive} if it admits finite coproducts, coproducts are disjoint
    (i.e.\  for objects $X, Y \in \Cat C$, the pullback $X \times_{X \sqcup Y} Y$ exists and is an initial object),
    and finite coproduct decompositions are stable under pullbacks.
\end{defn}

\begin{lem} \label{lemma:psig:extensive-topos}
    Suppose that $\Cat C$ is extensive.
    Then $\PSig{\Cat C} = \ShvTop{\sqcup}{\Cat C}$,
    where we write $\sqcup$ for the Grothendieck topology on $\Cat C$
    generated by covers of the form $\{ U_i \to U \}_{i \in I}$ with $I$ a finite set 
    such that $\sqcup_i U_i \to U$ is an equivalence.
    In particular, $\PSig{\Cat C}$ is a topos and $L_\Sigma$ is left exact.
\end{lem}
\begin{proof}
    This is \cite[Lemma 2.4]{bachmann2020norms}.
\end{proof}

Suppose from now on that $\Cat C$ is extensive,
so that $\PSig{\Cat C}$ is a topos,
and $L_\Sigma$ is the left adjoint of a
geometric morphism $\PSig{\Cat C} \to \PrShv{\Cat C}$.

\begin{lem} \label{lemma:psig-postnikov-complete}
    $\PSig{\Cat C}$ is Postnikov-complete.
\end{lem}
\begin{proof}
    See \cite[Lemma 2.6]{bachmann2020norms}.
\end{proof}

\begin{lem}
    We have a canonical equivalence
    $\Stab{\PSig{\Cat C}} \cong \PSigVal{\Cat C}{\Sp}$.
\end{lem}
\begin{proof}
    This is proven in \cite[Remark 1.2]{spectral-schemes}.
\end{proof}

\begin{lem} \label{lemma:psig:htpy-sheaf}
    Let $X \in \PSig{\Cat C}_*$ be a pointed sheaf.
    Then for every $U \in \Cat C$ and $n \ge 0$ we have
    $\pi_n(X)(U) = \pi_n(X(U))$.
\end{lem}
\begin{proof}
    It suffices to show that the homotopy presheaf $U \mapsto \pi_n(X(U))$
    is actually a sheaf. This is immediate since homotopy groups of anima
    preserve finite products.
\end{proof}

\begin{lem} \label{lemma:psig:classifying}
    Let $G \in \Grp{\PSig{\Cat C}}$ be a sheaf of groups.
    Then the classifying space can be computed on sections,
    i.e. for every $U \in \Cat C$ we have
    $BG(U) \cong B(G(U))$.
\end{lem}
\begin{proof}
    Using \cref{lemma:psig:htpy-sheaf}, it suffices to show that the classifying presheaf $U \mapsto B(G(U))$
    is actually a sheaf. This is clear since the classifying space
    of a product of two groups is the product of the classifying spaces.
\end{proof}

\begin{prop} \label{lemma:psig:pequivalence}
    A morphism $f \colon F \to G$ in $\PSig{\Cat C}$
    is a $p$-equivalence (in $\PSig{\Cat C}$) if and only if
    $\iota(f)$ is a $p$-equivalence in $\PrShv{\Cat C}$.
\end{prop}
\begin{proof}
    One direction is immediate: If $\iota f$ is a $p$-equivalence,
    so is $f = L_\Sigma (\iota f)$. For this, note that
    $L_\Sigma$ is the left adjoint of a geometric morphism
    $\PrShv{\Cat C} \rightleftarrows \PSig{\Cat C}$, and use
    \cref{lemma:peq-via-points}.

    So suppose that $f$ is a $p$-equivalence.
    Write $\Mod{\finfld{p}, \operatorname{gr}}$ for the category of 
    graded $\finfld{p}$-vectorspaces, and
    \begin{equation*}
        \coAlg{\Mod{\finfld{p},\operatorname{gr}}} \coloneqq \op{\Alg{\op{\Mod{\finfld{p},\operatorname{gr}}}}}
    \end{equation*}
    for the category
    of cocommutative graded coalgebras in $\finfld{p}$-vectorspaces. Note that the categorical product of coalgebras is given by
    the tensor-product of the underlying graded $\finfld{p}$-vectorspaces,
    i.e.\ the forgetful functor 
    \begin{equation*}
        U \colon \coAlg{\Mod{\finfld{p},\operatorname{gr}}} = \op{\Alg{\op{\Mod{\finfld{p},\operatorname{gr}}}}} \to \op{(\op{\Mod{\finfld{p},\operatorname{gr}}})} = \Mod{\finfld{p},\operatorname{gr}}
    \end{equation*}
    is symmetric monoidal where we equip $\coAlg{\Mod{\finfld{p},\operatorname{gr}}}$ with the categorical product,
    and $\Mod{\finfld{p},\operatorname{gr}}$ with the tensor product of graded $\finfld{p}$-vectorspaces.
    Note that for every $F \in \PrShv{\Cat C}$,
    the presheaf $H_*(F(-), \finfld{p}) \colon \op{\Cat C} \to \Mod{\finfld{p},\operatorname{gr}}$
    can be promoted to a presheaf of cocommutative graded coalgebras in $\finfld{p}$-vectorspaces 
    (see e.g.\ \cite[19.6.2]{tomdieck2008algebraic}).
    By abuse of notation, write again $H_*(F(-), \finfld{p}) \colon \op{\Cat C} \to \coAlg{\Mod{\finfld{p},\operatorname{gr}}}$ for this presheaf.
    If $F \in \PSig{\Cat C}$ is in the nonabelian derived category, then also $H_*(F(-), \finfld{p}) \in \PSigVal{\Cat C}{\coAlg{\Mod{\finfld{p},\operatorname{gr}}}}$:
    This is clear since the product of anima yields the tensor product on homology (by the Künneth formula,
    using that we take the homology with coefficients in a field),
    which is the categorical product in $\coAlg{\Mod{\finfld{p},\operatorname{gr}}}$.
    Now note that since $f$ is a $p$-equivalence,
    we know that $s^*f$ is a $p$-equivalence for all points $s$.
    This implies (using \cref{lemma:anima:peq-iff-fpeq})
    that $H_*(s^*f, \finfld{p})$ is an equivalence for all $s$.
    Since homology commutes with filtered colimits, it commutes with
    stalks, thus we get that $s^*H_*(f, \finfld{p})$ is an equivalence for all $s$
    (here we implicitly use that $H_*(F(-), \finfld{p}) \in \PSigVal{\Cat C}{\coAlg{\Mod{\finfld{p},\operatorname{gr}}}}$).
    Thus, using e.g.\ \cite[Example 2.13]{haine2021nonabelian} and the fact
    that $\coAlg{\Mod{\finfld{p},\operatorname{gr}}}$ is compactly generated
    (this is the fundamental theorem of coalgebras, see \cite[II.2.2.1]{sweedler1969hopf}),
    already $H_*(f, \finfld{p})$ is an equivalence.
    But this means, on every section $U \in \Cat C$ we have an isomorphism
    $H_*(F(U), \finfld{p}) \xrightarrow{\simeq} H_*(G(U), \finfld{p})$.
    Using \cref{lemma:anima:peq-iff-fpeq} again,
    we conclude that $f_U \colon F(U) \to G(U)$ is a $p$-equivalence for all $U$.
    Thus, $\iota f$ is a $p$-equivalence by \cref{lemma:presheaf:completion-sectionwise}.
\end{proof}

\begin{prop} \label{lemma:p-completion-in-psig}
    Write temporarily $L_p \coloneqq \completebr{-} \circ \iota \colon \PSig{\Cat C} \to \PrShv{\Cat C}$.
    If $F \in \PSig{\Cat C}$, then $L_p(F) \in \PSig{\Cat C}$
    and $L_p(F) = \complete{F}$.
\end{prop}
\begin{proof}
    Let $F \in \PSig{\Cat C}$.
    We need to prove that $L_p(F)$ transforms finite coproducts into finite products.
    Thus let $U, V \in \Cat C$.
    Then
    \begin{align*}
        L_p(F)(U \amalg V) & = \completebr{\iota F}(U \amalg V)                           \\
                           & \cong \completebr{F(U \amalg V)}                             \\
                           & \cong \completebr{F(U) \times F(V)}                          \\
                           & \cong \completebr{F(U)} \times \completebr{F(V)}             \\
                           & \cong \completebr{\iota F}(U) \times \completebr{\iota F}(V) \\
                           & = L_p(F)(U) \times L_p(F)(V),
    \end{align*}
    where the second and fifth equivalence are \cref{lemma:presheaf:completion-sectionwise},
    the third equivalence exists because $F \in \PSig{\Cat C}$,
    and the fourth equivalence holds because $p$-completion
    commutes with products, see \cref{lemma:topos:completion-products}.
    Thus, $L_p(F) \cong L_\Sigma(L_p(F))$.
    Since $L_\Sigma$ preserves $p$-equivalences, we get that
    $F \to L_p(F)$ is a $p$-equivalence.
    Thus, we are left to show that $L_p(F)$ is $p$-complete in $\PSig{\Cat C}$.
    Let $f \colon G \to G'$ be a $p$-equivalence in $\PSig{\Cat C}$.
    Then $\Map{\PSig{\Cat C}}(f, L_p(F)) \cong \Map{\PrShv{\Cat C}}(\iota f, \iota L_p(F)) \cong \Map{\PrShv{\Cat C}}(\iota f, \completebr{\iota F})$
    is an equivalence because $\iota f$ is a $p$-equivalence by \cref{lemma:psig:pequivalence}.
    We conclude that $L_p(F)$ is $p$-complete.
\end{proof}

\begin{lem} \label{lemma:psig-connectivity-of-completion}
    Let $F \in \PSig{\Cat C}$ be $n$-connective.
    Then $\complete{F}$ is $n$-connective.
\end{lem}
\begin{proof}
    By \cref{lemma:p-completion-in-psig}
    we can compute the $p$-completion on the underlying presheaf.
    Then the result follows from \cref{lemma:completions-respects-connectiveness,lemma:psig:htpy-sheaf}.
\end{proof}

\begin{lem} \label{lemma:psig:htpy-dim}
    $\PSig{\Cat C}$ is locally of homotopy dimension $0$.
    In particular, it is locally of cohomological dimension $0$,
    and thus for every $A \in \tstructheart{\PSigVal{\Cat C}{\Sp}}$,
    $\Gamma(U, A) \in \heart{\Sp}$ for all $U \in \Cat C$ (i.e.\ there is no sheaf cohomology).
\end{lem}
\begin{proof}
    Since the elements of $\Cat C$ generate $\PSig{\Cat C}$ under colimits,
    it suffices to show that for every $C \in \Cat C$ the topos 
    $\PSig{\Cat C}_{/C}$ is of homotopy dimension $0$.
    Note that $\PSig{\Cat C}_{/C} \cong \PSig{\Cat C_{/C}}$.
    Therefore, we may assume that $\Cat C$ has a final element,
    and we want to prove that $\PSig{\Cat C}$ has homotopy dimension $0$.

    Note that there is a unique geometric morphism $\operatorname{const} \colon \An \rightleftarrows \PSig{\Cat C} \colon \Gamma$.
    Since $\Cat C$ has a final object $*$, the functor $\Gamma$ is given by 
    evaluating at the final object.
    By \cite[Lemma 7.2.1.7]{highertopoi}, it suffices to show that $\Gamma$ preserves
    effective epimorphisms.
    By \cref{lemma:psig:htpy-sheaf}, the homotopy sheaves can be calculated as 
    the underlying homotopy presheaves.
    Therefore, we see that for an effective epimorphism $f \colon X \to Y$,
    that $\Gamma(f)$ is still surjective on $\pi_0$, i.e.\ $\Gamma(f)$ is an effective epimorphism.
    (Note that in the disjoint union topology a surjective map of sheaves of sets is already surjective on 
    sections).

    The last part is \cite[Corollary 7.2.2.30]{highertopoi}.
\end{proof}

\begin{lem} \label{lemma:postnikov-completion-of-psig}
    Let $F \in \PSig{\Cat C}$ be nilpotent.
    Then $\complete{F} = \completebr{\limil{n} \tau_{\le n} F} \cong \limil{n} \completebr{\tau_{\le n} F}$.
\end{lem}
\begin{proof}
    Using \cref{thm:topos:p-comp-commutes-with-post-tower},
    it suffices to show that $\PSig{\Cat C}$ is locally of finite uniform homotopy dimension.
    This is clear, since $\PSig{\Cat C}$ is locally of homotopy dimension $0$,
    see \cref{lemma:psig:htpy-dim}.
\end{proof}

Recall the $p$-adic t-structure from \cref{def:t-struct:defn}.

\begin{lem} \label{lemma:psig:t-exact-p-adic}
    The inclusion functor $\iota_\Sigma \colon \PSigVal{\Cat C}{\Sp} \to \PrShvVal{\Cat C}{\Sp}$
    is t-exact for the standard t-structures 
    and t-exact for the $p$-adic t-structures.
\end{lem}
\begin{proof}
    The claim about the standard t-structures is immediate as homotopy objects 
    can be computed on the level of presheaves.

    Using \cref{lemma:presheaf:evaluation-t-exact},
    it suffices to show that $E \in \PSigVal{\Cat C}{\Sp}$
    is connective (resp. coconnective) for the $p$-adic t-structure if and only if $E(U)$ 
    is connective (resp. coconnective) for the $p$-adic t-structure on $\Sp$ for all $U \in \Cat C$.
    Here, one argues as in the proof of \cref{lemma:presheaf:evaluation-t-exact},
    noting that the homotopy objects of $E$ are calculated as the homotopy objects of the underlying 
    presheaves, and using \cref{lemma:p-completion-in-psig}.
\end{proof}

\begin{lem} \label{lemma:psig:Li-computation}
    Let $A \in \heart{\PSigVal{\Cat C}{\Sp}}$.
    Then $(\mathbb L_i A)(U) \cong \mathbb L_i (A(U))$ for every $U \in \Cat C$.
    In particular, $\mathbb L_i A \in \heart{\PSigVal{\Cat C}{\Sp}}$.
\end{lem}
\begin{proof}
    First note that $A(U) \in \heart{\Sp}$ by \cref{lemma:psig:htpy-dim},
    so the statement makes sense.
    Note that the presheaf $U \mapsto \mathbb L_i (A(U))$ is actually a sheaf.
    This is clear since $\mathbb L_i$ is additive and thus preserves finite products.
    
    Thus, the lemma follows from the t-exactness of $\iota_\Sigma$ for the $p$-adic t-structures (\cref{lemma:psig:t-exact-p-adic})
    and \cref{lemma:presheaf:evaluation-t-exact}.
    The last claim follows, because $\iota_\Sigma$ is fully faithful and t-exact for the standard t-structures (by the same lemma)
    and the corresponding claim about presheaves.
\end{proof}

As in the case of presheaves, the heart of the $p$-adic t-structure has a very simple description:
\begin{lem} \label{lemma:psig:p-heart-description}
    We have $\pheart{\PSigVal{\Cat C}{\Sp}} \subset \heart{\PSigVal{\Cat C}{\Sp}}$,
    consisting exactly of the $p$-complete objects in the standard heart.

    In particular, if $E \in \PSigVal{\Cat C}{\Sp}$ is $p$-complete,
    then $\pi_n(E) \cong \pi_n^p(E)$.
\end{lem}
\begin{proof}
    The inclusion $\iota_\Sigma$
    is fully faithful 
    and t-exact for the standard t-structures 
    and t-exact for the $p$-adic t-structures by \cref{lemma:psig:t-exact-p-adic}.
    Thus, the lemma follows from \cref{lemma:presheaf:p-heart-description}
    (note that $\iota_\Sigma$ preserves $p$-complete objects, see \cref{lemma:t-struct:right-exact-commutes-with-completion}).
\end{proof}

\begin{defn} \label{def:psig:L1G}
    Let $G \in \Grp{\Disc{\PSig{\Cat C}}}$ be a nilpotent sheaf of groups
    (i.e.\  the conjugation action of $G$ on itself is nilpotent).
    We define
    \begin{equation*}
        \mathbb L_i G \coloneqq \pi_{i+1} (\completebr{BG})
    \end{equation*}
    for $i \ge 0$.
\end{defn}

\begin{rmk} \label{rmk:psig:LiG-zero}
    Since the $p$-completion of a 1-truncated nilpotent object is 2-truncated (see \cref{cor:topos:truncation-of-completion}),
    we see that $\mathbb L_i G = 0$ for all $i \ge 2$.
\end{rmk}

\begin{lem} \label{lemma:psig:L1G-eq-L1A}
    Let $A \in \heart{\PSigVal{\Cat C}{\Sp}} \cong \AbObj{\Disc{\PSig{\Cat C}}}$.
    Denote by $G$ the underlying nilpotent presheaf of groups (i.e.\ we forget that $A$ is abelian).
    Then $\mathbb L_i A \cong \mathbb L_i G$ for all $i \ge 0$.
\end{lem}
\begin{proof}
    Since homotopy sheaves (\cref{lemma:psig:htpy-sheaf}), classifying spaces (\cref{lemma:psig:classifying}) and $p$-completions (\cref{lemma:p-completion-in-psig}) in $\PSig{\Cat C}$
    can be computed in $\PrShv{\Cat C}$,
    we conclude that also $\mathbb L_i G$ can be computed in $\PrShv{\Cat C}$.
    Also, $\mathbb L_i A$ can be computed in $\PrShv{\Cat C}$ by \cref{lemma:presheaf:Li-sections,lemma:psig:Li-computation}.
    Thus, the lemma follows immediately from the corresponding \cref{lemma:presheaf:L1G-eq-L1A}.
\end{proof}

\begin{prop} \label{lemma:psig:short-exact-sequence}
    Let $X \in \PSig{\Cat C}_*$ be a pointed nilpotent sheaf.
    Then for every $n \ge 2$
    there exists a short exact sequence in $\tpstructheart{\PSigVal{\Cat C}{\Sp}}$
    (or a short exact sequence in $\Grp{\Disc{\PSig{\Cat C}}}$ if $n = 1$)
    \begin{equation*}
        0 \to \mathbb L_0 \pi_n(X) \to \pi_n(\complete X) \to \mathbb L_1 \pi_{n-1}(X) \to 0,
    \end{equation*}
    where we use \cref{def:psig:L1G} for $\mathbb L_i \pi_1(X)$.
    Note that this distinction does not matter if $\pi_1(X)$ is abelian,
    see \cref{lemma:psig:L1G-eq-L1A}.
    Here we define $\mathbb L_1 \pi_0(X) \coloneqq 0$, since $X$ is connected by assumption.
\end{prop}
\begin{proof}
    Note that everything can be computed on the underlying presheaves (\cref{lemma:psig:htpy-sheaf,lemma:psig:classifying,lemma:p-completion-in-psig,lemma:psig:Li-computation}),
    thus the lemma follows immediately from \cref{lemma:presheaf:short-exact-sequence}.
\end{proof}

\subsection{Completions via Embeddings} \label{section:embedding}
Let $\topos X$ be an $\infty$-topos.
Suppose moreover that there is a small extensive category $\Cat C$
and a geometric morphism
\begin{equation*}
    \nu^* \colon \topos X \rightleftarrows \PSig{\Cat C} \colon \nu_*,
\end{equation*}
such that the left adjoint $\nu^*$ is fully faithful.
We will freely use that $\nu^*$ and $\nu_*$ induce an adjoint pair
on stabilizations, see \cref{lemma:adjoints-on-stabilization}.
Note that since $\nu^*$ is fully faithful,
also the induced functor on stabilizations is fully faithful (see \cref{lemma:fully-faithful-on-stabilizations}).

\begin{lem}\label{lemma:embedding:enough-points}
    In this situation $\topos X$ is Postnikov-completete.
    In particular, $\topos X$ is hypercomplete.
\end{lem}
\begin{proof}
    We need to show that for every $X \in \topos X$
    the canonical map $X \to \limil{n} \tau_{\le n} X$ is an equivalence.
    \cref{lemma:psig-postnikov-complete} shows that $\PSig{\Cat C}$ is Postnikov-complete.
    Hence, the canonical map $\nu^*X \to \limil{n} \tau_{\le n} \nu^* X$ is an equivalence.
    We now compute
    \begin{align*}
        X & \cong \nu_* \nu^* X                        \\
          & \cong \nu_* \limil{n} \tau_{\le n} \nu^* X \\
          & \cong \limil{n} \nu_* \nu^* \tau_{\le n} X \\
          & \cong \limil{n} \tau_{\le n} X.
    \end{align*}
    Here, we used in the first and last equivalence that $\nu^*$ is fully faithful.
    The third equivalence holds because $\nu_*$ commutes with limits (as a right adjoint),
    and $\nu^*$ commutes with truncations, see \cite[Proposition 5.5.6.28]{highertopoi}.

    The last part follows from the first, see the proof of \cite[Corollary 7.2.1.12]{highertopoi},
    where only Postnikov-completeness of $\topos X$ is used.
\end{proof}

\begin{lem} \label{lemma:embedding:stable-completion}
    Let $E \in \spectra{X}$.
    Then $\complete E \cong \nu_* (\completebr{\nu^* E})$.
\end{lem}
\begin{proof}
    We have $\nu_* (\completebr{\nu^*E}) \cong \nu_* \limil{n} (\nu^*E)\sslash p^n \cong \limil{n} (\nu_*\nu^*E)\sslash p^n = \complete E$,
    where we used that $\nu_*$ commutes with limits and cofibers, and that $\nu^*$ is fully faithful,
    i.e.\ $\nu_* \nu^* \cong \operatorname{id}$.
\end{proof}

\begin{lem} \label{lemma:embedding:EM-space}
    Let $A \in \heart{\spectra X}$ and $n \ge 1$.
    Then $\complete{K(A, n)} \cong \tau_{\ge 1}\nu_*(\complete{K(\nu^*A, n)})$.
\end{lem}
\begin{proof}
    The statement makes sense: Note that $\nu^*A$ is in the heart of the standard t-structure,
    see \cref{lemma:stabilization:homotopy-objects}.
    Therefore, the Eilenberg-MacLane space $K(\nu^* A, n)$ is defined.

    We have the following chain of equivalences
    \begin{align*}
        \complete{K(A, n)} & \cong \pLoop(\tau_{\ge 1}\completebr{\Sigma^n A})                \\
                           & \cong \pLoop (\tau_{\ge 1} \nu_* \completebr{\nu^*\Sigma^n A})   \\
                           & \cong \tau_{\ge 1} \nu_* \pLoop (\completebr{\Sigma^n(\nu^*A)})  \\
                           & \cong \tau_{\ge 1}\nu_*(\complete{K(\nu^*A, n)}).
    \end{align*}
    The first and fourth equivalences are \cref{cor:completion-of-EM-space},
    noting that $\completebr{\Sigma^n(\nu^*A)}$ is already $n$-connective, see \cref{lemma:psig-connectivity-of-completion}.
    The second equivalence is \cref{lemma:embedding:stable-completion}.
    The third equivalence follows from the definition of the standard t-structure on
    $\spectra{X}$ and \cref{lemma:adjoints-on-stabilization}.
\end{proof}

We will repeatedly use the following fact about the interaction of connective covers
with limits and geometric morphisms:
\begin{lem}\label{lemma:embedding:connective-cover-idempotent-computation}
    Fix $n \ge 0$.
    Let $X \in \PSig{\Cat C}_*$ be a pointed space.
    We have an equivalence
    \begin{equation*}
        \tau_{\ge n} \nu_* X \cong \tau_{\ge n} \nu_* \tau_{\ge n} X.
    \end{equation*}

    Similar, if $X_i$ is an $I$-indexed system in $\topos X_*$ for some $\infty$-category $I$,
    then there is an equivalence
    \begin{equation*}
        \tau_{\ge n} \limil{k} X_k \cong \tau_{\ge n} \limil{k} \tau_{\ge n} X_k.
    \end{equation*}
\end{lem}
\begin{proof}
    Since $\nu_*$ commutes with limits, we have a canonical fiber sequence
    \begin{equation*}
        \nu_* \tau_{\ge n} X \to \nu_* X \to \nu_* \tau_{\le n-1} X.
    \end{equation*}
    Since $\nu_* \tau_{\le n-1}X$ is $(n-1)$-truncated (see \cite[Proposition 5.5.6.16]{highertopoi}),
    we conclude from the long exact sequence that for $k \ge n$
    we have isomorphisms $\pi_k(\nu_* \tau_{\ge n} X) \cong \pi_k(\nu_* X)$.
    Thus, using hypercompleteness of $\topos X$ (\cref{lemma:embedding:enough-points}), the induced map
    \begin{equation*}
        \tau_{\ge n} \nu_* X \cong \tau_{\ge n} \nu_* \tau_{\ge n} X
    \end{equation*}
    is an equivalence.

    In the case of limits one argues as above, and uses that a limit of fiber sequences
    is again a fiber sequence (as limits commute with limits), and that limits
    preserve $(n-1)$-truncated objects (see \cite[Proposition 5.5.6.5]{highertopoi}).
\end{proof}

\begin{lem} \label{lemma:embedding:finite-nilpotent-presentation}
    Let $F \in \topos X_*$ be nilpotent and $n$-truncated.
    Then $\tau_{\ge 1}\nu_*(\completebr{\nu^*F}) = \complete F$.
\end{lem}
\begin{proof}
    We do a proof by induction on $n$, the case $n = 0$ being trivial.
    So suppose we have proven the statement for $n \geq 0$.
    Since $F$ is nilpotent, its Postnikov tower has a principal refinement,
    see \cref{lemma:nilpotent:principal-refinement}.
    So assume by induction that the statement holds for $\tau_{\le n-1} F = F_{n,0}$.
    We proceed by induction on $1 \le k \le m_{n}$.
    From \cref{lemma:postnikov-fiber-sequence-completion} we know that
    \begin{equation*}
        \completebr{\nu^* F_{n, k}} = \tau_{\ge 1}\Fib{\completebr{\nu^* F_{n, k-1}} \to \complete{K(\nu^* A_{n, k}, n+1)}}
    \end{equation*}
    and therefore by applying $\tau_{\ge 1}\nu_*(-)$ we get
    \begin{align*}
        \tau_{\ge 1}\nu_* (\completebr{\nu^* F_{n, k}}) & \cong \tau_{\ge 1} \nu_* \tau_{\ge 1} \Fib{\completebr{\nu^* F_{n, k-1}} \to \complete{K(\nu^* A_{n, k}, n+1)}}                 \\
                                                        & \cong \tau_{\ge 1} \nu_* \Fib{\completebr{\nu^* F_{n, k-1}} \to \complete{K(\nu^* A_{n, k}, n+1)}}                              \\
                                                        & \cong \tau_{\ge 1} \Fib{\nu_*\completebr{\nu^* F_{n, k-1}} \to \nu_*\complete{K(\nu^* A_{n, k}, n+1)}}                          \\
                                                        & \cong \tau_{\ge 1} \Fib{\tau_{\ge 1} \nu_*\completebr{\nu^* F_{n, k-1}} \to \tau_{\ge 1}\nu_*\complete{K(\nu^* A_{n, k}, n+1)}} \\
                                                        & \cong \tau_{\ge 1} \Fib{\completebr{F_{n, k-1}} \to \complete{K(A_{n, k-1}, n+1)}}                                              \\
                                                        & \cong \completebr{F_{n, k}}.
    \end{align*}
    The second and fourth equivalences are \cref{lemma:embedding:connective-cover-idempotent-computation}.
    The third equivalence holds because $\nu_*$ preserves limits (as a right adjoint).
    The fifth equivalence holds by induction and \cref{lemma:embedding:EM-space}.
    The sixth equivalence is again \cref{lemma:postnikov-fiber-sequence-completion}.
    Thus, by induction, we conclude that the statement holds for $F_{n, m_n} = F_{n+1, 0} = \tau_{\le n} F = F$.
\end{proof}

\begin{lem} \label{lemma:embedding:completion-formula}
    Assume that $\topos X$ is locally of finite uniform homotopy dimension.
    Let $F \in \topos X_*$ be nilpotent.
    Then $\tau_{\ge 1}\nu_*(\completebr{\nu^*F}) = \complete F$.
\end{lem}
\begin{proof}
    We will freely use that $\topos X$ and $\PSig{\Cat C}$ are Postnikov-complete
    (\cref{lemma:embedding:enough-points,lemma:psig-postnikov-complete}).
    Note that $\nu^*$ commutes with truncations, see \cite[Proposition 5.5.6.28]{highertopoi}.
    Using \cref{lemma:postnikov-completion-of-psig},
    we get
    \begin{equation*}
        \completebr{\nu^* F} \cong \limil{n} \completebr{\nu^*\tau_{\le n}F}.
    \end{equation*}
    Applying $\nu_*$, we conclude
    \begin{equation*}
        \nu_* (\completebr{\nu^* F}) \cong \nu_* \limil{n} \completebr{\nu^*\tau_{\le n}F} \cong \limil{n} \nu_*\completebr{\nu^*\tau_{\le n}F},
    \end{equation*}
    where we use that $\nu_*$ is a right adjoint for the second equivalence.
    Thus,
    \begin{align*}
        \tau_{\ge 1}\nu_* (\completebr{\nu^* F}) 
        & \cong \tau_{\ge 1} \limil{n} \nu_* (\completebr{\nu^*\tau_{\le n}F})              \\
        & \cong \tau_{\ge 1} \limil{n} \tau_{\ge 1} \nu_* (\completebr{\nu^*\tau_{\le n}F}) \\
        & \cong \tau_{\ge 1} \limil{n} \completebr{\tau_{\le n}F}                           \\
        & \cong \tau_{\ge 1} \complete{F}                                                   \\
        & \cong \complete{F}.
    \end{align*}
    The second equivalence is \cref{lemma:embedding:connective-cover-idempotent-computation}.
    The third equivalence was proven in \cref{lemma:embedding:finite-nilpotent-presentation}.
    The fourth equivalence holds because $p$-completions can be computed on the Postnikov tower,
    see \cref{thm:topos:p-comp-commutes-with-post-tower} (here we use the assumption that $\topos X$ is
    locally of finite uniform homotopy dimension).
    The last equivalence follows because $p$-completions of connected spaces
    are connected, see \cref{lemma:peq-respects-pi0}.
\end{proof}

\begin{defn} \label{def:embedding:classical}
    Let $E \in \PSigVal{\Cat C}{\Sp}$.
    We say that $E$ is \emph{classical} if $E$ is in the essential image of
    $\nu^*$.
\end{defn}

\begin{rmk}
    Note that since $\nu^*$ is fully faithful, an
    $E \in \PSigVal{\Cat C}{\Sp}$ is classical if and only if
    $E \cong \nu^* \nu_* E$.
    Indeed, suppose that $E \cong \nu^* F$ for some $F \in \spectra{X}$.
    But then $\nu^* \nu_* E \cong \nu^* \nu_* \nu^* F \cong \nu^* F \cong E$
    using that $\nu^*$ is fully faithful.
\end{rmk}

\begin{lem} \label{lemma:embedding:pushdown-heart}
    Suppose that $A \in \tpstructheart{\PSigVal{\Cat C}{\Sp}}$
    and that $A \sslash p$ is classical.
    Then $\nu_* A \in \tpstructheart{\spectra X}$.
\end{lem}
\begin{proof}
    \cref{lemma:t-struct:right-t-exact} shows that $\nu_*$ is left t-exact with respect to the $p$-adic t-structure,
    therefore we get $\nu_* A \in \tpcocon{\spectra X}$.
    Thus, it suffices to show that $\nu_* A \in \tpcon{\spectra X}$.
    By assumption there is an $X \in \spectra{X}$ such that $\nu^* X \cong A \sslash p$.
    Note that since $A \in \tpstructheart{\PSigVal{\Cat C}{\Sp}}$
    we know that $A \sslash p \in \tcon{\PSigVal{\Cat C}{\Sp}}$ (see \cref{lemma:t-struct:mod-p-in-con-implies-in-tpcon}).
    But this implies that $X \in \tcon{\spectra{X}}$
    ($\nu^*\pi_k X \cong \pi_k \nu^* X \cong \pi_k (A \sslash p) = 0$
    for all $k < 0$ and $\nu^*$ is fully faithful).
    Now we have equivalences $X \cong \nu_* \nu^* X \cong \nu_* (A \sslash p) \cong (\nu_* A) \sslash p$,
    hence $(\nu_* A) \sslash p \in \tcon{\spectra X}$.
    Now we conclude again by \cref{lemma:t-struct:mod-p-in-con-implies-in-tpcon} that $\nu_* A \in \tpcon{\spectra X}$.
\end{proof}

\begin{cor} \label{cor:embedding:pushdown-ses}
    Suppose that we have a short exact sequence
    \begin{equation*}
        0 \to A \to B \to C \to 0
    \end{equation*}
    in $\tpstructheart{\PSigVal{\Cat C}{\Sp}}$ such that two out of $A \sslash p$, $B \sslash p$ and $C \sslash p$ are
    classical.
    Then also the third is classical, and we get a short exact sequence
    \begin{equation*}
        0 \to \nu_* A \to \nu_* B \to \nu_* C  \to 0
    \end{equation*}
    in $\tpstructheart{\spectra X}$.
\end{cor}
\begin{proof}
    First note that we have a morphism of fiber sequences
    given by the counit of the adjunction $\nu^* \dashv \nu_*$:
    \begin{center}
        \begin{tikzcd}
            \nu^* \nu_* (A \sslash p) \arrow[r]\arrow[d] &\nu^* \nu_* (B \sslash p) \arrow[r]\arrow[d] &\nu^* \nu_* (C \sslash p) \arrow[d] \\
            A \sslash p \arrow[r] &B \sslash p \arrow[r] &C \sslash p.
        \end{tikzcd}
    \end{center}
    By assumption, two of the vertical morphisms are isomorphisms, hence so is the third.
    Thus, we conclude that all of $A \sslash p$, $B\sslash p$ and $C\sslash p$
    are classical.
    The claim now follows immediately from \cref{lemma:embedding:pushdown-heart}.
\end{proof}

\begin{lem} \label{lemma:embedding:Li-computation}
    Let $A \in \tstructheart{\spectra{X}}$.
    Suppose that $(\mathbb L_1 \nu^* A)\sslash p$ is classical.
    Then $(\mathbb L_i \nu^* A) \sslash p$ is classical,
    and we have $\mathbb L_i A \cong \nu_* \mathbb L_i \nu^* A$ for all $i \in \Z$.
\end{lem}
\begin{proof}
    Since $\mathbb L_i = 0$ for all $i \neq 0, 1$ (see \cref{lemma:t-struct:Li-zero-if-i-neq-zero-one}),
    the claim needs only be checked for $i = 0, 1$.
    Note that by \cref{lemma:t-struct:decomposition-via-Li} we have a fiber sequence
    \begin{equation*}
        \Sigma \mathbb L_1 \nu^* A \to \completebr{\nu^* A} \to \mathbb L_0 \nu^* A.
    \end{equation*}
    Applying $(-) \sslash p$ we get a fiber sequence
    \begin{equation*}
        \Sigma (\mathbb L_1 \nu^* A) \sslash p \to \completebr{\nu^* A} \sslash p \to (\mathbb L_0 \nu^* A) \sslash p.
    \end{equation*}
    Note that the left term is classical by assumption.
    For the middle term we have equivalences $\completebr{\nu^* A} \sslash p \cong (\nu^* A) \sslash p \cong \nu^* (A \sslash p)$,
    i.e.\ it is also classical.
    Thus, we conclude that $(\mathbb L_0 \nu^* A) \sslash p$ is also classical by (a proof similar to) \cref{cor:embedding:pushdown-ses}.

    Applying $\nu_*(-)$ to the first fiber sequence, and noting that $\nu_* \completebr{\nu^* A} \cong \complete{A}$ by \cref{lemma:embedding:stable-completion},
    we arrive at the fiber sequence
    \begin{equation*}
        \Sigma \nu_* \mathbb L_1 \nu^*A \to \complete{A} \to \nu_* \mathbb L_0 \nu^* A.
    \end{equation*}
    Now by \cref{lemma:embedding:pushdown-heart} and since $(\mathbb L_i \nu^* A) \sslash p$
    is classical,
    we know that $\nu_* \mathbb L_i \nu^*A \in \tpstructheart{\spectra{X}}$.
    Note that we also have a fiber sequence
    \begin{equation*}
        \Sigma \mathbb L_1 A \to \complete{A} \to \mathbb L_0 A,
    \end{equation*}
    see again \cref{lemma:t-struct:decomposition-via-Li}.
    Now the lemma follows from the uniqueness of fiber sequences
    \begin{equation*}
        X \to \complete{A} \to Y
    \end{equation*}
    with $X \in \tpcon[1]{\spectra X}$ and $Y \in \tpcocon{\spectra X}$
    by the definition of a t-structure.
\end{proof}

\begin{defn} \label{def:embedding:L1G}
    Let $G \in \Grp{\Disc{\topos X}}$ be a nilpotent sheaf of groups.
    We define
    \begin{equation*}
        \mathbb L_i G \coloneqq \nu_* \mathbb L_i \nu^* G = \nu_* \pi_{i+1}(\completebr{B\nu^*G}) \in \spectra{X}
    \end{equation*}
    for $i \ge 1$, 
    using \cref{def:psig:L1G}.
    Similarly, we define
    \begin{equation*}
        \mathbb L_0 G \coloneqq \nu_* \mathbb L_0 \nu^* G = \nu_* \pi_1 (\completebr{B\nu^* G}) \in \Grp{\Disc{\topos X}},
    \end{equation*}
    where we view $\nu_*$ as a functor $\Grp{\Disc{\PSig{\Cat C}}} \to \Grp{\Disc{\topos X}}$.
\end{defn}

\begin{rmk}
    Note that $\mathbb L_i G \cong 0$ for all $i \ge 2$ since $\completebr{B\nu^* G}$ is 2-truncated 
    by \cref{cor:topos:truncation-of-completion}.
\end{rmk}

\begin{rmk}
    If $A \in \heart{\spectra{X}} \cong \AbObj{\Disc{\topos X}}$, then there are two conflicting notions of
    $\mathbb L_i A$: We could use \cref{def:t-struct:Li} or
    \cref{def:embedding:L1G} for the underlying sheaf of groups.
    Those two definitions are equivalent if $(\mathbb L_1 \nu^* A) \sslash p$ is
    classical, see \cref{lemma:embedding:L1G-eq-L1A} (where we use \cref{def:t-struct:Li}).
    Otherwise, it is not clear if the two notions agree.
    In the following, we always try to emphasize which definition we use, and whether
    the distinction does matter.
\end{rmk}

\begin{lem} \label{lemma:embedding:L1G-eq-L1A}
    Let $A \in \heart{\spectra{X}} \cong \AbObj{\Disc{\topos X}}$ be an abelian sheaf of groups.
    Denote by $G$ the underlying nilpotent sheaf of groups.
    Suppose that $(\mathbb L_1 \nu^* A)\sslash p$ is classical.
    Then $\mathbb L_i G \cong \mathbb L_i A$ for all $i \ge 0$.
\end{lem}
\begin{proof}
    Using \cref{lemma:embedding:Li-computation},
    it suffices to show that $\pi_{i+1}(\completebr{B\nu^*G}) = \mathbb L_i \nu^* A$.
    This was shown in \cref{lemma:psig:L1G-eq-L1A}.
\end{proof}

\begin{thm} \label{thm:embedding:short-exact-sequence}
    Let $X \in \topos X_*$ be pointed and nilpotent
    such that $(\mathbb L_1 \nu^* \pi_n X) \sslash p$
    is classical for every $n \ge 2$.
    Suppose further that either 
    \begin{itemize}
        \item 
            $\pi_1X$ is abelian and $(\mathbb L_1 \nu^* \pi_1 X) \sslash p$ is classical (where we use \cref{def:t-struct:Li}), or
        \item
            that $\mathbb L_1 \pi_1 X \in \tpstructheart{\spectra X}$ (where we use \cref{def:embedding:L1G}).
    \end{itemize}
    Then for every $n \ge 2$ there is a short exact sequence in $\tpstructheart{\spectra{X}}$
    (or a short exact sequence in $\Grp{\Disc{\topos X}}$ if $n = 1$)
    \begin{equation*}
        0 \to \mathbb L_0 \pi_n X \to \nu_* \pi_n (\completebr{\nu^* X}) \to \mathbb L_1 \pi_{n-1} X \to 0,
    \end{equation*}
    where we use \cref{def:embedding:L1G} for $\mathbb L_i \pi_1(X)$.
    Note that this distinction does not matter if $\pi_1(X)$ is abelian,
    see \cref{lemma:embedding:L1G-eq-L1A}.
    Here, we define $\mathbb L_1 \pi_0 X = 0$ (since $X$ is connected by assumption).

    Moreover, we get that $\pi_n (\completebr{\nu^* X}) \sslash p$ is classical for all $n \ge 2$.
\end{thm}
\begin{proof}
    We first prove the case $n \ge 2$.
    Using \cref{lemma:embedding:Li-computation} we conclude that also
    $(\mathbb L_i \nu^* \pi_n X) \sslash p$ is classical for all $n \ge 2$ and all $i$.
    \cref{lemma:psig:short-exact-sequence} gives us a short exact sequence in $\tpstructheart{\PSigVal{\Cat C}{\Sp}}$
    \begin{equation*}
        0 \to \mathbb L_0 \pi_n \nu^* X \to \pi_n (\completebr{\nu^* X}) \to \mathbb L_1 \pi_{n-1} \nu^* X \to 0.
    \end{equation*}
    This induces a fiber sequence 
    \begin{equation*}
        (\mathbb L_0 \pi_n \nu^* X) \sslash p \to (\pi_n (\completebr{\nu^* X})) \sslash p \to (\mathbb L_1 \pi_{n-1} \nu^* X) \sslash p,       
    \end{equation*}
    where the outer to parts are classical. Thus, the same is true for the middle,
    which proves the last statement.
    Using \cref{cor:embedding:pushdown-ses} (using the assumptions on $(\mathbb L_i \nu^* \pi_n X) \sslash p$), the above short exact sequence induces a short exact sequence
    in $\tpstructheart{\spectra{X}}$
    \begin{equation*}
        0 \to \nu_* \mathbb L_0 \pi_n \nu^* X \to \nu_* \pi_n (\completebr{\nu^* X}) \to \nu_* \mathbb L_1 \pi_{n-1} \nu^* X \to 0.
    \end{equation*}
    We conclude by noting that $\nu_* \mathbb L_i \pi_n \nu^* X \cong \nu_* \mathbb L_i \nu^* \pi_n X \cong \mathbb L_i \pi_n X$
    where the last equivalence is supplied by \cref{lemma:embedding:Li-computation}.

    For the case $n = 1$, we get a canonical equivalence in $\Grp{\Disc{\PSig{\Cat C}}}$ 
    from \cref{lemma:psig:short-exact-sequence}
    \begin{equation*}
        \mathbb L_0 \pi_1 \nu^* X \cong \pi_1 (\completebr{\nu^* X}).
    \end{equation*}
    Applying $\nu_*$, this induces an equivalence in $\Grp{\Disc{\topos X}}$
    \begin{equation*}
        \mathbb L_0 \pi_1 X = \nu_* \mathbb L_0 \pi_1 \nu^* X \cong \nu_* \pi_1 (\completebr{\nu^* X}),
    \end{equation*}
    which is what we wanted to show.
\end{proof}

\subsection{Comparison of the \texorpdfstring{$p$}{p}-adic Hearts}
We keep the notation from \cref{section:embedding}.
In this section, we prove a technical result about the functors on the hearts of the $p$-adic t-structures 
induced by the functors $\nu^* \dashv  \nu_*$.

\begin{defn}
    Let $\nupu \colon \tpstructheart{\spectra{X}} \to \tpstructheart{\PSigVal{\Cat C}{\Sp}}$ 
    be defined as the functor $\pi_0^p \circ \nu^*$ restricted to the heart.
    Similarly, let $\nupl \colon \tpstructheart{\PSigVal{\Cat C}{\Sp}} \to \tpstructheart{\spectra{X}}$ 
    be defined as the functor $\pi_0^p \circ \nu_*$ restricted to the heart.
\end{defn}

\begin{lem} \label{lemma:embedding:heart-adjoint}
    The functor $\nupu$ is left adjoint to $\nupl$.
    Moreover, $\nupu$ is right-exact and $\nupl$ is left-exact as functors 
    of abelian categories.
\end{lem}
\begin{proof}
    Note that $\nu^*$ is right t-exact and $\nu_*$ is left t-exact for the $p$-adic t-structures,
    see \cref{lemma:t-struct:right-t-exact}.
    Now the statements are \cite[Proposition 1.3.17 (i) and (iii)]{BeilinsonFaisceauxPerverse}.
\end{proof}

\begin{lem} \label{lemma:embedding:nupu-ker-0}
    Let $E \in \tpstructheart{\spectra{X}}$.
    Suppose that $\nupu E \cong 0$.
    Then $E \cong 0$.
\end{lem}
\begin{proof}
    By \cref{lemma:t-struct:right-t-exact} and the assumption,
    we see that $\nu^* E \in \tpcon[1]{\PSigVal{\Cat C}{\Sp}}$.
    By \cref{lemma:t-struct:conservative} (using that $\nu^*$ is conservative, since it is fully faithful),
    we conclude that $E \in \tpcon[1]{\spectra X}$.
    Since by assumption $E \in \tpstructheart{\spectra X}$,
    it follows that $E \cong 0$.
\end{proof}

\begin{lem} \label{lemma:embedding:pheart-range}
    Let $E \in \tpstructheart{\spectra X}$.
    Then $\completebr{\nu^* E} \in \tpcon{\PSigVal{\Cat C}{\Sp}} \cap \tpcocon[1]{\PSigVal{\Cat C}{\Sp}}$.
\end{lem}
\begin{proof}
    Note that $E \in \tcocon{\spectra{X}}$ by \cref{lemma:t-struct:char-of-cocon}.
    Thus, $\nu^*E \in \tcocon{\PSigVal{\Cat C}{\Sp}}$ (since $\nu^*$ is t-exact 
    for the standard t-structures, see e.g.\ \cref{lemma:stabilization:homotopy-objects}).
    On the other hand, $E \sslash p \in \tcon{\spectra{X}}$ by \cref{lemma:t-struct:mod-p-in-con-implies-in-tpcon}.
    Thus, also $\nu^* E \sslash p \in \tcon{\PSigVal{\Cat C}{\Sp}}$, again by the t-exactness of $\nu^*$.
    The lemma follows immediately from (1) and (3) of \cref{lemma:t-struct:Li-zero-if-i-neq-zero-one}.
\end{proof}

\begin{cor}\label{lemma:embedding:A-characterization}
    Let $E \in \tpstructheart{\spectra X}$.
    Then $\pi_1^p(\nu^* E) = 0$ if and only if $\completebr{\nu^* E} \in \tpcocon{\PSigVal{\Cat C}{\Sp}}$.
    In particular, in this case $\nupu E \cong \completebr{\nu^* E}$.
\end{cor}
\begin{proof}
    The first part is immediate from \cref{lemma:embedding:pheart-range}.
    For the last statement, note that
    \begin{equation*}
        \nupu E = \pi_0^p(\nu^* E) \cong \pi_0^p(\completebr{\nu^* E}) = \completebr{\nu^* E},
    \end{equation*}
    where we used \cref{lemma:t-struct:p-eq-iso-on-htpy-in-t-structure}.
\end{proof}

\begin{defn} \label{def:embedding:A}
    Let $\Cat A \subset \tpstructheart{\spectra X}$ be the full subcategory 
    spanned by objects $E$ such that $\pi_1^p(\nu^* E) \cong 0$.
\end{defn}

\begin{lem} \label{lemma:embedding:exact-on-A}
    Let $0 \to A \to B \to C \to 0$ be a short exact sequence in $\tpstructheart{\spectra X}$
    such that $C \in \Cat A$.
    Then $0 \to \nupu A \to \nupu B \to \nupu C \to 0$ is exact.
\end{lem}
\begin{proof}
    We already know that $\nupu$ is right exact, see \cref{lemma:embedding:heart-adjoint}.
    Moreover, $\nupu = \pi_0^p \circ \nu^*$.
    Thus, the result follows from the long exact sequence and the assumption on $C$.    
\end{proof}

\begin{lem} \label{lemma:embedding:pushdown-in-heart}
    Let $E \in \Cat A \subset \tpstructheart{\spectra X}$.
    Then $\nu_* \nupu E \cong E$.

    Moreover, $\nupl \nupu E \cong E$. In particular, $\nupu$ is fully faithful on $\Cat A$.
\end{lem}
\begin{proof}
    We compute 
    \begin{equation*}
        \nu_* \nupu E 
        \cong \nu_* \completebr{\nu^* E}
        \cong \completebr{\nu_* \nu^* E}
        \cong \complete{E}
        \cong E,
    \end{equation*}
    where we used \cref{lemma:embedding:A-characterization} in the first equivalence,
    \cref{lemma:embedding:stable-completion} in the second equivalence and
    the fully faithfulness of $\nu^*$ in the third equivalence.
    The fourth equivalence holds because $E \in \tpstructheart{\spectra X}$ is $p$-complete,
    see \cref{lemma:t-struct:char-of-cocon}.

    For the last part, we note that 
    $\nupl \nupu E = \pi_0^p(\nu_* \nupu E) \cong \pi_0^p(E) = E$,
    which follows from the calculation above.
    Note that this equivalence is the (inverse of the) unit of the adjunction $\nupu \dashv \nupl$,
    therefore it follows that $\nupl$ is fully faithful on $\Cat A$.
\end{proof}

\begin{cor} \label{cor:embedding:classical-char}
    Let $A \in \tpstructheart{\PSigVal{\Cat C}{\Sp}}$.
    Suppose that $A$ is in the essential image of $\nupu|_{\Cat A}$,
    i.e.\ there is an $A' \in \Cat A$ such that $\nupu A' \cong A$.

    Then $\nu_* A \cong A'$, in particular $\nu_* A \in \tpstructheart{\spectra X}$
    and $\nupu \nu_* A \cong A$.
\end{cor}
\begin{proof}
    This immediately from \cref{lemma:embedding:pushdown-in-heart},
    because $\nu_* A \cong \nu_* \nupu A' \cong A' \in \Cat A$,
    and $\nupu \nu_* A \cong \nupu A' \cong A$.
\end{proof}

\begin{lem} \label{lemma:embedding:ker-in-A}
    Let $f \colon A \to B$ be a morphism in $\tpstructheart{\spectra X}$,
    such that $A$ and $B$ are in $\Cat A$.
    Then also $\ker(f) \in \Cat A$.
\end{lem}
\begin{proof}
    Note that we have the following two fiber sequences in $\spectra{X}$:
    \begin{align*}
        A \xrightarrow{f} B \to \Cofib{f}, \\
        \Sigma \ker(f) \to \Cofib{f} \to \coker(f).
    \end{align*}
    Applying the exact functor $\completebr{\nu^*(-)}$ and using the assumptions on $A$ and $B$ (and \cref{lemma:embedding:A-characterization}),
    we conclude by the long exact sequence that $\completebr{\nu^*\Cofib{f}}$ lives in $p$-adic 
    degrees $0$ and $1$.
    We know from \cref{lemma:embedding:pheart-range}, that also $\completebr{\nu^*\coker(f)}$ lives in $p$-adic 
    degrees $0$ and $1$.
    Therefore, applying $\completebr{\nu^*(-)}$ to the second fiber sequence, 
    the long exact sequence implies that $\pi_1^p(\nu^* \ker(f)) \cong \pi_2^p(\completebr{\nu^*\Sigma \ker(f)}) = 0$,
    i.e.\ $\ker(f) \in \Cat A$.
\end{proof}

\begin{lem} \label{lemma:embedding:ses-calculation-middle-left}
    Let $0 \to A_1 \xrightarrow{f} A_2 \xrightarrow{g} A_3 \to 0$ be a short exact sequence in $\tpstructheart{\PSigVal{\Cat C}{\Sp}}$.
    Suppose that $A_3$ and one out of $A_1$ and $A_2$ satisfy that they are in the essential image of $\nupu|_{\Cat A}$.
    Then this is also true for the third.
\end{lem}
\begin{proof}
    We choose $A_3' \in \Cat A$ such that $\nupu A_3' \cong A_3$.
    Note that the short exact sequence in the $p$-adic heart gives a fiber sequence 
    $A_1 \to A_2 \to A_3$ in $\PSigVal{\Cat C}{\Sp}$.
    Applying the functor $\nu_*$ yields the fiber sequence 
    \begin{equation} \label{eq:embedding:ses-calculation-middle-left:fib-seq}
        \nu_* A_1 \to \nu_* A_2 \to A_3',
    \end{equation}
    where we used \cref{cor:embedding:classical-char}.

    We start with the case that the assumptions for $A_2$ and $A_3$ imply the statement for $A_1$.
    So choose $A_2' \in \Cat A$ such that $\nupu A_2' \cong A_2$.
    Since $\nupu$ is fully faithful on $\Cat A$,
    we know that $\nupu \nupl g \cong g$ (note that $g$ is a morphism between objects in the 
    essential image of $\nupu|_{\Cat A}$).
    Thus, again by \cref{cor:embedding:classical-char}, the fiber sequence \ref{eq:embedding:ses-calculation-middle-left:fib-seq}
    is equivalent to the fiber sequence 
    \begin{equation*}
        \nu_* A_1 \to \nupl A_2 \xrightarrow{\nupl g} \nupl A_3.
    \end{equation*}
    By \cref{lemma:t-struct:right-t-exact}, $\nu_* A_1 \in \tpcocon{\spectra X}$.
    Since $\nupl A_2$ and $\nupl A_3$ are living in $\tpstructheart{\spectra X}$,
    the long exact sequence show that $\nu_* A_1 \in \tpcon[-1]{\spectra X}$,
    and that $\pi_{-1}^p(\nu_* A_1) \cong \coker(\nupl g)$.
    Note that $\nupu \coker(\nupl g) \cong \coker(\nupu \nupl g) \cong \coker(g) \cong 0$,
    where we used that $\nupu$ is left exact, see \cref{lemma:embedding:heart-adjoint}.
    Using \cref{lemma:embedding:nupu-ker-0}, we see that $\coker(\nupl g) \cong 0$.
    This implies that $\nu_* A_1 \in \tpstructheart{\spectra X}$,
    in particular, $\nu_* A_1 \cong \ker(\nupl g)$.
    Since we know by \cref{lemma:embedding:ker-in-A} that $\Cat A$ is stable under kernels,
    we conclude $\nu_* A_1 \in \Cat A$.
    Therefore, $\nupu \nu_* A_1 \cong \completebr{\nu^* \nu_* A_1} \cong \complete{A_1} \cong A_1$,
    where we used \cref{lemma:embedding:A-characterization}, and the fact that $A_1$ is $p$-complete because 
    it lives in the $p$-adic heart.
    This proves that $A_1$ is in the essential image of $\nupu|_{\Cat A}$.

    We continue with the case that the assumptions for $A_1$ and $A_3$ imply the statement for $A_2$.
    So choose $A_1' \in \Cat A$ such that $\nupu A_1' \cong A_1$.
    Then by \cref{cor:embedding:classical-char} the fiber sequence \ref{eq:embedding:ses-calculation-middle-left:fib-seq}
    is equivalent to the fiber sequence 
    \begin{equation*}
        A_1' \to \nu_* A_2 \to A_3'.
    \end{equation*}
    Since the outer parts live in the $p$-adic heart, this is also true for $\nu_* A_2$.
    Now define $A_2' \coloneqq \nu_* A_2$.
    We immediately see that $A_2' \in \Cat A$ because $A_1'$ and $A_2'$ are
    (apply $\nu^*$ and use the long exact sequence for $\pi_n^p$).
    But then $\nupu A_2' \cong \completebr{\nu^* \nu_* A_2} \cong \complete{A_2} = A_2$,
    again by \cref{lemma:embedding:A-characterization} and the fact that $A_2$ lives in the $p$-adic heart and is thus $p$-complete.
    This proves the lemma.
\end{proof}

The following will be a useful criterion to determine when an object will be in the essential image of $\nupu|_{\Cat A}$:
\begin{prop} \label{prop:embedding:heart-exact-sequence}
    Let $0 \to A \xrightarrow{\alpha} B \xrightarrow{\beta} C \xrightarrow{\gamma} D$ be an exact sequence
    in $\tpstructheart{\PSigVal{\Cat C}{\Sp}}$.
    Suppose that there are $A'$, $C'$ and $D'$ in $\Cat A \subset \tpstructheart{\spectra{X}}$
    such that $\nupu A' \cong A$, $\nupu C' \cong C$ and $\nupu D' \cong D$.
    Suppose moreover that $\coker(\nupl \gamma) \in \Cat A$.
    
    Then $\nu_* B \in \Cat A \subset \tpstructheart{\spectra{X}}$,
    and $\nupu (\nu_* B) \cong B$.
\end{prop}
\begin{proof}
    It suffices to prove that $B$ is in the essential image of $\nupu|_{\Cat A}$,
    the claim then follows from \cref{cor:embedding:classical-char}.

    Write $K \coloneqq \ker(\gamma) \cong \operatorname{im}(\beta)$
    and $I \coloneqq \operatorname{im}(\gamma)$.
    We have exact sequences
    \begin{align*}
        0 \to A \xrightarrow{\alpha} B \to K \to 0, \\
        0 \to K \to C \to I \to 0,                  \\
        0 \to I \to D \to \coker(\gamma) \to 0.
    \end{align*}
    By applying \cref{lemma:embedding:ses-calculation-middle-left} three times,
    it suffices to show that $\coker(\gamma)$ is in the essential image of $\nupu|_{\Cat A}$.

    Since $\nupu$ is fully faithful on $\Cat A$ by \cref{lemma:embedding:pushdown-in-heart},
    we see that there is a morphism $\gamma' \colon C' \to D'$ such that 
    $\nupu (\gamma') \cong \gamma$.
    In particular, $\coker(\gamma') \cong \coker (\nupl \nupu \gamma') \cong \coker(\nupl \gamma)$,
    which lives in $\Cat A$ by assumption.
    Therefore, we see that $\coker(\gamma) \cong \coker(\nupu \gamma') \cong \nupu \coker(\gamma')$
    is in the essential image of $\nupu|_{\Cat A}$. Here we used that $\nupu$ is right-exact,
    see \cref{lemma:embedding:heart-adjoint}.
\end{proof}

We will also need the following lemma, which helps to determine when the pushforward $\nu_*$
of an object is actually in $\Cat A$:
\begin{lem} \label{lemma:embedding:pushdown-heart-criterium-A}
    Suppose that $A \in \pheart{\PSigVal{\Cat C}{\Sp}}$
    such that $A \sslash p$ is classical, and such that $\nu_* A \in \pheart{\spectra X}$.
    Then $\nu_* A \in \Cat A \subset \pheart{\spectra X}$.
\end{lem}
\begin{proof}
    Using \cref{lemma:embedding:A-characterization}, 
    we have to show that $\completebr{\nu^* \nu_* A} \in \tpcocon{\PSigVal{\Cat C}{\Sp}}$.
    Denote by $\varphi \colon \nu^* \nu_* A \to A$ the counit map.
    Since $A$ is $p$-complete (see e.g.\ \cref{lemma:t-struct:char-of-cocon}),
    $\varphi$ induces a map $\psi \colon \completebr{\nu^* \nu_* A} \to A$.
    Thus, if $\psi$ is an equivalence, we are done.
    For this, it suffices to show that $\varphi$ is a $p$-equivalence.
    Thus, we are reduced to show that $\varphi \sslash p \colon (\nu^* \nu_* A)\sslash p \to A \sslash p$
    is an equivalence.
    By exactness, the left term is equivalent to $\nu^* \nu_* (A \sslash p)$,
    and under this identification, the map $\varphi \sslash p$
    corresponds to the counit $\nu^* \nu_* (A \sslash p) \to A \sslash p$.
    But this is an equivalence since $A \sslash p$ is classical.
\end{proof}

\subsection{A Short Exact Sequence for Zariski Sheaves}
\label{section:zar}
Let $k$ be a field and denote by $\smooth{k}$ the category of quasi-compact smooth $k$-schemes.
Let $\ShvTop{\zar}{\smooth{k}}$ be the $\infty$-topos of sheaves on $\smooth{k}$
with respect to the Zariski topology,
i.e.\ covers are given by fpqc covers $\{ U_i \to U \}_i$
such that each $U_i \to U$ can be written as $\sqcup_j U_{i, j} \to U$
such that each $U_{i, j} \to U$ is an open immersion.

The following result is well-known:
\begin{lem} \label{lemma:zar:postnikov-complete}
    The topos $\ShvTop{\zar}{\smooth{k}}$ is Postnikov-complete.
    In particular, it is hypercomplete.
\end{lem}
\begin{proof}
    Let $X \in \ShvTop{\zar}{\smooth{k}}$.
    We have to show that $\limil{k} \tau_{\le k} X_k \cong X$.
    For $U \in \smooth{k}$, write $U_{\zar}$ for the small Zariski site over $U$ 
    (i.e.\ the poset of open subsets).
    There is an evident functor $f_U \colon \ShvTop{\zar}{\smooth{k}} \to \ShvTop{\zar}{U_{\zar}}$
    given by restriction.

    Note that $\ShvTop{\zar}{U_{\zar}}$ is Postnikov-complete (and thus also hypercomplete):
    It was proven in \cite[Corollary 7.2.4.17]{highertopoi} that it is locally of 
    homotopy dimension $\le \dim(U)$.
    Thus, the result follows from \cite[Proposition 7.2.1.10]{highertopoi}.

    The functor $f_U$ commutes with limits because limits of sheaves can be computed on sections.
    Moreover, $f_U$ commutes with truncations: This is clear,
    since the topos $\ShvTop{\zar}{U_{\zar}}$ is hypercomplete 
    and $f_U$ commutes with homotopy objects.
    This fact follows because $\pi_n(f_U(F))$ is the Zariski sheafification 
    of the presheaf $V \mapsto \pi_n(f_U(F)(V)) = \pi_n(F(V))$.
    But on the other hand, $\pi_n(F)$ is the Zariski sheafification 
    of the presheaf $V \mapsto \pi_n(F(V))$.
    Thus, the result follows from the Postnikov-completeness of $\ShvTop{\zar}{U_{\zar}}$ 
    for every $U$.
\end{proof}

We now show, using the theory developed in \cref{section:embedding},
that for certain nilpotent Zariski sheaves there is a short exact sequence
\begin{equation*}
    0 \to \mathbb L_0 \pi_n(X) \to \pi_n^p(X) \to \mathbb L_1 \pi_{n-1}(X) \to 0,
\end{equation*}
see \cref{thm:pro-zar:short-exact-sequence} for the precise statement.
Note that we have shown in \cref{appendix:pro-zar}, particularly in 
\cref{thm:pro-zar:topos,thm:pro-zar:embedding}
that there is a geometric morphism
\begin{equation*}
    \nu^* \colon \ShvTop{\zar}{\smooth{k}} \rightleftarrows \ShvTop{\prozartop}{\prozar{\smooth{k}}} \cong \PSig{W} \colon \nu_*,
\end{equation*}
where $W \subset \prozar{\smooth{k}}$ is the full subcategory of zw-contractible affine schemes,
see \cref{def:pro-zar:W-cat}
Hence, we can apply the results from \cref{section:embedding} to the (big)
Zariski $\infty$-topos.

\begin{rmk}
    At the end, we want to work with motivic spaces, which are in particular Nisnevich sheaves.
    Note that one could define the pro-Nisnevich topology,
    and prove that the Nisnevich topos on $\smooth{k}$ embeds into
    $\PSig{W_{\nis}}$ for a class $W_{\nis}$ of Nisnevich weakly contractible rings.
    But the pro-Nisnevich topos has too many objects:
    Write $\mu_{p^\infty} \subset \mathbb G_m$ for the pro-Nisnevich sheaf
    of $p$-power roots of unity (which is the left Kan extension of the Nisnevich sheaf 
    $\mu_{p^\infty}|_{\op{\smooth{k}}}$).
    But then a calculation shows that $(\mathbb L_1 \mu_{p^\infty}) \sslash p$ is not classical
    (in the sense of \cref{def:embedding:classical}).
    Thus, we cannot apply \cref{thm:embedding:short-exact-sequence}.
    As we will show below, this cannot happen if we work with the pro-Zariski topology.
\end{rmk}

\begin{defn} \label{def:pro-zar:gersten}
    Let $F \in \heart{\ZarShvSp{k}}$.
    We say that $F$ satisfies \emph{Gersten injectivity}
    if for every connected $U \in \smooth{k}$
    the canonical map $\heart{\Gamma}(U, F) \to \heart{\Gamma}(\eta, F)$ is injective where $\eta \in U$ is the generic point,
    and $\heart{\Gamma}(\eta, F)$ is the stalk of $F$ at $\eta$,
    i.e.\ we define $\heart{\Gamma}(\eta, F) \coloneqq \heart{\Gamma}(\eta, \nu^* F) \cong \colimil{\eta \to V \to U} \heart{\Gamma}(V, F)$,
    where the colimit runs over all Zariski morphisms $V \to U$
    that fit into a factorization $\eta \to V \to U$ of the morphism $\eta \to U$ (see \cref{cor:pro-zar:essential-image-heart} for the equivalence).

    Note that since $\eta \in \prozar{\smooth{k}}$ is zw-contractible (since it represents 
    a local ring of the Zariski topology, see \cref{def:pro-zar:zw-contractible} for the definition 
    of zw-contractible),
    we actually have $\heart{\Gamma}(\eta, \nu^* F) \cong \Gamma(\eta, \nu^* F)$,
    see \cref{lemma:psig:htpy-dim}.
\end{defn}

\begin{lem} \label{lemma:pro-zar:pdiv-of-gersten}
    Let $n \ge 1$ be an integer and $F \in \heart{\ZarShvSp{k}}$
    such that $F/p^n$ satisfies Gersten injectivity.
    Let $U \in \smooth{k}$ a connected smooth scheme,
    $\eta \in U$ its generic point and $x \in \heart{\Gamma}(U, F)$ a section.
    Suppose that there is $\tilde{y} \in \heart{\Gamma}(\eta, F)$ such
    that $p^n \tilde{y} = x|_{\eta}$.

    Then there is a Zariski cover $V \twoheadrightarrow U$ and a $y \in \heart{\Gamma}(V, F)$
    such that $p^n y = x|_V$.
\end{lem}
\begin{proof}
    By Gersten injectivity,
    the map $\heart{\Gamma}(U, F/p^n) \to \heart{\Gamma}(\eta, F/p^n)$ is injective.
    Note that $x|_\eta = 0$ in $\heart{\Gamma}(\eta, F/p^n)$.
    Thus, $x = 0$ in $\heart{\Gamma}(U, F/p^n)$.
    This means that there exists a Zariski cover $V \twoheadrightarrow U$
    and a $y \in \heart{\Gamma}(V, F)$ such that $p^n y = x|_V$.
\end{proof}

\begin{defn}
    Let $A \in \heart{\ShvTop{\prozartop}{\prozar{\smooth{k}}, \Sp}}$ be a sheaf of
    abelian groups on the pro-Zariski site.
    We say that an element $x \in \heart{\Gamma}(U, A)$
    is \emph{locally $p^n$-divisible} if
    there is a pro-Zariski cover $V \twoheadrightarrow U$
    and a $y \in \heart{\Gamma}(V, A)$ such that $p^ny =x|_V$,
    i.e. if $x$ lies in the sheaf-theoretic image 
    (calculated in the heart, which is an abelian category)
    of the morphism $p^n \colon A \to A$.

    We say that $x$ is \emph{locally arbitrary $p$-divisible} if
    $x$ is locally $p^n$-divisible for all $n \ge 1$.
\end{defn}

\begin{lem} \label{lemma:pro-zar:B-classical}
    Let $A \in \heart{\ShvTop{\prozartop}{\prozar{\smooth{k}}, \Sp}}$ be a sheaf of abelian groups on the pro-Zariski site.
    Define a subsheaf $B \subset A$ via
    \begin{equation*}
        B = A[p] \cap \bigcap_n \operatorname{im} (A \xrightarrow{p^n} A).
    \end{equation*}
    For $U \in \prozar{\smooth{k}}$ we have 
    \begin{equation*}
        \heart{\Gamma}(U, B) = \set{x \in \heart{\Gamma}(U, A)}{px = 0, x \text{ is locally arbitrary } p \text{ divisible}}.
    \end{equation*}
    If $A$ is classical (i.e. $A$ is in the essential image of $\nu^*$,
    see \cref{def:embedding:classical}) and $(\nu_*A)/p^n$ satisfies Gersten injectivity for every $n$, then $B$ is also classical.
\end{lem}
\begin{proof}
    The description of the sections of $B$ is clear,
    since limits of sheaves can be computed on sections, and $\pi_0 \colon \Sp \to \Ab$
    commutes with limits of coconnective spectra.

    Let $U_\infty \coloneqq \limil{i} U_i$ be the cofiltered limit
    of smooth schemes $U_i \in \smooth{k}$ where the transition morphisms $U_i \to U_j$
    are Zariski localizations.
    We need to show that the canonical map $\varphi \colon \colimil{i} \heart{\Gamma}(U_i, B) \to \heart{\Gamma}(U_\infty, B)$
    is an isomorphism (see \cref{cor:pro-zar:essential-image-heart}).
    Note that we have a commuting diagram
    \begin{center}
        \begin{tikzcd}
            \colimil{i} \heart{\Gamma}(U_i, B) \arrow[d] \arrow[r, "\varphi"] &\heart{\Gamma}(U_\infty, B) \arrow[d] \\
            \colimil{i} \heart{\Gamma}(U_i, A) \arrow[r, "\cong"] &\heart{\Gamma}(U_\infty, A).
        \end{tikzcd}
    \end{center}
    The lower horizontal arrow is an isomorphism because $A$ is classical, 
    see \cref{cor:pro-zar:essential-image-heart}.
    The left vertical arrow is injective since it is a filtered colimit of injections.
    This shows that $\varphi$ is injective.
    Let $x \in \heart{\Gamma}(U_\infty, B)$.
    In other words, $x \in \heart{\Gamma}(U_\infty, A)$, $px = 0$ and for every $n$
    there is a pro-Zariski cover $V_n \twoheadrightarrow U_\infty$
    and a $y_n \in \heart{\Gamma}(V_n, A)$ such that $p^ny_n = x|_{V_n}$.
    Since $A$ is classical, we have $\heart{\Gamma}(U_\infty, A) \cong \colimil{i} \heart{\Gamma}(U_i, A)$.
    We conclude that there exists an $i \in I$ and an $x_i \in \heart{\Gamma}(U_i, A)$
    such that $x_i|_{U_\infty} = x$.
    $U_i$ is of finite type over $\Spec{k}$, hence we can write $U_i = \sqcup_j U_{i, j}$
    as a finite coproduct, with $U_{i, j}$ the connected components of $U_i$.
    Moreover, since $U_i$ is smooth, we conclude that each $U_{i, j}$ is irreducible.
    For every $j$ write $\eta_j \in U_{i, j}$ for the generic point.
    Since $\sqcup_j U_{i, j} \to U_i$ is a Zariski cover,
    we conclude that $\heart{\Gamma}(\sqcup_j U_{i, j}, A) \cong \prod_j \heart{\Gamma}(U_{i, j}, A)$.
    Thus, $x_i$ corresponds to a tuple $(x_{i, j})_j$.
    Consider the canonical morphism $f \colon U_\infty \to U_i$.
    Let $j_0$ be an index.
    If $f$ does not hit $U_{i, j_0}$, i.e.\ $\operatorname{im}(f) \cap U_{i, j_0} = \emptyset$,
    we can replace $x_i$ by the tuple $(\tilde{x}_{i, j})_j$ with $\tilde{x}_{i, j} = x_{i, j}$
    if $j \neq j_0$ and $\tilde{x}_{i, j_0} = 0$,
    this still yields the same element $x \in \colimil{i}\heart{\Gamma}(U_i, A)$.
    Note that $0 \in \heart{\Gamma}(U_{i, j_0}, B)$.
    Thus, we may assume that $f$ hits $U_{i, j_0}$.
    Pro-Zariski morphisms are flat (see \cite[\href{https://stacks.math.columbia.edu/tag/05UT}{Tag 05UT}]{stacks-project} together
    with \cite[\href{https://stacks.math.columbia.edu/tag/00HT}{Tag 00HT (1)}]{stacks-project}) and hence lift generalizations (\cite[\href{https://stacks.math.columbia.edu/tag/03HV}{Tag 03HV}]{stacks-project}).
    Hence, there exists a point $\eta_\infty \in U_\infty$
    such that $f(\eta_\infty) = \eta_{j_0}$.
    Since pro-Zariski morphisms identify local rings (see \cite[\href{https://stacks.math.columbia.edu/tag/096T}{Tag 096T}]{stacks-project}),
    we conclude that $\eta_\infty$ is a generic point, and that $k(\eta_{j_0}) \cong k(\eta_\infty)$.
    Now let $n \in \N$.
    The same reasoning applies to the pro-Zariski cover $V_n \to U_\infty$,
    i.e.\ we find a generic point $\eta_n \in V_n$ mapping to $\eta_\infty$
    such that $k(\eta_n) \cong k(\eta_\infty) \cong k(\eta_{j_0})$.
    By assumption, there is $y_n \in \heart{\Gamma}(V_n, A)$ with $p^n y_n = x_i|_{V_n}$.
    Thus, $p^n y_n|_{\eta_n} = x_{i, j_0}|_{\eta_n}$.
    Using the isomorphism $k(\eta_n) \cong k(\eta_{j_0})$
    we thus find an element $\tilde{y}_n \in k(\eta_{j_0})$
    with $p^n \tilde{y}_n = x_{i, j_0}|_{\eta_{j_0}}$.
    Since $(\nu_*A)/p^n$ satisfies Gersten injectivity,
    we conclude by \cref{lemma:pro-zar:pdiv-of-gersten}
    that there is a Zariski cover $\tilde{V}_{n, j_0} \twoheadrightarrow U_{i, j_0}$
    such that $x_{i, j_0}|_{\tilde{V}_{n, j_0}}$ is $p^n$-divisible.
    Thus, we proved that $(x_{i, j})_j$ is locally arbitrarily
    $p$-divisible, hence $x_i \in \heart{\Gamma}(U_i, B)$.
    This shows that $\varphi$ is surjective.
\end{proof}

\begin{lem} \label{lemma:pro-zar:L1-computation}
    Let $A \in \heart{\ShvTop{\prozartop}{\prozar{\smooth{k}}, \Sp}}$.
    There is an equivalence $(\mathbb L_1 A) \sslash p \cong B$,
    where $B \in \tstructheart{\ShvTop{\prozartop}{\prozar{\smooth k},\Sp}}$ is defined
    as in \cref{lemma:pro-zar:B-classical}.
\end{lem}
\begin{proof}
    Consider the short exact sequence
    \begin{equation*}
        0 \to (\mathbb L_1 A) \sslash p \to A[p] \to \pi_1((\mathbb L_0 A) \sslash p) \to 0
    \end{equation*}
    from \cref{lemma:t-struct:L1-mod-p}. In particular, $(\mathbb L_1 A) \sslash p$
    is inside the heart.
    Note that since $\mathbb L_1 A \in \tpstructheart{\ShvTop{\prozartop}{\prozar{\smooth k}, \Sp}} \subset \tstructheart{\ShvTop{\prozartop}{\prozar{\smooth k},\Sp}}$
    (see \cref{lemma:psig:p-heart-description} for the inclusion),
    we see that $(\mathbb L_1 A) \sslash p \cong (\mathbb L_1 A) / p$,
    where $(-)/p$ is the endofunctor $\coker(- \xrightarrow{p} -)$ on the standard heart.
    We see that
    \begin{equation*}
        (\mathbb L_1 A) / p \cong (\pi_1(\limil{k} A \sslash p^k)) / p \cong (\limilheart{k} A[p^k]) / p \cong B.
    \end{equation*}
    For the first equivalence, note that $\mathbb L_1 A \cong \pi_n^p(A) \cong \pi_n^p(\complete{A}) \cong \pi_n(\complete A)$,
    where the first equivalence is the definition, the second is \cref{lemma:t-struct:p-eq-iso-on-htpy-in-t-structure},
    and the third is \cref{lemma:psig:p-heart-description}.

    For the last equivalence we used that an element $x \in \heart{\Gamma}(U, A)$
    is locally arbitrary $p$-divisible if and only if it is locally $\infty$-$p$-divisible,
    in the sense that there exists a (pro-Zariski) cover $V \to U$
    such that for every $n$ there is a $y_n \in \heart{\Gamma}(V, A)$ such that $p^ny_n = x|_V$.
    To show this, suppose that $x$ is locally arbitrary $p$-divisible, and
    choose covers $V_n \to U$ and $\tilde{y}_n \in \heart{\Gamma}(V_n, A)$ such that $p^n\tilde{y}_n = x|_{V_n}$.
    Then define $V \coloneqq \limil{n} V_1 \times_U \dots \times_U V_n$,
    this is a pro-Zariski cover of $U$.
    Then define $y_n \coloneqq \tilde{y}_n|_{V}$, they satisfy $p^ny_n = x|_V$.
    This shows that $x$ is locally $\infty$-$p$-divisible.

    Now note that $(\limilheart{k} A[p^k]) / p$ consists exactly of the $p$-torsion elements
    of $A$ that are locally $\infty$-$p$-divisible:
    By the above equivalences and short exact sequence, $(\limilheart{k} A[p^k]) / p$
    can be identified with a subsheaf of $A[p]$, via the map induced by the projection $\limilheart{k} A[p^k] \to A[p]$
    (note that $pA[p] = 0$).
    Now, an element of $x \in \heart{\Gamma}(U, A[p])$ lies in the image of this map, if and only if 
    there is a cover $V \to U$ and a compatible sequence $(y_n)_n \in \limilheart{k} \heart{\Gamma}(V, A[p^k])$
    such that $x|_V = y_0$. But such a compatible sequence in particular implies that $x|_V = y_0 = p^k y_k$
    for all $k$, i.e.\ $x$ is locally $\infty$-$p$-divisible.
    This concludes the proof.
\end{proof}

\begin{cor} \label{cor:pro-zar:L1-classical}
    Let $A \in \heart{\ZarShvSp{k}}$,
    such that $A/p^n$ satisfies Gersten injectivity for every $n \ge 1$.
    Then $(\mathbb L_1 \nu^*A) \sslash p$ is classical.
\end{cor}
\begin{proof}
    Combine \cref{lemma:pro-zar:B-classical,lemma:pro-zar:L1-computation}.
\end{proof}

\begin{defn} \label{def:pro-zar:completed-htpy}
    Let $X \in \ZarShv{k}_*$ be a pointed sheaf.
    We define for $n \ge 2$
    the \emph{$p$-completed homotopy groups} via
    \begin{equation*}
        \pi_n^p(X) \coloneqq \nu_* \pi_n(\completebr{\nu^* X}) \in \ZarShvSp{k},
    \end{equation*}
    and for $n = 1$ via 
    \begin{equation*}
        \pi_1^p(X) \coloneqq \nu_* \pi_1(\completebr{\nu^* X}) \in \Grp{\Disc{\ZarShv{k}}},
    \end{equation*}
    where we view $\nu_*$ as a functor 
    \begin{equation*}
        \nu_* \colon \Grp{\Disc{\ShvTop{\prozartop}{\prozar{k}}}} \to \Grp{\Disc{\ZarShv{k}}}.
    \end{equation*}
\end{defn}

\begin{rmk}
    The name "$p$-completed homotopy \emph{group}"
    instead of something like "$p$-completed homotopy \emph{spectrum}"
    is justified: We will show in \cref{thm:pro-zar:short-exact-sequence}
    that at least in good cases $\pi_n^p(X)$ actually lives in the abelian category $\tpstructheart{\ZarShvSp{k}}$ for all $n \ge 2$.
\end{rmk}

\begin{lem} \label{lemma:pro-zar:completed-htpy-stable-under-peq}
    Let $f \colon X \to Y$ be a morphism of pointed Zariski sheaves.
    Suppose that $f$ is a $p$-equivalence.
    Then $\pi_n^p(f) \colon \pi_n^p(X) \to \pi_n^p(Y)$ is an equivalence for all $n \ge 1$.
    In particular, $\pi_n^p(X) \cong \pi_n^p(\complete X)$.
\end{lem}
\begin{proof}
    Since $\nu^*$ preserves $p$-equivalences (see \cref{lemma:peq-via-points}), and $\completebr{-}$ transforms
    $p$-equivalences to equivalences, the result follows.
\end{proof}

\begin{thm} \label{thm:pro-zar:short-exact-sequence}
    Let $X \in \ZarShv{k}_*$ be a pointed nilpotent sheaf,
    such that $\pi_n(X)/p^k$ satisfies Gersten injectivity for every $k \ge 1$ and $n \ge 2$.
    Suppose moreover that either
    \begin{itemize}[itemsep=0em]
        \item $\pi_1(X)$ is abelian and $\pi_1(X)/p^k$ satisfies Gersten injectivity for every $k \ge 1$, or
        \item $\mathbb L_1 \pi_1(X) \in \tpstructheart{\ZarShvSp{k}}$, where we use \cref{def:embedding:L1G}.
    \end{itemize}
    Then for $n \ge 2$ there is a canonical short exact sequence in $\tpstructheart{\ZarShvSp{k}}$
    (or a canonical short exact sequence in $\Grp{\Disc{\ZarShv{k}}}$ if $n = 1$)
    \begin{equation*}
        0 \to \mathbb L_0 \pi_n (X) \to \pi^p_n(X) \to \mathbb L_1 \pi_{n-1}(X) \to 0,
    \end{equation*}
    where we use \cref{def:embedding:L1G} for $\mathbb L_i \pi_1(X)$.
    This distinction does not matter if $\pi_1(X)$ is abelian,
    see \cref{lemma:embedding:L1G-eq-L1A}.
    Here we use $\mathbb L_1 \pi_0(X) = 0$, since $X$ is connected.
    In particular, $\pi^p_n(X) \in \tpstructheart{\ZarShvSp{k}}$ for $n \ge 2$.
\end{thm}
\begin{proof}
    This follows immediately from \cref{thm:embedding:short-exact-sequence} and \cref{cor:pro-zar:L1-classical}.
\end{proof}

\begin{cor} \label{cor:pro-zar:dependence-on-truncation-or-cover}
    Let $X \in \ShvTop{\zar}{\smooth{k}}_*$ be a pointed nilpotent sheaf,
    satisfying the assumptions of \cref{thm:pro-zar:short-exact-sequence}.
    Fix $n \ge 2$.
    We have equivalences
    $\pi_n^p(X) \cong \pi_n^p(\tau_{\ge k} X) \cong \pi_n^p(\tau_{\le l} X)$
    for all $0 \le k \le n-1$ and all $l \ge n$.
\end{cor}
\begin{proof}
    This follows immediately from \cref{thm:pro-zar:short-exact-sequence}.
\end{proof}

We can establish a partial converse to \cref{lemma:pro-zar:completed-htpy-stable-under-peq}:
\begin{prop} \label{prop:pro-zar:peq-of-eq-on-p-htpy}
    Let $f \colon X \to Y \in \ZarShv{k}_*$ be a morphism of nilpotent pointed sheaves with abelian fundamental group,
    and suppose that $X$ and $Y$ satisfy the assumptions of \cref{thm:pro-zar:short-exact-sequence}.
    Suppose moreover that $\pi_n^p(f)$ is an equivalence for all $n \ge 1$.
    Then $f$ is a $p$-equivalence.
\end{prop}
\begin{proof}
    Note that we have a commutative square 
    \begin{center}
        \begin{tikzcd}
            \tau_{\ge 1} \nu_* (\completebr{\nu^* X}) \arrow[rrr, "\tau_{\ge 1}\nu_*(\completebr{\nu^* f})"] \arrow[d, "\cong"] &&& \tau_{\ge 1} \nu_* (\completebr{\nu^* Y}) \arrow[d, "\cong"] \\
            \complete{X} \arrow[rrr, "\complete{f}"] &&& \complete{Y},
        \end{tikzcd}
    \end{center}
    where the downward arrows are the equivalences from \cref{lemma:embedding:completion-formula}
    (a proof that $\ZarShv{k}$ is locally of finite uniform homotopy dimension can be found in \cref{lemma:nisnevich:locally-of-uniform-homotopy-dimension}),
    and the horizontal arrows are induced by $f$.
    Thus, the upper horizontal arrow is an equivalence if and only if the lower horizontal arrow is an equivalence.
    But $\complete{f}$ is an equivalence if and only if $f$ is a $p$-equivalence, see \cref{lemma:topos:peq-if-completion-eq}.
    Hence, in order to prove the lemma, it suffices to show that $\completebr{\nu^* f}$ 
    is an equivalence.
    By hypercompleteness of $\PSig{W}$, it suffices to show that $\pi_n(\completebr{\nu^* f})$
    is an equivalence for all $n \ge 1$ (note that $\nu^* X$ and $\nu^* Y$ are simply connected).
    By assumption, we know that $\nu_* \pi_n(\completebr{\nu^* f})$ is an equivalence for all $n \ge 1$.
    Note that we know from \cref{thm:embedding:short-exact-sequence,cor:pro-zar:L1-classical} 
    that $\pi_n(\completebr{\nu^* X}) \sslash p$ and $\pi_n(\completebr{\nu^* Y}) \sslash p$ 
    are classical for $n \ge 2$. We have also seen in \cref{thm:pro-zar:short-exact-sequence}
    that $\nu_* \pi_n(\completebr{\nu^* X})$ and $\nu_* \pi_n(\completebr{\nu^* X})$
    live in $\pheart{\ZarShvSp{k}}$ for $n \ge 2$.
    Thus, (the proof of) \cref{lemma:embedding:pushdown-heart-criterium-A}
    gives us a commuting square for all $n \ge 2$
    \begin{center}
        \begin{tikzcd}
            \completebr{\nu^* \nu_* \pi_n(\completebr{\nu^* X})} \arrow[rr]\arrow[d, "\cong"] && \completebr{\nu^* \nu_* \pi_n(\completebr{\nu^* Y})} \arrow[d, "\cong"] \\
            \pi_n(\completebr{\nu^* X}) \arrow[rr] && \pi_n(\completebr{\nu^* Y}),
        \end{tikzcd}
    \end{center}
    where the vertical arrows are equivalences and the horizontal arrows are induced by $f$.
    By assumption, the upper arrow is an equivalence, therefore the same holds for the lower arrow.
    Since $\pi_1(X)$ and $\pi_1(Y)$ are abelian, the same proof works for $n = 1$.
    This proves the proposition.
\end{proof}

\begin{rmk} \label{rmk:pro-zar:peq-of-eq-on-p-htpy-relaxed}
    The assumption that $\pi_1$ should be abelian in \cref{prop:pro-zar:peq-of-eq-on-p-htpy}
    is probably unnecessary, but a proof of this fact is unclear to the author.
    One would have to analyze how far $\mathbb L_0 \pi_1(\nu^* X) \in \Grp{\Disc{\PSig{\Cat C}}}$
    is from being classical (i.e.\ in the image of the functor $\nu^* \colon \Grp{\Disc{\ZarShv{k}}} \to \Grp{\Disc{\PSig{\Cat C}}}$).
    Note that we cannot use the ''classical mod $p$``TODO CHECK-techniques employed in the above proof 
    because of the nonabelian nature of the involved groups.
\end{rmk}

\section{Completions of Motivic Spaces}
\label{section:motivic}

Let $k$ be a perfect field and denote by $\smooth{k}$ the category of smooth $k$-schemes.
Let $\MotShv{k}$ be the $\infty$-topos of sheaves on $\smooth{k}$
with respect to the Nisnevich topology (see e.g. \cite[Definition 3.1.2]{morelvoevodsky}).
Note that a family of points of this $\infty$-topos is given by evaluation on henselian local rings $S_s^h$,
i.e.\ if $\mathcal F \in \MotShv{k}$, $S \in \smooth{k}$ and $s \in S$,
then $s^*\mathcal{F} \coloneqq \mathcal{F}(S_s^h) \coloneqq \colimil{s \to U \xrightarrow{et} S} \mathcal{F}(U)$
is the stalk of $\mathcal{F}$ at $S_s^h$, see e.g. \cite[Proposition A.3]{bachmann2020norms}.
For a point $S_s^h$, write $\mathcal I_s$ for the filtered category
of objects $s \to U \xrightarrow{et} S$.
Without loss of generality we may assume that the scheme $S$ defining a point $S_s^h$ is connected.
These points form a conservative family of points (again \cite[Proposition A.3]{bachmann2020norms}),
hence it follows from \cite[Remark 6.5.4.7]{highertopoi} that $\MotShv{k}$ is hypercomplete.
In fact, the Nisnevich topos is moreover Postnikov-complete.
As in the Zariski case, this is essentially well-known.
\begin{lem} \label{lemma:nisnevich:postnikov-complete}
    $\MotShv{k}$ is Postnikov-complete.
\end{lem}
\begin{proof}
    One argues exactly as in \cref{lemma:zar:postnikov-complete}.
    As geometric input, we use that for every $U \in \smooth{k}$ there is a functor
    $f_U \colon \MotShv{k} \to \ShvTop{\nis}{U_{et}}$ given by restriction,
    where $U_{et}$ is the category of étale $U$-schemes, with coverings given by Nisnevich coverings.
    As in the Zariski case, one argues that this functor commutes with limits and truncations.
    Then we use that $\ShvTop{\nis}{U_{et}}$ has homotopy dimension $\le \dim(U)$,
    which was proven in \cite[Theorem 3.7.7.1]{sag}.
\end{proof}

\subsection{Generalities on Motivic Spaces}

Recall the following definitions from \cite[Definition 0.7]{morel2012a1}:
\begin{defn} [$\AffSpc{1}$-invariance]
    \begin{enumerate}
        \item
            Let $X \in \MotShv{k}$ be a Nisnevich sheaf.
            We say that $X$ is \emph{$\AffSpc{1}$-invariant} if $X(S) \xrightarrow{\operatorname{pr}_S^*} X(S \times \AffSpc{1})$
            is an equivalence of anima for all $S \in \smooth{k}$.
        \item
            Similarly, we say that $E \in \MotShvSp{k}$
            is \emph{$\AffSpc{1}$-invariant} if $E(S) \xrightarrow{\operatorname{pr}_S^*} E(S \times \AffSpc{1})$
            is an equivalence of spectra for all $S$.
        \item
            If $G \in \Grp{\Disc{\MotShv{k}}}$ is a Nisnevich sheaf of groups,
            we say that $G$ is \emph{strongly $\AffSpc{1}$-invariant}
            if $H_{\nis}^n(X, A) \xrightarrow{\operatorname{pr}_S^*} H_{\nis}^n(X \times \AffSpc{1}, A)$
            is an isomorphism for all $A$ and $n = 0, 1$.
            Write $\GrpStr{k}$ for the full subcategory of strongly $\AffSpc{1}$-invariant Nisnevich sheaves of groups.
    \end{enumerate}
\end{defn}

\begin{defn}
    We write $\MotSpc{k} \subset \MotShv{k}$ for the full subcategory of $\AffSpc{1}$-invariant Nisnevich sheaves,
    and call this category the \emph{category of motivic spaces} (over $k$).

    We denote by $\SH{k} \coloneqq \Stab{\MotSpc{k}}$ the stabilization of the category of motivic spaces,
    and call this category the \emph{category of motivic $S^1$-spectra} (over $k$).
\end{defn}

\begin{lem} \label{lemma:motivic:motivic-spaces-adjunction}
    The inclusion functor $\iota_{\AffSpc{1}} \colon \MotSpc{k} \hookrightarrow \MotShv{k}$ 
    has a left adjoint $L_{\AffSpc{1}}$, and $\MotSpc{k}$ is presentable.
    
    We have an induced adjunction 
    \begin{equation*}
        L_{\AffSpc{1}} \colon \MotShvSp{k} \rightleftarrows \SH{k} \colon \iota_{\AffSpc{1}},
    \end{equation*}
    induced by the adjunction $L_{\AffSpc{1}} \dashv \iota_{\AffSpc{1}}$.
    The right adjoint $\iota_{\AffSpc{1}}$ is fully faithful, with essential image 
    those sheaves of spectra which are $\AffSpc{1}$-invariant.
\end{lem}
\begin{proof}
    The first statement is an application of \cite[Proposition 5.5.4.15]{highertopoi},
    noting that the $\AffSpc{1}$-invariant sheaves are the local objects for the 
    (small) set of morphisms $\set{\operatorname{pr}_X \colon \AffSpc{1}_X \to X}{X \in \smooth{k}}$.

    There is an induced adjunction on stabilizations with fully faithful right adjoint, 
    see \cref{lemma:adjoints-on-stabilization,lemma:stabilization:right-adjoint-fully-faithful}.
    For the statement about the essential image, see \cite[Chapter 4.2]{moreltrieste}.
\end{proof}

\begin{lem} \label{lemma:motivic:homotopy-t-structure}
    There is a t-structure on $\SH{k}$ (called the \emph{standard (or homotopy) t-structure}).
    This t-structure is uniquely characterized by the requirement that 
    $\iota_{\AffSpc{1}} \colon \SH{k} \to \MotShvSp{k}$
    is t-exact (for the standard t-structure on the second category).

    In particular, $\heart{\iota}_{\AffSpc{1}} \colon \heart{\SH{k}} \to \heart{\MotShvSp{k}}$
    is an exact fully faithful functor of abelian categories given by restriction of $\iota_{\AffSpc{1}}$.
    Its essential image is the intersection of $\SH{k}$ with $\tstructheart{\MotShvSp{k}}$.
    We will say that an element of $\heart{\MotShvSp{k}}$ which lies in the essential image 
    of $\iota_{\AffSpc{1}}$ is \emph{strictly $\AffSpc{1}$-invariant}.
\end{lem}
\begin{proof}
    Since $\iota_{\AffSpc{1}}$ is fully faithful, it is clear that t-exactness of
    this functor uniquely determines the t-structure (i.e.\ the t-structure must be given by 
    the intersection of $\SH{k}$ with the standard t-structure on $\MotShvSp{k}$).
    That this actually defines a t-structure is \cite[Theorem 4.3.4 (2)]{moreltrieste}.

    Since $\iota_{\AffSpc{1}}$ is fully faithful, exact and t-exact,
    it induces an exact embedding of the hearts \cite[Proposition 1.3.17(i)]{BeilinsonFaisceauxPerverse}.
    The description of the essential image is clear from the t-exactness of $\iota_{\AffSpc{1}}$.
\end{proof}

\begin{rmk} \label{rmk:motivic:strictly}
    Let $A \in \heart{\MotShvSp{k}}$.
    Then $A$ is strictly $\AffSpc{1}$-invariant,
    if and only if the underlying sheaf of abelian groups $\heart{\Gamma}(-, A)$
    is strictly $\AffSpc{1}$-invariant in the sense of \cite[Definition 0.7]{morel2012a1},
    i.e.\ the cohomology sheaves $H^{i}_{\nis}(-, \heart{\Gamma}(-, A)) \cong \pi_{-i}(\Gamma(-, A))$
    are $\AffSpc{1}$-invariant.
    Note that $\pi_{-i}(\Gamma(-, A))$ is clearly $\AffSpc{1}$-invariant because $A$ is.
\end{rmk}

\begin{rmk} \label{rmk:motivic:htpy-groups}
    Let $n \ge 2$.
    By \cite[Corollary 5.2]{morel2012a1} and \cref{rmk:motivic:strictly}, the functor $\pi_n \circ \iota_{\AffSpc{1}} \colon \MotSpc{k}_* \to \tstructheart{\MotShvSp{k}}$
    factors over the full subcategory of $\AffSpc{1}$-invariant sheaves of spectra.
    Thus, by \cref{lemma:motivic:homotopy-t-structure} it induces a functor $\pi_n \colon \MotSpc{k} \to \tstructheart{\SH{k}}$.
\end{rmk}

\begin{rmk} \label{rmk:motivic:htpy-group-1-abelian}
    We can also look at the case $n = 1$:
    By \cite[Corollary 5.2]{morel2012a1}, the functor $\pi_1 \circ \iota_{\AffSpc{1}} \colon \MotSpc{k}_* \to \Grp{\Disc{\MotShv{k}}}$
    factors through the category $\GrpStr{k}$.
    If $X$ is a motivic space with abelian $\pi_1(\iota_{\AffSpc{1}} X)$,
    then this group is moreover strictly $\AffSpc{1}$-invariant (see \cite[Theorem 4.46]{morel2012a1}).
    Therefore, we get a well-defined functor 
    $\pi_1 \colon \MotSpc{k}^{\operatorname{ab}} \to \heart{\SH{k}}$,
    where $\MotSpc{k}^{ab}$ is the category of motivic spaces with abelian fundamental group.
\end{rmk}

\begin{defn}
    We also define the adjunctions
    \begin{equation*}
        L_{\nis} \colon \ZarShv{k} \rightleftarrows \MotShv{k} \colon \iota_{\nis},
    \end{equation*}
    given by sheafification and inclusion (i.e.\ induced by the canonical morphism of sites), 
    and
    \begin{equation*}
        L_{\nis,\AffSpc{1}} \colon \ZarShv{k} \rightleftarrows \MotSpc{k} \colon \iota_{\nis,\AffSpc{1}},
    \end{equation*}
    given by $L_{\nis,\AffSpc{1}} \coloneqq L_{\AffSpc{1}} \circ L_{\nis}$ 
    and the fully faithful functor $\iota_{\nis, \AffSpc{1}} \coloneqq \iota_{\nis} \circ \iota_{\AffSpc{1}}$.
    Note that there are induced adjunctions (see \cref{lemma:adjoints-on-stabilization})
    \begin{align*}
        L_{\nis} \colon \ZarShvSp{k} &\rightleftarrows \MotShvSp{k} \colon \iota_{\nis}, \\
        L_{\nis,\AffSpc{1}} \colon \ZarShvSp{k} &\rightleftarrows \SH{k} \colon \iota_{\nis,\AffSpc{1}},
    \end{align*}
    where the right adjoints are again fully faithful (see \cref{lemma:stabilization:right-adjoint-fully-faithful}).
\end{defn}

We want to show now that $\iota_{\nis,\AffSpc{1}}$ is t-exact for the standard t-structures.
Note that this is rather surprising, as $\iota_{\nis,\AffSpc{1}}$ is defined as the composition 
of $\iota_{\nis}$ and $\iota_{\AffSpc{1}}$, and the former is not t-exact!
For this, we need the following general proposition:
\begin{prop} \label{prop:motivic:right-t}
    Let $\Cat D$ and $\Cat E$ be stable categories equipped with t-structures 
    $\tstruct{\Cat D}$ and $\tstruct{\Cat E}$,
    and let $F \colon \Cat D \to \Cat E$ be an exact functor.
    Assume moreover that
    \begin{enumerate}[label=(\arabic*),ref=assumption (\arabic*),itemsep=0em]
        \item \label{prop:motivic:right-t:limits} $F$ preserves limits,
        \item \label{prop:motivic:right-t:heart-right-exact} for all $X \in \heart{\Cat D}$ we have that $FX \in \heart{\Cat E}$,
        \item \label{prop:motivic:right-t:left-complete-D} the t-structure on $\Cat D$ is left-complete, and
        \item \label{prop:motivic:right-t:left-complete-E} the t-structure on $\Cat E$ is left-complete. 
    \end{enumerate}
    Then $F$ is right t-exact.
\end{prop}
\begin{proof}
    We first show that $F$ is t-exact on bounded objects, 
    i.e.\ we show that for all $m, n \in \Z$ and all $X \in \tcon[m]{\Cat D} \cap \tcocon[n]{\Cat D}$ 
    we have $FX \in \tcon[m]{\Cat E} \cap \tcocon[n]{\Cat E}$.
    Note that by shifting, it suffices to consider the case $m = 0$ (and thus $n \ge 0$, for $n < 0$ the statement is vacuous).

    We proceed by induction on $n$, the case $n = 0$ follows from \ref{prop:motivic:right-t:heart-right-exact}.
    So suppose the statement is true for $n \ge 0$, and let $X \in \tcon{\Cat D} \cap \tcocon[n+1]{\Cat D}$.
    Consider the fiber sequence $\Sigma^{n+1} \pi_{n+1} X \to X \to \tau_{\le n} X$.
    Applying $F$ yields the fiber sequence $\Sigma^{n+1} F\pi_{n+1} X \to FX \to F\tau_{\le n} X$.
    By induction, we see that $F\tau_{\le n} X \in \tcon{\Cat E} \cap \tcocon[n]{\Cat E} \subset \tcon{\Cat E} \cap \tcocon[n+1]{\Cat E}$,
    and $\Sigma^{n+1} F\pi_{n+1} X \in \Sigma^{n+1} \heart{\Cat E} \subset \tcon{\Cat E} \cap \tcocon[n+1]{\Cat E}$ by \ref{prop:motivic:right-t:heart-right-exact}.
    Thus, since $\tcon{\Cat E}$ and $\tcocon[n+1]{\Cat E}$ are stable under extensions, we get that $FX \in \tcon{\Cat E} \cap \tcocon[n+1]{\Cat E}$.

    Now, let $X \in \tcon{\Cat D}$ be a general connective object.
    Then, since the t-structure on $\Cat D$ is left-complete by \ref{prop:motivic:right-t:left-complete-D}, we can write 
    $X \cong \limil{n} \tau_{\le n} X$. Since $F$ commutes with limits (\ref{prop:motivic:right-t:limits}), we can thus write 
    $FX \cong \limil{n} F \tau_{\le n} X$.

    Using \cite[Proposition 1.2.1.17 (2)]{higheralgebra} and the left-completeness of $\Cat E$ 
    (\ref{prop:motivic:right-t:left-complete-E}),
    it suffices to show that $F \tau_{\le n} X$ is connective for every $n$,
    and that $\tau_{\le n} F \tau_{\le n+1} X \cong F \tau_{\le n} X$;
    this then implies that $\limil{n} F \tau_{\le n} X$ is connective.
    We have seen above that $F \tau_{\le n} X$ is connective for every $n$.
    
    So suppose that $n \ge 0$.
    Consider the fiber sequence 
    \begin{equation*}
        \Sigma^{n+1} \pi_{n+1} F \tau_{\le n+1} X \to F \tau_{\le n+1} X \to \tau_{\le n} F \tau_{\le n+1} F.
    \end{equation*}
    Note that there is also a fiber sequence 
    \begin{equation*}
        \Sigma^{n+1} \pi_{n+1} X \to \tau_{\le n+1} X \to \tau_{\le n} X,
    \end{equation*}
    which after applying $F$ yields 
    \begin{equation*}
        F \Sigma^{n+1} \pi_{n+1} X \to F \tau_{\le n+1} X \to F\tau_{\le n} X.
    \end{equation*}
    Thus, in order to show that $\tau_{\le n} F\tau_{\le n+1} X \cong F \tau_{\le n} X$,
    it suffices to show that $F \pi_{n+1} X \cong \pi_{n+1} F \tau_{\le n+1} X$.
    This follows immediately from t-exactness on bounded objects,
    i.e. we get (since $\tau_{\le n+1} X$ is bounded) 
    $\pi_{n+1} F \tau_{\le n+1} X \cong \pi_{n+1} \tau_{\le n+1} FX \cong \pi_{n+1} FX$.
\end{proof}

\begin{lem} \label{lemma:motivic:iota-A1-nis-t-exact-std}
    The functor $\iota_{\nis,\AffSpc{1}}$ is t-exact for the 
    standard t-structures.

    In particular, $\heart{\iota_{\nis,\AffSpc{1}}} \colon \heart{\SH{k}} \to \tstructheart{\ShvTop{\zar}{\smooth{k}, \Sp}}$
    is an exact fully faithful functor of abelian categories and given by restriction of $\iota_{\nis, \AffSpc{1}}$.
\end{lem}
\begin{proof}
    We see that $\iota_{\nis,\AffSpc{1}}$ is left t-exact as the composition of a t-exact functor (\cref{lemma:motivic:homotopy-t-structure})
    and a left t-exact functor (note that $\iota_{\nis}$ is right adjoint to the t-exact functor $L_{\nis}$
    (see \cref{lemma:stabilization:homotopy-objects} for the t-exactness),
    and use \cite[Proposition 1.3.17 (iii)]{BeilinsonFaisceauxPerverse}).
    Thus, it suffices to see that the functor is right t-exact.
    We first prove the following: If $A \in \heart{\SH{k}}$,
    then also $\iota_{\nis,\AffSpc{1}} A \in \heart{\ZarShvSp{k}}$.
    Write $H \colon \AbObj{\Disc{\MotShv{k}}} \cong \heart{\MotShvSp{k}}$
    (and similar for Zariski sheaves).
    Since this is an equivalence, we know that there is an $A' \in \AbObj{\Disc{\MotShv{k}}}$
    with $HA' \cong \iota_{\AffSpc{1}}A$.
    Note that since $\iota_{\AffSpc{1}} A$ is $\AffSpc{1}$-invariant, we know that $A'$ is strictly 
    $\AffSpc{1}$-invariant, see \cref{rmk:motivic:strictly}.
    It suffices to show that $\iota_{\nis} HA' \cong H\iota_{\nis} A'$,
    where $\iota_{\nis} A' \in \AbObj{\Disc{\MotShv{k}}}$ is the application of the \emph{underived}
    functor $\iota_{\nis} \colon \MotShv{k} \to \ZarShv{k}$ with the induced structure of an abelian group object.
    In order to prove this equivalence, by Whitehead's theorem it suffices to prove that 
    for all $n$ and all $U \in \smooth{k}$
    the canonical map $\pi_n((\iota_{\nis} HA')(U)) \to \pi_n((H\iota_{\nis} A')(U))$ is an equivalence.
    But note that we have equivalences 
    \begin{equation*}
        \pi_n((\iota_{\nis} HA')(U)) = \pi_n((HA')(U)) \cong H^{-n}_{\nis}(U, A')
    \end{equation*}
    and 
    \begin{equation*}
        \pi_n((H \iota_{\nis}A')(U)) \cong H^{-n}_{\zar}(U, \iota_{\nis} A').
    \end{equation*}
    But the right-hand sides agree by \cite[Theorem 4.5]{AsokGersten}
    (The reference uses that $k$ is an infinite field.
    If $k$ is a finite field, we can argue as in the above reference,
    using the Gabber presentation lemma for finite fields, see \cite[Theorem 1.1]{hogadi2020gabber}).

    Thus, we can apply \cref{prop:motivic:right-t} with $\iota_{\nis,\AffSpc{1}}$: 
    Note that $\iota_{\nis,\AffSpc{1}}$ preserves limits because it is a right adjoint, 
    and that the standard t-structure on $\ZarShvSp{k}$ is left-complete because $\ZarShv{k}$ is Postnikov-complete,
    see \cref{lemma:zar:postnikov-complete} and the proof of \cite[Corollary 1.3.3.11]{sag}.
    Note that also $\MotShvSp{k}$ is left-complete with respect to the standard t-structure, 
    because $\MotShv{k}$ is Postnikov-complete, see \cref{lemma:nisnevich:postnikov-complete}.
    Thus, it follows that also $\SH{k} \subset \MotShvSp{k}$ is left-complete,
    since the functor $\iota_{\AffSpc{1}}$ is an exact and t-exact fully faithful functor which commutes 
    with limits (as a right-adjoint): 
    Indeed, if $X \in \SH{k}$, then we have 
    \begin{equation*}
        \iota_{\AffSpc{1}} X \cong \limil{k} \tau_{\le k} \iota_{\AffSpc{1}} X \cong \iota_{\AffSpc{1}} \limil{k} \tau_{\le k} X.
    \end{equation*}
    Since $\iota_{\AffSpc{1}}$ is fully faithful, it is in particular conservative, i.e. $X \cong \limil{k} \tau_{\le k} X$,
    which is what we wanted to show.
    Hence, \cref{prop:motivic:right-t} implies that $\iota_{\nis,\AffSpc{1}}$ is right t-exact.
\end{proof}

\begin{lem} \label{lemma:motivic:iota-a1-EM}
    Let $A \in \heart{\SH{k}}$ and $n \ge 0$. 
    Then $\iota_{\AffSpc{1}} K(A, n) \cong K(\heart{\iota}_{\AffSpc{1}} A, n)$
    and $\iota_{\nis,\AffSpc{1}} K(A, n) \cong K(\heart{\iota}_{\nis,\AffSpc{1}} A, n)$ 
\end{lem}
\begin{proof}
    We calculate 
    \begin{equation*}
        K(\heart{\iota}_{\AffSpc{1}}A, n) 
        = \pLoop \Sigma^n \heart{\iota}_{\AffSpc{1}}A 
        \cong \pLoop \Sigma^n \iota_{\AffSpc{1}} A 
        \cong \iota_{\AffSpc{1}} \pLoop \Sigma^n A 
        = \iota_{\AffSpc{1}} K(A, n),
    \end{equation*}
    where we used that $\heart{\iota}_{\AffSpc{1}}A \cong \iota_{\AffSpc{1}} A$
    (because $\iota_{\AffSpc{1}}$ is t-exact for the standard t-structures, see \cref{lemma:motivic:homotopy-t-structure}),
    and \cref{lemma:adjoints-on-stabilization}.

    The same proof works for the second statement, using t-exactness of $\iota_{\nis,\AffSpc{1}}$
    for the standard t-structures, see \cref{lemma:motivic:iota-A1-nis-t-exact-std}.
\end{proof}

\begin{lem} \label{lemma:motivic:connective-cover-descent-to-a1}
    For every $n \ge 0$ the functor 
    $\tau_{\ge n} \colon \MotShv{k}_* \to \MotShv{k}_*$
    restricts to a functor $\tau_{\ge n} \colon \MotSpc{k}_* \to \MotSpc{k}_*$.

    In other words, there is a functor $\tau_{\ge n}$ such that the following square commutes:
    \begin{center}
        \begin{tikzcd}
            \MotSpc{k}_* \arrow[d, hook, "{\iota_{\AffSpc{1}}}"] \arrow[r, "{\tau_{\ge n}}"] &\MotSpc{k}_* \arrow[d, hook, "{\iota_{\AffSpc{1}}}"] \\
            \MotShv{k}_* \arrow[r, "{\tau_{\ge n}}"] &\MotShv{k}_*.
        \end{tikzcd}
    \end{center}
\end{lem}
\begin{proof}
    Let $n \ge 0$, and fix a pointed motivic space $X \in \MotSpc{k}_*$.
    It suffices to show that $\tau_{\ge n} \iota_{\AffSpc{1}} X$ is again $\AffSpc{1}$-invariant.

    If $n = 0$ there is nothing to prove, so we can assume $n \ge 1$.
    Using \cite[Corollary 5.3]{morel2012a1},
    it suffices to prove that $\pi_1(\tau_{\ge n} \iota_{\AffSpc{1}}X)$ is strongly $\AffSpc{1}$-invariant
    and $\pi_k(\tau_{\ge n} \iota_{\AffSpc{1}}X)$ is strictly $\AffSpc{1}$-invariant for all $k \ge 2$.
    This is clear if $n > k$, since $0$ is strictly $\AffSpc{1}$-invariant.
    If $n \le k$, we use \cite[Corollary 5.2]{morel2012a1}
    to conclude that $\pi_k(\tau_{\ge n}\iota_{\AffSpc{1}}X) \cong \pi_k(\iota_{\AffSpc{1}}X)$ is strictly $\AffSpc{1}$-invariant.
    The same proof works for $\pi_1$ if $n = 1$, again using \cite[Corollary 5.2]{morel2012a1}.
\end{proof}

\begin{lem} \label{lemma:motivic:truncation-cover-iota-if-connected}
    Let $X \in \MotSpc{k}_*$ be a pointed connected motivic space, i.e.\ 
    it is in the image of $\tau_{\ge 1} \colon \MotSpc{k}_* \to \MotSpc{k}_*$ from \cref{lemma:motivic:connective-cover-descent-to-a1}.
    For all $n \ge 1$ there are equivalences 
    \begin{align*}
        \tau_{\le n} \iota_{\nis,\AffSpc{1}} X &\cong \iota_{\nis} \tau_{\le n} \iota_{\AffSpc{1}} X.
    \end{align*}
\end{lem}
\begin{proof}
    Let $k \ge 1$.
    Then there is a fiber sequence 
    \begin{equation*}
        K(\pi_k(\iota_{\AffSpc{1}} X), k) \to \tau_{\le k} \iota_{\AffSpc{1}} X \to \tau_{\le k-1} \iota_{\AffSpc{1}} X.
    \end{equation*}
    Thus, since $\iota_{\nis}$ preserves limits (it is a right adjoint), we get a fiber sequence
    \begin{equation*}
        \iota_{\nis} K(\pi_k(\iota_{\AffSpc{1}} X), k) \to \iota_{\nis} \tau_{\le k} \iota_{\AffSpc{1}} X \to \iota_{\nis} \tau_{\le k-1} \iota_{\AffSpc{1}} X.
    \end{equation*}
    By definition, we have $\iota_{\AffSpc{1}} \pi_k(X) = \pi_k(\iota_{\AffSpc{1}} X)$.
    \cref{lemma:motivic:iota-a1-EM} now gives us equivalences 
    \begin{equation*}
        \iota_{\nis} K(\pi_k(\iota_{\AffSpc{1}} X), k) 
        \cong \iota_{\nis, \AffSpc{1}} K(\pi_k(X), k) 
        \cong K(\heart{\iota}_{\nis, \AffSpc{1}} \pi_k(X), k).
    \end{equation*}

    Moreover, $\limil{k} \iota_{\nis} \tau_{\le k} \iota_{\AffSpc{1}} X \cong \iota_{\nis} \limil{k} \tau_{\le k} \iota_{\AffSpc{1}} X \cong \iota_{\nis,\AffSpc{1}} X$,
    since $\iota_{\nis}$ preserves limits and $\MotShv{k}$ is Postnikov-complete (\cref{lemma:nisnevich:postnikov-complete}).
    Since $\iota_{\nis} \tau_{\le k} \iota_{\AffSpc{1}} X$ is still $k$-truncated (as the right adjoint of a geometric
    morphism preserves truncated objects, see \cite[Proposition 6.3.1.9]{highertopoi}),
    and the fibers of $\iota_{\nis} \tau_{\le k} \iota_{\AffSpc{1}} X \to \iota_{\nis} \tau_{\le k-1} \iota_{\AffSpc{1}} X$ are Eilenberg-MacLane objects
    in degree $k$, we conclude by induction on $k$ that actually $(\iota_{\nis} \tau_{\le k} \iota_{\AffSpc{1}} X)_k$ is the Postnikov tower of $\iota_{\nis,\AffSpc{1}} X$,
    i.e.\ $\iota_{\nis} \tau_{\le k} \iota_{\AffSpc{1}} X \cong \tau_{\le k} \iota_{\nis,\AffSpc{1}} X$.
\end{proof}

\begin{lem} \label{lemma:motivic:connected-cover-iota}
    Let $X \in \MotSpc{k}_*$ be a pointed motivic space.
    Then $\tau_{\ge n} \iota_{\nis,\AffSpc{1}} X \cong \iota_{\nis} \tau_{\ge n} \iota_{\AffSpc{1}} X$ 
    for all $n \ge 0$.
\end{lem}
\begin{proof}
    If $n = 0$ then there is nothing to prove.
    So suppose that $n \ge 1$.
    Write $\ZarShv{k}_{\ge n,*} \subset \ZarShv{k}_*$ for the full subcategory of $n$-connective 
    pointed Zariski sheaves.
    We begin by showing that for every $Y \in \MotShv{k}$,
    the canonical map $\tau_{\ge n} \iota_{\nis} \tau_{\ge n} Y \to \tau_{\ge n} \iota_{\nis} Y$ is an equivalence.
    Note that there is a fiber sequence
    \begin{equation*}
        \tau_{\ge n} Y \to Y \to \tau_{\le n-1} Y.
    \end{equation*}
    Applying the right adjoint $\iota_{\nis}$ yields the fiber sequence 
    \begin{equation*}
        \iota_{\nis} \tau_{\ge n} Y \to \iota_{\nis} Y \to \iota_{\nis} \tau_{\le n-1} Y.
    \end{equation*}
    Note that if we view $\tau_{\ge n}$ as a functor $\ZarShv{k}_* \to \ZarShv{k}_{\ge n,*}$,
    then it preserves limits because it is right adjoint to the inclusion.
    Therefore, applying $\tau_{\ge n}$ yields a fiber sequence (in $\ZarShv{k}_{\ge n, *}$)
    \begin{equation*}
        \tau_{\ge n} \iota_{\nis} \tau_{\ge n} Y \to \tau_{\ge n} \iota_{\nis} Y \to \tau_{\ge n} \iota_{\nis} \tau_{\le n-1} Y.
    \end{equation*}
    Since $\iota_{\nis}$ preserves ($n-1$)-truncated objects (this is proven in \cite[Proposition 6.3.1.9]{highertopoi}, since $\iota_{\nis}$
    is the right adjoint of a geometric morphism), the right term vanishes.
    Therefore we have an equivalence $\tau_{\ge n} \iota_{\nis} \tau_{\ge n} Y \cong \tau_{\ge n} \iota_{\nis} Y$
    in $\ZarShv{k}_{\ge n,*}$, and therefore also in $\ZarShv{k}_*$.

    Therefore, it suffices to show that $\iota_{\nis} \tau_{\ge n} \iota_{\AffSpc{1}} X$ is already 
    $n$-connective for every $X \in \MotSpc{k}_*$.
    Note first that by \cref{lemma:motivic:connective-cover-descent-to-a1}, there is a pointed motivic space
    $Y \coloneqq \tau_{\ge n} X$ with $\tau_{\ge n} \iota_{\AffSpc{1}} X \cong \iota_{\AffSpc{1}} Y$,
    and $Y$ is a pointed, $n$-connective motivic space.
    Note that $\iota_{\nis,\AffSpc{1}} Y$ is $n$-connective if and only if 
    $\tau_{\le n} \iota_{\nis,\AffSpc{1}} Y$ is $n$-connective.
    We know from \cref{lemma:motivic:truncation-cover-iota-if-connected}, 
    that $\tau_{\le n} \iota_{\nis,\AffSpc{1}} Y \cong \iota_{\nis} \tau_{\le n} \iota_{\AffSpc{1}} Y$.
    Therefore, we may assume that $\iota_{\AffSpc{1}} Y$ is $n$-connective and $n$-truncated, 
    i.e.\ $\iota_{\AffSpc{1}} Y \cong \iota_{\AffSpc{1}} K(A, n)$ for some $A \in \heart{\SH{k}}$.
    But now we have that 
    $\iota_{\nis,\AffSpc{1}} Y \cong \iota_{\nis,\AffSpc{1}} K(A, n) \cong K(\heart{\iota}_{\nis,\AffSpc{1}}, n)$ 
    by \cref{lemma:motivic:iota-a1-EM},
    which is in particular $n$-connective.
    This proves the lemma.
\end{proof}

\begin{cor}\label{cor:motivic:iota-pin-if-connected}
    Let $X \in \MotSpc{k}_*$ be a pointed motivic space, i.e.\ 
    If $n \ge 2$, there are equivalences
    \begin{equation*}
        \pi_n (\iota_{\nis,\AffSpc{1}} X) \cong \heart{\iota}_{\nis} \pi_n(\iota_{\AffSpc{1}} X) \cong \heart{\iota_{\nis, \AffSpc{1}}} \pi_n(X),
    \end{equation*}
    and if $n = 1$, we have an isomorphism
    \begin{equation*}
        \pi_1 (\iota_{\nis,\AffSpc{1}} X) \cong \iota_{\nis} \pi_1(\iota_{\AffSpc{1}} X),
    \end{equation*}
    where we view $\iota_{\nis}$ as a functor 
    \begin{equation*}
        \Grp{\Disc{\MotShv{k}}} \to \Grp{\Disc{\ZarShv{k}}}.
    \end{equation*}
\end{cor}
\begin{proof}
    From \cref{lemma:motivic:truncation-cover-iota-if-connected,lemma:motivic:connected-cover-iota} 
    we are immediately able to conclude that 
    $\pi_n(\iota_{\nis, \AffSpc{1}} X) \cong \iota_{\nis} \pi_n(\iota_{\AffSpc{1}} X)$.
    Moreover, by definition $\iota_{\AffSpc{1}} \pi_n(X) = \pi_n(\iota_{\AffSpc{1}} X)$,
    therefore we also get an equivalence $\pi_n(\iota_{\nis, \AffSpc{1}} X) \cong \iota_{\nis, \AffSpc{1}} \pi_n(X)$.
    Since everything is in the heart of the standard t-structure, we get the desired equivalences.
    If $n = 1$, then the same proof works, but we ignore the hearts and view 
    $\iota_{\nis}$ as a functor $\Grp{\Disc{\MotShv{k}}} \to \Grp{\Disc{\ZarShv{k}}}$.
\end{proof}

\begin{lem} \label{lemma:motivic:iota-A1-nis-t-exact-p-adic}
    The functor $\iota_{\nis,\AffSpc{1}} \colon \SH{k} \to \ZarShvSp{k}$ is t-exact 
    for the $p$-adic t-structures.

    In particular, it induces a fully faithful exact functor 
    \begin{equation*}
        \pheart{\iota}_{\nis,\AffSpc{1}} \colon \pheart{\SH{k}} \to \pheart{\ZarShvSp{k}}.
    \end{equation*}
\end{lem}
\begin{proof}
    By \cref{lemma:motivic:iota-A1-nis-t-exact-std}, $\iota_{\nis, \AffSpc{1}}$ is t-exact for the standard t-structures.
    Therefore $L_{\nis, \AffSpc{1}}$ is right t-exact for the standard t-structures by \cite[Proposition 1.3.17(iii)]{BeilinsonFaisceauxPerverse}.
    Now \cref{lemma:t-struct:right-t-exact} applied to $L = \iota_{\nis,\AffSpc{1}}$ 
    implies that $\iota_{\nis,\AffSpc{1}}$ is right t-exact,
    whereas the same lemma applied to $L = L_{\nis,\AffSpc{1}}$ and $R = \iota_{\nis,\AffSpc{1}}$ 
    implies that $\iota_{\nis,\AffSpc{1}}$ is left t-exact.
    This proves the first part of the lemma.

    The last part is \cite[Proposition 1.3.17(i)]{BeilinsonFaisceauxPerverse}.
\end{proof}

\subsection{\texorpdfstring{$\AffSpc{1}$}{A-1}-Invariance of the \texorpdfstring{$p$}{p}-Completion}
The category of motivic spaces is not an $\infty$-topos.
Nonetheless, it is presentable (see \cref{lemma:motivic:motivic-spaces-adjunction}).
Therefore, \cref{section:presentable-completion}
applies and gives us a notion of $p$-equivalence, and a $p$-completion functor
$\completebr{-} \colon \MotSpc{k} \to \MotSpc{k}$.
In this section we prove that at least for nilpotent motivic spaces,
the $p$-completion of the underlying Nisnevich sheaf is still $\AffSpc{1}$-invariant,
and agrees with the $p$-completion of $X$ in the category of pointed connected motivic spaces, see \cref{thm:motivic:iota-a1-completion}.

\begin{rmk}
    We will also show in \cref{lemma:motivic:iota-completion}
    that the $p$-completion of a nilpotent motivic space agrees with the 
    $p$-completion of the underlying Zariski sheaf.
    This is unclear for arbitrary Nisnevich sheaves, even if we assume nilpotence.
\end{rmk}

Recall that Asok-Fasel-Hopkins defined in \cite[Definition 3.3.1]{asok2022localization}
what a nilpotent motivic space is.

\begin{lem} \label{lemma:motivic:nilpotent}
    A pointed motivic space $X \in \MotSpc{k}_*$
    is nilpotent if and only if $\iota_{\AffSpc{1}} X$ is nilpotent as a Nisnevich sheaf
    in the sense of \cref{def:nilpotent:defn}.
\end{lem}
\begin{proof}
    One direction is clear from the definitions, since the homotopy groups (and the action of $\pi_1$) 
    of a motivic space are the same as the homotopy groups (and the action of $\pi_1$) of the underlying Nisnevich sheaf of anima.
    For the other direction one uses \cite[Proposition 3.2.3]{asok2022localization} (and its variant 
    for actions of $\pi_1$ on $\pi_n$) to conclude that every nilpotent Nisnevich sheaf of groups 
    which is strictly $\AffSpc{1}$-invariant is already $\AffSpc{1}$-nilpotent.
\end{proof}

\begin{lem} \label{lemma:nisnevich:locally-of-uniform-homotopy-dimension}
    $\MotShv{k}$ and $\ZarShv{k}$ are locally of finite uniform homotopy dimension.
\end{lem}
\begin{proof}
    Let $\Cat S$ be the collection of all points of $\MotShv{k}$,
    and $\htpydim \colon \Cat S \to \N$ be the function
    $S_s^h \mapsto \dim(S)$.

    Let $F \in \MotShv{k}$ be $k$-connective,
    $S_s^h$ be a point and $U \in \mathcal I_s$.
    Then $U \to S$ is an étale neighborhood of $s$, and thus $\dim(U) = \dim(S)$ (by the assumption
    on the connectedness of $S$).
    Denote by $\topos{X}_U$ the category of sheaves
    on the site of étale morphisms over $U$ with Nisnevich covers.
    There is a functor $f_U \colon \MotShv{k} \to \topos{X}_U$
    given by restriction.
    Note that $F(U) \cong (f_U F)(U)$.
    Since by \cite[Theorem 3.7.7.1]{sag} $\topos{X}_U$ has homotopy dimension $\le \dim(S)$,
    we conclude that $F(U)$ is $k-\htpydim(s)$-connective (note that $f_UF$ is still $k$-connective,
    as $f_U$ commutes with homotopy objects, to prove this, one argues exactly as in the Zariski case,
    see the proof of \cref{lemma:zar:postnikov-complete}).

    For the Zariski $\infty$-topos one argues similar, noting that the points of the Zariski $\infty$-topos
    are given by the local schemes $S_s$.
    To see that the small Zariski $\infty$-topos over a smooth scheme $U$ has homotopy dimension $\le \dim(U)$,
    one uses \cite[Corollary 7.2.4.17]{highertopoi}.
\end{proof}

\begin{cor} \label{cor:nisnevich:postnikov-completion}
    Let $X \in \MotShv{k}_*$ or $X \in \ZarShv{k}_*$ be nilpotent.
    Then $\complete{X} = \limil{n} \completebr{\tau_{\le n}X}$.
\end{cor}
\begin{proof}
    This is \cref{thm:topos:p-comp-commutes-with-post-tower},
    together with \cref{lemma:nisnevich:locally-of-uniform-homotopy-dimension}.
    Here we use that the Zariski and Nisnevich topoi are Postnikov-complete,
    see \cref{lemma:zar:postnikov-complete} and \cref{lemma:nisnevich:postnikov-complete}.
\end{proof}

\begin{prop} \label{thm:motivic:completion-a1-invariant}
    Let $X \in \MotSpc{k}_*$ be nilpotent.
    Then the $p$-completion $\completebr{\iota_{\AffSpc{1}} X}$ is an $\AffSpc{1}$-invariant sheaf.
\end{prop}
\begin{proof}
    By \cref{cor:nisnevich:postnikov-completion} there are equivalences
    $\completebr{\iota_{\AffSpc{1}}X} \cong \completebr{\limil{n} \tau_{\le n}\iota_{\AffSpc{1}}X} \cong \limil{n} \completebr{\tau_{\le n}\iota_{\AffSpc{1}}X}$.
    Since the limit of $\AffSpc{1}$-invariant sheaves is $\AffSpc{1}$-invariant (as the inclusion $\iota_{\AffSpc{1}}$ is a right adjoint,
    i.e.\ commutes with limits),
    we can assume that $X$ is $n$-truncated (i.e. $\iota_{\AffSpc{1}} X$ is $n$-truncated).
    We proceed by induction on $n$, the case $n = 0$ being trivial.
    Using \cite[Theorem 3.3.13]{asok2022localization}
    the Postnikov tower of $X$ has a principal refinement consisting
    of (nilpotent) motivic spaces $X_{n, k}$,
    and sheaves of spectra $A_{n, k+1} \in \heart{\SH{k}}$,
    such that there are fiber sequences 
    \begin{equation*}
        X_{n, k+1} \to X_{n, k} \to K(A_{n, k+1}, n+1)
    \end{equation*}
    and equivalences $X_{n, 0} \cong \tau_{\le n} X$.
    Applying $\iota_{\AffSpc{1}}$ to the fiber sequences gives the fiber sequence 
    \begin{equation*}
        \iota_{\AffSpc{1}}X_{n, k+1} \to \iota_{\AffSpc{1}}X_{n, k} \to K(\heart{\iota}_{\AffSpc{1}}A_{n, k+1}, n+1),
    \end{equation*}
    where we used \cref{lemma:motivic:iota-a1-EM}.
    Note that by \cref{lemma:motivic:nilpotent}, all of those sheaves are nilpotent.

    We can thus proceed by induction on $0 \le k \le m_n$.
    We know that $\completebr{\iota_{\AffSpc{1}}X_{n, 0}} \cong \completebr{\tau_{\le n} \iota_{\AffSpc{1}}X}$
    is $\AffSpc{1}$-invariant by induction (on $n$).
    Thus suppose we have shown that $\completebr{\iota_{\AffSpc{1}}X_{n, k}}$ is $\AffSpc{1}$-invariant, $k < m_n$.
    Using the above fiber sequence, we can compute the $p$-completion using \cref{lemma:postnikov-fiber-sequence-completion}:
    \begin{equation*}
        \completebr{\iota_{\AffSpc{1}}X_{n,k+1}} = \tau_{\ge 1} \Fib{\completebr{\iota_{\AffSpc{1}}X_{n,k}} \to \completebr{K(\heart{\iota}_{\AffSpc{1}}A_{n,k}, n+1)}}.
    \end{equation*}
    Since fibers and connected covers (\cref{lemma:motivic:connective-cover-descent-to-a1}) of $\AffSpc{1}$-invariant sheaves are $\AffSpc{1}$-invariant,
    we can reduce  to the case $X = K(\heart{\iota}_{\AffSpc{1}}A, n)$ for some $A \in \heart{\SH{k}}$ and $n \ge 2$.
    
    But then $\complete{X} \cong \tau_{\ge 1} \pLoop \left(\completebr{\Sigma^n \heart{\iota}_{\AffSpc{1}}A}\right)$.
    Since connected covers of $\AffSpc{1}$-invariant sheaves are $\AffSpc{1}$-invaraint (again by \cref{lemma:motivic:connective-cover-descent-to-a1}),
    it suffices to show that $\completebr{\heart{\iota}_{\AffSpc{1}} A}$ is $\AffSpc{1}$-invariant.
    But this is just a limit of $\AffSpc{1}$-invariant sheaves of spectra,
    and therefore $\AffSpc{1}$-invariant (as $\iota_{\AffSpc{1}}$ is a right adjoint).
\end{proof}

\begin{rmk}
    We now want to show that the $p$-completion of a nilpotent motivic space is the same as the 
    $p$-completion of the underlying Nisnevich sheaf.
    In order to do this, 
    one needs to show that the motivic space $L_{\AffSpc{1}} (\completebr{\iota_{\AffSpc{1} X}})$
    is again $p$-complete.
    We would like to argue again using the principal refinement of the Postnikov tower,
    and write this motivic space as a repeated limit of $p$-completions of Eilenberg Mac-Lane spaces.
    Unfortunately, this approach has a major drawback: By calculating $p$-completions on the Postnikov tower,
    connective covers will appear. This introduces a problem: Since the category of motivic spaces is not 
    an $\infty$-topos, we cannot use the arguments from \cref{section:topos} to conclude
    that the connective cover of a $p$-complete space is again $p$-complete,
    since it is not at all clear that the $p$-completion of motivic spaces respects $\pi_0$.
    We can correct this error by working in the category of connected motivic spaces (in particular,
    every nilpotent motivic space is connected).
    This also leads to the following conjecture:
\end{rmk}

\begin{conj} \label{conj:motivic:conjecture-A}
    Let $X \in \MotSpc{k}_*$ be a pointed motivic space.
    If $X$ is $p$-complete, then also $\tau_{\ge 1} X$ is $p$-complete.
\end{conj}

We now introduce the category of pointed connected motivic spaces:
\begin{defn}
    Write $\MotSpc{k}_{\ge 1,*}$ for the category of \emph{pointed connected motivic spaces},
    i.e.\ the full subcategory of $\MotSpc{k}_*$ spanned by objects $X$
    such that the underlying Nisnevich sheaf $\iota_{\AffSpc{1}} X$ is connected (i.e.\ $\pi_0(\iota_{\AffSpc{1}} X) = *$).
\end{defn}

\begin{rmk}
    Note that we have homotopy sheaves $\pi_n \colon \MotSpc{k}_{\ge 1,*} \to \heart{\SH{k}}$
    for $n \ge 2$, and $\pi_1 \colon \MotSpc{k}_{\ge 1,*} \to \GrpStr{k}$.
\end{rmk}

\begin{rmk}
    Note that $\MotSpc{k}_{\ge 1,*}$ is presentable:
    It is stable under all colimits in $\MotSpc{k}_*$, and
    is the preimage of the terminal category $*$ under the accessible functor 
    $\pi_0 \circ \iota_{\AffSpc{1}} \colon \MotSpc{k}_* \to \Disc{\MotShv{k}}$,
    thus also accessible by \cite[Proposition 5.4.6.6]{highertopoi}.
    Hence, we can apply \cref{section:presentable-completion}
    and get a $p$-completion functor on this category.
\end{rmk}

Using the presentability of $\MotSpc{k}_{\ge 1,*}$ and the observation 
that the inclusion $\MotSpc{k}_{\ge 1,*} \to \MotSpc{k}_*$
preserves colimits (this follows from the fact that $L_{\AffSpc{1}}$ preserves connected objects),
the adjoint functor theorem gives us a right adjoint.
\begin{defn}
    Write $\iota_{\ge 1} \colon \MotSpc{k}_{\ge 1,*} \rightleftarrows \MotSpc{k}_* \colon \tau_{\ge 1}$
    for the canonical adjunction.
    We define as shorthand the following notations:
    \begin{align*}
        \iota_{\AffSpc{1},\ge 1} &\coloneqq \iota_{\AffSpc{1}} \iota_{\ge 1} \colon \MotSpc{k}_{\ge 1,*} \to \MotShv{k}_*, \text{ and} \\
        \iota_{\nis,\AffSpc{1},\ge 1} &\coloneqq \iota_{\nis} \iota_{\AffSpc{1}} \iota_{\ge 1} \colon \MotSpc{k}_{\ge 1,*} \to \ZarShv{k}_*.
    \end{align*}
\end{defn}

\begin{lem} \label{lemma:motivic:connected-stabilization}
    We have an equivalence of categories 
    $\SH{k} \cong \Stab{\MotSpc{k}_{\ge 1,*}}$.

    In particular, we have a commuting diagram 
    \begin{center}
        \begin{tikzcd}
            \MotSpc{k}_{\ge 1,*} \arrow[r, hook] \arrow[dr, "\Sus"] &\MotSpc{k}_* \arrow[d, "\Sus"] \\
            &\SH{k}.
        \end{tikzcd}
    \end{center}
    Thus, if $f \colon X \to Y$ is a morphism of connected pointed motivic spaces,
    then it is a $p$-equivalence if and only if the underlying morphism of pointed motivic spaces $\iota_{\ge 1} f$ is a $p$-equivalence.
\end{lem}
\begin{proof}
    Recall from \cite[Remark 1.4.2.25]{higheralgebra} that there are equivalences of $\infty$-categories 
    \begin{equation*}
        \SH{k} \cong \limil{} (\dots \xrightarrow{\Omega} \MotSpc{k}_* \xrightarrow{\Omega} \MotSpc{k}_*)
    \end{equation*}
    and 
    \begin{equation*}
        \Stab{\MotSpc{k}_{\ge 1,*}} \cong \limil{} (\dots \xrightarrow{\Omega} \MotSpc{k}_{\ge 1,*} \xrightarrow{\Omega} \MotSpc{k}_{\ge 1,*}).
    \end{equation*}
    The result follows by a cofinality argument, using that we have equivalences $\Omega \tau_{\ge 1} X \cong \Omega X$ 
    for every pointed motivic space $X$.
\end{proof}

Using the last lemma, from now on we will identify the stabilization of $\MotSpc{k}_{\ge 1,*}$ with $\SH{k}$.

\begin{defn}
    Let $X \in \MotSpc{k}_{\ge 1,*}$.
    We say that $X$ is \emph{nilpotent} if the underlying motivic space is nilpotent.
\end{defn}

\begin{thm} \label{thm:motivic:iota-a1-completion}
    Let $X \in \MotSpc{k}_{\ge 1,*}$ be a nilpotent pointed motivic space
    (note that every nilpotent space is connected).
    We have a canonical equivalence $\iota_{\AffSpc{1},\ge 1} (\complete{X}) \cong \completebr{\iota_{\AffSpc{1},\ge 1} X}$.
    In other words, the $p$-completion of a nilpotent pointed connected motivic space can be computed
    on the underlying Nisnevich sheaf. 
\end{thm}
\begin{proof}
    Let $\iota_{\AffSpc{1},\ge 1} X \to \completebr{\iota_{\AffSpc{1},\ge 1} X}$ be the canonical $p$-equivalence.
    Applying $L_{\AffSpc{1}}$ yields the $p$-equivalence (in $\MotSpc{k}_*$)
    \begin{equation*}
        \iota_{\ge 1} X \cong L_{\AffSpc{1}} \iota_{\AffSpc{1}, \ge 1} X \to L_{\AffSpc{1}} \left(\completebr{\iota_{\AffSpc{1}, \ge 1} X} \right).
    \end{equation*}
    Note that $\iota_{\ge 1}X$ is connected by assumption, and that the right-hand side is connected 
    because the $p$-completion in an $\infty$-topos preserves connected objects (see \cref{lemma:peq-respects-pi0}),
    and the same is true for $L_{\AffSpc{1}}$, see \cite[Corollary 3.2.5]{moreltrieste}.
    Thus, this is a morphism in $\MotSpc{k}_{\ge 1,*}$,
    and hence we have a $p$-equivalence $X \to \tau_{\ge 1} L_{\AffSpc{1}} (\completebr{\iota_{\AffSpc{1}, \ge 1} X})$,
    see \cref{lemma:motivic:connected-stabilization}.
    It suffices to show that the right object is $p$-complete:
    Then $p$-completion induces an equivalence 
    $\complete{X} \cong \tau_{\ge 1} L_{\AffSpc{1}} (\completebr{\iota_{\AffSpc{1},\ge 1} X})$.
    Applying $\iota_{\AffSpc{1},\ge 1}$ then induces an equivalence 
    \begin{equation*}
        \iota_{\AffSpc{1},\ge 1} (\complete{X}) \cong \iota_{\AffSpc{1}, \ge 1} \tau_{\ge 1} L_{\AffSpc{1}} \left(\completebr{\iota_{\AffSpc{1}, \ge 1} X}\right) \cong \completebr{\iota_{\AffSpc{1}, \ge 1} X},
    \end{equation*}
    where we used in the last equivalence that $\completebr{\iota_{\AffSpc{1}} X}$ is already connected (by the above discussion) 
    and $\AffSpc{1}$-invariant (see \cref{thm:motivic:completion-a1-invariant}).

    In order to see that $\tau_{\ge 1} L_{\AffSpc{1}} \completebr{\iota_{\AffSpc{1}, \ge 1} X}$ is $p$-complete,
    we first reduce to the case that $X$ is truncated:
    For this, we calculate 
    \begin{align*}
        \tau_{\ge 1} L_{\AffSpc{1}} (\completebr{\iota_{\AffSpc{1}, \ge 1} X})
        &\cong \tau_{\ge 1} L_{\AffSpc{1}} \limil{n} \completebr{\tau_{\le n} \iota_{\AffSpc{1}, \ge 1} X} \\
        &\cong \tau_{\ge 1} L_{\AffSpc{1}} \limil{n} \iota_{\AffSpc{1}} L_{\AffSpc{1}} (\completebr{\tau_{\le n} \iota_{\AffSpc{1}, \ge 1} X}) \\
        &\cong \tau_{\ge 1} L_{\AffSpc{1}} \iota_{\AffSpc{1}} \limil{n} L_{\AffSpc{1}} (\completebr{\tau_{\le n} \iota_{\AffSpc{1}, \ge 1} X}) \\
        &\cong \tau_{\ge 1} \limil{n} L_{\AffSpc{1}} (\completebr{\tau_{\le n} \iota_{\AffSpc{1}, \ge 1} X}) \\
        &\cong \limil{n} \tau_{\ge 1}  L_{\AffSpc{1}} (\completebr{\tau_{\le n} \iota_{\AffSpc{1}, \ge 1} X}),
    \end{align*}
    where we used \cref{cor:nisnevich:postnikov-completion} for the first equivalence,
    and that $\completebr{\tau_{\le n} \iota_{\AffSpc{1}, \ge 1} X}$ is $\AffSpc{1}$-invariant in the second 
    equivalence (see \cref{thm:motivic:completion-a1-invariant}, using that $\tau_{\le n} \iota_{\AffSpc{1}, \ge 1} X$ is nilpotent).
    The third equivalence holds because $\iota_{\AffSpc{1}}$ commutes with limits,
    the fourth equivalence is fully faithfulness of $\iota_{\AffSpc{1}}$,
    and the last equivalence uses that $\tau_{\ge 1}$ is a right adjoint.
    Since limits of $p$-complete objects are $p$-complete, it suffices to prove the statement for truncated nilpotent connected motivic spaces.
    
    Proceeding as in the proof of the last proposition, we choose a principal refinement of the Postnikov tower (note 
    that all the $X_{n, k}$ are automatically connected since they are nilpotent),
    and do double induction on $n$ and $k$ (with notation as in the proof of \cref{thm:motivic:completion-a1-invariant}).
    Therefore, we assume that the statement is true for $X_{n, k}$ (i.e.\ $\tau_{\ge 1} L_{\AffSpc{1}}(\completebr{\iota_{\AffSpc{1},\ge 1} X})$ 
    is $p$-complete), and that there is a fiber sequence
    \begin{equation*}
        \iota_{\AffSpc{1}, \ge 1}X_{n, k+1} \to \iota_{\AffSpc{1}, \ge 1}X_{n, k} \to K(\heart{\iota}_{\AffSpc{1}}A_{n, k+1}, n+1).
    \end{equation*}
    Using the above fiber sequence, we can compute the $p$-completion using \cref{lemma:postnikov-fiber-sequence-completion}.
    Applying $\tau_{\ge 1}L_{\AffSpc{1}}$, we calculate
    \begin{align*}
        &\tau_{\ge 1}L_{\AffSpc{1}} \left(\completebr{\iota_{\AffSpc{1}, \ge 1}X_{n,k+1}}\right) \\
        &\cong \tau_{\ge 1} L_{\AffSpc{1}} \tau_{\ge 1} \Fib{\completebr{\iota_{\AffSpc{1}, \ge 1}X_{n,k}} \to \completebr{K(\heart{\iota}_{\AffSpc{1}}A_{n,k}, n+1)}} \\
        &\cong \tau_{\ge 1} L_{\AffSpc{1}} \tau_{\ge 1} \Fib{\iota_{\AffSpc{1}} L_{\AffSpc{1}} \left(\completebr{\iota_{\AffSpc{1}, \ge 1}X_{n,k}}\right) \to \iota_{\AffSpc{1}} L_{\AffSpc{1}} \left(\completebr{K(\heart{\iota}_{\AffSpc{1}}A_{n,k}, n+1)}\right)} \\
        &\cong \tau_{\ge 1} L_{\AffSpc{1}} \tau_{\ge 1} \iota_{\AffSpc{1}} \Fib{L_{\AffSpc{1}} \left(\completebr{\iota_{\AffSpc{1}, \ge 1}X_{n,k}}\right) \to L_{\AffSpc{1}} \left(\completebr{K(\heart{\iota}_{\AffSpc{1}}A_{n,k}, n+1)}\right)} \\
        &\cong \tau_{\ge 1} L_{\AffSpc{1}} \iota_{\AffSpc{1},\ge 1} \tau_{\ge 1} \Fib{L_{\AffSpc{1}} \left(\completebr{\iota_{\AffSpc{1}, \ge 1}X_{n,k}}\right) \to L_{\AffSpc{1}} \left(\completebr{K(\heart{\iota}_{\AffSpc{1}}A_{n,k}, n+1)}\right)} \\
        &\cong \tau_{\ge 1} \iota_{\ge 1} \tau_{\ge 1} \Fib{L_{\AffSpc{1}} \left(\completebr{\iota_{\AffSpc{1}}X_{n,k}}\right) \to L_{\AffSpc{1}} \left(\completebr{K(\heart{\iota}_{\AffSpc{1}}A_{n,k}, n+1)}\right)} \\
        &\cong \tau_{\ge 1} \Fib{L_{\AffSpc{1}} \left(\completebr{\iota_{\AffSpc{1},\ge 1}X_{n,k}}\right) \to L_{\AffSpc{1}} \left(\completebr{K(\heart{\iota}_{\AffSpc{1}}A_{n,k}, n+1)}\right)} \\
        &\cong \Fib{\tau_{\ge 1} L_{\AffSpc{1}} \left(\completebr{\iota_{\AffSpc{1},\ge 1}X_{n,k}}\right) \to \tau_{\ge 1} L_{\AffSpc{1}} \left(\completebr{K(\heart{\iota}_{\AffSpc{1}}A_{n,k}, n+1)}\right)},
    \end{align*}
    Here, the second equivalence holds because both $p$-completions on the right 
    are actually $\AffSpc{1}$-invariant, see again \cref{thm:motivic:completion-a1-invariant}.
    The third, fourth and fifth equivalences hold because $\iota_{\AffSpc{1}}$ commutes with limits and
    the connective cover (\cref{lemma:motivic:connective-cover-descent-to-a1}),
    and is fully faithful.
    The sixth equivalence is fully faithfulness of $\iota_{\ge 1}$,
    and the last equivalence holds because $\tau_{\ge 1}$ commutes with limits.
    By induction, $\tau_{\ge 1} L_{\AffSpc{1}} \left(\completebr{\iota_{\AffSpc{1},\ge 1}X_{n,k}}\right)$ is $p$-complete.
    Since fibers of $p$-complete objects are $p$-complete,
    we have reduced to the case of an Eilenberg-MacLane space.

    So suppose that $n \ge 2$ and $A \in \heart{\SH{k}}$ is strictly $\AffSpc{1}$-invariant.
    We need to show that $\tau_{\ge 1} L_{\AffSpc{1}} \left(\completebr{K(\heart{\iota}_{\AffSpc{1}} A, n)}\right)$
    is $p$-complete (in connected motivic spaces).
    We compute 
    \begin{align*}
        \tau_{\ge 1} L_{\AffSpc{1}} \left(\completebr{K(\heart{\iota}_{\AffSpc{1}} A, n)}\right) 
        &\cong \tau_{\ge 1} L_{\AffSpc{1}} \tau_{\ge 1} \pLoop \left(\completebr{\Sigma^n \heart{\iota}_{\AffSpc{1}} A}\right) \\
        &\cong \tau_{\ge 1} L_{\AffSpc{1}} \tau_{\ge 1} \pLoop \left(\completebr{\Sigma^n \iota_{\AffSpc{1}} A}\right) \\
        &\cong \tau_{\ge 1} L_{\AffSpc{1}} \tau_{\ge 1} \pLoop \iota_{\AffSpc{1}} \left(\completebr{\Sigma^n A}\right) \\
        &\cong \tau_{\ge 1} L_{\AffSpc{1}} \tau_{\ge 1} \iota_{\AffSpc{1}} \pLoop \left(\completebr{\Sigma^n A}\right) \\
        &\cong \tau_{\ge 1} L_{\AffSpc{1}} \iota_{\AffSpc{1}, \ge 1} \tau_{\ge 1} \pLoop \left(\completebr{\Sigma^n A}\right) \\
        &\cong \tau_{\ge 1} \iota_{\ge 1} \tau_{\ge 1} \pLoop \left(\completebr{\Sigma^n A}\right) \\
        &\cong \tau_{\ge 1} \pLoop \left(\completebr{\Sigma^n A}\right),
    \end{align*}
    where we used \cref{cor:completion-of-EM-space} in the first equivalence,
    and t-exactness of $\iota_{\AffSpc{1}}$ (\cref{lemma:motivic:homotopy-t-structure}) in the second equivalence.
    The third equivalence holds because $\iota_{\AffSpc{1}}$ commutes with limits,
    the fourth equivalence is \cref{lemma:adjoints-on-stabilization},
    and the fifth is \cref{lemma:motivic:connective-cover-descent-to-a1}.
    The last two equivalences use fully faithfulness of $\iota_{\AffSpc{1}}$ 
    and $\iota_{\ge 1}$.
    The theorem follows because $\tau_{\ge 1}\pLoop$ preserves $p$-complete objects (as its left adjoint $\Sus \colon \MotSpc{k}_{\ge 1, *} \to \SH{k}$ preserves $p$-equivalences by definition).
\end{proof}

\begin{rmk}
    Note that if \cref{conj:motivic:conjecture-A} is true,
    then the same reasoning allows us to prove the following result:
    If $X \in \MotSpc{k}_*$ is a pointed nilpotent space,
    then $\completebr{\iota_{\AffSpc{1}} X} \cong \iota_{\AffSpc{1}} \complete{X}$. 
\end{rmk}

The same technique allows us to prove a related result: The $p$-completion of the underlying Nisnevich sheaf of a nilpotent motivic space 
is also the $p$-completion of the underlying Zariski sheaf.
For this, we need the following lemma:

\begin{lem} \label{lemma:motivic:completion-EM-iota}
    Let $A \in \heart{\SH{k}}$ and $n \ge 2$.

    There is an equivalence $\iota_{\nis}(\complete{K(\heart{\iota_{\AffSpc{1}}} A, n)}) \cong \complete{K(\heart{\iota_{\nis, \AffSpc{1}}} A, n)}$.
\end{lem}
\begin{proof}
    Note that since $\iota_{\AffSpc{1}}$ and $\iota_{\nis,\AffSpc{1}}$ are t-exact for the standard t-structures 
    (see \cref{lemma:motivic:homotopy-t-structure,lemma:motivic:iota-A1-nis-t-exact-std}),
    we see that $\heart{\iota_{\AffSpc{1}}} A \cong \iota_{\AffSpc{1}} A$,
    and similarly, $\heart{\iota_{\nis, \AffSpc{1}}} A \cong \iota_{\nis, \AffSpc{1}} A$.
    Therefore, we see that $K(\heart{\iota_{\AffSpc{1}}} A, n) \cong \pLoop \Sigma^n \iota_{\AffSpc{1}} A$,
    and $K(\heart{\iota_{\nis,\AffSpc{1}}} A, n) \cong \pLoop \Sigma^n \iota_{\nis,\AffSpc{1}} A$.
    Thus, it suffices to show that there is an equivalence 
    \begin{equation*}
        \iota_{\nis}(\completebr{\pLoop \Sigma^n \iota_{\AffSpc{1}} A}) \cong \completebr{\pLoop \Sigma^n \iota_{\nis, \AffSpc{1}} A}.
    \end{equation*}
    We now calculate
    \begin{align*}
        \iota_{\nis} \left(\completebr{\pLoop \Sigma^n\iota_{\AffSpc{1}} A}\right) 
        & \cong \iota_{\nis} \pLoop \tau_{\ge 1} \left(\completebr{\Sigma^n \iota_{\AffSpc{1}} A}\right) \\
        & \cong \pLoop \iota_{\nis} \tau_{\ge 1} \left(\completebr{\Sigma^n \iota_{\AffSpc{1}} A}\right) \\
        & \cong \pLoop \iota_{\nis} \tau_{\ge 1} \iota_{\AffSpc{1}} \left(\completebr{\Sigma^n A}\right) \\
        & \cong \pLoop \tau_{\ge 1} \iota_{\nis, \AffSpc{1}} \left(\completebr{\Sigma^n A}\right) \\
        & \cong \pLoop \tau_{\ge 1} \left(\completebr{\Sigma^n \iota_{\nis, \AffSpc{1}} A}\right) \\
        & \cong \completebr{\pLoop \Sigma^n \iota_{\nis, \AffSpc{1}} A}.
    \end{align*}
    Here, the first and last equivalences are \cref{cor:completion-of-EM-space},
    the second equivalence is \cref{lemma:adjoints-on-stabilization},
    the third and fifth equivalences follow from \cref{lemma:t-struct:right-exact-commutes-with-completion}
    and the exactness of $\iota_{\AffSpc{1}}$ and $\iota_{\nis,\AffSpc{1}}$,
    and the fourth equivalence is \cref{lemma:motivic:connected-cover-iota}.
\end{proof}

\begin{thm} \label{lemma:motivic:iota-completion}
    Let $X \in \MotSpc{k}_*$ be nilpotent.
    Then $\iota_{\nis}(\completebr{\iota_{\AffSpc{1}}X}) \cong \completebr{\iota_{\nis,\AffSpc{1}} X}$.

    In particular, if we regard $X$ as an object of $\MotSpc{k}_{\ge 1,*}$ we get an equivalence 
    $\iota_{\nis, \AffSpc{1},\ge 1} (\complete{X}) \cong \completebr{\iota_{\nis,\AffSpc{1},\ge 1} X}$
    by combining this result with \cref{thm:motivic:iota-a1-completion}.
\end{thm}
\begin{proof}
    First, assume that $X$ is $n$-truncated for some $n$.
    As above, we choose a principal refinement of the Postnikov tower of $X$,
    with $X_{n,k} \in \MotSpc{k}_*$ and $A_{n, k} \in \heart{\SH{k}}$.
    We proceed by double induction on $n$ and $k$, the case $n = 0$ being trivial.
    As above, we have a fiber sequence
    \begin{equation*}
        \iota_{\AffSpc{1}} X_{n, k+1} \to \iota_{\AffSpc{1}} X_{n, k} \to K(\heart{\iota}_{\AffSpc{1}}A_{n, k+1}, n+1).
    \end{equation*}
    Applying $\iota_{\nis}$, we get a fiber sequence
    \begin{equation*}
        \iota_{\nis} \iota_{\AffSpc{1}} X_{n, k+1} \to \iota_{\nis} \iota_{\AffSpc{1}} X_{n, k} \to K(\heart{\iota}_{\nis}\heart{\iota}_{\AffSpc{1}} A_{n, k+1}, n+1),
    \end{equation*}
    where we used \cref{lemma:motivic:iota-a1-EM}.
    We now compute
    \begin{align*}
        \completebr{\iota_{\nis} \iota_{\AffSpc{1}} X_{n, k+1}}
         & \cong \tau_{\ge 1} \Fib{\completebr{\iota_{\nis} \iota_{\AffSpc{1}} X_{n, k}} \to \complete{K(\heart{\iota}_{\nis}\heart{\iota}_{\AffSpc{1}}A_{n, k+1}, n+1)}}     \\
         & \cong \tau_{\ge 1} \Fib{\iota_{\nis} (\completebr{\iota_{\AffSpc{1}}X_{n, k}}) \to \iota_{\nis} (\complete{K(\heart{\iota}_{\AffSpc{1}} A_{n, k+1}, n+1)})} \\
         & \cong \tau_{\ge 1} \iota_{\nis} \Fib {\completebr{\iota_{\AffSpc{1}}X_{n, k}} \to \complete{K(\heart{\iota}_{\AffSpc{1}}A_{n, k+1}, n+1)}}          \\
         & \cong \iota_{\nis} \tau_{\ge 1} \Fib {\completebr{\iota_{\AffSpc{1}}X_{n, k}} \to \complete{K(\heart{\iota}_{\AffSpc{1}}A_{n, k+1}, n+1)}}          \\
         & \cong \iota_{\nis} (\completebr{\iota_{\AffSpc{1}} X_{n, k+1}}).
    \end{align*}
    Here, the first and last equivalences are \cref{lemma:postnikov-fiber-sequence-completion},
    the second equivalence follows from induction and \cref{lemma:motivic:completion-EM-iota},
    the third equivalence exists because $\iota_{\nis}$ commutes with limits (as a right adjoint),
    and the fourth equivalence is \cref{lemma:motivic:connected-cover-iota} (noting that the fiber is $\AffSpc{1}$-invariant 
    as a limit of $\AffSpc{1}$-invariant sheaves).
    This proves the claim.

    We will now deduce the general case. We have the following chain of equivalences:
    \begin{align*}
        \iota_{\nis} \left(\completebr{\iota_{\AffSpc{1}}X}\right) 
        & \cong \iota_{\nis} \limil{n} \completebr{\tau_{\le n} \iota_{\AffSpc{1}}X} \\
        & \cong \limil{n} \iota_{\nis} \left(\completebr{\tau_{\le n} \iota_{\AffSpc{1}}X}\right) \\
        & \cong \limil{n} \completebr{\iota_{\nis} \tau_{\le n} \iota_{\AffSpc{1}} X} \\
        & \cong \limil{n} \completebr{\tau_{\le n} \iota_{\nis,\AffSpc{1}} X} \\
        & \cong \completebr{\iota_{\nis,\AffSpc{1}} X}.
    \end{align*}
    The first and last equivalences are \cref{cor:nisnevich:postnikov-completion}.
    The second equivalence holds because $\iota$ commutes with limits (as a right adjoint).
    The third equivalence was proven above, since $\tau_{\le n} X$ is $n$-truncated.
    The fourth equivalence is \cref{lemma:motivic:truncation-cover-iota-if-connected} (note that $X$ is connected because it is nilpotent).
    This proves the theorem.
\end{proof}

\begin{rmk}
    Again, if \cref{conj:motivic:conjecture-A} is true,
    then we get the following:
    If $X \in \MotSpc{k}_*$ is a pointed nilpotent space,
    then $\completebr{\iota_{\nis,\AffSpc{1}}X} \cong \iota_{\nis,\AffSpc{1}} \complete{X}$. 
\end{rmk}

\subsection{A Short Exact Sequence for Motivic Spaces}

We want to establish a short exact sequence for the homotopy objects of the $p$-completion of motivic spaces,
similar to the one for Zariski sheaves from \cref{thm:pro-zar:short-exact-sequence}.

\begin{lem} \label{lemma:motivic:gersten-of-strictly}
    Let $A \in \heart{\SH{k}}$.
    Then $\heart{\iota}_{\nis, \AffSpc{1}} A$ satisfies Gersten injectivity (\cref{def:pro-zar:gersten}).
\end{lem}
\begin{proof}
    This is proven in \cite[Lemma 4.6]{AsokGersten}, if $k$ is an infinite field.
    If $k$ is a finite field, we can argue as in the above reference,
    using the Gabber presentation lemma for finite fields, see \cite[Theorem 1.1]{hogadi2020gabber}.
\end{proof}

\begin{lem} \label{lemma:motivic:iota-nis-a1-Li}
    Let $A \in \heart{\SH{k}}$.
    Then $\iota_{\nis,\AffSpc{1}} \mathbb L_i A \cong \mathbb L_i \heart{\iota_{\nis,\AffSpc{1}}} A$.
\end{lem}
\begin{proof}
    Since $\iota_{\nis,\AffSpc{1}}$ is t-exact for the standard t-structures (\cref{lemma:motivic:iota-A1-nis-t-exact-std}),
    we see that $\heart{\iota_{\nis,\AffSpc{1}}} A \cong \iota_{\nis, \AffSpc{1}} A$.
    Moreover, the same functor is also t-exact for the $p$-adic t-structures (\cref{lemma:motivic:iota-A1-nis-t-exact-p-adic}).
    Therefore, we compute 
    \begin{equation*}
        \iota_{\nis,\AffSpc{1}} \mathbb L_i A = \iota_{\nis,\AffSpc{1}} \pi_i^p A \cong \pi_i^p \iota_{\nis,\AffSpc{1}} A = \mathbb L_i \heart{\iota_{\nis,\AffSpc{1}}} A.
    \end{equation*}
    Note that $\mathbb L_i$ is just given by the functor $\pi_i^p$ restricted to the standard heart.
\end{proof}

\begin{cor} \label{cor:motivic:iota-li-pin}
    Let $X \in \MotSpc{k}_*$ be a pointed motivic space.
    We have canonical equivalences $\mathbb L_i \pi_n (\iota_{\nis,\AffSpc{1}} X) \cong \iota_{\nis,\AffSpc{1}} \mathbb L_i \pi_n(X)$
    and $L_{\nis,\AffSpc{1}} \mathbb L_i \pi_n (\iota_{\nis,\AffSpc{1}} X) \cong \mathbb L_i \pi_n(X)$
    for all $i$ and $n \ge 2$.
    If $\pi_1(X)$ is abelian, then the same is true for $n = 1$.
\end{cor}
\begin{proof}
    We have the following sequence of equivalences:
    \begin{equation*}
        \mathbb L_i \pi_n (\iota_{\nis,\AffSpc{1}} X)
        \cong \mathbb L_i \heart{\iota_{\nis,\AffSpc{1}}} \pi_n (X)
        \cong \iota_{\nis,\AffSpc{1}} \mathbb L_i \pi_n (X),
    \end{equation*}
    where the first equivalence is given by \cref{cor:motivic:iota-pin-if-connected},
    and the second equivalence by \cref{lemma:motivic:iota-nis-a1-Li}.
    Applying $L_{\nis,\AffSpc{1}}$ we arrive at the equivalence 
    \begin{equation*}
        L_{\nis,\AffSpc{1}} \mathbb L_i \pi_n (\iota_{\nis,\AffSpc{1}} X) \cong L_{\nis,\AffSpc{1}} \iota_{\nis,\AffSpc{1}} \mathbb L_i \pi_n (X) \cong \mathbb L_i \pi_n (X),
    \end{equation*}
    where the second equivalence used the fully faithfulness of $\iota_{\nis,\AffSpc{1}}$.
    If $\pi_1(X)$ is abelian, then we can regard it as an object of $\heart{\SH{k}}$ (see \cref{rmk:motivic:htpy-group-1-abelian}).
    In this case, the same proof works.
\end{proof}

\begin{lem} \label{lemma:motivic:gersten-of-A1}
    Let $X \in \MotSpc{k}_*$ be a pointed motivic space.
    Then $\pi_n(\iota_{\nis, \AffSpc{1}} X)/p^k$ satisfies Gersten injectivity for all $k \ge 1$ and $n \ge 2$.
    If $\pi_1(X)$ is abelian, then the result also holds for $n = 1$.
\end{lem}
\begin{proof}
    Fix $n \ge 2$ and $k \ge 1$.
    We have equivalences 
    \begin{equation*}
        \pi_n(\iota_{\nis, \AffSpc{1}} X) / p^k \cong (\heart{\iota_{\nis, \AffSpc{1}}} \pi_n(X)) / p^k \cong \heart{\iota_{\nis, \AffSpc{1}}} (\pi_n(X) / p^k),
    \end{equation*}
    where we used \cref{cor:motivic:iota-pin-if-connected} in the first equivalence and exactness of $\heart{\iota_{\nis, \AffSpc{1}}}$ in the second equivalence,
    see \cref{lemma:motivic:iota-A1-nis-t-exact-std}.
    Thus, we conclude by \cref{lemma:motivic:gersten-of-strictly} that $\pi_n(\iota_{\nis, \AffSpc{1}} X)/p^k$
    satisfies Gersten injectivity.

    If $\pi_1(X)$ is abelian, then we can regard it as an object of $\heart{\SH{k}}$ (see \cref{rmk:motivic:htpy-group-1-abelian}).
    In this case, the same proof works.
\end{proof}

\begin{lem} \label{lemma:motivic:abelian-Li-heart}
    Let $A \in \heart{\SH{k}}$.
    Then $\nu_* \mathbb L_i \nu^* \heart{\iota_{\nis, \AffSpc{1}}} A \cong \mathbb L_i \heart{\iota_{\nis, \AffSpc{1}}} A$ for all $i$.

    In particular, $\nu_* \mathbb L_i \nu^* \heart{\iota_{\nis, \AffSpc{1}}} A \in \tpstructheart{\ShvTop{\zar}{\smooth k, \Sp}}$ for all $i$.
    Moreover, we have that $\nu_* \mathbb L_i \nu^* \heart{\iota_{\nis, \AffSpc{1}}} A \in \Cat A$,
    where $\Cat A$ is the subcategory of $\tpstructheart{\ShvTop{\zar}{\smooth{k}, \Sp}}$
    from \cref{def:embedding:A}.
\end{lem}
\begin{proof}
    By exactness of $\heart{\iota_{\nis, \AffSpc{1}}}$ (see \cref{lemma:motivic:iota-A1-nis-t-exact-std}),
    for every $k \ge 1$ there are equivalences $(\heart{\iota_{\nis, \AffSpc{1}}}A) / p^k \cong \heart{\iota_{\nis, \AffSpc{1}}}(A/p^k)$.
    Thus, by \cref{lemma:motivic:gersten-of-strictly}, $(\heart{\iota_{\nis, \AffSpc{1}}}A) / p^k$
    satisfies Gersten injectivity for all $k$.
    This implies that $(\mathbb L_1 \nu^* \heart{\iota_{\nis, \AffSpc{1}}} A)\sslash p$ is classical, see \cref{cor:pro-zar:L1-classical}.
    Thus, the equivalence is provided by \cref{lemma:embedding:Li-computation}.
    Note that the same lemma shows that also $(\mathbb L_1 \nu^* \heart{\iota_{\nis, \AffSpc{1}}} A)\sslash p$ 
    is classical for all $i$.
    Thus, the statement about $\Cat A$ follows immediately from \cref{lemma:embedding:pushdown-heart-criterium-A}.
\end{proof}

We will need a non-abelian variant of \cref{lemma:motivic:iota-a1-EM}:
\begin{lem} \label{lemma:motivic:classifying-space-iota-nis}
    Suppose $G \in \Grp{\Disc{\MotShv{k}}}$
    is strongly $\AffSpc{1}$-invariant.
    Then $B \iota_{\nis} G \cong \iota_{\nis} BG$.
\end{lem}
\begin{proof}
    Since both objects are Zariski sheaves,
    it suffices to prove that for all $T = \Spec{\Cat O_{U, u}}$ the spectra of the local rings of 
    a scheme $U \in \smooth{k}$ with point $u \in U$,
    the canonical map $(B \iota_{\nis} G)(T) \to (\iota_{\nis} BG)(T)$ is an equivalence.
    Here, for a Zariski sheaf $F$ we define $F(T) \coloneqq (\nu^*F)(T) \cong \colimil{T \to V \subset U} F(V)$,
    where the colimit runs over all open neighborhoods of $T$ in $U$.
    By Whitehead's theorem and the fact that both anima are $1$-truncated,
    we can reduce to showing that the canonical map induces an equivalence 
    $\pi_k((B \iota_{\nis} G)(T)) \cong \pi_k((\iota_{\nis} BG)(T))$ for $k = 0,1$ 
    and all choices of basepoints.
    Note that both sheaves have a canonical basepoint $*$, and that we have 
    $\pi_k((B \iota_{\nis} G)(U), *) \cong H^{1-k}(U, \iota_{\nis} G)$
    and $\pi_k((\iota_{\nis} BG)(U), *) = \pi_k((BG)(U)) \cong H^{1-k}(U, G)$ for all $U \in \smooth{k}$, see \cite[Proposition 4.1.16]{morelvoevodsky}.
    Note that we have isomorphisms of cohomology groups $H^{1-k}(U, \iota_{\nis} G) \cong H^{1-k}(U, G)$ for all $k$ and $U$ by \cite[Theorem 4.5]{AsokGersten}
    (The reference uses that $k$ is an infinite field.
    If $k$ is a finite field, we can argue as in the above reference,
    using the Gabber presentation lemma for finite fields, see \cite[Theorem 1.1]{hogadi2020gabber}).

    In particular, since homotopy groups and cohomology are compatible with filtered colimits,
    we get $\pi_0((B \iota_{\nis} G)(T)) \cong H^{1} (T, \iota_{\nis} G) = 0$, since Zariski cohomology is Zariski-locally trivial.

    Thus, we immediately see that both anima in question are connected, 
    and we have to prove the equivalence on $\pi_1$ only over the canonical basepoint, 
    which we have seen above.
\end{proof}

Recall the category $\Cat A$ from \cref{def:embedding:A}.
\begin{lem} \label{lemma:motivic:iota-p-adic-in-A}
    Let $C \in \pheart{\SH{k}}$.
    Then $\pheart{\iota}_{\nis,\AffSpc{1}} \in \Cat A \subset \pheart{\MotShvSp{k}}$.
\end{lem}
\begin{proof}
    Write $C' \coloneqq \pheart{\iota}_{\nis,\AffSpc{1}} C \cong \iota_{\nis,\AffSpc{1}} C$ (see \cref{lemma:motivic:iota-A1-nis-t-exact-p-adic}
    for the equivalence).
    We have to show that $\pi_1^p(\nu^* C') \cong 0$.
    Note that by \cref{lemma:short-exact-sequence-of-t-structure} there is a short exact sequence 
    \begin{equation*}
        0 \to \mathbb L_0 \pi_1(\nu^* C') \to \pi_1^p(\nu^* C') \to \mathbb L_1 \pi_0(\nu^* C') \to 0.
    \end{equation*}
    By \cref{lemma:t-struct:char-of-cocon}, we know that $C' \in \tcocon[0]{\ZarShvSp{k}}$.
    Thus, $\pi_1(\nu^* C') \cong \nu^* \pi_1(C') \cong 0$.
    Hence, it suffices to prove that $\mathbb L_1 \pi_0(\nu^* C') \cong \mathbb L_1 \nu^* \pi_0(C') = 0$.
    But note that $\pi_0(C') = \pi_0(\iota_{\nis,\AffSpc{1}} C) \cong \iota_{\nis,\AffSpc{1}} \pi_0(C)$
    by \cref{lemma:motivic:iota-A1-nis-t-exact-std}.
    Since $\mathbb L_1 \nu^* \iota_{\nis,\AffSpc{1}} \pi_0(C)$ is $p$-complete (e.g. by  \cref{lemma:t-struct:char-of-cocon}),
    it suffices to show that $(\mathbb L_1 \nu^* \iota_{\nis,\AffSpc{1}} \pi_0(C)) \sslash p = 0$.
    Note that this sheaf is classical by \cref{cor:pro-zar:L1-classical},
    where we used that $(\iota_{\nis,\AffSpc{1}} \pi_0(C)) / p^n \cong \iota_{\nis,\AffSpc{1}} (\pi_0(C) / p^n)$
    satisfies Gersten injectivity (see \cref{lemma:motivic:iota-A1-nis-t-exact-std} for the first equivalence,
    and \cref{lemma:motivic:gersten-of-strictly} for the claim about the Gersten injectivity).
    Thus, we calculate 
    \begin{align*}
        (\mathbb L_1 \nu^* \iota_{\nis,\AffSpc{1}} \pi_0(C)) \sslash p 
        &\cong \nu^* \nu_* ((\mathbb L_1 \nu^* \iota_{\nis,\AffSpc{1}} \pi_0(C)) \sslash p) \\
        &\cong \nu^* ((\nu_* \mathbb L_1 \nu^* \iota_{\nis,\AffSpc{1}} \pi_0(C)) \sslash p) \\
        &\cong \nu^* ((\mathbb L_1 \iota_{\nis,\AffSpc{1}} \pi_0(C)) \sslash p) \\
        &\cong \nu^* ((\mathbb L_1 \pi_0(\iota_{\nis,\AffSpc{1}} C)) \sslash p),
    \end{align*}
    where we used that the sheaf is classical in the first equivalence,
    exactness of $\nu_*$ in the second equivalence,
    \cref{lemma:embedding:Li-computation} in the third equivalence,
    and \cref{lemma:motivic:iota-A1-nis-t-exact-std} in the last equivalence.
    Therefore, it suffices to prove that $\mathbb L_1 \pi_0(\iota_{\nis,\AffSpc{1}} C) = 0$.
    Again, \cref{lemma:short-exact-sequence-of-t-structure} supplies us with a short exact sequence 
    \begin{equation*}
        0 \to \mathbb L_0 \pi_1(\iota_{\nis,\AffSpc{1}} C) \to \pi_1^p(\iota_{\nis,\AffSpc{1}} C) \to \mathbb L_1 \pi_0(\iota_{\nis,\AffSpc{1}} C) \to 0.
    \end{equation*}
    But we have $\pi_1^p(\iota_{\nis,\AffSpc{1}} C) \cong \iota_{\nis,\AffSpc{1}} \pi_1^p(C) \cong 0$,
    where we used \cref{lemma:motivic:iota-A1-nis-t-exact-p-adic} 
    in the first equivalence and the assumption that $C \in \pheart{\SH{k}}$ in the second equivalence.
    This proves the lemma.
\end{proof}

\begin{lem} \label{lemma:motivic:pi1-in-heart}
    Let $G \in \Grp{\Disc{\MotShv{k}}}$ be a nilpotent sheaf of groups,
    which is strongly $\AffSpc{1}$-invariant.
    Then $\mathbb L_1 \iota_{\nis} G \in \tpstructheart{\ZarShvSp{k}}$,
    where we use \cref{def:embedding:L1G}.
\end{lem}
\begin{proof}
    Using \cite[Proposition 3.2.3]{asok2022localization},
    we see that $G$ is in particular $\AffSpc{1}$-nilpotent, in the sense
    of \cite[Definition 3.2.1 (3)]{asok2022localization}.
    Thus, there is a $G$-central series $G = G_0 \supset G_{1} \supset \dots \supset G_n = 1$
    (i.e.\ the $G_i$ are sheaves of normal subgroups and the quotients $A_i \coloneqq G_{i} /G_{i+1}$ have trivial $G$ action (via conjugation)),
    such that the $A_i$ are again strongly $\AffSpc{1}$-invariant.
    Moreover, the $G_i$ are strongly $\AffSpc{1}$-invariant (\cite[Remark 3.2.2 (1)]{asok2022localization}).
    Since the $A_i$ are abelian (\cite[Remark 3.2.2 (3)]{asok2022localization}) and
    strongly $\AffSpc{1}$-invariant, they are strictly $\AffSpc{1}$-invariant by \cite[Theorem 4.46]{morel2012a1}.
    Note that we have central extensions of groups
    \begin{equation*}
        1 \to A_i \to G/G_{i+1} \to G/G_{i} \to 1,
    \end{equation*}
    see \cite[Remark 3.2.2 (3)]{asok2022localization}.
    This extension is classified by a fiber sequence
    \begin{equation*}
        B(G/G_{i+1}) \to B(G/G_{i}) \to K(\heart{\iota}_{\AffSpc{1}} \tilde{A}_i, 2),
    \end{equation*}
    where $\tilde{A_i} \in \heart{\SH{k}}$ corresponds to the strictly $\AffSpc{1}$-invariant 
    sheaf of abelian groups $A_i$.
    Thus, we can proceed by induction.
    Recall the definition of the full subcategory $\Cat A \subset \pheart{\ZarShvSp{k}}$ from \cref{def:embedding:A}.
    We will inductively prove that $\mathbb L_1 \iota_{\nis} (G/G_i) \in \Cat A \subset \tpstructheart{\ZarShvSp{k}}$,
    and that $\mathbb L_1 \iota_{\nis} (G/G_i)$ is actually an $\AffSpc{1}$-invariant Nisnevich sheaf of spectra 
    living in the $p$-adic heart, i.e.\ there is a $B \in \pheart{\SH{k}}$
    with $\iota_{\nis,\AffSpc{1}} B \cong \mathbb L_1 \iota_{\nis} (G/G_i)$.
    The base case $G/G_0 = G/G = 1$ is trivial.

    So suppose the statement holds for $G/G_{i}$.
    Since $\iota_{\nis}$ preserves limits (as a right adjoint) and $\nu^*$
    preserves finite limits (as the left adjoint of a geometric morphism),
    we get a fiber sequence 
    \begin{equation*}
        \nu^* \iota_{\nis} B(G/G_{i+1}) \to \nu^* \iota_{\nis} B(G/G_{i}) \to \nu^* \iota_{\nis} K(\heart{\iota}_{\AffSpc{1}} \tilde{A}_i, 2).
    \end{equation*}
    Since all involved groups are strongly $\AffSpc{1}$-invariant and nilpotent,
    this fiber sequence is equivalent to the fiber sequence 
    \begin{equation*}
        \nu^* B(\iota_{\nis} (G/G_{i+1})) \to \nu^* B(\iota_{\nis} (G/G_i)) \to \nu^* K(\heart{\iota}_{\nis,\AffSpc{1}} \tilde{A}_i, 2),
    \end{equation*}
    see \cref{lemma:motivic:iota-a1-EM,lemma:motivic:classifying-space-iota-nis}. Now \cref{lemma:fiber-lemma} implies that we have a fiber sequence 
    \begin{equation*}
        \tau_{\ge 1} \completebr{\nu^* B(\iota_{\nis} (G/G_{i+1}))} \to \completebr{\nu^* B(\iota_{\nis} (G/G_i))} \to \completebr{\nu^* K(\heart{\iota}_{\nis,\AffSpc{1}} \tilde{A}_i, 2)}.
    \end{equation*}
    But $\completebr{\nu^* B(\iota_{\nis} (G/G_{i+1}))}$ is already connected, see \cref{lemma:psig-connectivity-of-completion}.
    Thus, we arrive at the fiber sequence
    \begin{equation*}
        \completebr{\nu^* B\iota_{\nis}(G/G_{i+1})} \to \completebr{\nu^* B\iota_{\nis}(G/G_{i})} \to \completebr{\nu^*K(\heart{\iota}_{\nis,\AffSpc{1}} \tilde{A}_i, 2)}.
    \end{equation*}
    Thus, using the long exact sequence and the fact that the $p$-completion of a $k$-truncated object is
    $(k+1)$-truncated (see \cref{cor:topos:truncation-of-completion}), we get an exact sequence in $\tstructheart{\PSigVal{W}{\Sp}}$
    \begin{equation*}
        0 \to \pi_3\completebr{\nu^*K(\heart{\iota}_{\nis,\AffSpc{1}} \tilde{A}_i, 2)} \to \mathbb L_1 \nu^* \iota_{\nis}(G/G_{i+1}) \to \mathbb L_1 \nu^* \iota_{\nis}(G/G_{i}) \to \pi_2 \completebr{\nu^*K(\heart{\iota}_{\nis,\AffSpc{1}} \tilde{A}_i, 2)},
    \end{equation*}
    where we use \cref{def:psig:L1G} for $\mathbb L_1$.
    Using \cref{lemma:psig:short-exact-sequence}, we can identify 
    \begin{equation*}
        \pi_k \left(\completebr{\nu^*K(\heart{\iota}_{\nis,\AffSpc{1}} \tilde{A}_i, 2)}\right) \cong \mathbb L_{k-2} \nu^* \heart{\iota}_{\nis,\AffSpc{1}} \tilde{A}_i
    \end{equation*}
    for $k = 2,3$.
    Thus, we arrive at the exact sequence in $\tpstructheart{\PSigVal{W}{\Sp}}$:
    \begin{equation*}
        0 \to \mathbb L_1 \nu^* \heart{\iota}_{\nis,\AffSpc{1}} \tilde{A}_i \to \mathbb L_1 \nu^* \iota_{\nis} (G/G_{i+1}) \to \mathbb L_1 \nu^* \iota_{\nis} (G/G_{i}) \to \mathbb L_0 \nu^* \heart{\iota}_{\nis,\AffSpc{1}} \tilde{A}_{i}.
    \end{equation*}

    We want to apply \cref{prop:embedding:heart-exact-sequence} to this exact sequence.
    We first check the assumptions on the outer two terms involving $\tilde{A}_i$.
    We know that $\nu_* \mathbb L_k \nu^* \heart{\iota}_{\nis, \AffSpc{1}} \tilde{A}_i \cong \mathbb L_k \heart{\iota}_{\nis,\AffSpc{1}} \tilde{A}_i$
    for all $k$,
    and that it lives in $\Cat A$, see \cref{lemma:motivic:abelian-Li-heart}
    Therefore, we also get $\nupu \mathbb L_k \heart{\iota}_{\nis,\AffSpc{1}} \tilde{A}_i \cong \nupu \nu_* \mathbb L_k \nu^* \heart{\iota}_{\nis,\AffSpc{1}} \tilde{A}_i \cong \mathbb L_k \nu^* \heart{\iota}_{\nis, \AffSpc{1}} \tilde{A}_i$
    for all $k$, see \cref{cor:embedding:classical-char} for the second equivalence.

    By induction, $\mathbb L_1 \iota_{\nis} (G/G_i) \cong \nu_* \mathbb L_1 \nu^* \iota_{\nis} (G/G_i) \in \Cat A \subset \tpstructheart{\ZarShvSp{k}}$.
    In particular, $\nupu \mathbb L_1 \iota_{\nis} (G/G_i) \cong \nupu \nu_* \mathbb L_1 \nu^* \iota_{\nis} (G/G_i) \cong \mathbb L_1 \nu^* \iota_{\nis} (G/G_i)$,
    where we also used \cref{cor:embedding:classical-char} for the second equivalence.

    Thus, we are left to show that $\coker(\mathbb L_1 \iota_{\nis} (G/G_{i}) \to \mathbb L_0 \heart{\iota}_{\nis,\AffSpc{1}} \tilde{A}_{i}) \in \Cat A$:
    First note that $\mathbb L_0 \heart{\iota}_{\nis,\AffSpc{1}} \tilde{A}_{i} \cong \pi_0^p(\iota_{\nis, \AffSpc{1}} \tilde{A}_i) \cong \iota_{\nis,\AffSpc{1}} \pi_0^p (\tilde{A}_i) \cong \pheart{\iota}_{\nis,\AffSpc{1}} \pi_0^p (\tilde{A}_i)$
    by t-exactness for the standard and $p$-adic t-structures of $\iota_{\nis,\AffSpc{1}}$,
    see \cref{lemma:motivic:iota-A1-nis-t-exact-std,lemma:motivic:iota-A1-nis-t-exact-p-adic}.
    By induction, there is a $B \in \pheart{\SH{k}}$ with $\pheart{\iota}_{\nis,\AffSpc{1}} B \cong \mathbb L_1 \iota_{\nis} (G/G_{i})$.
    Therefore, again by exactness and fully faithfulness of $\pheart{\iota}_{\nis,\AffSpc{1}}$, the cokernel is also of the form $\pheart{\iota}_{\nis,\AffSpc{1}} C$ for some $C \in \pheart{\SH{k}}$.
    Thus, we immediately get that the cokernel is in $\Cat A$, see \cref{lemma:motivic:iota-p-adic-in-A}.

    We can now apply \cref{prop:embedding:heart-exact-sequence},
    which allows us to deduce that also 
    $\mathbb L_1 \iota_{\nis} (G/G_{i+1}) = \nu_* \mathbb L_1 \nu^* \iota_{\nis} (G/G_{i+1}) \in \Cat A \subset \heart{\ZarShvSp{k}}$.

    Moreover, there is now an exact sequence 
    \begin{equation*}
        0 \to \mathbb L_1 \heart{\iota}_{\nis,\AffSpc{1}} \tilde{A_i} \to \mathbb L_1 \iota_{\nis}(G/G_{i+1}) \to K \to 0,
    \end{equation*}
    where $K \coloneqq \ker(\mathbb L_1 \iota_{\nis}(G/G_i) \to \mathbb L_0 \heart{\iota_{\nis,\AffSpc{1}}} \tilde{A_i})$.
    We have seen above that $\mathbb L_k \heart{\iota_{\nis,\AffSpc{1}}} \tilde{A_i}$
    is in fact an $\AffSpc{1}$-invariant Nisnevich sheaf of spectra, living in the $p$-adic heart (for $k = 0,1$).
    By induction, the same is true for $\mathbb L_1 \iota_{\nis}(G/G_i)$.
    Exactness of $\pheart{\iota}_{\nis,\AffSpc{1}}$ implies that this also holds for the kernel $K$.
    Thus, $\mathbb L_1 \iota_{\nis}(G/G_{i+1})$ sits in a short exact sequence where 
    the outer terms are $\AffSpc{1}$-invariant Nisnevich sheaves of spectra, living in the $p$-adic heart.
    From this we deduce immediately that the same is true for $\mathbb L_1 \iota_{\nis}(G/G_{i+1})$.
    This concludes the induction.
\end{proof}

\begin{defn} \label{def:motivic:completed-htpy}
    Let $X \in \MotSpc{k}_*$ be a pointed motivic space.
    For every $n \ge 2$ we define the \emph{$p$-completed homotopy groups} of $X$
    via 
    \begin{equation*}
        \pi^p_n(X) \coloneqq L_{\nis, \AffSpc{1}} \pi^p_n (\iota_{\nis, \AffSpc{1}} X) \in \SH{k},
    \end{equation*}
    and for $n = 1$ via 
    \begin{equation*}
        \pi^p_1(X) \coloneqq L_{\nis} \pi^p_1(\iota_{\nis, \AffSpc{1}} X) \in \Grp{\Disc{\MotShv{k}}}.
    \end{equation*}
    (Recall \cref{def:pro-zar:completed-htpy} for the $p$-completed homotopy groups
    of $\iota_{\nis,\AffSpc{1}} X$.)
\end{defn}

\begin{rmk}
    Let $X \in \MotSpc{k}_*$ be a pointed space.
    We will show in \cref{thm:motivic:ses}
    that if $X$ is nilpotent, then actually
    $\pi^p_n(X) \in \tpstructheart{\SH{k}}$ if $n \ge 2$.
    Thus, the name $p$-completed homotopy \emph{group} is justified.
\end{rmk}

\begin{lem} \label{lemma:motivic:p-complete-htpy-ios-to-completion}
    Let $X \in \MotSpc{k}_*$ be nilpotent.
    Then the canonical map induces an equivalence $\pi_n^p(X) \to \pi_n^p(\iota_{\ge 1} (\completebr{\tau_{\ge 1} X}))$
    for all $n \ge 1$.
\end{lem}
\begin{proof}
    We know $\iota_{\nis, \AffSpc{1},\ge 1} (\completebr{\tau_{\ge 1} X}) \cong \completebr{\iota_{\nis,\AffSpc{1},\ge 1} \tau_{\ge 1}X} \cong \completebr{\iota_{\nis,\AffSpc{1}} X}$
    from \cref{lemma:motivic:iota-completion} and the fact that $X$ is connected because it is nilpotent.
    Therefore, the map $\iota_{\nis, \AffSpc{1}} X \to \iota_{\nis, \AffSpc{1},\ge 1}(\completebr{\tau_{\ge 1}X})$ is a $p$-equivalence.
    Thus, we conclude from \cref{lemma:pro-zar:completed-htpy-stable-under-peq}
    that $\pi_n^p(X) \to \pi_n^p(\iota_{\ge 1}(\completebr{\tau_{\ge 1}X}))$ is an equivalence 
    (note that by definition $\pi_n^p(X) = L_{\nis, \AffSpc{1}}\pi_n^p(\iota_{\nis, \AffSpc{1}} X)$
    and $\pi_n^p(\iota_{\ge 1}(\completebr{\tau_{\ge 1}X})) = L_{\nis, \AffSpc{1}}\pi_n^p(\iota_{\nis, \AffSpc{1},\ge 1}(\completebr{\tau_{\ge 1}X}))$
    if $n \ge 2$, and similarly for $n = 1$).
\end{proof}

\begin{prop} \label{lemma:motivic:p-complete-htpy-iso-if-peq}
    Let $f \colon X \to Y$ be a morphism in $\MotSpc{k}_*$ of pointed nilpotent spaces,
    and $n \ge 1$.
    If $f$ is a $p$-equivalence, then $\pi_n^p(f)$ is an equivalence.
\end{prop}
\begin{proof}
    Note that we can regard $f$ as a morphism in $\MotSpc{k}_{\ge 1, *}$ 
    since nilpotent spaces are connected.
    In particular, it is also a $p$-equivalence in this category, see \cref{lemma:motivic:connected-stabilization}.
    Thus, we can assume that $f$ is a $p$-equivalence in $\MotSpc{k}_{\ge 1, *}$,
    and we want to prove that $\pi_n^p(\iota_{\ge 1}(f))$ is an equivalence.

    By $p$-completing, we get a commutative square
    \begin{center}
        \begin{tikzcd}
            X \arrow[r, "f"] \arrow[d] &Y \arrow[d] \\
            \complete{X} \arrow[r, "\complete{f}"] & \complete{Y}
        \end{tikzcd}
    \end{center}
    where the downwards arrows are the canonical $p$-equivalences.
    Applying the functor $\pi_n^p(\iota_{\ge 1}(-))$ for $n \ge 1$, we arrive at the square
    \begin{center}
        \begin{tikzcd}
            \pi_n^p(\iota_{\ge 1}(X)) \arrow[r, "\pi_n^p(\iota_{\ge 1}(f))"] \arrow[d] &\pi_n^p(\iota_{\ge 1}(Y)) \arrow[d] \\
            \pi_n^p(\iota_{\ge 1}(\complete{X})) \arrow[r, "\pi_n^p(\iota_{\ge 1}(\complete{f}))"] & \pi_n^p(\iota_{\ge 1}(\complete{Y})).
        \end{tikzcd}
    \end{center}
    Since $f$ is a $p$-equivalence, we know that $\complete{f}$ is an equivalence.
    In particular, $\pi_n^p(\iota_{\ge 1}(\complete{f}))$ is an equivalence.
    The two vertical maps are equivalences by \cref{lemma:motivic:p-complete-htpy-ios-to-completion}.
    From this we conclude that also $\pi_n^p(\iota_{\ge 1}(f))$ is an equivalence.
\end{proof}

\begin{defn} \label{def:motivic:L1G}
    Let $G \in \GrpStr{\Disc{\MotShv{k}}}$ be a strictly $\AffSpc{1}$-invariant nilpotent sheaf of groups.
    We define 
    \begin{equation*}
        \mathbb L_1 G \coloneqq L_{\nis} \mathbb L_1 \iota_{\nis} G \in \MotShvSp{k},
    \end{equation*}
    where we use \cref{def:embedding:L1G}, and 
    \begin{equation*}
        \mathbb L_0 G \coloneqq L_{\nis} \mathbb L_0 \iota_{\nis} G \in \Grp{\MotShv{k}}.
    \end{equation*}
\end{defn}

\begin{thm} \label{thm:motivic:ses}
    Let $X \in \MotSpc{k}_*$ be a pointed nilpotent motivic space.
    Then for every $n \ge 2$,
    there is a canonical short exact sequence in $\pheart{\SH{k}}$
    (or a short exact sequence in $\GrpStr{\Disc{\ShvNis{k}}}$ if $n = 1$)
    \begin{equation*}
        0 \to \mathbb L_0 \pi_n (X) \to \pi_n^p (X) \to \mathbb L_1 \pi_{n-1}(X) \to 0,
    \end{equation*}
    where we use \cref{def:motivic:L1G} for $\mathbb L_i \pi_1(X)$.
    In particular, $\pi_n^p(X) \in \pheart{\SH{k}}$.
    Here we set $\mathbb L_1 \pi_0(X) = 0$ since $X$ is connected.

    Moreover, for $n \ge 2$ the unit map induces an equivalence
    \begin{equation*}
        \iota_{\nis,\AffSpc{1}} \pi_n^p(X) = \iota_{\nis, \AffSpc{1}} L_{\nis, \AffSpc{1}} \pi_n^p(\iota_{\nis, \AffSpc{1}} X) \cong \pi_n^p(\iota_{\nis, \AffSpc{1}} X),
    \end{equation*}
    i.e.\ $\pi_n^p(\iota_{\nis, \AffSpc{1}} X)$ is already an $\AffSpc{1}$-invariant Nisnevich sheaf of spectra.
    If $\pi_1(X)$ is abelian, the same is true for $\pi_1^p(\iota_{\nis,\AffSpc{1}} X)$.
\end{thm}
\begin{proof}
    Note that $\pi_n(\iota_{\nis,\AffSpc{1}} X)/p^k$ satisfies Gersten injectivity for all $n \ge 2$ 
    and all $k \ge 1$, see \cref{lemma:motivic:gersten-of-A1}.
    By this lemma, the same is true if $\pi_1(X)$ is abelian.
    If not, then we can still conclude by \cref{lemma:motivic:pi1-in-heart}
    that $\mathbb L_1 \pi_1(\iota_{\nis, \AffSpc{1}} X) \in \tpstructheart{\ZarShvSp{k}}$
    (note that $\pi_1(\iota_\AffSpc{1}X)$ is strongly $\AffSpc{1}$-invariant by \cite[Corollary 5.2]{morel2012a1},
    and that $\pi_1(\iota_{\nis, \AffSpc{1}} X) = \iota_{\nis} \pi_1(\iota_{\AffSpc{1}}X)$ by \cref{cor:motivic:iota-pin-if-connected}).

    Thus, for $n \ge 2$ we have a short exact sequence
    \begin{equation*}
        0 \to \mathbb L_0 \pi_n (\iota_{\nis,\AffSpc{1}} X) \to \pi_n^p (\iota_{\nis,\AffSpc{1}} X) \to \mathbb L_1 \pi_{n-1} (\iota_{\nis,\AffSpc{1}} X) \to 0
    \end{equation*}
    in $\tpstructheart{\ShvTop{\zar}{\smooth k,\Sp}}$ by \cref{thm:pro-zar:short-exact-sequence}.
    Applying $L_{\nis,\AffSpc{1}}$ we get a fiber sequence 
    \begin{equation*}
        L_{\nis,\AffSpc{1}} \mathbb L_0 \pi_n (\iota_{\nis,\AffSpc{1}} X) \to L_{\nis,\AffSpc{1}} \pi_n^p (\iota_{\nis,\AffSpc{1}} X) \to L_{\nis,\AffSpc{1}} \mathbb L_1 \pi_{n-1} (\iota_{\nis,\AffSpc{1}} X).
    \end{equation*}
    Using \cref{cor:motivic:iota-li-pin}, we compute that $L_{\nis,\AffSpc{1}} \mathbb L_i \pi_k (\iota_{\nis,\AffSpc{1}} X) \cong \mathbb L_i \pi_k(X)$
    (if $k = 1$, then this is just the definition).
    Moreover, $L_{\nis,\AffSpc{1}} \pi_n^p(\iota_{\nis,\AffSpc{1}} X) = \pi_n^p(X)$ by \cref{def:motivic:completed-htpy}.
    Thus, we get a fiber sequence
    \begin{equation*}
        \mathbb L_0 \pi_n (X) \to \pi_n^p (X) \to \mathbb L_1 \pi_{n-1} (X).
    \end{equation*}
    Note that the outer terms are in $\tpstructheart{\MotShvSp{k}}$ by definition.
    Thus, using the long exact sequence, we conclude that also $\pi^p_n(X) \in \tpstructheart{\MotShvSp{k}}$
    and the fiber sequence yields an exact sequence
    \begin{equation*}
        0 \to \mathbb L_0 \pi_n (X) \to \pi_n^p (X) \to \mathbb L_1 \pi_{n-1} (X) \to 0.
    \end{equation*}

    The last statement follows since we have (again by \cref{cor:motivic:iota-li-pin})
    \begin{equation*}
        \mathbb L_i \pi_k(\iota_{\nis,\AffSpc{1}} X) \cong \iota_{\nis,\AffSpc{1}} \mathbb L_i \pi_k(X),
    \end{equation*}
    i.e.\ the $\mathbb L_i \pi_k(\iota_{\nis,\AffSpc{1}} X)$ are $\AffSpc{1}$-invariant Nisnevich sheaves (of spectra),
    and thus $\pi_n^p (\iota_{\nis,\AffSpc{1}} X)$ sits in the middle of an exact sequence,
    where the outer terms are in $\pheart{\SH{k}} \subset \SH{k} \hookrightarrow \ShvTop{\zar}{\smooth k,\Sp}$.
    Thus, since stable subcategories are stable under extensions, the result follows.
    If $\pi_1(X)$ is abelian, the same proof works.

    If $n = 1$, \cref{thm:pro-zar:short-exact-sequence} instead supplies us with a short exact sequence 
    \begin{equation*}
        0 \to \mathbb L_0 \pi_1 (\iota_{\nis,\AffSpc{1}} X) \to \pi_1^p (\iota_{\nis,\AffSpc{1}} X) \to \mathbb L_1 \pi_{0} (\iota_{\nis,\AffSpc{1}} X) \to 0
    \end{equation*}
    in $\Grp{\Disc{\ZarShv{k}}}$,
    where $\mathbb L_1 \pi_0 (\iota_{\nis, \AffSpc{1}} X) = 0$,
    i.e.\ this is an equivalence 
    \begin{equation*}
        \mathbb L_0 \pi_1 (\iota_{\nis,\AffSpc{1}} X) \cong \pi_1^p (\iota_{\nis,\AffSpc{1}} X).
    \end{equation*}
    Applying $L_{\nis}$, we get an equivalence 
    \begin{equation*}
        L_{\nis} \mathbb L_0 \pi_1 (\iota_{\nis,\AffSpc{1}} X) \cong L_{\nis} \pi_1^p (\iota_{\nis,\AffSpc{1}} X)
    \end{equation*}
    in $\Grp{\Disc{\MotShv{k}}}$.
    Note that by definition, the right-hand side is $\pi_1^p(X)$,
    and the left-hand side is $\mathbb L_0 \pi_1(X)$.
    We thus get the required short exact sequence.
\end{proof}

\begin{cor} \label{cor:motivic:ses-with-p-completion}
    Let $X \in \MotSpc{k}_*$ be nilpotent.
    Then for $n \ge 2$ there is a canonical short exact sequence in $\pheart{\SH{k}}$ 
    (or in $\Grp{\MotShv{k}}$ if $n = 1$)
    \begin{equation*}
        0 \to \mathbb L_0 \pi_n(X) \to \pi_n^p (\iota_{\ge 1} (\completebr{\tau_{\ge 1} X})) \to \mathbb L_1 \pi_{n-1}(X) \to 0.
    \end{equation*}
\end{cor}
\begin{proof}
    This follows immediately from \cref{thm:motivic:ses} and \cref{lemma:motivic:p-complete-htpy-ios-to-completion}.
\end{proof}

\begin{rmk}
    Let $X \in \MotSpc{k}_*$ be nilpotent and $n \ge 1$.
    If \cref{conj:motivic:conjecture-A} is true,
    then we get moreover a short exact sequence
    \begin{equation*}
        0 \to \mathbb L_0 \pi_n(X) \to \pi_n^p (\complete{X}) \to \mathbb L_1 \pi_{n-1}(X) \to 0.
    \end{equation*}
\end{rmk}

We can now also establish a (partial) converse to \cref{lemma:motivic:p-complete-htpy-iso-if-peq}:
\begin{prop} \label{prop:motivic:peq-of-iso-in-phtpy}
    Let $f \colon X \to Y \in \MotSpc{k}_*$ be a morphism of pointed nilpotent spaces with abelian fundamental groups.
    Assume that $\pi_n^p(f)$ is an isomorphism for all $n \ge 1$.
    Then $f$ is a $p$-equivalence.
\end{prop}
\begin{proof}
    It follows from \cref{thm:pro-zar:short-exact-sequence} 
    that $\pi_n^p(\iota_{\nis,\AffSpc{1}} X)$ and $\pi_n^p(\iota_{\nis, \AffSpc{1}} Y)$
    are already $\AffSpc{1}$-invariant Nisnevich sheaves for all $n \ge 1$.
    Therefore, we conclude that $\pi_n^p(\iota_{\nis,\AffSpc{1}} f)$ 
    is an isomorphism for all $n \ge 1$ (note that $L_{\nis,\AffSpc{1}} \pi_n^p(\iota_{\nis,\AffSpc{1}} f) = \pi_n^p(f)$
    are isomorphisms by assumption).
    
    Note that by the proof of \cref{thm:motivic:ses} we conclude that $\iota_{\nis,\AffSpc{1}} X$
    and $\iota_{\nis,\AffSpc{1}} Y$ satisfy the conditions of \cref{thm:pro-zar:short-exact-sequence}.
    Therefore, \cref{prop:pro-zar:peq-of-eq-on-p-htpy} 
    implies that $\iota_{\nis,\AffSpc{1}} f$ is a $p$-equivalence.
    Hence, also $f \cong L_{\nis,\AffSpc{1}} \iota_{\nis,\AffSpc{1}} f$ is a $p$-equivalence.
    Here, we used that $\iota_{\nis,\AffSpc{1}}$ is fully faithful and \cref{lemma:peq-via-points}.
    This proves the proposition.
\end{proof}

\begin{rmk} 
    As in the case of \cref{prop:pro-zar:peq-of-eq-on-p-htpy},
    the assumptions that $\pi_1(X)$ and $\pi_1(Y)$ are abelian can probably be relaxed,
    but a proof of this statement is unclear to the author,
    see also \cref{rmk:pro-zar:peq-of-eq-on-p-htpy-relaxed}
\end{rmk}

\newpage
\appendix
\section{Background Material}
\subsection{Stabilization}
\label{appendix:stab}

We prove some basic results about the stabilization of adjoint functors of presentable $\infty$-categories.
All the results are well-known, but hard to track down in the literature.

\begin{lem} \label{lemma:adjoints-on-stabilization}
    Let $f^* \colon \topos X \rightleftarrows \topos Y \colon f_*$
    be an adjunction of presentable $\infty$-categories.

    Then $f^*$ and $f_*$ induce an adjunction
    \begin{equation*}
        f^* \colon \spectra{X} \rightleftarrows \spectra{Y} \colon f_*
    \end{equation*}
    of exact functors such that the following diagrams of functors commute (up to homotopy):
    \begin{center}
        \begin{tikzcd}
            \spectra{X} \arrow[r, "f^*"] & \spectra{Y} \\
            \topos X_* \arrow[u, "\Sus"] \arrow[r, "f^*"] &\topos Y_* \arrow[u, "\Sus"]
        \end{tikzcd}
        \begin{tikzcd}
            \spectra{X} \arrow[d, "\pLoop"] & \spectra{Y} \arrow[d, "\pLoop"]\arrow[l, "f_*"] \\
            \topos X_* &\topos Y_*. \arrow[l, "f_*"]
        \end{tikzcd}
    \end{center}
\end{lem}
\begin{proof}
    \cite[Propositions 1.4.2.22 and 1.4.4.4]{higheralgebra} imply the
    existence of a limit-preserving exact functor $f_* \colon \spectra Y \to \spectra X$
    that fits into the right diagram (see also the proof of \cite[Corollary 1.4.4.5]{higheralgebra}).
    Using \cite[Proposition 1.4.4.4 (3)]{higheralgebra},
    we see that this functor admits a left adjoint $f^*$.
    By uniqueness of adjoints, we conclude that the left diagram is commutative.
\end{proof}

\begin{lem} \label{lemma:stabilization:right-adjoint-fully-faithful}
    In the situation of \cref{lemma:adjoints-on-stabilization},
    assume moreover that $f_* \colon \topos Y \to \topos X$ is fully faithful.
    Then also $f_* \colon \spectra Y \to \spectra X$ is fully faithful.
\end{lem}
\begin{proof}
    The category $\spectra{X}$ can be defined as the category 
    of excisive functors from finite anima to $\topos X$, see \cite[Definition 1.4.2.8]{higheralgebra}.
    Note that the functor $f_*$ is given by postcomposing with the functor $f_* \colon \topos Y \to \topos X$.
    Thus, the result follows, since postcomposition with 
    a fully faithful functor is already fully faithful on functor categories.
\end{proof}

\begin{lem} \label{lemma:stabilization:geometric-morphism-commutes-with-loopspace}
    In the situation of \cref{lemma:adjoints-on-stabilization},
    assume moreover that $f^*$ is left exact.
    Then we have canonical equivalences
    \begin{equation*}
        f^* \Loop (-) \cong \Loop f^* (-)
    \end{equation*}
    and
    \begin{equation*}
        f^* \pLoop (-) \cong \pLoop f^* (-).
    \end{equation*}
    Moreover, if $f^* \colon \topos X \to \topos Y$ is conservative,
    so is $f^* \colon \spectra{X} \to \spectra{Y}$.
\end{lem}
\begin{proof}
    The category $\spectra{X}$ can be defined as the category 
    of excisive functors from finite anima to $\topos X$, see \cite[Definition 1.4.2.8]{higheralgebra}.
    Note that the functor $\Loop$ is given by evaluating at the finite anima $S^0$, see \cite[Notation 1.4.2.20]{higheralgebra}.
    In contrast, the functor $f^*$ is given by postcomposing excisive functors with $f^* \colon \topos X \to \topos Y$.
    It is therefore clear that $f^* \Loop (-) \cong \Loop f^* (-)$.

    Suppose now that $f^* \colon \topos X \to \topos Y$ is conservative.
    Let $g \colon E \to F$ be a morphism in $\spectra X$ such that $f^*g$ is an equivalence.
    In order to show that $g$ is an equivalence, it suffices to show that $\pLoop\Sigma^n g$ is
    an equivalence for all $n$.
    Since $f^* \colon \topos X \to \topos Y$ is conservative,
    it thus suffices to show that $f^* \pLoop \Sigma^n g$ is an equivalence.
    But we have
    \begin{equation*}
        f^* \pLoop \Sigma^n g \cong \pLoop f^* \Sigma^n g \cong \pLoop\Sigma^n f^*g,
    \end{equation*}
    which is an equivalence by assumption.
\end{proof}

\begin{lem} \label{lemma:fully-faithful-on-stabilizations}
    In the situation of \cref{lemma:adjoints-on-stabilization},
    assume moreover that $f^* \colon \topos X \to \topos Y$ is fully faithful and left exact.
    Then $f^* \colon \spectra{X} \to \spectra{Y}$ is fully faithful.
\end{lem}
\begin{proof}
    We need to show that $f_*f^* \cong \operatorname{id}_{\spectra{X}}$.
    So let $E \in \spectra{X}$.
    In order to show that $f_*f^* E \cong E$,
    it suffices to show that for all $n$,
    $\Loop \Sigma^n f_*f^*E \cong \Loop \Sigma^n E$.
    But we have
    \begin{align*}
        \Loop \Sigma^n f_*f^* E & \cong \Loop f_* f^* \Sigma^n E \\
                                & \cong f_* f^* \Loop \Sigma^n E \\
                                & \cong \Loop \Sigma^n E.
    \end{align*}
    The first equivalence is clear because $f_*$ and $f^*$ are exact.
    The second equivalence uses \cref{lemma:stabilization:geometric-morphism-commutes-with-loopspace}.
    The last equivalence follows because $f^* \colon \topos X \to \topos Y$
    is fully faithful.
\end{proof}

The stabilization of a presentable $\infty$-category $\topos X$ has a canonical t-structure,
which we call the \emph{standard t-structure}:
\begin{lem} \label{lemma:stabilization:t-structure}
    The category $\spectra{X}$ has an accessible t-structure (the \emph{standard (or homotopy) t-structure}),
    given by $\tcocon[-1]{\spectra X} = \set{E \in \spectra{X}}{\Loop E \cong *}$.
    This t-structure is right-separated (i.e.\ $\bigcap_n \tcocon[n]{\spectra{X}} = 0$).
\end{lem}
\begin{proof}
    The existence of the t-structure is \cite[Proposition 1.4.3.4]{higheralgebra}.

    For the other statement, we essentially copy the proof of \cite[Proposition 1.3.2.7 (3)]{sag}.
    Let $F \in \bigcap_n \tcocon[n]{\spectra{X}}$.
    By definition, this says that $\pLoop \Sigma^n F \cong *$ for every $n$.
    Since the functors $\pLoop \Sigma^n$ are jointly conservative and preserve final objects (as they 
    commute with limits), it follows that $F = 0$, i.e.\ the t-structure is right-separated.
\end{proof}

\begin{lem} \label{lemma:stabilization:homotopy-objects}
    In the situation of \cref{lemma:adjoints-on-stabilization},
    assume moreover that $\topos X$ and $\topos Y$ are $\infty$-topoi,
    and that $f^*$ is left exact (i.e.\ $(f^*, f_*)$ is a geometric morphism).
    Let $A \in \tstructheart{\spectra{X}}$ be in the heart of the standard t-structure.
    Then $f^*A \cong f^{*,\heartsuit}A$.
    Similarly, if $E \in \spectra X$, then
    $\pi_n(f^*E) \cong f^* \pi_n(E)$.
\end{lem}
\begin{proof}
    \cite[Remark 1.3.2.8]{sag} shows that $f^*$ is t-exact with respect to the standard t-structures.
\end{proof}

\begin{lem} \label{lemma:stabilization:stab-of-grothendieck-topos}
    Let $(\Cat C, \tau)$ be a site.
    Write $\ShvTop{\tau}{\Cat C, \Cat V}$ for the category of 
    sheaves on $\Cat C$ (in the $\tau$-topology) with values in a presentable $\infty$-category $\Cat V$.
    Then there is an equivalence 
    \begin{equation*}
        \Stab{\ShvTop{\tau}{\Cat C, \An}} \cong \ShvTop{\tau}{\Cat C, \Sp}.
    \end{equation*}
\end{lem}
\begin{proof}
    This is \cite[Remark 1.3.2.2]{sag}, together with \cite[Proposition 1.3.1.7]{sag}.
\end{proof}

\subsection{Nilpotent Objects} \label{appendix:nilpotent}
Let $\topos X$ be a hypercomplete $\infty$-topos.
Recall the following definition:
\begin{defn} \label{def:nilpotent:nilpotent-action}
    Let $G$ and $H$ be group objects in $\Disc{\topos X}$,
    with an action of $G$ on $H$.
    A \emph{$G$-central series} is a finite decreasing filtration $H = H_0 \supset \dots \supset H_n = 1$
    such that $H_i$ is normal and $G$-stable, $H_i / H_{i+1}$ is abelian
    and the induced action of $G$ on $H_i / H_{i+1}$ is trivial.

    The action of $G$ on $H$ is called \emph{nilpotent}
    if there exists a $G$-central series of $H$.

    We say that $G$ is \emph{nilpotent}
    if the action of $G$ on itself via conjugation is nilpotent.
\end{defn}

\begin{lem} \label{lemma:nilpotent:abelian-implies-nilpotent}
    Let $G$ be a group object in $\Disc{\topos X}$.
    If $G$ is abelian then $G$ is nilpotent.
\end{lem}
\begin{proof}
    One can choose the $G$-central series $G \supset 1$,
    since the conjugation action is trivial.
\end{proof}

\begin{defn} \label{def:nilpotent:defn}
    Let $X \in \topos X_*$ be a pointed object.
    We say that $X$ is \emph{nilpotent} if $X$ is connected,
    $\pi_1(X)$ is a nilpotent group object
    and the action of $\pi_1(X)$ on $\pi_n(X)$
    is nilpotent for all $n \ge 2$.
\end{defn}

\begin{lem} \label{lemma:nilpotent:loopspace}
    Let $X \in \topos X_*$ be a pointed object.
    Then $\tau_{\ge 1} \Omega X$ is nilpotent.
\end{lem}
\begin{proof}
    Note that $\tau_{\ge 1} \Omega X \cong \Omega \tau_{\ge 2} X$.
    Since $\tau_{\ge 2} X$ is simply connected, it is in particular nilpotent.
    Now note that $\Omega \tau_{\ge 2} X = \Fib{* \to \tau_{\ge 2} X}$.
    Thus, we conclude by \cite[Proposition 2.2.4]{asok2022localization}
    that $\tau_{\ge 1} \Omega X$ is nilpotent.
\end{proof}

\begin{lem} \label{lemma:nilpotent:fiber}
    Let $f \colon X \to Y$ be a morphism of pointed nilpotent objects in $\topos X_*$.
    Then $\tau_{\ge 1} \Fib{f}$ is nilpotent.
\end{lem}
\begin{proof}
    Following the proof in \cite[Proposition 2.2.4]{asok2022localization},
    we see that $\pi_1(\Fib{f})$ is a nilpotent group,
    with a nilpotent action on $\pi_n(\Fib{f})$ for all $n \ge 2$.
    Thus, $\tau_{\ge 1} \Fib{f}$ is nilpotent.
\end{proof}

\begin{lem} \label{lemma:nilpotent:posttower}
    Let $X \in \topos X_*$ be a pointed object.
    Suppose that $\tau_{\le n} X$ is nilpotent for every $n$.
    Then $X$ is nilpotent.
\end{lem}
\begin{proof}
    Since $\tau_{\le 1} X$ is connected, also $X$ is connected.
    Since the action of $\pi_1(X)$ on $\pi_n(X)$
    is the same as the action of $\pi_1(\tau_{\le n} X)$ on $\pi_n(\tau_{\le n} X)$,
    the lemma follows.
\end{proof}

\begin{defn} \label{def:nilpotent:principal-refinement}
    Let $X \in \topos X_*$ be a connected space.
    Consider the Postnikov tower of $X$
    given by
    \begin{equation*}
        \dots \to \tau_{\le n} X \xrightarrow{p_n} \tau_{\le n-1} X \to \dots \to \tau_{\le 0} X = *.
    \end{equation*}
    We say that the Postnikov tower of $X$ admits a \emph{principal refinement}
    if for each $n \ge 1$ there exists a factorization of $p_n$
    as
    \begin{equation*}
        \tau_{\le n} X = X_{n, m_n} \xrightarrow{p_{n, m_n}} X_{n, m_n-1} \to \dots \to X_{n, 1} \xrightarrow{p_{n,1}} X_{n, 0} = \tau_{\le n-1} X,
    \end{equation*}
    with $m_n \ge 1$, such that each $p_{n, k}$ fits into a fiber sequence
    \begin{equation*}
        X_{n, k} \xrightarrow{p_{n, k}} X_{n, k-1} \to K(A_{n, k}, n + 1)
    \end{equation*}
    with $A_{n, k}$ an abelian group object in $\Disc{\topos X}$.
\end{defn}

\begin{lem} \label{lemma:nilpotent:principal-refinement}
    Let $X \in \topos X_*$ be a pointed object.
    Then $X$ is nilpotent if and only if the Postnikov tower of $X$ admits a principal refinement.
\end{lem}
\begin{proof}
    The proof is analogous to the proof of \cite[Theorem 3.3.13]{asok2022localization},
    applied to the morphism $f \colon X \to *$.
\end{proof}

\subsection{Completions of Anima}

In this section, we collect some results about the $p$-completion of anima.
Essentially everything in this section already appeared in \cite{bousfield2009homotopy}.

\begin{defn}
    Let $f \colon X \to Y$ be a morphism of anima.
    We say that $f$ is an \emph{$\finfld{p}$-equivalence} if
    $f$ induces an isomorphism of homology
    $f_* \colon H_*(X, \finfld{p}) \xrightarrow{\simeq} H_*(Y, \finfld{p})$.
\end{defn}

\begin{lem} \label{lemma:anima:peq-iff-fpeq}
    Let $f \colon X \to Y$ be a morphism of anima.
    Then $f$ is a $p$-equivalence if and only if $f$ is an $\finfld{p}$-equivalence.
\end{lem}
\begin{proof}
    See e.g. \cite[Theorem 2.6]{barthelbousfield}.
    Note that $\pSus f$ is a morphism of connective spectra.
\end{proof}

The following results are from \cite{MCAT}.
We will use without comment that a $p$-equivalence
is the same as an $\finfld{p}$-equivalence, see \cref{lemma:anima:peq-iff-fpeq}.

\begin{lem} \label{lemma:anima:connectivenes-of-completion}
    Let $X$ be an $n$-connective pointed anima for some $n \ge 0$.
    Then $\complete{X}$ is $n$-connective.
\end{lem}
\begin{proof}
    For $n = 0$ the result is vacuous,
    and for $n = 1$ the result directly follows from \cref{lemma:peq-respects-pi0}.
    If $n > 1$ then $X$ is simply connected and thus nilpotent.
    We conclude by using the short exact sequence from \cite[Theorem 11.1.2 (ii)]{MCAT}.
\end{proof}

\begin{lem} \label{lemma:anima:fiber-lemma}
    Let $F \to X \to Y$ be a fiber sequence of pointed anima, with $X$ and $Y$ nilpotent.
    Then $\completebr{\tau_{\ge 1} F} = \tau_{\ge 1} \Fib{\complete X \to \complete Y}$.
\end{lem}
\begin{proof}
    This was proven in \cite[Proposition 11.2.5]{MCAT},
    under the additional assumption that the involved spaces
    have finitely generated homotopy groups.
    The original reference, without the finiteness assumptions,
    is \cite[Lemma 4.8]{bousfield2009homotopy}.
\end{proof}

\begin{lem}\label{lemma:anima:refined-fiber-lemma}
    Suppose that there is a commutative diagram of fiber sequences of pointed anima
    \begin{center}
        \begin{tikzcd}
            F \arrow[r] \arrow[d, "f_F"] &X \arrow[r] \arrow[d, "f_X"] &Y \arrow[d, "f_Y"] \\
            F' \arrow[r] &X' \arrow[r] &Y',
        \end{tikzcd}
    \end{center}
    such that $X, Y, X'$ and $Y'$ are nilpotent
    and $f_X$ and $f_Y$ are $p$-equivalences.
    Then $\tau_{\ge 1} F \xrightarrow{\tau_{\ge 1} f_F} \tau_{\ge 1} F'$
    is a $p$-equivalence.
\end{lem}
\begin{proof}
    By \cref{lemma:anima:fiber-lemma},
    we conclude that $\completebr{\tau_{\ge 1} F} \cong \tau_{\ge 1} \Fib{\complete{X} \to \complete{Y}}$,
    and similarly $\completebr{\tau_{\ge 1} F'} \cong \tau_{\ge 1} \Fib{\complete{X'} \to \complete{Y'}}$.
    Since $f_X$ and $f_Y$ are $p$-equivalences,
    it follows that $\complete{X} \cong \complete{X'}$ and $\complete{Y} \cong \complete{Y'}$.
    Thus, we have $\completebr{\tau_{\ge 1} F} \cong \completebr{\tau_{\ge 1} F'}$,
    i.e.\ $\tau_{\ge 1} f_F$ is a $p$-equivalence.
\end{proof}

\begin{defn}
    For each $i$ write
    \begin{equation*}
        L_i \colon \Ab \xrightarrow{(-)[0]} \Cat{D}(\Z) \xrightarrow{\limil{n} (-) \sslash p^n} \Cat{D}(\Z) \xrightarrow{\pi_i(-)} \Ab.
    \end{equation*}
    We call these functors the \emph{derived $p$-completion functors}
    on the category of abelian groups.
\end{defn}

\begin{lem} \label{lemma:anima:t-struct-description}
    Recall the $p$-adic t-structure from \cref{def:t-struct:defn}, now applied to the category
    of spectra.
    Then
    \begin{enumerate}[label=(\arabic*),ref=(\arabic*),itemsep=0em]
        \item $\tpstructheart{\Sp} \subset \tstructheart{\Sp}$,
        \item if $E$ is a $p$-complete spectrum, then $\pi_n(E) = \pi_n^p(E)$, and
        \item there are canonical isomorphisms $\mathbb L_i \cong L_i$
    \end{enumerate}
\end{lem}
\begin{proof}
    We first prove (1).
    By definition and \cref{lemma:t-struct:char-of-cocon},
    we see that $E \in \tpstructheart{\Sp}$
    if and only if $\pi_i(E)$ is uniquely $p$-divisible for all $i < -1$,
    $\pi_{-1}(E)$ is $p$-divisible,
    $\pi_0(E)$ has bounded $p$-divisibility,
    and $E = \complete{E} = \tau_{\le 0}E$.
    The conditions on the negative homotopy groups imply
    that $\complete E$ is connective:
    Indeed, from \cite[Theorem 2.6]{barthelbousfield},
    we have for every $n \in \Z$ the following short exact sequence:
    \begin{equation*}
        0 \to L_0 \pi_n(E) \to \pi_n (\complete E) \to L_1 \pi_{n-1}(E) \to 0.
    \end{equation*}
    If $\pi_{n-1}(E)$ is uniquely $p$-divisible, it has in particular no $p$-torsion.
    Thus, following \cite[Corollary 10.1.15]{MCAT}
    (using that $\mathbb H_p \cong L_1$, see \cite[Proposition 10.1.17]{MCAT}),
    we see that $L_1 \pi_{n-1}(E) = 0$.
    On the other hand, if $\pi_{n}(E)$ is $p$-divisible, we
    see that $L_0(\pi_n(E)) = 0$ following (the proof of the abelian case of) \cite[Proposition 10.4.7 (iii)]{MCAT}
    (using that $\mathbb E_p \cong L_0$, see \cite[Proposition 10.1.17]{MCAT}).
    Thus, $E = \complete E$ is connective.
    Hence, $E = \pi_0(E)$ is in $\tstructheart{\Sp}$.

    In order to prove (2), suppose now that $E$ is $p$-complete. Let $n \in \Z$ be arbitrary.
    There is a fiber sequence
    \begin{equation*}
        \tau^p_{\ge n} E \to E \to \tau^p_{\le n-1} E.
    \end{equation*}
    From the discussion directly above, we see that $\tau^p_{\ge n} E$
    is in fact $n$-connective.
    On the other hand, it is immediate from \cref{lemma:t-struct:char-of-cocon}
    that $\tau^p_{\le n-1} E$ is actually $(n-1)$-truncated.
    Thus, by the uniqueness of a decomposition in $n$-connective and $(n-1)$-coconnective parts
    in a t-structure, we see that actually
    $\tau^p_{\ge n} E \cong \tau_{\ge n} E$ and $\tau^p_{\le n-1} E \cong \tau_{\le n-1} E$ for all $n \in \Z$.
    This immediately implies that $\pi_n^p(E) \cong \pi_n(E)$ for all $n \in \Z$.

    It remains to show (3). This follows directly from the fact that
    $\mathcal D(\Ab) \cong \operatorname{Mod}_{H\Z} \to \Sp$ is a limit-preserving
    exact and t-exact functor,
    and that $\mathbb L_i A = \pi_n^p(A) \cong \pi_n^p(\complete{A}) \cong \pi_n(\complete{A})$ (using (2), since $\complete{A}$ is $p$-complete).
\end{proof}

\begin{defn} \label{def:anima:L1G}
    Let $G$ be a nilpotent group.
    We define $\mathbb L_i G \coloneqq \pi_{i+1} (\completebr{BG})$.
\end{defn}

\begin{lem} \label{lemma:anima:L1g-eq-L1A}
    Let $A$ be an abelian group,
    and let $G$ be the underlying nilpotent group (i.e.\ we forget that $A$ is abelian).
    Then $\mathbb L_i A \cong \mathbb L_i G$ for all $i \ge 0$.
\end{lem}
\begin{proof}
    This follows for example from \cite[Theorem 10.3.2]{MCAT}.
\end{proof}

\begin{lem} \label{lemma:anima:short-exact-sequence}
    Let $X$ be a nilpotent, pointed anima.
    For every $n \ge 1$ there is a short exact sequence (functorial in $X$)
    \begin{equation*}
        0 \to \mathbb L_0 \pi_n X \to \pi_n \complete{X} \to \mathbb L_1 \pi_{n-1} X \to 0,
    \end{equation*}
    where we use \cref{def:anima:L1G} for $\mathbb L_i \pi_1(X)$.
    Note that this distinction does not matter if $\pi_1(X)$ is abelian,
    see \cref{lemma:anima:L1g-eq-L1A}.
    Note that we use the definition $\mathbb L_1 \pi_{0} X = 0$.
\end{lem}
\begin{proof}
    \cite[Theorem 11.1.2 (ii)]{MCAT} provides a short exact sequence
    \begin{equation*}
        0 \to L_0 \pi_n (X) \to \pi_n (\complete{X}) \to L_1 \pi_{n-1} (X) \to 0.
    \end{equation*}
    The lemma follows from \cref{lemma:anima:t-struct-description},
    and the fact that our definition of $\mathbb L_i G$ is the same as the
    definition of $L_i G$ in \cite[Section 10.4]{MCAT} for nilpotent groups $G$
    (note that they use the notation $\mathbb E_p$ and $\mathbb H_p$ for what we
    call $L_0$ and $L_1$, see \cite[Proposition 10.1.17]{MCAT}).
\end{proof}

\begin{lem} \label{lemma:anima:completion-of-infinite-loop-space}
    Let $E$ be a 1-connective spectrum.
    Then $\pLoop (\complete E) = \completebr{\pLoop E}$.
\end{lem}
\begin{proof}
    Using the short exact sequence from \cref{lemma:anima:short-exact-sequence},
    we conclude that the homotopy groups of $\completebr{\pLoop E}$
    fit into short exact sequences
    \begin{equation*}
        0 \to \mathbb L_0 \pi_n(\pLoop E) \to \pi_n(\completebr{\pLoop E}) \to \mathbb L_1 \pi_{n-1}(\pLoop E) \to 0.
    \end{equation*}
    By \cref{lemma:short-exact-sequence-of-t-structure},
    the homotopy groups of $\complete E$ fit into a short exact sequence
    \begin{equation*}
        0 \to \mathbb L_0  \pi_n(E) \to \pi_n^p(\complete E) \to \mathbb L_1 \pi_{n-1}(E) \to 0.
    \end{equation*}
    Thus, the lemma follows from Whitehead's theorem and the fact that
    $\pi_n(\pLoop E) \cong \pi_n(E)$ and $\pi_n^p(\complete E) \cong \pi_n(\complete E) \cong \pi_n(\pLoop (\complete E))$ (see \cref{lemma:anima:t-struct-description}).
\end{proof}

\begin{lem} \label{lemma:anima:infinite-loop-spaces-preserves-peq}
    Let $E \to F$ be a $p$-equivalence of 1-connective spectra.
    Then $\pLoop E \to \pLoop F$
    is a $p$-equivalence.
\end{lem}
\begin{proof}
    Since $\complete E \cong \complete F$ is an equivalence by assumption,
    we conclude by the last \cref{lemma:anima:completion-of-infinite-loop-space}
    that also $\completebr{\pLoop E} \cong \completebr{\pLoop F}$,
    i.e. that $\pLoop E \to \pLoop F$ is a $p$-equivalence.
\end{proof}

\begin{defn}
    Let $X_k$ be an $\N$-indexed inverse system of pointed connected anima.
    We say that it is a \emph{weak Postnikov tower} of anima
    if $\tau_{\le k} X_{k+1} \cong \tau_{\le k} X_{k}$ for all $k \ge 0$
    (i.e.\ the maps $X_{k+1} \to X_{k}$ are $k$-connective for all $k$).
\end{defn}
We want to prove that the suspension spectrum commutes with the limit of weak Postnikov towers.
For this, we need the following well-known statement:
\begin{lem} \label{lemma:anima:BM-connectivity}
    Let $f \colon X \to Y$ be a morphism of pointed anima.
    Suppose that $f$ is $k$-connective for some $k$.
    Then $\Sus f \colon \Sus X \to \Sus Y$ is $k$-connective.
\end{lem}
\begin{proof}
    Let $F \coloneqq \Fib{f}$ be the fiber.
    By assumption, we have that $F$ is $k$-connective.
    Let $C \coloneqq \Cofib{f}$ be the cofiber.
    By the Blakers-Massy Theorem (see e.g. \cite[Theorem 6.4.1]{tomdieck2008algebraic})
    that $C$ is $(k+1)$-connective.
    Since $\Sus$ preserves colimits (as it is a left adjoint),
    we get a cofiber sequence of spectra $\Sus X \to \Sus Y \to \Sus C$.
    Again by Blakers-Massey (or it's corollary, the Freudenthal Suspension Theorem),
    we conclude that $\Sus C$ is $(k+1)$-connective.
    Thus, since $\Sp$ is stable, we see that there is a fiber sequence 
    $\Omega \Sus C \to \Sus X \to \Sus Y$.
    Note that $\Omega \Sus C$ is $k$-connective.
    This proves that $\Sus f$ is $k$-connective.
\end{proof}

\begin{lem} \label{lemma:anima:suspension-spectrum-of-weak-postnikov-tower}
    Let $X_k$ be a weak Postnikov tower of anima.
    Then
    \begin{equation*}
        \Sus \limil{k} X_k \cong \limil{k} \Sus X_k.
    \end{equation*}
\end{lem}
\begin{proof}
    By assumption, each of the morphisms $X_{k+1} \to X_k$ is $k$-connective.
    By \cref{lemma:anima:BM-connectivity} we see that $\Sus X_{k+1} \to \Sus X_k$ is $k$-connective.

    Since by assumption the homotopy groups of the system $X_k$ stabilize, 
    we see by \cite[Proposition 2.2.9]{MCAT} that $\limil{n} X_n \to X_k$ is $k$-connective for every $k$.
    Therefore, again by \cref{lemma:anima:BM-connectivity} 
    also the morphism $\Sus \limil{n} X_n \to \Sus X_k$ is $k$-connective.
    
    Note that the $k$-connectivity of $\Sus X_{k+1} \to \Sus X_k$ implies that the projection 
    $\limil{n} \Sus X_n \to \Sus X_k$ is $k$-connective for every $k$.

    Thus, we see that $\pi_k(\limil{n} \Sus X_n) \cong \pi_k (\Sus X_{k}) \cong \pi_k (\Sus \limil{n} X_n)$.
    We conclude by Whitehead's theorem.
\end{proof}

The above statement about weak Postnikov towers now allows us to conclude that $p$-equivalences 
of weak Postnikov towers induce $p$-equivalences on the limits of the towers:
\begin{lem} \label{lemma:anima:seq-limit-of-peq-of-anima}
    Suppose there are $\N$-indexed inverse systems of pointed connected anima $X_k$ and $Y_k$,
    and for any $n \ge 0$ there exists a $k_n \ge 0$ such that
    $\pi_n(X_k) \cong \pi_n(X_{k_n})$ and $\pi_n(Y_k) \cong \pi_n(Y_{k_n})$ for all $k \ge k_n$.
    Suppose further that there is a morphism of systems $f_k \colon X_k \to Y_k$
    such that each $f_k$ is a $p$-equivalence.
    Then $f \colon \limil{k} X_k \to \limil{k} Y_k$ is a $p$-equivalence.
\end{lem}
\begin{proof}
    Up to replacing $\N$ by a cofinal subset, we may assume that $k_n = n$ for each $n$.
    Note that we have equivalences $\tau_{\le n-1} X_n \cong \tau_{\le n-1} X_{n-1}$
    by assumption. Thus, the system $X_k$ is a weak Postnikov tower.
    This allows us to conclude from \cref{lemma:anima:suspension-spectrum-of-weak-postnikov-tower},
    that $\Sus \limil{k} X_k \cong \limil{k} \Sus X_k$
    and $\Sus \limil{k} Y_k \cong \limil{k} \Sus Y_k$.
    Thus, $\Sus f \cong \Sus \limil{k} f_k \cong \limil{k} \Sus f_k$.
    We now conclude that $f$ is a $p$-equivalence
    because $\Sus f\sslash p \cong (\limil{k} \Sus f_k)\sslash p \cong \limil{k} ((\Sus f_k)\sslash p)$
    is a limit of equivalences.
\end{proof}

\subsection{Conservativity of the Free Sheaf Functor}
Let $\topos X$ be a 1-topos, i.e.\ the category of sheaves of sets on some site $(\Cat C, \tau)$.
Let $R$ be a ring, this defines a presentable 1-category $\Mod{R,\topos X}$ of $R$-modules internal to $\topos X$, 
together with a conservative forgetful functor $\iota \colon \Mod{R,\topos X} \to \topos X$.
This forgetful functor commutes with limits and filtered colimits, and thus has a left adjoint $R[-] \colon \topos X \to \Mod{R, \topos X}$ 
by presentability.
Note that for $X \in \topos X$, the value $R[X]$ is given by the sheafification of the presheaf of $R$-modules $U \mapsto R[X(U)]$,
where $R[X(U)]$ is the free $R$-module on generators $X(U)$.
This can be seen by comparing right adjoints.
Our goal in this section is to prove that $R[-]$ is conservative.

\begin{defn}
    Let $\Cat C$ be a 1-category, and $f \colon X \to Y$ a morphism in $\Cat C$.
    Then $f$ is called an \emph{extremal monomorphism}
    if $f$ is a monomorphism and for all factorizations $f = i \circ p$ with $p$ an epimorphism,
    we already have that $p$ is an isomorphism.
\end{defn}

The following is (the dual of) a well-known result in category theory: 
\begin{lem} \label{lem:free:adjunction-conservative-extremal}
    Let $L \colon \Cat C \rightleftarrows \Cat D \colon R$ be an adjunction of 1-categories,
    and write $\eta \colon \operatorname{id} \to RL$ for the unit map.
    Suppose moreover that
    $\eta_X \colon X \to RLX$ is an extremal monomorphism for all $X \in \Cat C$.
    Then $L$ is conservative.
\end{lem}
\begin{proof}
    Let $f \colon X \to Y$ be a morphism in $\Cat C$ such that $Lf$ is an isomorphism.
    We have to show that $f$ is an isomorphism.
    By naturality of $\eta$, we get a commutative square 
    \begin{center}
        \begin{tikzcd}
            X \arrow[d, "f"] \arrow[r, "\eta_X"] &RLX \arrow[d, "RLf"] \\
            Y \arrow[r, "\eta_Y"] &RLY.
        \end{tikzcd}
    \end{center}
    Note that the right vertical map is an isomorphism,
    and the horizontal maps are extremal monomorphisms.
    Thus, by the definition of extremal monomorphism,
    it suffices to show that $f$ is an epimorphism.

    So suppose that there is $T \in \Cat C$ and $h_1, h_2 \colon Y \to T$ such that 
    $h_1 f = h_2 f$. We need to show that $h_1 = h_2$.
    By functoriality, we have $RLh_1 \circ RLf = RLh_2 \circ RLf$.
    Since $RLf$ is an isomorphism, we conclude that $RLh_1 = RLh_2$.
    By naturality of $\eta$, we thus get the following equality:
    \begin{equation*}
        \eta_T \circ h_1 = RLh_1 \circ \eta_Y = RLh_2 \circ \eta_Y = \eta_T \circ h_2.
    \end{equation*}
    We conclude that $h_1 = h_2$ because $\eta_T$ is a monomorphism by assumption.
\end{proof}

In order to apply the above, we need the following two lemmas:
\begin{lem} \label{lem:free:balanced-mono-is-extremal}
    Suppose that $\Cat C$ is a balanced category 
    (i.e.\ every morphism $f$ which is both monic and epic is already an isomorphism), 
    and that $f \colon X \to Y$ in $\Cat C$ is a monomorphism.
    Then $f$ is an extremal monomorphism.
\end{lem}
\begin{proof}
    Suppose that we have a factorization $f = i \circ p$ with $p$ an epimorphism.
    We need to show that $p$ is an isomorphism.
    Since $\Cat C$ is balanced, it suffices to show that $p$ is a monomorphism,
    which follows immediately from the assumption that $f$ is a monomorphism.
\end{proof}

\begin{lem} \label{lem:free:unit-is-mono}
    For every $X \in \topos X$, the unit $X \to \iota R[X]$ is a monomorphism.
\end{lem}
\begin{proof}
    Write $F$ for the presheaf (of $R$-modules) $U \mapsto R[X(U)]$,
    such that $R[X]$ is the sheafification of $F$.
    Note that the map $X \to F$ is clearly a monomorphism,
    because on each level it is just the canonical map $X(U) \to R[X(U)]$,
    which maps an element $x \in X(U)$ to the corresponding basis element of $R[X(U)]$.
    Now observe that sheafification preserves monomorphisms: 
    Indeed, monomorphisms 
    $f \colon A \to B$ can be characterized as 
    the existence of pullback squares of the form
    \begin{center}
        \begin{tikzcd}
            A \arrow[r, equal] \arrow[d, equal] &A \arrow[d, "f"] \\
            A \arrow[r, "f"] &B,
        \end{tikzcd}
    \end{center}
    which are preserved because sheafification is left exact.
    But since $X$ is already a sheaf by assumption,
    we conclude that $X \to R[X]$ is a monomorphism.
\end{proof}

This allows us to conclude:
\begin{prop} \label{prop:free:conservativity}
    The functor $R[-] \colon \topos X \to \Mod{R, \topos X}$
    is conservative.
\end{prop}
\begin{proof}
    Since every 1-topos is a balanced category,
    it follows from \cref{lem:free:balanced-mono-is-extremal,lem:free:unit-is-mono}
    that the unit $X \to \iota R[X]$ is an extremal monomorphism for all $X \in \topos X$.
    Thus, we conclude from \cref{lem:free:adjunction-conservative-extremal}
    that $R[-]$ is conservative.
\end{proof}

\section{The Pro-Zariski Topology}
\label{appendix:pro-zar}
Let $k$ be a field and denote by $\smooth{k}$ the category of quasi-compact smooth $k$-schemes.
Let $\ShvTop{\zar}{\smooth{k}}$ be the $\infty$-topos of sheaves on $\smooth{k}$
with respect to the Zariski topology,
i.e.\ covers are given by fpqc covers $\{ U_i \to U \}_i$
such that each $U_i \to U$ can be written as $\sqcup_j U_{i, j} \to U$
such that each $U_{i, j} \to U$ is an open immersion.
In this section, we develop an analog of the pro-étale topology from \cite{bhatt2014proetale},
adapted for the Zariski topology.
We use this pro-Zariski topology to show that $\ShvTop{\zar}{\smooth{k}}$
can be embedded into a topos of the form $\PSig{W}$,
where the category $W$ will be realized by zw-contractible rings,
an analog of w-contractible rings from \cite{bhatt2014proetale}.
We will begin with a general discussion with categories of sheaves on locally weakly contractible sites,
and then specialize this discussion to the pro-Zariski topos.

\subsection{Locally Weakly Contractible \texorpdfstring{$\infty$}{Infinity}-Topoi}
\label{section:weakly-contractible}
The goal of this section is to prove that the topos of hypercomplete sheaves
on a locally weakly contractible site $(\Cat C, \tau)$ (see \cref{def:weakly-contractible:lwc})
is always of the form $\PSig{W}$ for a suitable subcategory $W \subset \Cat C$
of weakly contractible objects.
Since we will deal with hypercomplete and non-hypercomplete sheaves,
if $(\Cat C, \tau)$ is a site, then denote the categories of
sheaves on this site (resp. hypercomplete sheaves on this site)
by $\ShvTopNH{\tau}{\Cat C}$ (resp. $\ShvTopH{\tau}{\Cat C}$).
Moreover, denote the sheafification adjunction by
\begin{equation*}
    L_{nh} \colon \PrShv{\Cat C} \rightleftarrows \ShvTopNH{\tau}{\Cat C} \colon \iota_{nh}
\end{equation*}
and
\begin{equation*}
    L_{h} \colon \PrShv{\Cat C} \rightleftarrows \ShvTopH{\tau}{\Cat C} \colon \iota_{h},
\end{equation*}
respectively. Note that $L_h$ factors over $L_{nh}$,
write
\begin{equation*}
    L_{hyp} \colon \ShvTopNH{\tau}{\Cat C} \rightleftarrows \ShvTopH{\tau}{\Cat C} \colon \iota_{hyp}
\end{equation*}
for the geometric morphism corresponding to this factorization.

\begin{defn} \label{def:weakly-contractible:sigma}
    Let $(\Cat C, \tau)$ be a site which admits finite coproducts.
    We say that the topology $\tau$ is a \emph{$\Sigma$-topology}
    if every finite collection of morphisms $\{ U_i \to U \}_i$
    such that $\sqcup_i U_i \to U$ is an isomorphism
    is a cover in the $\tau$-topology.
\end{defn}

\begin{defn} \label{def:weakly-contractible:wc}
    Let $(\Cat C, \tau)$ be a site.
    We say that an object $w \in \Cat C$ is \emph{weakly contractible}
    if every cover by a single morphism $U \twoheadrightarrow w$
    has a splitting.
\end{defn}

\begin{defn} \label{def:weakly-contractible:lwc}
    Let $(\Cat C, \tau)$ be a site.
    We say that $\Cat C$ is \emph{locally weakly contractible}, if
    there is a subcategory $W \subset \Cat C$ such that
    \begin{enumerate}[label=(\textbf{LWC \arabic*}),ref=(LWC \arabic*),itemsep=0em,left=1em]
        \item \label{def:weakly-contractible:lwc:coproducts-pullbacks} $\Cat C$ has finite coproducts, and 
              finite coproducts distribute over all pullbacks that exist in $\Cat C$, i.e.\ if $(U_i)_i$ is a family of objects in $\Cat C$, $f \colon X \to Y$ a morphism in $\Cat C$,
              and $g_i \colon U_i \to Y$ morphisms, then $(\sqcup_i U_i) \times_Y X \cong \sqcup_i (U_i \times_Y X)$,
        \item \label{def:weakly-contractible:lwc:wc} every object $w \in W$ is weakly contractible (\cref{def:weakly-contractible:wc}),
        \item \label{def:weakly-contractible:lwc:w-coproducts} $W$ is closed under finite coproducts in $\Cat C$,
        \item \label{def:weakly-contractible:lwc:quasicompact} every object $w \in W$ is quasi-compact, i.e.\ every cover of $w$ can be refined by a finite cover,
        \item \label{def:weakly-contractible:lwc:sigma} the topology is a $\Sigma$-topology (\cref{def:weakly-contractible:sigma}),
        \item \label{def:weakly-contractible:lwc:wc-cover} every object $X \in \Cat C$ has a cover $w \twoheadrightarrow X$
              by a weakly contractible object $w \in W$, and
        \item \label{def:weakly-contractible:lwc:extensive} the category $W$ is extensive, see \cref{def:psig:extensive}.
    \end{enumerate}
\end{defn}

Suppose from now on that $(\Cat C, \tau)$ is a locally weakly contractible site.
Since by assumption \ref{def:weakly-contractible:lwc:extensive} the category $W$ 
is extensive, we see that $\PSig{W}$ is an $\infty$-topos, see \cref{lemma:psig:extensive-topos}.
In particular, we have a geometric morphism
\begin{equation*}
    L_\Sigma \colon \PrShv{W} \rightleftarrows \PSig{W} \colon \iota_\Sigma.
\end{equation*}
The fully faithful inclusion $W \to \Cat C$
induces an adjunction of presheaf categories
\begin{equation*}
    j^* \colon \PrShv{W} \rightleftarrows \PrShv{\Cat C} \colon j_*,
\end{equation*}
where $j_*$ is given by restriction, and $j^*$ is given by left Kan extension
(see \cite[Corollary 4.3.2.14]{highertopoi} for the existence of left Kan extensions of presheaves,
and \cite[Corollary 4.3.2.16 and Proposition 4.3.2.17]{highertopoi} for a proof
that the left Kan extension functor exists and is left adjoint to the restriction functor).
Write $\pi_n^{pre}$ for the homotopy objects in a presheaf category, i.e.\ the functor
given by postcomposing with the functor $\pi_n \colon \An \to \operatorname{Set}$.

\begin{lem} \label{lemma:weakly-contractible:discrete-recovery}
    Let $F \in \Disc{\PrShv{\Cat C}}$ be a $0$-truncated presheaf (i.e.\ a presheaf of sets),
    such that $j_* F = \iota_\Sigma L_\Sigma j_* F \in \PrShv{W}$.
    Then the canonical map $L_{nh}j^*j_* F \to L_{nh} F$ is an equivalence,
    and for all $w \in W$ we have an equivalence $(L_{nh} F)(w) \cong F(w)$.
\end{lem}
\begin{proof}
    Since everything is $0$-truncated, this is a statement about sheaves of sets.
    In particular, $L_{nh} G \cong G^{++}$,
    where $(-)^+$ is the plus construction, see e.g.\ \cite[\href{https://stacks.math.columbia.edu/tag/00W4}{Tag 00W1}]{stacks-project}.
    Since the $w \in W$ generate the topos $\Disc{\ShvTopNH{\tau}{\Cat C}}$ (this follows from assumption \ref{def:weakly-contractible:lwc:wc-cover}),
    it suffices to prove that $(L_{nh} j^* j_* F)(w) \to (L_{nh} F)(w)$
    is an equivalence for all $w \in W$.
    Moreover, since $(j^* j_* F)(w) \cong (j_* F)(w) \cong F(w)$,
    it suffices to prove that $(L_{nh} G)(w) = G(w)$
    for every presheaf $G \in \Disc{\PrShv{\Cat C}}$
    with $j_* G \cong \iota_{\Sigma} L_\Sigma j_* G$ and $w \in W$.
    Thus, it suffices to show that $G^+(w) = G(w)$.
    Let $\{ U_i \to w \}_i$ be a cover of $w$.
    We can refine this cover by a cover $\{ w_i \to w \}_i$
    with $w_i \in W$, by assumption \ref{def:weakly-contractible:lwc:wc-cover}.
    We may assume that this cover is finite since $w$ is quasi-compact, see assumption \ref{def:weakly-contractible:lwc:quasicompact}.
    Thus, by the definition of $G^+(w) = \colimil{\Cat U \in J_w^{\operatorname{op}}} H^0(\mathcal U, G)$
    (see the discussion right before \cite[\href{https://stacks.math.columbia.edu/tag/00W4}{Tag 00W4}]{stacks-project}
    for the notation),
    we can run the colimit only over covers $\{w_i \to w\}$ with $w_i \in W$.
    But now since $j_* G \cong \iota_\Sigma L_\Sigma j_* G$,
    we know that $\prod_i G(w_i) \cong G(\sqcup_i w_i)$.
    Thus, since coproducts distribute over pullbacks in $\Cat C$ by assumption \ref{def:weakly-contractible:lwc:coproducts-pullbacks}, we see that the \v{C}ech-nerves 
    of $\{w_i \to w\}_i$ and $\{ \sqcup_i w_i \to w \}$ agree.
    Therefore, we may assume that the cover is in fact a single morphism $\mathcal U = \{ v \to w \}$,
    with $v = \sqcup_i w_i \in W$ because objects in $W$ are stable under coproducts by assumption \ref{def:weakly-contractible:lwc:w-coproducts}.
    This morphism has a split by assumption \ref{def:weakly-contractible:lwc:wc}.
    Hence, the \v{C}ech nerve is homotopy equivalent to (the constant simplicial object) $w$,
    see (the dual version of) \cite[\href{https://stacks.math.columbia.edu/tag/019Z}{Tag 019Z}]{stacks-project}.
    Thus, $H^0(\mathcal U, G) = G(w)$.
    Since this is true for a cofinal family of covers, we conclude $G^+(w) = G(w)$.
\end{proof}

\begin{lem} \label{lemma:weakly-contractible:recovery}
    Let $F \in \PrShv{\Cat C}$ be a presheaf,
    such that $j_* F = \iota_\Sigma L_\Sigma j_* F \in \PrShv{W}$.
    Then the canonical map $L_{h}j^*j_* F \to L_{h} F$ is an equivalence,
    and for all $w \in W$, we have an equivalence $(L_{h} F)(w) \cong F(w)$.
\end{lem}
\begin{proof}
    Write $\epsilon \colon j^* j_* F \to F$ for the counit of the adjunction $j^* \dashv j_*$.
    For the first statement, by hypercompleteness it suffices to show that for each $n$,
    each $U \in \ShvTopH{\tau}{\Cat C}$ and each morphism
    $x \colon U \to L_{h} j^* j_* F$ (i.e.\ each choice of basepoint
    in the overtopos $\ShvTopH{\tau}{\Cat C}_{/U}$)
    the morphism $\pi_n((L_{h} j^* j_* F)|_U, x) \to \pi_n((L_{h} F)|_U, \epsilon \circ x)$
    induced by $\epsilon$ is an equivalence for all $n \ge 0$ (in the case $n = 0$,
    we can do the same calculations as below, but do it without the choice of a basepoint).
    But for every presheaf $G \in \PrShv{\Cat C}$ (and object $U$ and basepoint $x \colon U \to L_h G$),
    we have a chain of equivalences $\pi_n((L_{h} G)|_U, x) \cong \pi_n ((L_{nh} G)|_{\iota_{hyp}U}, \iota_{hyp}(x)) \cong L_{nh} \pi_n^{pre}(G|_{\iota_{h}U}, \iota_h x)$,
    where the first equivalence follows since $L_{h}$ factors over $L_{nh}$,
    and this factorization is the universal functor out of
    $\ShvTopNH{\tau}{\Cat C}$ that inverts $\pi_*$-isomorphisms
    (i.e.\ morphisms $f$ such that $\pi_k(f)$ is an isomorphism for all $k$).
    Thus, it suffices to prove that the canonical morphism
    \begin{equation*}
        L_{nh} \pi_n^{pre} ((j^* j_* F)|_{\iota_{h}U}, \iota_h(x)) \xrightarrow{- \circ \epsilon} L_{nh} \pi_n^{pre}(F|_{\iota_h U}, \epsilon \circ \iota_h(x))
    \end{equation*}
    is an equivalence.
    We know that $\pi_n^{pre} ((j^* j_* F)|_{\iota_{h}U}, \iota_h(x)) \cong j^* j_* \pi_n^{pre}(F|_{\iota_{h}U}, \epsilon \circ \iota_h(x))$,
    since $j^*$ is a geometric morphism and thus commutes with homotopy objects,
    and $j_*$ is just the restriction of functors.
    Thus, the result follows from \cref{lemma:weakly-contractible:discrete-recovery},
    if $j_* \pi_n^{pre}(F|_{\iota_{h}U}, \epsilon \circ \iota_h(x)) \cong \iota_\Sigma L_\Sigma j_* \pi_n^{pre}(F|_{\iota_{h}U}, \epsilon \circ \iota_h(x))$.
    But this is clear since $j_* \pi_n^{pre}(F|_{\iota_{h}U}, \epsilon \circ \iota_h(x)) \cong \pi_n^{pre} (j_* F|_{j_* \iota_h U}, j_* \iota_h(x))$ (again since
    $j_*$ is just the restriction of functors), since $j_* F \cong \iota_\Sigma L_\Sigma j_* F$ by assumption
    and since the homotopy presheaf $\pi_n^{pre}((j_* F)|_{j_* \iota_h U}, j_* \iota_h(x))$
    is the homotopy object of $(j_* F)|_{j_* \iota_h U}$ in $\PSig{W}_{/j_* \iota_h U}$
    with respect to the given basepoint, see \cref{lemma:psig:htpy-sheaf}.

    For the second point, choose again a $U$ and $x$ as above.
    Note that by the above and \cref{lemma:weakly-contractible:discrete-recovery},
    we get
    \begin{align*}
        (\pi_n((L_h F)|_U, x))(w) & \cong (L_{nh} \pi_n^{pre}(F|_{\iota_h U}, \iota_h(x)))(w) \\
                                  & \cong (\pi_n^{pre}(F|_{\iota_h U}, \iota_h(x)))(w)        \\
                                  & = \pi_n(F|_{\iota_h U}(w), \iota_h(x)(w)).
    \end{align*}
    On the other hand, since $j_* \pi_n^{pre}((L_h F)|_U, x) \cong \iota_\Sigma L_\Sigma j_* \pi_n^{pre}((L_h F)|_U, x)$,
    we again conclude by \cref{lemma:weakly-contractible:discrete-recovery}
    that
    \begin{equation*}
        (\pi_n^{pre}((L_h F)|_U, x))(w) \cong (L_{nh} \pi_n^{pre} ((L_h F)|_U, x))(w) = (\pi_n((L_h F)|_U, x))(w).
    \end{equation*}
    Thus, we conclude that for all $n$, $U$ and $x$ we have an isomorphism
    \begin{align*}
        \pi_n(F|_{\iota_h U}(w), \iota_h(x)(w)) & \cong (\pi_n((L_h F)|_U, x))(w)       \\
                                                & \cong (\pi_n^{pre}((L_h F)|_U, x))(w) \\
                                                & = \pi_n((L_h F)|_U(w), x(w)).
    \end{align*}
    By Whitehead's lemma, we conclude that $F(w) \cong L_h F(w)$.
\end{proof}

\begin{lem} \label{lemma:weakly-contractible:factor-over-psig}
    The unit $j_* \iota_{h} \to \iota_\Sigma L_\Sigma j_* \iota_h$ is an equivalence.
    In particular, for every sheaf $F \in \ShvTopH{\tau}{\Cat C}$,
    there is a canonical equivalence $j_* \iota_{h} F \cong \iota_\Sigma L_\Sigma j_* \iota_h F$.
\end{lem}
\begin{proof}
    Fix $F \in \ShvTopH{\tau}{\Cat C}$.
    Since $W$ is extensive by assumption \ref{def:weakly-contractible:lwc:extensive},
    using \cref{lemma:psig:extensive-topos} it suffices to show that
    $j_* \iota_{h} F$ has descent for disjoint union covers in $W$.
    But those covers are in particular in $\tau$ by assumption \ref{def:weakly-contractible:lwc:sigma}.
    Thus, we conclude since $F$ is a $\tau$-sheaf.
\end{proof}

\begin{lem} \label{lemma:weakly-contractible:psig-comparison}
    The adjunction $j^* \colon \PrShv{W} \rightleftarrows \PrShv{\Cat C} \colon j_*$
    induces an adjunction
    \begin{equation*}
        p^* \colon \PSig{W} \rightleftarrows \ShvTopH{\tau}{\Cat C} \colon p_*,
    \end{equation*}
    where the left adjoint is given by $p^* \coloneqq L_h j^* \iota_\Sigma$,
    and the right adjoint is given by $p_* \coloneqq L_\Sigma j_* \iota_h$.
    Moreover, this adjunction is an equivalence.
\end{lem}
\begin{proof}
    We first show that there is an adjunction $p^* \dashv p_*$:
    We construct the unit as the composition 
    \begin{equation*}
        \operatorname{id} \cong L_{\Sigma} \iota_{\Sigma} \to L_{\Sigma} j_* j^* \iota_{\Sigma} \to L_{\Sigma} j_* \iota_{h} L_h j^* \iota_{\Sigma} = p_* p^*.
    \end{equation*}
    Here, the first arrow is the inverse of the counit of the adjunction $L_{\Sigma} \dashv \iota_{\Sigma}$,
    note that it is invertible because $\iota_{\Sigma}$ is fully faithful.
    The next two arrows are the units of the adjunctions $j^* \dashv j_*$ and $L_{h} \dashv \iota_h$.
    The last equality are the definitions of $p^*$ and $p_*$.
    It is now clear that this defines the unit of an adjunction, because it is equivalent to the composition of 
    the units of two adjunctions.
    Thus, we get the required adjunction via \cite[Proposition 5.2.2.8]{highertopoi}.
    We need to show that the counit and unit maps are equivalences.

    So let $F \in \ShvTopH{\tau}{\Cat C}$.
    Then $p^* p_* F = L_h j^* \iota_\Sigma L_\Sigma j_* \iota_h F$.
    Since we know that $j_* \iota_h F \cong \iota_\Sigma L_\Sigma j_* \iota_{h} F$ from \cref{lemma:weakly-contractible:factor-over-psig},
    we conclude that $L_h j^* \iota_\Sigma L_\Sigma j_* \iota_h F \cong L_h j^* j_* \iota_h F \cong L_h \iota_h F \cong F$,
    where we used \cref{lemma:weakly-contractible:recovery} for the middle equivalence.

    On the other hand, let $F \in \PSig{W}$.
    We want to prove that for all $w \in W$,
    we have $(p_* p^* F)(w) \cong F(w)$.
    We compute
    \begin{align*}
        (p_* p^* F)(w) & = (L_\Sigma j_* \iota_h L_h j^* \iota_\Sigma F)(w)              \\
                       & = (\iota_\Sigma L_\Sigma j_* \iota_h L_h j^* \iota_\Sigma F)(w) \\
                       & \cong (j_* \iota_h L_h j^* \iota_\Sigma F)(w)                   \\
                       & = (L_h j^* \iota_\Sigma F)(w)                                   \\
                       & \cong (j^* \iota_\Sigma F)(w)                                   \\
                       & \cong (\iota_\Sigma F)(w)                                       \\
                       & = F(w),
    \end{align*}
    where we use the last conclusion from \cref{lemma:weakly-contractible:recovery} in the
    fifth equivalence.
\end{proof}

In the last part of this section, we want to establish a condition which allows us to conclude 
that an inclusion of sites actually induces a fully faithful geometric morphism of (hypercomplete) 
$\infty$-topoi.

\begin{prop} \label{prop:weakly-contractible:inclusion}
    Let $(\Cat C', \tau') \subseteq (\Cat C, \tau)$
    be a full subcategory such that any $\tau'$-cover $\{U_i \to U\}_i$
    is also a $\tau$-cover.
    Suppose that $\ShvTopH{\tau}{\Cat C}$ and $\ShvTopH{\tau'}{\Cat C'}$ are Postnikov-complete.
    Write
    \begin{align*}
        L_h \colon \PrShv{\Cat C} \rightleftarrows \ShvTopH{\tau}{\Cat C} \colon \iota_h \\
        L_h' \colon \PrShv{\Cat C'} \rightleftarrows \ShvTopH{\tau'}{\Cat C'} \colon \iota'_h
    \end{align*}
    for the sheafification adjunctions.

    Write $k \colon \Cat C' \hookrightarrow \Cat C$ for the inclusion.
    This induces an adjoint pair
    \begin{equation*}
        k^* \colon \PrShv{\Cat C'} \rightleftarrows \PrShv{\Cat C} \colon k_*,
    \end{equation*}
    where $k_*$ is restriction and $k^*$ is left Kan extension.
    Then we have the following:
    \begin{itemize}
        \item These functors then induce an adjoint pair
              \begin{equation*}
                  j^* \colon \ShvTopH{\tau'}{\Cat C'} \rightleftarrows \ShvTopH{\tau}{\Cat C} \colon j_*,
              \end{equation*}
              where $j^*$ is given by $L_h k^* \iota_h'$
              and $j_*$ is given by $L_h' k_* \iota_h$.
        \item This adjoint pair is a geometric morphism of $\infty$-topoi.
        \item We have a natural equivalence $\iota_h' j_* \cong k_* \iota_h$ (i.e.\ the restriction of a $\tau$-hypersheaf to $\Cat C'$ is a $\tau'$-hypersheaf).
    \end{itemize}

    Assume moreover that if $F \in \ShvTopH{\tau'}{\Cat C'}$ is $n$-truncated for some $n$,
    then $\iota_h j^* F \cong k^* \iota_h' F$ (i.e.\ the left Kan extension of an $n$-truncated $\tau'$-hypersheaf is already a $\tau$-hypersheaf).

    Then $j^*$ is fully faithful.
\end{prop}
\begin{proof}
    We first prove that there is an equivalence $\iota_h' j_* \cong k_* \iota_h$.
    This follows immediately from the fact that every $\tau'$-hypercover is in particular
    a $\tau$-hypercover, thus every $\tau$-hypersheaf is automatically a $\tau'$-hypersheaf.

    We now prove that there is an adjunction $j^* \dashv j_*$:
    We construct the unit as the composition 
    \begin{equation*}
        \operatorname{id} \cong L_{h'} \iota_{h'} \to L_{h'} k_* k^* \iota_{h'} \to L_{h'} k_* \iota_{h} L_h k^* \iota_{h'} = j_* j^*.
    \end{equation*}
    Here, the first arrow is the inverse of the counit of the adjunction $L_{h'} \dashv \iota_{h'}$,
    note that it is invertible because $\iota_{h'}$ is fully faithful.
    The next two arrows are the units of the adjunctions $k^* \dashv k_*$ and $L_{h} \dashv \iota_h$.
    The last equality is the definition of $j_*$ and $j^*$.
    It is now clear that this defines the unit of an adjunction, because it is equivalent to the composition of 
    the units of two adjunctions.
    Thus, we get the required adjunction via \cite[Proposition 5.2.2.8]{highertopoi}.

    In particular, we see that $j^*$ is (the left adjoint of) a geometric morphism,
    because it has a right adjoint and preserves finite limits (since $\iota_h'$
    preserves limits as a right adjoint, and $k^*$ and $L_h$ preserve
    finite limits as the left adjoints of geometric morphisms).

    Assume from now on that if $F \in \ShvTopH{\tau'}{\Cat C'}$ is $n$-truncated for some $n$,
    then $\iota_h j^* F \cong k^* \iota_h' F$.
    In order to show that $j^*$ is fully faithful,
    we first show that it is fully faithful on $n$-truncated objects.
    For this, it suffices to show that for every $n$-truncated $F \in \ShvTopH{\tau'}{\Cat C'}$,
    the natural map $F \to j_* j^* F$ is an equivalence.
    But we compute
    \begin{equation*}
        j_* j^* F \cong L_h' k_* \iota_h j^* F \cong L_h' k_* k^* \iota_h' F \cong L_h' \iota_h' F \cong F,
    \end{equation*}
    where we used for the first equivalence the definition of $j_*$,
    in the second equivalence that $\iota_h j^* F \cong k^* \iota_h' F$ since $F$ is $n$-truncated,
    in the third equivalence that $k^*$ is fully faithful,
    and in the last equivalence that $\iota_h'$ is fully faithful.

    We now want to show that $j^*$ is fully faithful.
    Again, it therefore suffices that the canonical morphism $F \to j_* j^* F$
    is an equivalence for all $F \in \ShvTopH{\tau'}{\Cat C'}$.
    We have a chain of equivalences
    \begin{align*}
        j_* j^* F
         & \cong j_* \limil{n} \tau_{\le n} j^* F \\
         & \cong \limil{n} j_* j^* \tau_{\le n} F \\
         & \cong \limil{n} \tau_{\le n} F         \\
         & \cong F.
    \end{align*}
    Here, the first equivalence uses Postnikov-completeness of $\ShvTopH{\tau'}{\Cat C'}$,
    the second equivalence uses that $j_*$ commutes with limits (it is right adjoint to $j^*$)
    and that $j^*$ commutes with truncations (see \cite[Proposition 6.3.1.9]{highertopoi}),
    the third equivalence holds because we have seen that $j^*$ is fully faithful on $n$-truncated objects,
    and the last equivalence uses Postnikov-completeness of $\ShvTopH{\tau}{\Cat C}$.
    This finishes the proposition.
\end{proof}

\subsection{The Pro-Zariski Topos}

Recall the following definition from \cite[\href{https://stacks.math.columbia.edu/tag/0965}{Tag 0965}]{stacks-project}:
\begin{defn}
    Let $f \colon A \to B$ be a ring map.
    We say that
    \begin{enumerate}[label=(\arabic*),ref=(\arabic*),itemsep=0em]
        \item $f$ is a \emph{local isomorphism} if for every prime $\pideal q \subset B$ there exists a $g \in B$, $g\notin \pideal q$
              such that $A \to B_g$ induces an open immersion $\Spec{B_g} \to \Spec{A}$,
        \item $f$ is an \emph{ind-Zariski map} if $f$ is a filtered colimit of local isomorphisms,
        \item $f$ is an \emph{ind-Zariski cover} if $f$ is a faithfully flat ind-Zariski map.
    \end{enumerate}
\end{defn}

\begin{defn} \label{def:pro-zar:zw-contractible}
    A ring $A$ is called \emph{zw-contractible}
    if it satisfies the equivalent conditions from \cite[\href{https://stacks.math.columbia.edu/tag/09AZ}{Tag 09AZ}]{stacks-project},
    i.e.\ if any faithfully flat ind-Zariski map $A \to B$ has a retraction.
\end{defn}

\begin{lem} \label{lemma:pro-zar:existence-contractible-cover}
    Let $A$ be a ring.
    Then there exists an ind-Zariski cover $A \to \overline{A}$
    such that $\overline{A}$ is zw-contractible.
\end{lem}
\begin{proof}
    This is \cite[\href{https://stacks.math.columbia.edu/tag/09B0}{Tag 09B0}]{stacks-project}.
\end{proof}

\begin{defn}
    Let $f \colon X \to Y$ be a morphism of schemes.
    We say that $f$ is a \emph{Zariski localization} if $f$
    is isomorphic to $\amalg_{i \in I} U_i \to Y$ with $I$ a finite set
    and $U_i \to Y$ open immersions.
    We say that $f$ is a \emph{pro-Zariski localization}
    if $f$ is isomorphic to a cofiltered limit $\limil{i} f_i \colon \limil{i} X_i \to Y$
    such that each $f_i$ is a Zariski localization
    (and hence all transition maps $X_i \to X_j$ are also Zariski localizations).
\end{defn}

\begin{defn}
    Write $\prozar{\smooth{k}}$ for the full subcategory of schemes over $k$
    consisting of pro-Zariski schemes over $\smooth{k}$,
    i.e. morphisms $X \to \Spec{k}$ such that $X$ can be written as a cofiltered limit $X = \limil{i} X_i$
    with $X_i \to \Spec{k}$ smooth such that all transition morphisms $X_i \to X_j$ are Zariski localizations.
    Write $\prozaraff{\smooth{k}} \subset \prozar{\smooth{k}} $ for the full subcategory consisting of affine schemes.
\end{defn}

\begin{lem} \label{lemma:pro-zar:existence-coproducts-pullbacks}
    The category $\prozar{\smooth{k}}$
    has finite coproducts and the inclusion into $\sch{k}$ preserves them.

    Similarly, $\prozar{\smooth{k}}$ has pullbacks along pro-Zariski localization
    and the inclusion into $\sch{k}$ preserves those pullbacks.
\end{lem}
\begin{proof}
    For the first part, let $I$ be a finite set,
    and $(X_i)_{i \in I}$ be a family of schemes $X_i \in \prozar{\smooth k}$.
    Write $X_i \cong \limil{j \in J_i} X_{i, j}$ as a cofiltered limit
    with $X_{i, j} \to \Spec{k}$ smooth such that the transition morphisms are Zariski localizations.
    We get
    \begin{equation*}
        \sqcup_{i} X_i \cong \sqcup_i \limil{j \in J_i} X_{i, j} \cong \limil{(j_i)_i \in \prod_i J_i} \sqcup_{i} X_{i, j_i},
    \end{equation*}
    where the second isomorphism exists because cofiltered limits commute with finite colimits
    and a cofinality argument.
    Hence, the coproduct is again in $\prozar{\smooth k}$.

    We now prove the second part.
    So suppose that $X, U$ and $V$ are in $\prozar{\smooth{k}}$,
    and that there are morphisms $f \colon X \to U$ and $g \colon V \to U$ with $g$ a pro-Zariski morphism.
    Since all limits are cofiltered, we can choose a common filtered category $I$ and presentations
    $X = \limil{i} X_i$, $U = \limil{i} U_i$ and $V = \limil{i} V_i$,
    with $X_i$, $U_i$ and $V_i$ in $\smooth{k}$, with Zariski localizations as transition maps,
    and such that $g_i \colon V_i \to U_i$ is a Zariski localization,
    i.e. $g_i$ is of the form $\amalg_{j \in J} V_{i, j} \to U_i$ for some finite set $J$,
    such that $V_{i, j} \to U_i$ is an open immersion.
    Then $X_i \times_{U_i} V_i \in \smooth{k}$:
    Indeed, it suffices to show that $X_i \times_{U_i} V_{i, j}$ is smooth for every $j \in J$,
    but this is just an open subscheme of $X_i$.
    Note that the transition morphisms $X_i \times_{U_i} V_i \to X_j \times_{U_j} V_j$
    are Zariski localizations (as a composition of basechanges of Zariski localizations).
    Thus, $X \times_U V \cong \limil{i} X_i \times_{U_i} V_i$
    is again in $\prozar{\smooth{k}}$.
\end{proof}

\begin{defn}
    Let $\mathcal U \coloneqq \{f_i \colon U_i \to U\}_{i \in I}$ be a family of morphisms in $\prozar{\smooth k}$.
    We say that $\mathcal U$ is a \emph{pro-Zariski cover}
    if and only if $f_i$ is pro-Zariski for all $i$
    and the $f_i$ form an fpqc-cover.
\end{defn}

\begin{rmk}
    Let $\Spec{f} \colon \Spec{B} \to \Spec{A}$ be a morphism of schemes in $\prozaraff{\smooth k}$.
    Then $\{\Spec{f}\}$ is a pro-Zariski cover if and only if $f \colon A \to B$
    is an ind-Zariski cover.
    To see this, it suffices to show that $\Spec{f}$ is a Zariski-localization
    if and only if $f$ is a local isomorphism.
    This follows from \cite[\href{https://stacks.math.columbia.edu/tag/096J}{Tag 096J}]{stacks-project}.
\end{rmk}

\begin{lem} \label{lemma:pro-zar:sites}
    The categories $\prozar{\smooth{k}}$ and $\prozaraff{\smooth{k}}$
    together with the class of pro-Zariski covers form sites in
    the sense of \cite[\href{https://stacks.math.columbia.edu/tag/00VH}{Tag 00VH}]{stacks-project}.
    Moreover, the natural inclusion $\prozaraff{\smooth{k}} \subset \prozar{\smooth{k}}$
    is a morphism of sites in the sense of \cite[\href{https://stacks.math.columbia.edu/tag/00X1}{Tag 00X1}]{stacks-project}.
\end{lem}
\begin{proof}
    For the first statement, the only nontrivial part is the existence of
    pullbacks of covers, which was proven in \cref{lemma:pro-zar:existence-coproducts-pullbacks}.
    The last assertion is clear from \cite[\href{https://stacks.math.columbia.edu/tag/00X6}{Tag 00X6}]{stacks-project},
    since the inclusion commutes with limits (as limits of affine schemes are affine).
\end{proof}

\begin{defn}
    Let $(\prozar{\smooth{k}}, \prozartop)$ and $(\prozaraff{\smooth{k}}, \prozartop)$ be the
    sites from \cref{lemma:pro-zar:sites}.
\end{defn}

\begin{lem} \label{lemma:pro-zar:affine-basis}
    The geometric morphisms
    \begin{equation*}
        \ShvTopNH{\prozartop}{\prozaraff{\smooth{k}}} \rightleftarrows \ShvTopNH{\prozartop}{\prozar{\smooth{k}}}
    \end{equation*}
    and
    \begin{equation*}
        \ShvTopH{\prozartop}{\prozaraff{\smooth{k}}} \rightleftarrows \ShvTopH{\prozartop}{\prozar{\smooth{k}}}
    \end{equation*}
    induced by the morphism of sites are equivalences.
\end{lem}
\begin{proof}
    The first morphism is an equivalence by \cite[Lemma C.3]{hoyois2015quadratic}.
    Thus, it also induces an equivalence after hypercompletion.
\end{proof}

\begin{defn} \label{def:pro-zar:W-cat}
    Let $W \subset \prozaraff{\smooth k}$ be the full subcategory spanned
    by the (spectra of) zw-contractible rings (see \cref{def:pro-zar:zw-contractible}).
\end{defn}

\begin{lem} \label{lemma:pro-zar:w-extensive}
    $W$ is an extensive category and
    $\PSig{W}$ is an $\infty$-topos given by sheaves on $W$ with respect to the disjoint union topology.
\end{lem}
\begin{proof}
    The category of schemes is extensive,
    and $W$ is a full subcategory stable under summands and finite products.
    From this we immediately conclude that $W$ is extensive.
    The last statement is \cref{lemma:psig:extensive-topos}.
\end{proof}

\begin{lem} \label{lemma:pro-zar:lwc}
    The site $(\prozaraff{\smooth{k}}, \prozartop)$ is locally weakly contractible.
\end{lem}
\begin{proof}
    The pro-Zariski topology is a $\Sigma$-topology, since a clopen immersion
    is in particular a pro-Zariski morphism.
    The pro-Zariski topology on $\prozaraff{\smooth{k}}$ is finitary (cf.\ \cite[Definition A.3.1.1]{sag}) by definition,
    so every object is quasi-compact.
    The category $W$ is exactly the subcategory of weakly contractible objects by definition.
    Every element in $\prozaraff{\smooth{k}}$ has a cover by a weakly contractible object,
    this is the content of \cref{lemma:pro-zar:existence-contractible-cover}.
    We have seen that $W$ is extensive, see \cref{lemma:pro-zar:w-extensive}.
    This proves the lemma.
\end{proof}

\begin{thm} \label{thm:pro-zar:topos}
    We have an equivalence of categories
    \begin{equation*}
        \ShvTopH{\prozartop}{\prozar{\smooth{k}}} \cong \PSig{W}.
    \end{equation*}
\end{thm}
\begin{proof}
    There is a chain of equivalences
    \begin{equation*}
        \ShvTopH{\prozartop}{\prozar{\smooth{k}}} \cong \ShvTopH{\prozartop}{\prozaraff{\smooth{k}}} \cong \PSig{W},
    \end{equation*}
    where the equivalences are supplied by \cref{lemma:pro-zar:affine-basis,lemma:weakly-contractible:psig-comparison}.
    Here we used that the affine pro-Zariski site is locally weakly contractible, see \cref{lemma:pro-zar:lwc}.
\end{proof}

We now want to embed the category of Zariski sheaves on $\smooth{k}$
into the category of hypercomplete pro-Zariski sheaves on $\prozar{\smooth{k}}$. 
\begin{thm} \label{thm:pro-zar:embedding}
    There is a geometric morphism
    \begin{equation*}
        \nu^* \colon \ShvTopH{\zar}{\smooth{k}} \rightleftarrows \ShvTopH{\prozartop}{\prozar{\smooth{k}}} \cong \PSig{W} \colon \nu_*,
    \end{equation*}
    where the right adjoint is given by restriction,
    and the left adjoint is fully faithful.

    Moreover, an $n$-truncated sheaf $F \in \ShvTopH{\prozartop}{\prozar{\smooth{k}}}$
    is in the essential image of $\nu^*$ (i.e.\ it is classical in the notation of \cref{def:embedding:classical})
    if and only if for all $U \in \prozar{\smooth{k}}$
    and all presentations of $U$ as cofiltered limit $U \cong \limil{i} U_i$
    (with the $U_i \in \smooth{k}$ such that the transition morphisms $U_i \to U_j$
    are Zariski) the canonical map $\colimil{i} F(U_i) \to F(U)$ is an equivalence.
\end{thm}
\begin{proof}
    We want to apply \cref{prop:weakly-contractible:inclusion}
    with $\Cat C = \prozar{\smooth k}$ with the pro-Zariski topology and $\Cat C' = \smooth{k}$
    with the Zariski topology, where we use the notation from \cref{prop:weakly-contractible:inclusion}.

    We have seen in \cref{lemma:zar:postnikov-complete} that $\ShvTopH{\zar}{\smooth{k}} \cong \ShvTopNH{\zar}{\smooth{k}}$
    is Postnikov-complete.
    Note that $\ShvTopH{\prozartop}{\prozar{\smooth{k}}} \cong \PSig{W}$ by \cref{thm:pro-zar:topos},
    thus this $\infty$-topos is also Postnikov-complete, see \cref{lemma:psig-postnikov-complete}.

    It remains to prove that $\iota_h j^* F \cong k^* \iota_h' F$
    for every $n$-truncated Zariski sheaf $F \in \ShvTopH{\zar}{\smooth{k}}$,
    i.e.\ we have to show that the presheaf $k^* \iota_h' F$
    is already a pro-Zariski hypersheaf.
    But note that $\ShvTopH{\zar}{\smooth{k}}_{\le n} \cong \ShvTopNH{\zar}{\smooth{k}}_{\le n}$
    (since every $\infty$-connective object in $\ShvTopNH{\zar}{\smooth{k}}$ which is also $n$-truncated 
    is automatically $0$),
    so it suffices to proof that $k^* \iota_h' F$ is a pro-Zariski sheaf.
    Note that by definition if $U \in \prozar{\smooth{k}}$
    is a scheme with presentation as a cofiltered limit $U = \limil{i} U_i$ with $U_i \in \smooth{k}$,
    $(k^* \iota_h' F)(U) \cong \colimil{i} F(U_i)$.

    Using \cref{lemma:pro-zar:affine-basis},
    it suffices to show that $k^* \iota_h' F$ has descent for all pro-Zariski covers $\{V_j \to V\}_j$
    with $V_j$ and $V$ in $\prozaraff{\smooth k}$, i.e.\ all schemes are affine.
    First note that $k^* \iota_h' F$ is a Zariski sheaf: If $\Spec{B} = \bigcup_j U_j$ is a finite
    union of affine open subschemes, and $B$ is a filtered colimit of smooth algebras $B_i$ (where 
    the transition maps are Zariski),
    then this union is pulled back from some $B_i$ (since open immersions are of finite presentation).
    But $F$ is a Zariski sheaf on $\smooth{k}$ by assumption.
    Now let $\{V_j \to V\}_j$ be some pro-Zariski cover.
    Note that $\{V_j \to \sqcup_k V_k\}$ is a Zariski cover.
    Thus, since $k^* \iota_h'F$ satisfies Zariski descent, we can reduce to the case that the cover 
    is of the form $\{ \Spec{f} \}$ for a single ind-Zariski cover $f \colon B \to C$.
    Write $C = \colimil{i} C_i$ as a filtered colimit of Zariski covers $B \to C_i$.
    Again, since $k^* \iota_h' F$ satisfies Zariski descent, we have descent for these covers.
    Thus, the claim follows by taking filtered colimits (note that filtered colimits commute with finite limits,
    and since $k^* \iota_h' F$ is $n$-truncated, the sheaf axiom is actually a finite limit).
    This proves the theorem.
\end{proof}

\begin{cor} \label{cor:pro-zar:essential-image-heart}
    Let $A \in \tstructheart{\ShvTop{\prozartop}{\prozar{\smooth{k}}, \Sp}}$.
    Then $A$ is in the essential image of $\nu^*$ 
    if and only if for all $U \in \prozar{\smooth{k}}$
    and all presentations of $U$ as cofiltered limit $U \cong \limil{i} U_i$
    (with the $U_i \in \smooth{k}$ such that the transition morphisms $U_i \to U_j$
    are Zariski) the canonical map $\colimil{i} \heart{\Gamma}(U_i, A) \to \heart{\Gamma}(U, A)$ is an equivalence.
\end{cor}
\begin{proof}
    Recall that the equivalence of abelian categories 
    \begin{equation*}
        \tstructheart{\ShvTop{\prozartop}{\prozar{\smooth{k}}, \Sp}} \xrightarrow{\cong} \AbObj{\Disc{\ShvTop{\prozartop}{\prozar{\smooth{k}}}}}
    \end{equation*}
    is given by $A \mapsto \heart{\Gamma}(-, A)$.
    Note that the sheaf $\heart{\Gamma}(-, A)$ is $0$-truncated.
    Thus, the result follows immediately from \cref{thm:pro-zar:embedding}.
\end{proof}

\newpage
\bibliographystyle{alpha}
\bibliography{bibliography}

\end{document}

%% file: preamble.tex
\usepackage[utf8]{inputenc}
\usepackage[english]{babel}

\usepackage{amsmath, amsthm, amssymb, amsfonts}
\usepackage{mathtools}

\usepackage{tikz-cd}

\usepackage{hyperref}
\usepackage[capitalize, english, noabbrev]{cleveref}

\crefname{lem}{Lemma}{Lemmas}
\crefname{thm}{Theorem}{Theorems}
\crefname{cor}{Corollary}{Corollaries}
\crefname{prop}{Proposition}{Propositions}

\usepackage{enumitem}

\newcommand{\sslash}{/\mkern-6mu/}

\theoremstyle{definition}
\newtheorem{defn}{Definition}[section]

\theoremstyle{plain}
\newtheorem{lem}[defn]{Lemma}

\theoremstyle{plain}
\newtheorem{prop}[defn]{Proposition}

\theoremstyle{plain}
\newtheorem{cor}[defn]{Corollary}

\theoremstyle{plain}
\newtheorem{thm}[defn]{Theorem}

\theoremstyle{plain}

\theoremstyle{remark}
\newtheorem{rmk}[defn]{Remark}

\theoremstyle{remark}

\theoremstyle{plain}
\newtheorem{conj}[defn]{Conjecture}

\newcommand{\Z}{{\mathbb{Z}}}
\newcommand{\N}{{\mathbb{N}}}
\newcommand{\Q}{{\mathbb{Q}}}

\newcommand{\A}{{\mathbb{A}}}

\newcommand{\Spec}[1]{\operatorname{Spec}(#1)}
\newcommand{\AffSpc}[1]{{\A^{#1}}}

\newcommand{\PrShv}[1]{{\mathcal{P}(#1)}}
\newcommand{\PrShvVal}[2]{{\mathcal{P}(#1, #2)}}
\newcommand{\PSig}[1]{{\mathcal{P}_{\Sigma}(#1)}}
\newcommand{\PSigVal}[2]{{\mathcal{P}_{\Sigma}(#1, #2)}}

\newcommand{\coker}{\operatorname{coker}}
\newcommand{\colim}[1]{{\underset{\substack{#1}}{\operatorname{colim}}\,}}
\newcommand{\colimil}[1]{{{\operatorname{colim}}_{#1}\,}}
\renewcommand{\lim}[1]{{\underset{\substack{#1}}{\operatorname{lim}}\,}}
\newcommand{\limil}[1]{{{\operatorname{lim}}_{#1}\,}}
\newcommand{\limilheart}[1]{{{\operatorname{lim}}^\heartsuit_{#1}\,}}
\newcommand{\limone}[1]{{\underset{\substack{#1}}{\operatorname{lim}}^{1}\,}}

\newcommand{\Map}[1]{{\operatorname{Map}_{#1}}}
\newcommand{\Fun}[1]{{\operatorname{Fun}_{#1}}}

\newcommand{\Sp}{{\operatorname{Sp}}}

\newcommand{\op}[1]{{#1}^{\operatorname{op}}}

\newcommand{\Sus}{{\Sigma^\infty}}
\newcommand{\pSus}{{\Sigma^\infty_+}}
\newcommand{\Loop}{{\Omega^\infty}}
\newcommand{\pLoop}{{\Omega_*^\infty}}

\newcommand{\complete}[1]{{{#1}{}^{\wedge}_{p}}}

\newcommand{\completebr}[1]{{{\left(#1\right)}^{\wedge}_{p}}}

\newcommand{\Fib}[1]{{\operatorname{fib}\mathopen{}\left(#1\right)\mathclose{}}}
\newcommand{\Cofib}[1]{{\operatorname{cofib}\mathopen{}\left(#1\right)\mathclose{}}}

\newcommand{\topos}[1]{{\mathcal {#1}}}
\newcommand{\spectra}[1]{{\operatorname{Sp}(\topos {#1}})}

\newcommand{\Stab}[1]{{\operatorname{Sp}(#1)}}

\newcommand{\An}{{\mathcal An}}

\newcommand{\Ab}{{\mathcal{A}b}}
\newcommand{\AbObj}[1]{{\mathcal{A}b(#1)}}

\newcommand{\Grp}[1]{{\mathcal{G}rp(#1)}}
\newcommand{\GrpStr}[1]{{\mathcal{G}rp_{\operatorname{str}}(#1)}}
\newcommand{\Disc}[1]{{\operatorname{Disc}(#1)}}

\newcommand{\Hom}[1]{{\operatorname{Hom}_{#1}}}
\newcommand{\Cat}[1]{\mathcal {#1}}

\newcommand{\tcon}[2][0]{{#2}_{\ge #1}}
\newcommand{\tcocon}[2][0]{{#2}_{\le #1}}
\newcommand{\tstruct}[1]{(\tcon{#1}, \tcocon{#1})}
\newcommand{\tstructheart}[1]{{#1}^\heartsuit}
\newcommand{\heart}[1]{#1^\heartsuit}
\newcommand{\pheart}[1]{#1^{p\heartsuit}}

\newcommand{\tpcon}[2][0]{{#2}^{p}_{\ge #1}}
\newcommand{\tpcocon}[2][0]{{#2}^{p}_{\le #1}}
\newcommand{\tpstruct}[1]{(\tpcon{#1}, \tpcocon{#1})}
\newcommand{\tpstructheart}[1]{{#1}^{p\heartsuit}}

\newcommand{\finfld}[1]{\mathbb{F}_{#1}}

\newcommand{\set}[2]{\left\{\, {#1} \,\middle\vert\, {#2} \,\right\}}

\newcommand{\Shv}[1]{\operatorname{Shv}({#1})}
\newcommand{\ShvTop}[2]{\operatorname{Shv}_{#1}({#2})}
\newcommand{\ShvTopH}[2]{\operatorname{Shv}^{\operatorname{h}}_{#1}({#2})}
\newcommand{\ShvTopNH}[2]{\operatorname{Shv}^{\operatorname{nh}}_{#1}({#2})}

\newcommand{\ShvNis}[1]{\operatorname{Shv}_{\operatorname{nis}}(#1)}

\newcommand{\MotSpc}[1]{\operatorname{Spc}(#1)}
\newcommand{\SH}[1]{\operatorname{SH}^{S^1}\mkern-6mu(#1)}
\newcommand{\MotShv}[1]{\operatorname{Shv}_{\operatorname{nis}}(\operatorname{Sm}_{#1})}
\newcommand{\MotShvSp}[1]{\operatorname{Shv}_{\operatorname{nis}}(\operatorname{Sm}_{#1}, \Sp)}
\newcommand{\ZarShv}[1]{\operatorname{Shv}_{\operatorname{zar}}(\operatorname{Sm}_{#1})}
\newcommand{\ZarShvSp}[1]{\operatorname{Shv}_{\operatorname{zar}}(\operatorname{Sm}_{#1}, \Sp)}

\newcommand{\pideal}[1]{{\mathfrak #1}}

\newcommand{\sch}[1]{\operatorname{Sch}_{#1}} 
\newcommand{\smooth}[1]{\operatorname{Sm}_{#1}} 
\newcommand{\prozar}[1]{\operatorname{ProZar}({#1})} 
\newcommand{\prozaraff}[1]{\operatorname{ProZarAff}({#1})} 

\newcommand{\nis}{\operatorname{nis}}
\newcommand{\zar}{\operatorname{zar}}
\newcommand{\prozartop}{\operatorname{prozar}} 

\newcommand{\nupu}{\nu^{*,p\heartsuit}}
\newcommand{\nupl}{\nu_*^{p\heartsuit}}

\newcommand{\Mod}[1]{\operatorname{Mod}_{#1}}

\newcommand{\coAlg}[1]{{\operatorname{CoAlg}(#1)}}
\newcommand{\Alg}[1]{{\operatorname{CAlg}(#1)}}

\newcommand{\htpydim}{\operatorname{htpydim}}

\author{Klaus Mattis\footnote{\href{mailto:klaus.mattis@uni-mainz.de}{klaus.mattis@uni-mainz.de}}}
\date{\today}